\documentclass{amsart}
\usepackage{amssymb}
\usepackage{amsfonts}

\theoremstyle{plain}

\newtheorem{algorithm}{Algorithm}[section]

\newtheorem{theorem}[algorithm] {Theorem}

\newtheorem{corollary}[algorithm]{Corollary}

\newtheorem{lemma}[algorithm]{Lemma}

\newtheorem{proposition}[algorithm]{Proposition}
\newtheorem{remark}[algorithm]{Remark}

\input{tcilatex}

\begin{document}
\title[An Exotic Sphere with Positive Curvature]{An Exotic Sphere with
Positive Sectional Curvature}
\dedicatory{In memory of Detlef Gromoll}
\author{Peter Petersen}
\address{Department of Mathematics\\
UCLA}
\email{petersen@math.ucla.edu}
\author{Frederick Wilhelm}
\address{Department of Mathematics\\
UCR}
\email{fred@math.ucr.edu}
\date{May 6, 2008}
\maketitle

During the 1950s, a famous theorem in geometry and some perplexing examples
in topology were discovered that turned out to have unexpected connections.
In geometry, the development was the Quarter Pinched Sphere Theorem. (\cite%
{Berg1}, \cite{Kling}, and \cite{Rau})

\medskip

\noindent \textbf{Theorem }\emph{(Rauch-Berger-Klingenberg, 1952-1961)}%
\textrm{\ }\emph{If a simply connected, complete manifold has sectional
curvature between }$1/4$\emph{\ and }$1$\emph{, i.e.,}%
\begin{equation*}
1/4<\sec \leq 1,
\end{equation*}%
\emph{then the manifold is homeomorphic to a sphere.}

\medskip

The topological examples were \cite{Miln}

\medskip

\noindent \textbf{Theorem }\emph{(Milnor, 1956) There are }$7$\emph{%
-manifolds that are homeomorphic to, but not diffeomorphic to, the }$7$\emph{%
-sphere.}

\medskip

The latter result raised the question as to whether or not the conclusion in
the former is optimal. After a long history of partial solutions, this
problem has been finally solved.

\medskip

\noindent \textbf{Theorem }\emph{(Brendle-Schoen, 2007) Let }$M$\emph{\ be a
complete, Riemannian manifold and }$f:M\longrightarrow \left( 0,\infty
\right) $\emph{\ a }$C^{\infty }$\emph{--function so that at each point }$x$ 
\emph{of }$M$\emph{\ the sectional curvature satisfies}%
\begin{equation*}
\frac{f\left( x\right) }{4}<\sec _{x}\leq f\left( x\right) .
\end{equation*}%
\emph{Then }$M$\emph{\ is diffeomorphic to a spherical space form.}

\medskip

Prior to this major breakthrough, there were many partial results. Starting
with Gromoll and Shikata (\cite{Grom} and \cite{Shik}) and more recently
Suyama (\cite{Suy}) it was shown that if one allows for a stronger pinching
hypothesis $\delta \leq \sec \leq 1$ for some $\delta $ close to $1,$ then,
in the simply connected case, the manifold is diffeomorphic to a sphere. In
the opposite direction, Weiss showed that not all exotic spheres admit
quarter pinched metrics \cite{Weis}.

Unfortunately, this body of technically difficult geometry and topology
might have been about a vacuous subject. Until now there has not been a
single example of an exotic sphere with positive sectional curvature.

To some extent this problem was alleviated in 1974 by Gromoll and Meyer \cite%
{GromMey}.

\medskip

\noindent \textbf{Theorem }\emph{(Gromoll-Meyer, 1974) There is an exotic }$%
7 $\emph{--sphere with nonnegative sectional curvature and positive
sectional curvature at a point.\medskip }

A metric with this type of curvature is called \emph{quasi-positively curved}%
, and positive curvature almost everywhere is referred to as \emph{almost
positive curvature}. In 1970 Aubin showed the following. (See \cite{Aub} and
also \cite{Ehrl} for a similar result for scalar curvature.)

\medskip

\noindent \textbf{Theorem }\emph{(Aubin, 1970) Any complete metric with
quasi-positive Ricci curvature can be perturbed to one with positive Ricci
curvature.}

\medskip

Coupled with the Gromoll-Meyer example, this raised the question of whether
one could obtain a positively curved exotic sphere via a perturbation
argument. Some partial justification for this came with Hamilton's Ricci
flow and his observation that a metric with quasi-positive curvature
operator can be perturbed to one with positive curvature operator (see \cite%
{Ham}).

This did not change the situation for sectional curvature. For a long time,
it was not clear whether the appropriate context for this problem was the
Gromoll-Meyer sphere itself or more generally an arbitrary quasi-positively
curved manifold. The mystery was due to an appalling lack of examples. For a
25--year period the Gromoll-Meyer sphere and the flag type example in \cite%
{Esch1} were the only known examples with quasi-positive curvature that were
not known to also admit positive curvature.

This changed around the year 2000 with the body of work \cite{PetWilh}, \cite%
{Tapp1}, \cite{Wilh2}, and \cite{Wilk} that gave us many examples of almost
positive curvature. In particular, \cite{Wilk} gives examples with almost
positive sectional curvature that do not admit positive sectional curvature,
the most dramatic being a metric on $\mathbb{R}P^{3}\times \mathbb{R}P^{2}.$
We also learned in \cite{Wilh2} that the Gromoll-Meyer sphere admits almost
positive sectional curvature. (See \cite{EschKer} for a more recent and much
shorter proof.) Here we show that this space actually admits positive
curvature.

\medskip

\noindent \textbf{Theorem }\emph{The Gromoll-Meyer exotic sphere admits
positive sectional curvature.}

\medskip

On the other hand, we know from the theorem of Brendle and Schoen that the
Gromoll-Meyer sphere cannot carry pointwise, $\frac{1}{4}$--pinched,
positive curvature. In addition, we know from \cite{Weis} that it cannot
carry%
\begin{equation*}
\mathrm{sec\ }\geq 1\text{ and radius }>\frac{\pi }{2}
\end{equation*}%
and from \cite{GrovWilh} that it also can not admit 
\begin{equation*}
\mathrm{sec\ }\geq 1\text{ and four points at pairwise distance}>\frac{\pi }{%
2}.
\end{equation*}%
We still do not know whether any exotic sphere can admit 
\begin{equation*}
\mathrm{sec\ }\geq 1\text{ and diameter }>\frac{\pi }{2}.
\end{equation*}%
The Diameter Sphere Theorem says that such manifolds are topological spheres
(\cite{Berg3}, \cite{GrovShio}). We also do not know the diffeomorphism
classification of \textquotedblleft almost $\frac{1}{4}$--pinched%
\textquotedblright , positively curved manifolds. According to \cite{AbrMey}
and \cite{Berg4} such spaces are either diffeomorphic to CROSSes or
topological spheres.

The class with $\mathrm{sec\ }\geq 1$ and diameter $>\frac{\pi }{2}$
includes the globally $\frac{1}{4}$--pinched, simply connected, class,
apparently as a tiny subset. Indeed, globally $\frac{1}{4}$--pinched spheres
have uniform lower injectivity radius bounds, whereas manifolds with $%
\mathrm{sec\ }\geq 1$ and diameter $>\frac{\pi }{2}$ can be Gromov-Hausdorff
close to intervals.

In contrast to the situation for sectional curvature, quite a bit is known
about manifolds with positive scalar curvature, Ricci curvature, and
curvature operator. Starting with the work of Hitchin, it became clear that
not all exotic spheres can admit positive scalar curvature. In fact, the
class of simply connected manifolds that admit positive scalar curvature is
pretty well understood, thanks to work of Lichnerowicz, Hitchin, Schoen-Yau,
Gromov-Lawson and most recently Stolz \cite{Stol}. Since it is usually hard
to understand metrics without any symmetries, it is also interesting to note
that Lawson-Yau have shown that any manifold admitting a nontrivial $S^{3}$
action carries a metric of positive scalar curvature. In particular, exotic
spheres that admit nontrivial $S^{3}$ actions carry metrics of positive
scalar curvature. Poor and Wraith have also found a lot of exotic spheres
that admit positive Ricci curvature (\cite{Poor} and \cite{Wrai}). By
contrast B\"{o}hm-Wilking in \cite{BohmWilk} showed that manifolds with
positive curvature operator all admit metrics with constant curvature and
hence no exotic spheres occur. This result is also a key ingredient in the
differentiable sphere theorem by Brendle-Schoen mentioned above.

We construct our example as a deformation of a metric with nonnegative
sectional curvature, so it is interesting to ponder the possible difference
between the classes of manifolds with positive curvature and those with
merely nonnegative curvature. For the three tensorial curvatures, much is
known. For sectional curvature, the grim fact remains that there are no
known differences between nonnegative and positive curvature for simply
connected manifolds. Probably the most promising conjectured obstruction for
passing from nonnegative to positive curvature is admitting a free torus
action. Thus Lie groups of higher rank, starting with $S^{3}\times S^{3}$,
might be the simplest nonnegatively curved spaces that do not carry metrics
with positive curvature. The Hopf conjecture about the Euler characteristic
being positive for even dimensional positively curved manifolds is another
possible obstruction to $S^{3}\times S^{3}$ having positive sectional
curvature. The other Hopf problem about whether or not $S^{2}\times S^{2}$
admits positive sectional curvature is probably much more subtle.

Although our argument is very long, we will quickly establish that there is
a good chance to have positive curvature on the Gromoll-Meyer sphere, $%
\Sigma ^{7}$. Indeed, in the first section, we start with the metric from 
\cite{Wilh2} and show that by scaling the fibers of the submersion $\Sigma
^{7}\longrightarrow S^{4},$ we get integrally positive curvature over the
sections that have zero curvature in \cite{Wilh2}. More precisely, the zero
locus in \cite{Wilh2} consists of a (large) family of totally geodesic $2$%
--dimensional tori. We will show that after scaling the fibers of $\Sigma
^{7}\longrightarrow S^{4},$ the integral of the curvature over any of these
tori becomes positive. The computation is fairly abstract, and the argument
is made in these abstract terms, so no knowledge of the metric of \cite%
{Wilh2} is required.

The difficulties of obtaining positive curvature after the perturbation of
section 1 cannot be over stated. After scaling the fibers, the curvature is
no longer nonnegative, and although the integral is positive, this
positivity is to a higher order than the size of the perturbation. This
higher order positivity is the best that we can hope for. Due to the
presence of totally geodesic tori, there can be no perturbation of the
metric that is positive to first order on sectional curvature \cite{Stra}.
The technical significance of this can be observed by assuming that one has
a $C^{\infty }$ family of metrics $\left\{ g_{t}\right\} _{t\in \mathbb{R}}$
with $g_{0}$ a metric of nonnegative curvature. If, in addition, 
\begin{equation*}
\left. \frac{\partial }{\partial t}\mathrm{sec}_{g_{t}}\,P\right\vert
_{t=0}>0
\end{equation*}%
for all planes $P$ so that $\mathrm{sec}_{g_{0}}\,P=0,$ then $g_{t}$ has
positive curvature for all sufficiently small $t>0.$ Since no such
perturbation of the metric in \cite{Wilh2} is possible, it will not be
enough for us to consider the effect of our deformation on the set, $Z,$ of
zero planes of the metric in \cite{Wilh2}. Instead we will have to check
that the curvature becomes positive in an entire neighborhood of $Z.$ This
will involve understanding the change of the full curvature tensor.

According to recent work of Tapp, any zero plane in a Riemannian submersion
of a biinvariant metric on a compact Lie group exponentiates to a flat. Thus
any attempt at perturbing any of the known quasipositively curved examples
to positive curvature would have to tackle this issue \cite{Tapp2}.

In contrast to the metric of \cite{EschKer}, the metric in \cite{Wilh2} does
not come from a left (or right) invariant metric on $Sp\left( 2\right) .$ So
although the Gromoll--Meyer sphere is a quotient of the Lie group $Sp\left(
2\right) ,$ we do not use Lie theory for any of our curvature computations
or even for the definition of our metric. Our choice here is perhaps a
matter of taste. The overriding idea is that although none of the metrics
considered lift to left invariant ones on $Sp\left( 2\right) ,$ there is
still a lot of structure. Our goal is to exploit this structure to simplify
the exposition as much as we can.

Our substitute for Lie theory is the pull-back construction of \cite{Wilh1}.
In fact, the current paper is a continuation of \cite{PetWilh}, \cite{Wilh1}%
, and \cite{Wilh2}. The reader who wants a thorough understanding of our
argument will ultimately want to read these earlier papers. We have,
nevertheless, endeavored to make this paper as self-contained as possible by
reviewing the basic definitions, notations, and results of \cite{PetWilh}, 
\cite{Wilh1}, and \cite{Wilh2} in sections 2, 3, and 4. It should be
possible to skip the earlier papers on a first read, recognizing that
although most of the relevant results have been restated, the proofs and
computations are not reviewed here. On the other hand, Riemannian
submersions play a central role throughout the paper; so the reader will
need a working knowledge of \cite{On}.

After establishing the existence of integrally positive curvature and
reviewing the required background, we give a detailed and technical summary
of the remainder of the argument in section 5. Unfortunately, aspects of the
specific geometry of the Gromoll-Meyer sphere are scattered throughout the
paper, starting with section 2; so it was not possible to write section 5 in
a way that was completely independent of the review sections. Instead we
offer the following less detailed summary with the hope that it will suffice
for the moment.

Starting from the Gromoll-Meyer metric the deformations to get positive
curvature are

\begin{description}
\item[(1)] The $\left( h_{1}\oplus h_{2}\right) $--Cheeger deformation,
described in section 3

\item[(2)] The redistribution, described in section 6.

\item[(3)] The $\left( U\oplus D\right) $--Cheeger deformation, described in
section 3

\item[(4)] The scaling of the fibers, described in section 1

\item[(5)] The partial conformal change, described in section 10

\item[(6)] The $\Delta \left( U,D\right) $ Cheeger deformation and a further 
$h_{1}$--deformation.
\end{description}

We let $g_{1},$ $g_{1,2},$ $g_{1,2,3},$ ect. be the metrics obtained after
doing deformations (1), (1) and (2), or (1), (2), and (3) respectively.

It also makes sense to talk about metrics like $g_{1,3},$ i.e. the metric
obtained from doing just deformations (1) and (3) without deformation (2).

All of the deformations occur on $Sp\left( 2\right) .$ So at each stage we
verify invariance of the metric under the various group actions that we
need. For the purpose of this discussion we let $g_{1},$ $g_{1,2},$ $%
g_{1,2,3},$ ect. stand for the indicated metric on both $Sp\left( 2\right) $
and $\Sigma ^{7}.$

$g_{1,3}$ is the metric of [Wilh2] that has almost positive curvature on $%
\Sigma ^{7}$. $g_{1,2,3}$ is also almost positively curvature on $\Sigma
^{7} $, and has precisely the same zero planes as $g_{1,3}.$ Some specific
positive curvatures of $g_{1,3}$ are redistributed in $g_{1,2,3}$. The
reasons for this are technical, but as far as we can tell without
deformation (2) our methods will not produce positive curvature. It does not
seem likely that either $g_{1,2}$ or $g_{1,2,3}$ are nonnegatively curved on 
$Sp\left( 2\right) ,$ but we have not verified this.

Deformation (4), scaling the fibers of $Sp\left( 2\right) \longrightarrow
S^{4},$ is the raison d'\^{e}tre of this paper. $g_{1,2,3,4}$ has some
negative curvatures, but has the redeeming feature that the integral of the
curvatures of the zero planes of $g_{1,3}$ is positive. In fact this
integral is positive over any of the flat tori of $g_{1,3}.$

The role of deformation (5) is to even out the positive integral. The
curvatures of the flat tori of $g_{1,3}$ are pointwise positive with respect
to $g_{1,2,3,4,5}.$

To understand the role of deformation ($6),$ recall that we have to check
that we have positive curvature not only on the $0$--planes of $g_{1,3},$
but in an entire neighborhood (of uniform size) of the zero planes of $%
g_{1,3}.$ To do this suppose that our zero planes have the form 
\begin{equation*}
P=\mathrm{span}\left\{ \zeta ,W\right\} .
\end{equation*}%
We have to understand what happens when the plane is perturbed by moving its
foot point, and also what happens when the plane moves within the fibers of
the Grassmannian.

To deal with the foot points, we extend $\zeta $ and $W$ to families of
vectors $\mathcal{F}_{\zeta }$ and $\mathcal{F}_{W}$ on $Sp\left( 2\right) .$
These families can be multivalued and $\mathcal{F}_{W}$ contains some
vectors that are not horizontal for the Gromoll-Meyer submersion. All pairs $%
\left\{ \zeta ,W\right\} $ that contain zero planes of $\left( \Sigma
^{7},g_{1,3}\right) $ are contained in these families, and the families are
defined in a fixed neighborhood of the $0$--locus of $g_{1,3}.$ All of our
arguments are valid for all pairs $\left\{ z,V\right\} $ with $z\in \mathcal{%
F}_{\zeta }$ and $V\in \mathcal{F}_{W}$, provided $z$ and $V$ have the same
foot point. In this manner, we can focus our attention on fiberwise
deformations of the zero planes.

To do this we consider planes of the form 
\begin{equation*}
P=\mathrm{span}\left\{ \zeta +\sigma z,W+\tau V\right\}
\end{equation*}%
where $\sigma ,\tau $ are real numbers and $z$ and $V$ are tangent vectors.
Ultimately we show that all values of all curvature polynomials%
\begin{equation*}
P\left( \sigma ,\tau \right) =\mathrm{curv}\left( \zeta +\sigma z,W+\tau
V\right)
\end{equation*}%
are positive.

Allowing $\sigma ,\tau $, $z$ and $V$ to range through all possible values
describes an open dense subset in the Grassmannian fiber. The complement of
this open dense set consists of planes that have either no $z$ component or
no $W$ component. These curvatures can be computed as combinations of
quartic, cubic, and quadratic terms in suitable polynomials $P\left( \sigma
,\tau \right) .$ In sections 12 and 13 we show that these
combinations/curvatures do not decrease much under our deformations (in a
proportional sense); so the entire Grassmannian is positively curved.

The role of the Cheeger deformations in ($6)$ is that any fixed plane with a
nondegenerate projection to the vertical space of $\Sigma
^{7}\longrightarrow S^{4}$ becomes positively curved, provided these
deformations are carried out for a sufficiently long time. Although the zero
planes $P=\mathrm{span}\left\{ \zeta ,W\right\} $ all have degenerate
projections to the vertical space of $\Sigma ^{7}\longrightarrow S^{4},$
there are of course nearby planes whose projections are nondegenerate.
Exploiting this idea we get

\begin{proposition}
\label{Cheeger Reduction}If all curvature polynomials whose corresponding
planes have degenerate projection onto the vertical space of $\Sigma
^{7}\longrightarrow S^{4}$ are positive on $\left( \Sigma
^{7},g_{1,2,3,4,5}\right) ,$ then $\left( \Sigma ^{7}g_{1,2,3,4,5,6}\right) $
is positively curved, provided the Cheeger deformations in (6) are carried
out for a sufficiently long time.
\end{proposition}

\begin{proof}
The assumptions imply that a neighborhood $N$ of the $0$--locus of $g_{1,3}$
is positively curved with respect to $g_{1,2,3,4,5}.$ The complement of this
neighborhood is compact, so $g_{1,2,3,4,5,6}$ is positively curved on the
whole complement, provided the Cheeger deformations in (6) are carried out
for enough time. Since Cheeger deformations preserve positive curvature $%
g_{1,2,3,4,5,6}$ is also positively curved on $N$. So $g_{1,2,3,4,5,6}$ is
positively curved.
\end{proof}

Thus the deformations in (6) allow us the computational convenience of
assuming that the vector \textquotedblleft $z$\textquotedblright\ is in the
horizontal space of $\Sigma ^{7}\longrightarrow S^{4}.$

In the sequel, we will not use the notation $g_{1},g_{1,2},g_{1,2,3}$, ect.
. Rather we will use more suggestive notation for these metrics, which we
will specify in Section 5.

\noindent \emph{Acknowledgments: }The authors are grateful to the referee
for finding a mistake in an earlier draft in Lemma \ref{abstract qudrat def}%
, to Karsten Grove for listening to an extended outline of our proof and
making a valuable expository suggestion, to Kriss Tapp for helping us find a
mistake in an earlier proof, to Bulkard Wilking for helping us find a
mistake in a related argument and for enlightening conversations about this
work, and to Paula Bergen for copy editing.

\section{Integrally Positive Curvature}

Here we show that it is possible to perturb the metric from \cite{Wilh2} to
one that has more positive curvature but also has some negative curvatures.
The sense in which the curvature has increased is specified in the theorem
below. The idea is that if we integrate the curvatures of the planes that
used to have zero curvature, then the answer is positive after the
perturbation. The theorem is not specific to the Gromoll-Meyer sphere.

\begin{theorem}
\label{Integraly Positive} Let $\left( M,g_{0}\right) $ be a Riemannian
manifold with nonnegative sectional curvature and 
\begin{equation*}
\pi :\left( M,g_{0}\right) \longrightarrow B
\end{equation*}%
a Riemannian submersion. Further assume that $G$ is an isometric group
action on $M$ that is by symmetries of $\pi $ and that the intrinsic metrics
on the principal orbits of $G$ in $B$ are homotheties of each other.

Let $T\subset M$ be a totally geodesic, flat torus spanned by geodesic
fields $X$ and $W$ such that $X$ is horizontal for $\pi $ and $D\pi \left(
W\right) =H_{w}$ is a Killing field for the $G$--action on $B.$ We suppose
further that $X$ is invariant under $G,$ $D\pi \left( X\right) $ is
orthogonal to the orbits of $G,$ and the normal distribution to the orbits
of $G$ on $B$ is integrable. Let $g_{s}$ be the metric obtained from $g_{0}$
by scaling the lengths of the fibers of $\pi $ by 
\begin{equation*}
\sqrt{1-s^{2}}.
\end{equation*}%
Let $c$ be an integral curve of $d\pi \left( X\right) $ from a zero of $%
\left\vert H_{w}\right\vert $ to a maximum of $\left\vert H_{w}\right\vert $
along $c,$whose interior passes through principle orbits. Then 
\begin{equation*}
\int_{c}\mathrm{curv}_{g_{s}}\left( X,W\right) =s^{4}\int_{c}\left(
D_{X}\left( \left\vert H_{w}\right\vert \right) \right) ^{2}.
\end{equation*}%
In particular, the curvature of \textrm{span}$\left\{ X,W\right\} $ is
integrally positive along $c,$ provided $H_{w}$ is not identically $0$ along 
$c.$
\end{theorem}

Here and throughout the paper we set 
\begin{equation*}
\mathrm{curv}\left( X,W\right) \equiv R\left( X,W,W,X\right) .
\end{equation*}

The formulas for the curvature tensor of metrics obtained by warping the
fibers of a Riemannian submersion by a \emph{function} on the base were
computed by Detlef Gromoll and his Stony Brook students in various classes
over the years. We were made aware of them via lecture notes by Carlos Duran 
\cite{GromDur}. They will appear shortly in the textbook \cite{GromWals}. In
the case when the function is constant, these formulas are necessarily much
simpler and can also be found in \cite{Bes}, where scaling the fibers by a
constant is referred to as the \textquotedblleft canonical
variation\textquotedblright . To ultimately get positive curvature on the
Gromoll-Meyer sphere, we have to control the curvature tensor in an entire
neighborhood in the Grassmannian, so we will need several of these formulas.
In fact, since the particular \textquotedblleft $W$\textquotedblright\ that
we have in mind is neither horizontal nor vertical for $\pi ,$ we need
multiple formulas just to find \textrm{curv}$\left( X,W\right) .$

For vertical vectors $U,V\in \mathcal{V}$ and horizontal vectors $X,Y,Z\in 
\mathcal{H},$ for $\pi :M\rightarrow B$ we have\addtocounter{algorithm}{1} 
\begin{eqnarray}
\left( R^{g_{s}}\left( X,V\right) U\right) ^{\mathcal{H}} &=&\left(
1-s^{2}\right) \left( R\left( X,V\right) U\right) ^{\mathcal{H}}+\left(
1-s^{2}\right) s^{2}A_{A_{X}U}V  \notag \\
R^{g_{s}}(V,X)Y &=&\left( 1-s^{2}\right) R(V,X)Y+s^{2}\left( R(V,X)Y\right)
^{\mathcal{V}}+s^{2}A_{X}A_{Y}V  \notag \\
R^{g_{s}}\left( X,Y\right) Z &=&\left( 1-s^{2}\right) R\left( X,Y\right)
Z+s^{2}\left( R\left( X,Y\right) Z\right) ^{\mathcal{V}}+s^{2}R^{B}\left(
X,Y\right) Z  \label{Detlef equation}
\end{eqnarray}%
The superscripts $^{\mathcal{H}}$ and $^{\mathcal{V}}$ denote the horizontal
and vertical parts of the vectors, $R$ and $A$ are the curvature and $A$%
-tensors for the unperturbed metric $g,$ $R^{g_{s}}$ denotes the new
curvature tensor of $g_{s},$ and $R^{B}$ is the curvature tensor of the base.

To eventually understand the curvature in a neighborhood of the
Gromoll-Meyer $0$-locus, we will need formulas for 
\begin{eqnarray*}
&&R^{g_{s}}\left( W,X\right) X\text{ and} \\
&&\left( R^{g_{s}}\left( X,W\right) W\right) ^{\mathcal{H}}
\end{eqnarray*}%
where $X$ is as above and $W$ is an arbitrary vector in $TM.$

\begin{lemma}
\label{Detlef}Let 
\begin{equation*}
\pi :\left( M,g_{0}\right) \longrightarrow B
\end{equation*}%
be as above. Let $X$ be a horizontal vector for $\pi $ and let $W$ be an
arbitrary vector in $TM.$ Then%
\begin{eqnarray*}
R^{g_{s}}\left( W,X\right) X &=&\left( 1-s^{2}\right) R(W,X)X+s^{2}\left(
R(W,X)X\right) ^{\mathcal{V}} \\
&&+s^{2}R^{B}\left( W^{\mathcal{H}},X\right) X+s^{2}A_{X}A_{X}W^{\mathcal{V}}
\\
\left( R^{g_{s}}\left( X,W\right) W\right) ^{\mathcal{H}} &=&\left(
1-s^{2}\right) \left( R\left( X,W\right) W\right) ^{\mathcal{H}} \\
&&+\left( 1-s^{2}\right) s^{2}A_{A_{X}W^{\mathcal{V}}}W^{\mathcal{V}%
}+s^{2}R^{B}\left( X,W^{\mathcal{H}}\right) W^{\mathcal{H}}
\end{eqnarray*}
\end{lemma}

\begin{remark}
Notice that the first curvature terms vanish in both formulas on the totally
geodesic torus.
\end{remark}

\begin{proof}
We split $W=W^{\mathcal{V}}+W^{\mathcal{H}}$ and get%
\begin{eqnarray*}
R^{g_{s}}\left( W,X\right) X &=&R^{g_{s}}\left( W^{\mathcal{V}},X\right)
X+R^{g_{s}}\left( W^{\mathcal{H}},X\right) X \\
&=&\left( 1-s^{2}\right) R(W^{\mathcal{V}},X)X+s^{2}\left( R(W^{\mathcal{V}%
},X)X\right) ^{\mathcal{V}}+s^{2}A_{X}A_{X}W^{\mathcal{V}} \\
&&+\left( 1-s^{2}\right) R\left( W^{\mathcal{H}},X\right) X+s^{2}\left(
R\left( W^{\mathcal{H}},X\right) X\right) ^{\mathcal{V}}+s^{2}R^{B}\left( W^{%
\mathcal{H}},X\right) X \\
&=&\left( 1-s^{2}\right) R(W,X)X+s^{2}\left( R(W,X)X\right) ^{\mathcal{V}%
}+s^{2}R^{B}\left( W^{\mathcal{H}},X\right) X+s^{2}A_{X}A_{X}W^{\mathcal{V}}
\end{eqnarray*}%
To find the other curvature we use%
\begin{eqnarray*}
R^{g_{s}}\left( X,W\right) W &=&R^{g_{s}}\left( X,W^{\mathcal{V}}\right) W^{%
\mathcal{V}}+R^{g_{s}}\left( X,W^{\mathcal{H}}\right) W^{\mathcal{V}} \\
&&+R^{g_{s}}\left( X,W^{\mathcal{V}}\right) W^{\mathcal{H}}+R^{g_{s}}\left(
X,W^{\mathcal{H}}\right) W^{\mathcal{H}}
\end{eqnarray*}%
Since $A_{X}A_{W^{\mathcal{H}}}W^{\mathcal{V}}$ and $A_{W^{\mathcal{H}%
}}A_{X}W^{\mathcal{V}}$ are vertical the above curvature formulas imply 
\begin{eqnarray*}
\left( R^{g_{s}}\left( X,W^{\mathcal{H}}\right) W^{\mathcal{V}}\right) ^{%
\mathcal{H}} &=&\left( 1-s^{2}\right) \left( R\left( X,W^{\mathcal{H}%
}\right) W^{\mathcal{V}}\right) ^{\mathcal{H}} \\
\left( R^{g_{s}}\left( X,W^{\mathcal{V}}\right) W^{\mathcal{H}}\right) ^{%
\mathcal{H}} &=&\left( 1-s^{2}\right) \left( R\left( X,W^{\mathcal{V}%
}\right) W^{\mathcal{H}}\right) ^{\mathcal{H}}.
\end{eqnarray*}%
In addition we have%
\begin{eqnarray*}
\left( R^{g_{s}}\left( X,W^{\mathcal{V}}\right) W^{\mathcal{V}}\right) ^{%
\mathcal{H}} &=&\left( 1-s^{2}\right) \left( R\left( X,W^{\mathcal{V}%
}\right) W^{\mathcal{V}}\right) ^{\mathcal{H}}+\left( 1-s^{2}\right)
s^{2}A_{A_{X}W^{\mathcal{V}}}W^{\mathcal{V}} \\
\left( R^{g_{s}}\left( X,W^{\mathcal{H}}\right) W^{\mathcal{H}}\right) ^{%
\mathcal{H}} &=&\left( 1-s^{2}\right) \left( R\left( X,W^{\mathcal{H}%
}\right) W^{\mathcal{H}}\right) ^{\mathcal{H}}+s^{2}R^{B}\left( X,W^{%
\mathcal{H}}\right) W^{\mathcal{H}}.
\end{eqnarray*}%
Therefore%
\begin{equation*}
\left( R^{g_{s}}\left( X,W\right) W\right) ^{\mathcal{H}}=\left(
1-s^{2}\right) \left( R\left( X,W\right) W\right) ^{\mathcal{H}}+\left(
1-s^{2}\right) s^{2}A_{A_{X}W^{\mathcal{V}}}W^{\mathcal{V}}+s^{2}R^{B}\left(
X,W^{\mathcal{H}}\right) W^{\mathcal{H}}
\end{equation*}%
as claimed.
\end{proof}

Now let $X$ and $W$ be as in the theorem. We set $H_{w}=D\pi \left( W^{%
\mathcal{H}}\right) $ and $V=W^{\mathcal{V}}$. To prove the theorem we need
to find $\mathrm{curv}_{B}\left( X,H_{w}\right) $ and $A_{X}V.$

\begin{lemma}
\label{R^B ( H, X) X}%
\begin{equation*}
R^{B}\left( H_{w},X\right) X=-\left( \frac{D_{X}D_{X}\left\vert
H_{w}\right\vert }{\left\vert H_{w}\right\vert }\right) H_{w}
\end{equation*}
\end{lemma}

\begin{proof}
Since $X$ is invariant under $G,$ $\left[ X,H_{w}\right] \equiv 0.$ Since $X$
is also a geodesic field%
\begin{equation*}
R^{B}\left( H_{w},X\right) X=-\nabla _{X}\nabla _{H_{w}}X.
\end{equation*}%
Similarly, since the normal distribution to the orbits of $G$ on $B$ is
integrable we can extend any normal vector $z$ to a $G$--invariant normal
field $Z$, and get that all terms of the Koszul formula for 
\begin{equation*}
\left\langle \nabla _{H_{w}}X,Z\right\rangle
\end{equation*}%
vanish. In particular, $\nabla _{H_{w}}X$ is tangent to the orbits of $G.$

If $K$ is another Killing field we have that $X$ commutes with $K$ as well
as $H_{w},$ and $\left[ K,H_{w}\right] $ is perpendicular to $X$ as it is
again a Killing field. Combining this with our hypothesis that the intrinsic
metrics on the principal orbits of $G$ in $B$ are homotheties of each other,
we see from Koszul's formula that $\nabla _{H_{w}}X$ is proportional to $%
H_{w}$ and can be calculated by%
\begin{eqnarray*}
\left\langle \nabla _{H_{w}}X,H_{w}\right\rangle &=&\left\langle \nabla
_{X}H_{w},H_{w}\right\rangle \\
&=&\frac{1}{2}D_{X}\left\vert H_{w}\right\vert ^{2} \\
&=&\left\vert H_{w}\right\vert D_{X}\left\vert H_{w}\right\vert ,\text{ so}
\\
\nabla _{H_{w}}X &=&\frac{D_{X}\left\vert H_{w}\right\vert }{\left\vert
H_{w}\right\vert }H_{w}.
\end{eqnarray*}%
Thus 
\begin{eqnarray*}
R^{B}\left( H_{w},X\right) X &=&-\nabla _{X}\left( \frac{D_{X}\left\vert
H_{w}\right\vert }{\left\vert H_{w}\right\vert }H_{w}\right) \\
&=&-D_{X}\left( \frac{D_{X}\left\vert H_{w}\right\vert }{\left\vert
H_{w}\right\vert }\right) H_{w}-\left( \frac{D_{X}\left\vert
H_{w}\right\vert }{\left\vert H_{w}\right\vert }\nabla _{X}H_{w}\right) \\
&=&-\left( \frac{\left\vert H_{w}\right\vert D_{X}D_{X}\left\vert
H_{w}\right\vert -\left( D_{X}\left\vert H_{w}\right\vert \right) ^{2}}{%
\left\vert H_{w}\right\vert ^{2}}\right) H_{w}-\left( \frac{D_{X}\left\vert
H_{w}\right\vert }{\left\vert H_{w}\right\vert }\right) ^{2}H_{w} \\
&=&-\left( \frac{D_{X}D_{X}\left\vert H_{w}\right\vert }{\left\vert
H_{w}\right\vert }\right) H_{w}.
\end{eqnarray*}
\end{proof}

\begin{lemma}
\label{R^B (X, H)H}%
\begin{equation*}
R^{B}\left( X,H_{w}\right) H_{w}=-\left\vert H_{w}\right\vert \nabla
_{X}\left( \mathrm{grad}\left\vert H_{w}\right\vert \right) .
\end{equation*}
\end{lemma}

\begin{proof}
Let $Z$ be any vector field. Using that $H_{w}$ is a Killing field we get 
\begin{eqnarray*}
\left\langle \nabla _{H_{w}}H_{w},Z\right\rangle &=&-\left\langle \nabla
_{Z}H_{w},H_{w}\right\rangle \\
&=&-\frac{1}{2}D_{Z}\left\langle H_{w},H_{w}\right\rangle \\
&=&-\frac{1}{2}D_{Z}\left\vert H_{w}\right\vert ^{2} \\
&=&-\left\vert H_{w}\right\vert D_{Z}\left\vert H_{w}\right\vert \\
&=&-\left\langle \left\vert H_{w}\right\vert \mathrm{grad}\left\vert
H_{w}\right\vert ,Z\right\rangle
\end{eqnarray*}%
showing that 
\begin{equation*}
\nabla _{H_{w}}H_{w}=-\left\vert H_{w}\right\vert \mathrm{grad}\left\vert
H_{w}\right\vert .
\end{equation*}%
Thus%
\begin{eqnarray*}
R^{B}\left( X,H_{w}\right) H_{w} &=&\nabla _{X}\nabla _{H_{w}}H_{w}-\nabla
_{H_{w}}\nabla _{X}H_{w} \\
&=&-\nabla _{X}\left( \left\vert H_{w}\right\vert \mathrm{grad}\left\vert
H_{w}\right\vert \right) -\nabla _{H_{w}}\left( \frac{D_{X}\left\vert
H_{w}\right\vert }{\left\vert H_{w}\right\vert }H_{w}\right) \\
&=&-\left( D_{X}\left\vert H_{w}\right\vert \right) \mathrm{grad}\left\vert
H_{w}\right\vert -\left( \left\vert H_{w}\right\vert \nabla _{X}\mathrm{grad}%
\left\vert H_{w}\right\vert \right) -\frac{D_{X}\left\vert H_{w}\right\vert 
}{\left\vert H_{w}\right\vert }\nabla _{H_{w}}H_{w} \\
&=&-\left( D_{X}\left\vert H_{w}\right\vert \right) \mathrm{grad}\left\vert
H_{w}\right\vert -\left( \left\vert H_{w}\right\vert \nabla _{X}\mathrm{grad}%
\left\vert H_{w}\right\vert \right) +\frac{D_{X}\left\vert H_{w}\right\vert 
}{\left\vert H_{w}\right\vert }\left\vert H_{w}\right\vert \mathrm{grad}%
\left\vert H_{w}\right\vert \\
&=&-\left( \left\vert H_{w}\right\vert \nabla _{X}\mathrm{grad}\left\vert
H_{w}\right\vert \right)
\end{eqnarray*}
\end{proof}

It follows that \addtocounter{algorithm}{1} 
\begin{eqnarray}
\mathrm{curv}_{B}\left( X,H_{w}\right) &=&\left\langle R^{B}\left(
H_{w},X\right) X,H_{w}\right\rangle  \notag \\
&=&-\left( \frac{D_{X}D_{X}\left\vert H_{w}\right\vert }{\left\vert
H_{w}\right\vert }\right) \left\langle H_{w},H_{w}\right\rangle  \notag \\
&=&-\left\vert H_{w}\right\vert \left( D_{X}D_{X}\left\vert H_{w}\right\vert
\right) .  \label{curv of base}
\end{eqnarray}

Next we focus on $\left\vert A_{X}V\right\vert ^{2}.$

\begin{lemma}
\label{Abstract A--tensor}%
\begin{equation*}
A_{X}V=-\frac{D_{X}\left\vert H_{w}\right\vert }{\left\vert H_{w}\right\vert 
}H_{w}.
\end{equation*}
\end{lemma}

\begin{proof}
Since $X$ and $W$ are commuting geodesic fields on a totally geodesic flat
torus, $\nabla _{X}W=0.$

So%
\begin{eqnarray*}
A_{X}V &=&\left( \nabla _{X}V\right) ^{\mathcal{H}} \\
&=&\left( \nabla _{X}W-\nabla _{X}H_{w}\right) ^{\mathcal{H}} \\
&=&-\left( \nabla _{X}H_{w}\right) ^{\mathcal{H}} \\
&=&-\nabla _{H_{w}}^{B}X \\
&=&-\frac{D_{X}\left\vert H_{w}\right\vert }{\left\vert H_{w}\right\vert }%
H_{w}
\end{eqnarray*}
\end{proof}

Combining this $A$--tensor formula with equation \ref{curv of base} and
Lemma \ref{Detlef} yields%
\begin{eqnarray*}
\mathrm{curv}_{g_{s}}\left( X,W\right) &=&\left( 1-s^{2}\right) \mathrm{curv}%
\left( X,W\right) +s^{2}\mathrm{curv}_{B}\left( X,H_{w}\right)
-s^{2}\left\vert A_{X}V\right\vert ^{2}+s^{4}\left\vert A_{X}V\right\vert
^{2} \\
&=&\left( 1-s^{2}\right) \mathrm{curv}\left( X,W\right) -s^{2}\left(
\left\vert H_{w}\right\vert \left( D_{X}D_{X}\left\vert H_{w}\right\vert
\right) \right) -s^{2}\left( D_{X}\left\vert H_{w}\right\vert \right)
^{2}+s^{4}\left( D_{X}\left\vert H_{w}\right\vert \right) ^{2}
\end{eqnarray*}%
Since $\mathrm{curv}\left( X,W\right) =0,$ this further simplifies to %
\addtocounter{algorithm}{1} 
\begin{equation}
\mathrm{curv}_{g_{s}}\left( X,W\right) =-s^{2}\left( D_{X}\left( \left\vert
H_{w}\right\vert D_{X}\left\vert H_{w}\right\vert \right) \right)
+s^{4}\left( D_{X}\left\vert H_{w}\right\vert \right) ^{2}.
\label{curv formula}
\end{equation}

If $c$ is an integral curve of $X$ from a zero of $H_{w}$ to a maximum of $%
\left\vert H_{w}\right\vert $ along $c,$ then the first term integrates to $%
0 $ along $c,$ yielding%
\begin{equation*}
\int_{c}\mathrm{curv}_{g_{s}}\left( X,W\right) =s^{4}\int_{c}\left(
D_{X}\left\vert H_{w}\right\vert \right) ^{2}
\end{equation*}%
as desired.

As we've mentioned, to get positive curvature on the Gromoll-Meyer sphere we
will have to understand the full curvature tensor. Combining the
calculations above we have

\begin{lemma}
\label{Abstract (1,3) tensors} Let $X$ and $W$ be as in Theorem \ref%
{Integraly Positive}. Then 
\begin{eqnarray*}
R^{g_{s}}\left( W,X\right) X &=&-s^{2}\left( \frac{D_{X}D_{X}\left\vert
H_{w}\right\vert }{\left\vert H_{w}\right\vert }\right) H_{w}-s^{2}\frac{%
D_{X}\left\vert H_{w}\right\vert }{\left\vert H_{w}\right\vert }A_{X}H_{w} \\
\left( R^{g_{s}}\left( X,W\right) W\right) ^{\mathcal{H}} &=&-\left(
1-s^{2}\right) s^{2}\frac{D_{X}\left\vert H_{w}\right\vert }{\left\vert
H_{w}\right\vert }A_{H_{w}}W^{\mathcal{V}}-s^{2}\left\vert H_{w}\right\vert
\nabla _{X}\left( \mathrm{grad}\left\vert H_{w}\right\vert \right) .
\end{eqnarray*}
\end{lemma}

\begin{remark}
The two $A$--tensors $A_{X}H_{w}$ and $A_{H_{w}}W^{\mathcal{V}}$ involve
derivatives of vectors that are not tangent or normal to the totally
geodesic tori. They cannot be determined abstractly, and are in fact
dependent on the particular geometry. We give estimates for them in the case
of the Gromoll-Meyer sphere in Lemma \ref{A--tensor estimate} below.
\end{remark}

\section{Review of the geometry of $Sp\left( 2\right) $}

The next three sections are a review of \cite{PetWilh}, \cite{Wilh1}, and 
\cite{Wilh2}.

We let $h:S^{7}\longrightarrow S^{4}$ and $\tilde{h}:S^{7}\longrightarrow
S^{4}$ be the Hopf fibrations corresponding to the right $A^{h}$ and left $%
A^{\tilde{h}}$ actions of $S^{3}$ on $S^{7}$.

Points on $S^{7}$ are denoted by pairs of quaternions written as column
vectors. The quotient map for action on the right is 
\begin{equation*}
h:\left( 
\begin{array}{c}
a \\ 
c%
\end{array}%
\right) \mapsto (a\bar{c},\frac{1}{2}(|a|^{2}-|c|^{2})),
\end{equation*}%
and the quotient map for action on the left is 
\begin{equation*}
\tilde{h}:\left( 
\begin{array}{c}
a \\ 
c%
\end{array}%
\right) \mapsto (\bar{a}c,\frac{1}{2}(|a|^{2}-|c|^{2})).
\end{equation*}%
The image is $S^{4}(\frac{1}{2})\subset \mathbb{H}\oplus \mathbb{R}$ \cite%
{Wilh1}.

\begin{proposition}
(The Pullback Identification) $Sp(2)$ is diffeomorphic to the total space of
the pullback of the Hopf fibration $S^{7}\overset{h}{\longrightarrow }S^{4}$
via $S^{7}\overset{-I\circ h}{\longrightarrow }S^{4}$, where $S^{4}\overset{%
-I}{\longrightarrow }S^{4}$ is the antipodal map. In fact, the biinvariant
metric on $Sp(2)$ is isometric (up to rescaling) to the subspace metric on
the pullback 
\begin{equation*}
\left( -I\circ h\right) ^{\ast }\left( S^{7}\right) \subset S^{7}\left(
1\right) \times S^{7}\left( 1\right) ,
\end{equation*}%
where $S^{7}\left( 1\right) $ is the unit 7-sphere and $S^{7}\left( 1\right)
\times S^{7}\left( 1\right) $ has the product metric.
\end{proposition}

In \cite{GromMey} it was shown that $\Sigma ^{7}$ is the quotient of the $%
S^{3}$-action on $Sp(2)$ given by 
\begin{equation*}
A_{2,-1}\left( q,\ \left( 
\begin{array}{cc}
a & b \\ 
c & d%
\end{array}%
\right) \right) =\left( 
\begin{array}{cc}
qa\bar{q} & qb \\ 
qc\bar{q} & qd%
\end{array}%
\right) .
\end{equation*}%
We let $q_{2,-1}:Sp(2)\longrightarrow \Sigma ^{7}$ denote the quotient map.
It was observed by Gromoll and Meyer that $\Sigma ^{7}$ is the $S^{3}$%
--bundle over $S^{4}$ of \textquotedblleft type $\left( 2,-1\right) $%
\textquotedblright , using the classification convention of \cite{Miln}. The
submersion $p_{2,-1}:\Sigma ^{7}\longrightarrow S^{4}$ is induced by 
\begin{equation*}
\widetilde{h}\circ p_{2}|_{Sp\left( 2\right) }:Sp\left( 2\right)
\longrightarrow S^{4},
\end{equation*}%
where $p_{2}:S^{7}\times S^{7}\longrightarrow S^{7}$ is projection onto the
second factor.

The Gromoll-Meyer metric on $\Sigma ^{7}$ is induced by the biinvariant
metric via $q_{2,-1}$. The metric studied in \cite{Wilh2}, $g_{\nu _{1},\nu
_{2},l_{1}^{u},l_{1}^{d}}$, is induced via $q_{2,-1}$ by the perturbation of
the biinvariant metric that was studied in \cite{PetWilh}. We will review
the definition of this metric in the next section.

The isometry group of the metric discovered by Gromoll and Meyer is $%
O(2)\times SO\left( 3\right) .$ The $O(2)$-action is induced on $\Sigma ^{7}$
by the action $A_{O(2)}$ on $Sp\left( 2\right) $ defined as 
\begin{eqnarray*}
O(2)\times Sp(2) &\longrightarrow &Sp(2) \\
(A,U) &\mapsto &AU.
\end{eqnarray*}%
The $SO(3)$-action is induced on $\Sigma ^{7}$ by the $S^{3}$--action $%
A^{h_{2}}$ on $Sp\left( 2\right) $ defined as 
\begin{eqnarray*}
S^{3}\times Sp\left( 2\right) &\longrightarrow &Sp\left( 2\right) \\
\left( q,\left( 
\begin{array}{cc}
a & b \\ 
c & d%
\end{array}%
\right) \right) &\longmapsto &\left( 
\begin{array}{cc}
a & b\bar{q} \\ 
c & d\bar{q}%
\end{array}%
\right) .
\end{eqnarray*}%
As in \cite{Wilh1} we have

\begin{proposition}
Every point in $\Sigma ^{7}$ has a point in its orbit under $A_{SO\left(
2\right) }\times A^{h_{2}}$ that can be represented in $Sp(2)$ by a point of
the form 
\begin{equation*}
\left( \left( 
\begin{array}{c}
\cos t \\ 
\alpha \sin t%
\end{array}%
\right) p,\left( 
\begin{array}{c}
\alpha \sin t \\ 
\cos t%
\end{array}%
\right) \right)
\end{equation*}%
with $t\in \left[ 0,\frac{\pi }{4}\right] ,$ $p,\alpha \in S^{3}\subset 
\mathbb{H},$ and $\func{Re}\left( \alpha \right) =0$.
\end{proposition}

Since only $A^{h_{2}}$ acts by isometries with respect to the metrics we
study, the points in the previous proposition have to be multiplied by $%
SO\left( 2\right) $ to get

\begin{proposition}
Every point in $\Sigma ^{7}$ has a representative point $\left(
N_{1}p,N_{2}\right) $ in its orbit under $A^{h_{2}}$ that in $Sp(2)$ has the
form 
\begin{eqnarray*}
\left( N_{1}p,N_{2}\right) &=&\left( 
\begin{array}{cc}
\cos \theta & \sin \theta \\ 
-\sin \theta & \cos \theta%
\end{array}%
\right) \left( \left( 
\begin{array}{c}
\cos t \\ 
\alpha \sin t%
\end{array}%
\right) p,\left( 
\begin{array}{c}
\alpha \sin t \\ 
\cos t%
\end{array}%
\right) \right) \\
&=&\left( \left( 
\begin{array}{c}
\cos \theta \cos t+\alpha \sin \theta \sin t \\ 
-\sin \theta \cos t+\alpha \cos \theta \sin t%
\end{array}%
\right) p,\left( 
\begin{array}{c}
\sin \theta \cos t+\alpha \cos \theta \sin t \\ 
\cos \theta \cos t-\alpha \sin \theta \sin t%
\end{array}%
\right) \right)
\end{eqnarray*}%
with $t\in \left[ 0,\frac{\pi }{4}\right] ,$ $\theta \in \left[ 0,\pi \right]
,$ $p,\alpha \in S^{3},$ and $\func{Re}\left( \alpha \right) =0$.
\end{proposition}

We have a similar representation in $S^{7}.$

\begin{corollary}
Every point in $S^{7}$ has a point in its orbit under $A^{\tilde{h}}\times
A^{h}$ of the form 
\begin{equation*}
N=\left( 
\begin{array}{cc}
\cos \theta & \sin \theta \\ 
-\sin \theta & \cos \theta%
\end{array}%
\right) \left( 
\begin{array}{c}
\cos t \\ 
\alpha \sin t%
\end{array}%
\right) =\left( 
\begin{array}{c}
\cos \theta \cos t+\alpha \sin \theta \sin t \\ 
-\sin \theta \cos t+\alpha \cos \theta \sin t%
\end{array}%
\right)
\end{equation*}%
with $t\in \left[ 0,\frac{\pi }{4}\right] ,$ $\theta \in \left[ 0,\pi \right]
,$ $\alpha \in S^{3},$ and $\func{Re}\left( \alpha \right) =0$.
\end{corollary}

The $h$--fiber of $N$ consists of the points 
\begin{equation*}
\left\{ Np:p\in S^{3}\right\} .
\end{equation*}

We need a basis for the tangent space of $Sp\left( 2\right) $ that is well
adapted to the Gromoll-Meyer sphere and its symmetry group. It turns out
that a left invariant framing is ill suited for this purpose; rather we use
a basis that comes from $S^{7}$ via the embedding $Sp\left( 2\right) \subset
S^{7}\times S^{7}.$ To get the correct basis we point out

\begin{proposition}
\label{hopf symmetries}$SO\left( 2\right) \times A^{h}$ acts on $S^{7}$ by
symmetries of $\tilde{h}.$ The action induced on $S^{4}$ has $\mathbb{Z}_{2}$%
--kernel and induces an effective $SO\left( 2\right) \times SO\left(
3\right) $ action that respects the join decomposition $S^{4}=S^{1}\ast
S^{2}.$ The $SO\left( 2\right) $--factor acts in the standard way on $S^{1}$
and as the identity on $S^{2}.$ The $SO\left( 3\right) $ action is standard
on the $S^{2}$--factor and the identity on the $S^{1}$--factor. (See \cite%
{GluWarZil}, cf also the proof of Proposition 1.2 in \cite{Wilh1}.)
\end{proposition}

\begin{remark}
At a representative point%
\begin{eqnarray*}
\left( N_{1}p,N_{2}\right) &=&\left( 
\begin{array}{cc}
\cos \theta & \sin \theta \\ 
-\sin \theta & \cos \theta%
\end{array}%
\right) \left( \left( 
\begin{array}{c}
\cos t \\ 
\alpha \sin t%
\end{array}%
\right) p,\left( 
\begin{array}{c}
\alpha \sin t \\ 
\cos t%
\end{array}%
\right) \right) \\
&=&\left( \left( 
\begin{array}{c}
\cos \theta \cos t+\alpha \sin \theta \sin t \\ 
-\sin \theta \cos t+\alpha \cos \theta \sin t%
\end{array}%
\right) p,\left( 
\begin{array}{c}
\sin \theta \cos t+\alpha \cos \theta \sin t \\ 
\cos \theta \cos t-\alpha \sin \theta \sin t%
\end{array}%
\right) \right) ,
\end{eqnarray*}%
the parameter $\theta ,$ is the \textquotedblleft $S^{1}$\textquotedblright
--coordinate in $S^{1}\ast S^{2},$ $\alpha $ is the $S^{2}$--coordinate, $t$
is the distance to the singular $S^{1}$ in $S^{1}\ast S^{2}$ and $p$
parameterizes the fibers of $p_{2,-1}:\Sigma ^{7}\longrightarrow S^{4},$
giving us a partial coordinate system $\left( t,\theta ,\alpha ,p\right) $
for $\Sigma ^{7}$. We denote the singular $S^{1}$ in $S^{1}\ast S^{2}$ by $%
S_{\mathbb{R}}^{1}$ and we denote the singular $S^{2}$ by $S_{\mathrm{Im}%
}^{2}.$ The points in $S_{\mathbb{R}}^{1}$ are represented in $Sp\left(
2\right) $ by the points with $t=0,$ and $S_{\mathrm{Im}}^{2}$ corresponds
to the set where $t=\frac{\pi }{4}.$ Thus%
\begin{eqnarray*}
S_{\mathbb{R}}^{1} &=&\tilde{h}\circ p_{2}|_{Sp\left( 2\right) }\left\{
\left( \left( 
\begin{array}{c}
\cos \theta \\ 
-\sin \theta%
\end{array}%
\right) p,\left( 
\begin{array}{c}
\sin \theta \\ 
\cos \theta%
\end{array}%
\right) \right) \in Sp\left( 2\right) :\theta \in \left[ 0,\pi \right] ,p\in
S^{3}\right\} \text{ and} \\
S_{\mathrm{Im}}^{2} &=&\tilde{h}\circ p_{2}|_{Sp\left( 2\right) }\left\{
\left( \frac{1}{\sqrt{2}}\left( 
\begin{array}{c}
\cos \theta +\alpha \sin \theta \\ 
-\sin \theta +\alpha \cos \theta%
\end{array}%
\right) p,\frac{1}{\sqrt{2}}\left( 
\begin{array}{c}
\sin \theta +\alpha \cos \theta \\ 
\cos \theta -\alpha \sin \theta%
\end{array}%
\right) \right) \in Sp\left( 2\right) :\right. \\
&&\left. \theta \in \left[ 0,\pi \right] ,\alpha ,p\in S^{3},\text{ and }%
\func{Re}\left( \alpha \right) =0\right\} .
\end{eqnarray*}
\end{remark}

Throughout the paper, $\gamma _{1}$ and $\gamma _{2}$ will be purely
imaginary unit quaternions that satisfy $\gamma _{1}\gamma _{2}=\alpha $.
Using such a choice for $\gamma _{1}$ and $\gamma _{2}$ gets us a basis for
the vertical space of $h$ at $N\subset S^{7}$ by setting%
\begin{eqnarray*}
\mathfrak{v} &=&N\alpha p, \\
\vartheta _{1} &=&N\gamma _{1}p, \\
\vartheta _{2} &=&N\gamma _{2}p.
\end{eqnarray*}%
The fibers of $h$ and $\tilde{h}$ have a one-dimensional intersection when $%
t>0$ and coincide when $t=0.$ $\mathfrak{v}$ is tangent to this intersection.

We get a basis for the horizontal space of $h$ by selecting a suitable
vector perpendicular to $N.$ When $\theta =0$ a natural choice is%
\begin{equation*}
\hat{N}=\left( 
\begin{array}{c}
-\sin t \\ 
\alpha \cos t%
\end{array}%
\right) .
\end{equation*}%
For general $\theta $ we just multiply by an element in $SO\left( 2\right) $
and get%
\begin{eqnarray*}
\hat{N} &=&\left( 
\begin{array}{cc}
\cos \theta & \sin \theta \\ 
-\sin \theta & \cos \theta%
\end{array}%
\right) \left( 
\begin{array}{c}
-\sin t \\ 
\alpha \cos t%
\end{array}%
\right) \\
&=&\left( 
\begin{array}{c}
-\cos \theta \sin t+\alpha \sin \theta \cos t \\ 
\sin \theta \sin t+\alpha \cos \theta \cos t%
\end{array}%
\right) .
\end{eqnarray*}%
With this choice we define the basis for the horizontal space as%
\begin{eqnarray*}
x &=&\hat{N}p, \\
y &=&\hat{N}\bar{\alpha}p \\
\eta _{1} &=&\hat{N}\gamma _{1}p, \\
\eta _{2} &=&\hat{N}\gamma _{2}p.
\end{eqnarray*}

These vectors are well-adapted to the Gromoll-Meyer sphere since $x$ is
normal to the $S^{1}\times S^{2}$s in $S^{1}\ast S^{2}=S^{4},$ $y$ is
tangent to the $S^{1}$s in $S^{1}\times S^{2}\subset S^{1}\ast S^{2}=S^{4},$
and the $\eta $s are tangent to the $S^{2}$s in $S^{1}\times S^{2}\subset
S^{1}\ast S^{2}=S^{4}.$

We call $x,y,$ and $\mathfrak{v}$, $\alpha $--vectors, and we call $\eta
_{1},\eta _{2},$ $\vartheta _{1},$ and $\vartheta _{2}$, $\gamma $--vectors.

When $t=0$, our formula for $Np$ becomes%
\begin{equation*}
Np=\left( 
\begin{array}{c}
\cos \theta \\ 
-\sin \theta%
\end{array}%
\right) p
\end{equation*}%
which has no $``\alpha $\textquotedblright . So the vectors 
\begin{equation*}
\mathfrak{v},\vartheta _{1},\vartheta _{2}
\end{equation*}%
become indistinguishable. This reflects the fact that the fibers of $h$ and $%
\tilde{h}$ coincide when $t=0.$ Similarly our formulas for the vectors%
\begin{equation*}
x,\eta _{1},\eta _{2}
\end{equation*}%
become indistinguishable at $t=0.$ This reflects the fact that the set where 
$t=0$ in $S^{4}$ is the \textquotedblleft singular\textquotedblright\ $%
S^{1}\subset S^{1}\ast S^{2}=S^{4},$ i.e. the place where the $S^{2}$s are
\textquotedblleft collapsed\textquotedblright . On the other, hand at $t=0$, 
$y$ becomes 
\begin{equation*}
\left( 
\begin{array}{c}
\alpha \sin \theta \\ 
\alpha \cos \theta%
\end{array}%
\right) \bar{\alpha}p=\left( 
\begin{array}{c}
\sin \theta \\ 
\cos \theta%
\end{array}%
\right) p
\end{equation*}%
and hence is well defined, reflecting the fact that $y$ is tangent to the
circles of the join decomposition.

\begin{proposition}
On $S^{7}$ the \textquotedblleft combined Hopf action\textquotedblright\ $A^{%
\tilde{h}}\times A^{h}$ leaves the splitting 
\begin{equation*}
\mathrm{span}\left\{ x,\eta _{1},\eta _{2}\right\} \oplus \mathrm{span}%
\left\{ y\right\} \oplus \mathrm{span}\left\{ \mathfrak{v},\vartheta
_{1},\vartheta _{2}\right\}
\end{equation*}%
invariant and leaves the splitting 
\begin{equation*}
\mathrm{span}\left\{ x\right\} \oplus \mathrm{span}\left\{ y\right\} \oplus 
\mathrm{span}\left\{ \eta _{1},\eta _{2}\right\} \oplus \mathrm{span}\left\{ 
\mathfrak{v}\right\} \oplus \mathrm{span}\left\{ \vartheta _{1},\vartheta
_{2}\right\}
\end{equation*}%
invariant when $t>0.$
\end{proposition}

\begin{proof}
Since $A^{\tilde{h}}$ acts by symmetries of $h,$ it at least preserves the
horizontal and vertical splitting of $h$. But it also leaves its own
horizontal and vertical spaces invariant. The $A^{\tilde{h}}$--invariance of 
$\mathrm{span}\left\{ \mathfrak{v}\right\} \oplus \mathrm{span}\left\{
\vartheta _{1},\vartheta _{2}\right\} $ when $t>0$ follows from the fact
that $\mathrm{span}\left\{ \mathfrak{v}\right\} $ is the intersection of the
two vertical spaces and $\mathrm{span}\left\{ \vartheta _{1},\vartheta
_{2}\right\} $ its orthogonal complement in the vertical space of $h.$ The $%
A^{\tilde{h}}$--invariance of $\mathrm{span}\left\{ x\right\} \oplus \mathrm{%
span}\left\{ y\right\} \oplus \mathrm{span}\left\{ \eta _{1},\eta
_{2}\right\} $ when $t>0$ follows from the fact that at the level of $S^{4},$
$A^{\tilde{h}}$ preserves our join decomposition. Finally, $\mathrm{span}%
\left\{ y\right\} $ is $A^{\tilde{h}}$--invariant when $t=0$ since on $%
S^{4}, $ the set where $t=0$ is the fixed point set of $A^{\tilde{h}},$ and $%
\mathrm{span}\left\{ y\right\} $ is the tangent space to this fixed point
set.

A similar argument gives us the statement for $A^{h}.$
\end{proof}

As observed in \cite{PetWilh}, $TSp\left( 2\right) $ has a splitting%
\begin{equation*}
TSp\left( 2\right) =V_{1}\oplus V_{2}\oplus H,
\end{equation*}%
where $V_{1}$ and $V_{2}$ are the vertical spaces for the Hopf fibrations
that describe 
\begin{equation*}
Sp\left( 2\right) \equiv \left( -I\circ h\right) ^{\ast }\left( S^{7}\right)
\subset S^{7}\left( 1\right) \times S^{7}\left( 1\right) ,
\end{equation*}%
and $H$ is the orthogonal complement of $V_{1}\oplus V_{2}$ with respect to
the biinvariant metric.

The vectors 
\begin{eqnarray*}
\left( \mathfrak{v},0\right) &=&\left( N_{1}\alpha p,0\right) , \\
\left( \vartheta _{1},0\right) &=&\left( N_{1}\gamma _{1}p,0\right) , \\
\left( \vartheta _{2},0\right) &=&\left( N_{1}\gamma _{2}p,0\right)
\end{eqnarray*}%
form an orthogonal basis for $V_{1}.$ Similarly,%
\begin{eqnarray*}
\left( 0,\mathfrak{v}\right) &=&\left( 0,N_{2}\alpha \right) , \\
\left( 0,\vartheta _{1}\right) &=&\left( 0,N_{2}\gamma _{1}\right) , \\
\left( 0,\vartheta _{2}\right) &=&\left( 0,N_{2}\gamma _{2}\right)
\end{eqnarray*}%
form a orthogonal basis for $V_{2}.$

To get a basis for $H$ at representative points we define%
\begin{eqnarray*}
\hat{N}_{1} &=&\left( 
\begin{array}{cc}
\cos \theta & \sin \theta \\ 
-\sin \theta & \cos \theta%
\end{array}%
\right) \left( 
\begin{array}{c}
-\sin t \\ 
\alpha \cos t%
\end{array}%
\right) , \\
\hat{N}_{2} &=&\left( 
\begin{array}{cc}
\cos \theta & \sin \theta \\ 
-\sin \theta & \cos \theta%
\end{array}%
\right) \left( 
\begin{array}{c}
\alpha \cos t \\ 
-\sin t%
\end{array}%
\right)
\end{eqnarray*}%
and%
\begin{eqnarray*}
x^{2,0} &=&\left( \hat{N}_{1}p,\hat{N}_{2}\right) , \\
y^{2,0} &=&\left( \hat{N}_{1}\bar{\alpha}p,\hat{N}_{2}\alpha \right) \\
\left( \eta _{1,}\eta _{1}\right) &=&\left( \hat{N}_{1}\gamma _{1}p,\hat{N}%
_{2}\gamma _{1}\right) \\
\left( \eta _{2,}\eta _{2}\right) &=&\left( \hat{N}_{1}\gamma _{2}p,\hat{N}%
_{2}\gamma _{2}\right)
\end{eqnarray*}

We refer the reader to \cite{Wilh1} for the computations that show that $%
x^{2,0},y^{2,0},\left( \eta _{1,}\eta _{1}\right) ,$ and $\left( \eta
_{2,}\eta _{2}\right) $ are tangent to $Sp\left( 2\right) $. A corollary of
the previous proposition is

\begin{corollary}
\label{Invariance Corrollary} The Gromoll-Meyer action $A^{2,-1}\times
A^{h_{2}}$ leaves%
\begin{eqnarray*}
&&\mathrm{span}\left\{ x^{2,0},\left( \eta _{1},\eta _{1}\right) ,\left(
\eta _{2},\eta _{2}\right) \right\} \oplus \mathrm{span}\left\{
y^{2,0}\right\} \\
&&\oplus \mathrm{span}\left\{ \left( \mathfrak{v},0\right) ,\left( \vartheta
_{1},0\right) ,\left( \vartheta _{2},0\right) \right\} \oplus \mathrm{span}%
\left\{ \left( 0,\mathfrak{v}\right) ,\left( 0,\vartheta _{1}\right) ,\left(
0,\vartheta _{2}\right) \right\}
\end{eqnarray*}%
invariant and leaves the splitting 
\begin{eqnarray*}
&&\mathrm{span}\left\{ x^{2,0}\right\} \oplus \mathrm{span}\left\{
y^{2,0}\right\} \oplus \mathrm{span}\left\{ \left( \eta _{1},\eta
_{1}\right) ,\left( \eta _{2},\eta _{2}\right) \right\} \\
&&\oplus \mathrm{span}\left\{ \left( \mathfrak{v},0\right) \right\} \oplus 
\mathrm{span}\left\{ \left( \vartheta _{1},0\right) ,\left( \vartheta
_{2},0\right) \right\} \oplus \mathrm{span}\left\{ \left( 0,\mathfrak{v}%
\right) \right\} \oplus \mathrm{span}\left\{ \left( 0,\vartheta _{1}\right)
,\left( 0,\vartheta _{2}\right) \right\}
\end{eqnarray*}%
invariant when $t>0.$
\end{corollary}

\section{Cheeger Deformations}

The metric studied in \cite{Wilh2} is induced via $q_{2,-1}$ by the
perturbation of the biinvariant metric that was studied in \cite{PetWilh}.
We start by reviewing its construction.

In \cite{Cheeg} a general method for perturbing the metric $g$ on a manifold 
$M$ of nonnegative sectional curvature was proposed. Various special cases
of this method were first studied in \cite{Berg2}, \cite{BourDesSent}, and 
\cite{Wal}.

If $G$ is a compact group of isometries of $\left( M,g\right) $, then we let 
$G$ act on $G\times M$ by 
\begin{equation*}
q\cdot (p,m)=(pq^{-1},qm).
\end{equation*}%
If $b$ is a biinvariant metric on $G$, then for each $l>0$ we get a product
metric $l^{2}b+g$ on $G\times M.$ The quotient of this action then induces a
new metric, $g_{l},$ of nonnegative sectional curvature on $M$. It was
observed in \cite{Cheeg} that we may expect the new metric to have fewer $0$%
--curvatures and symmetries than the original metric, $g=g_{\infty }$. The
quotient map of this action is denoted by \addtocounter{algorithm}{1} 
\begin{equation}
q_{G\times M}:G\times M\longrightarrow M.  \label{Cheeger submersion}
\end{equation}

In \cite{PetWilh} we studied the effect of perturbing the biinvariant metric
on $Sp(2)$ using Cheeger's method and the $S^{3}\times S^{3}\times
S^{3}\times S^{3}$ action induced by the commuting $S^{3}$-actions 
\begin{equation*}
\begin{array}{l}
A^{u}\left( p_{1},\left( 
\begin{array}{cc}
a & b \\ 
c & d%
\end{array}%
\right) \right) =\left( 
\begin{array}{cc}
p_{1}a & p_{1}b \\ 
c & d%
\end{array}%
\right) , \\ 
A^{d}\left( p_{2},\left( 
\begin{array}{cc}
a & b \\ 
c & d%
\end{array}%
\right) \right) =\left( 
\begin{array}{cc}
a & b \\ 
p_{2}c & p_{2}d%
\end{array}%
\right) , \\ 
A^{h_{1}}\left( q_{1},\left( 
\begin{array}{cc}
a & b \\ 
c & d%
\end{array}%
\right) \right) =\left( 
\begin{array}{cc}
a\bar{q}_{1} & b \\ 
c\bar{q}_{1} & d%
\end{array}%
\right) , \\ 
A^{h_{2}}\left( q_{2},\left( 
\begin{array}{cc}
a & b \\ 
c & d%
\end{array}%
\right) \right) =\left( 
\begin{array}{cc}
a & b\bar{q}_{2} \\ 
c & d\bar{q}_{2}%
\end{array}%
\right) .%
\end{array}%
\end{equation*}

If $\xi \in TM$, then $\hat{\xi}\in T(G\times M)$\label{hat} denotes the
horizontal vector, with respect to $q_{G\times M},$ satisfying $dp_{2}\left( 
\hat{\xi}\right) =\xi ,$ where $p_{2}:G\times M\longrightarrow M$ is the
projection onto the second factor. Similarly if $P\subset TM$ is a tangent
plane, then $\hat{P}\subset T(G\times M)$ is the horizontal plane satisfying 
$dp_{2}(\hat{P})=P$. Cheeger's observation was that (\cite{Cheeg}, cf \cite%
{PetWilh}, Proposition 1.10)

\begin{proposition}
\label{Cheeger's curvature condition}

\begin{description}
\item[(i)] If the curvature of $P$ is positive with respect to $g_{\infty }$%
, then the curvature of 
\begin{equation*}
dq_{G\times M}(\hat{P})
\end{equation*}%
is positive with respect $g_{l}$.

\item[(ii)] The curvature of $dq_{G\times M}(\hat{P})$ is positive with
respect to $g_{l}$ if the $A$-tensor of $q_{G\times M}$ is nonzero on $\hat{P%
}$.

\item[(iii)] If $G=S^{3}$, then the curvature of $dq_{G\times M}(\hat{P})$
is positive if the projection of $P$ onto $TO_{G}$ is nondegenerate.

\item[(iv)] If the curvature of $\hat{P}$ is $0$ and $A^{q_{G\times M}}$
vanishes on $\hat{P}$, then the curvature of $dq_{G\times M}(\hat{P})$ is $0$%
.
\end{description}
\end{proposition}

\begin{remark}
According to \cite{Tapp2}, no new positive curvature can be created via (ii)
if $M$ is a Lie group with a biinvariant metric.
\end{remark}

Following \cite{PetWilh} and \cite{Wilh2}, our computations will be based on
deformations of the biinvariant metric on $Sp\left( 2\right) .$ The
biinvariant metric induced by $Sp\left( 2\right) \subset S^{7}(1)\times
S^{7}(1)$ is called $b.$ The biinvariant metric we use is scaled so that the
vectors $x^{2,0}$ etc. have unit length. Thus we use $\frac{1}{2}b,$ also
called $b_{\frac{1}{\sqrt{2}}}$in \cite{PetWilh} and \cite{Wilh2}, which is
induced by $Sp\left( 2\right) \subset S^{7}(\frac{1}{\sqrt{2}})\times S^{7}(%
\frac{1}{\sqrt{2}})$, where $S^{7}(\frac{1}{\sqrt{2}})$ is the sphere of
radius $\frac{1}{\sqrt{2}}$.

The effect of the Cheeger perturbation $A^{h_{1}}\times A^{h_{2}}$ is to
scale $V_{1}$ and $V_{2}$ and to preserve the splitting $V_{1}\oplus
V_{2}\oplus H$ and $\frac{1}{2}b|_{H}$. The amount of the scaling is $<1$
and converges to $1$ as the scale on the $S^{3}$-factor in $\left(
S^{3}\times S^{3}\right) \times Sp(2)$ converges to $\infty $ and converges
to $0$ when the $S^{3}\times S^{3}$ factor is scaled to a point. We will
call the resulting scales on $V_{1}$ and $V_{2}$, $\nu _{1}$ and $\nu _{2}$.
To simplify the exposition, we set $\nu =\nu _{1}=\nu _{2}$ and call the
resulting metric $g_{\nu }$.

It follows that $g_{\nu }$ is the restriction to $Sp(2)$ of the product
metric $S_{\nu }^{7}\times S_{\nu }^{7}$ where $S_{\nu }^{7}$ denotes the
Berger metric obtained from $S^{7}(\frac{1}{\sqrt{2}})$ by scaling the
fibers of $h$ by $\nu \sqrt{2}$.

The following results can be found in \cite{PetWilh}.

\begin{proposition}
\label{S^3 times S^3 times S^3 times S^3} Let $g_{\nu ,l}$ denote a metric
obtained from the biinvariant metric on $Sp(2)$ via Cheeger's method using
the $S^{3}\times S^{3}\times S^{3}\times S^{3}$-action, $A^{u}\times
A^{d}\times A^{h_{1}}\times A^{h_{2}}$.

Then $A^{u}\times A^{d}\times A^{h_{1}}\times A^{h_{2}}$ is by isometries
with respect to $g_{\nu ,l}$. In particular, $A_{2,-1}$ is by isometries
with respect to $g_{\nu ,l}$, and hence $g_{\nu ,l}$ induces a metric of
nonnegative curvature on the Gromoll-Meyer sphere, $\Sigma ^{7}$.
\end{proposition}

\begin{proposition}
\label{new horizontal space} Let $A_{H}:H\times M\longrightarrow M$ be an
action that is by isometries with respect to both $g_{\infty }$ and $g_{l}$.
Let $H_{A_{H}}$ denote the distribution of vectors that are perpendicular to
the orbits of $A_{H}$.

$P$ is in $H_{A_{H}}$ with respect to $g_{\infty }$ if and only if $%
dq_{G\times M}(\hat{P})$ is in $H_{A_{H}}$ with respect to $g_{l}$. In fact, 
\begin{equation*}
g_{\infty }\left( u,w\right) =g_{l}\left( u,dq_{G\times M}\left( \hat{w}%
\right) \right)
\end{equation*}%
for all $u,w\in TM.$
\end{proposition}

\noindent \textbf{Notational Convention: }Let \label{notational convention
copy(1)}%
\begin{equation*}
q_{G\times M}:G\times \left( M,g_{\infty }\right) \longrightarrow \left(
M,g_{l}\right)
\end{equation*}%
be a Cheeger submersion. Suppose that $\pi :M\longrightarrow B$ is a
Riemannian submersion with respect to both $g_{\infty }$ and $g_{l}$. It
follows that $z$ is horizontal for $\pi :M\longrightarrow B$ with respect to 
$g_{\infty }$ if and only if $dq_{G\times M}\left( \hat{z}\right) $ is
horizontal for $\pi $ with respect to $g_{l}.$ To keep the notation simpler,
we can think of this correspondence as a parameterization of the horizontal
space, $H_{\pi ,g_{l}},$ of $\pi $ with respect to $g_{l}$ by the horizontal
space, $H_{\pi ,\ g_{\infty }}$ of $\pi $ with respect to $g_{\infty }$. We
can then denote vectors and planes in $H_{\pi ,\ g_{l}}$ by the
corresponding vectors and planes in $H_{\pi ,\ g_{\infty }}$. We will do
this for the $\left( A^{u}\oplus A^{d}\right) $--Cheeger deformation, but
not for the $\left( A^{h_{1}}\oplus A^{h_{2}}\right) $--Cheeger deformation.

Note that if $t\in \lbrack 0,\frac{\pi }{4})$ then the orthogonal projection 
$p_{V_{h},V_{\tilde{h}}}:V_{h}\longrightarrow V_{\tilde{h}}$ with respect to
the unit metric on $S^{7}$ is an isomorphism. In fact the matrix of $%
p_{V_{h},V_{\tilde{h}}}$ with respect to the ordered bases $\mathfrak{v},$ $%
\vartheta _{1},$ $\vartheta _{2}$ and $\mathfrak{v},$ $\tilde{\vartheta}%
_{1}, $ $\tilde{\vartheta}_{2}$ is 
\begin{equation*}
\left( 
\begin{array}{ccc}
1 & 0 & 0 \\ 
0 & \cos (2t) & 0 \\ 
0 & 0 & \cos (2t)%
\end{array}%
\right) .
\end{equation*}%
The horizontal space of $q_{2,-1}$ with respect to $g_{\nu }$ is given by

\begin{proposition}
\label{H_p_m,-1} \cite{Wilh2} For $t\in \lbrack 0,\frac{\pi }{4})$ the
horizontal space of $q_{2,-1}$ with respect to $g_{\nu }$ at the
representative point $\left( N_{1}p,N_{2}\right) $ is spanned by 
\begin{eqnarray*}
&&\left\{ x^{2,0},\;y^{2,0},\;\left( \eta _{1},\eta _{1}+\tan (2t)\frac{%
\vartheta _{1}}{\nu ^{2}}\right) ,\;\left( \eta _{2},\eta _{2}+\tan (2t)%
\frac{\vartheta _{2}}{\nu ^{2}}\right) ,\right. \\
&&\left( -\frac{\mathfrak{v}}{\nu ^{2}},\frac{\mathfrak{v}}{\nu ^{2}}-\frac{%
p_{V_{h},V_{\tilde{h}}}^{-1}(\bar{p}\alpha pN_{2})}{\nu ^{2}}\right)
,\;\left( -\frac{\vartheta _{1}}{\nu ^{2}},\frac{\vartheta _{1}}{\nu ^{2}}-%
\frac{p_{V_{h},V_{\tilde{h}}}^{-1}(\bar{p}\gamma _{1}pN_{2})}{\nu ^{2}}%
\right) , \\
&&\left. \left( -\frac{\vartheta _{2}}{\nu ^{2}},\frac{\vartheta _{2}}{\nu
^{2}}-\frac{p_{V_{h},V_{\tilde{h}}}^{-1}(\bar{p}\gamma _{2}pN_{2})}{\nu ^{2}}%
\right) \right\}
\end{eqnarray*}
\end{proposition}

\textbf{Notation:} We will call the seven vectors in Proposition \ref%
{H_p_m,-1}, $x^{2,0},\;y^{2,0},\;\eta _{1}^{2,0},$ $\eta _{2}^{2,0}$, $%
\mathfrak{v}^{2,-1}$, $\vartheta _{1}^{2,-1},$ and $\vartheta _{2}^{2,-1}$
respectively. We will call the span of the first four $H_{2,-1}$ and the
span of the last three $V_{2,-1}$.

Although our partial framing of $TSp\left( 2\right) $ is well adapted to
study the Gromoll-Meyer sphere it is neither left nor right invariant. For
example, the left invariant field that equals $x^{2,0}$ at 
\begin{equation*}
Q=\left( \left( 
\begin{array}{c}
\cos t \\ 
\alpha \sin t%
\end{array}%
\right) ,\left( 
\begin{array}{c}
\alpha \sin t \\ 
\cos t%
\end{array}%
\right) \right)
\end{equation*}%
is 
\begin{eqnarray*}
\left( L_{Q}\right) _{\ast }\left( \left( 
\begin{array}{c}
0 \\ 
\alpha%
\end{array}%
\begin{array}{c}
\alpha \\ 
0%
\end{array}%
\right) \right) &=&\left( 
\begin{array}{c}
\cos t \\ 
\alpha \sin t%
\end{array}%
\begin{array}{c}
\alpha \sin t \\ 
\cos t%
\end{array}%
\right) \left( 
\begin{array}{c}
0 \\ 
\alpha%
\end{array}%
\begin{array}{c}
\alpha \\ 
0%
\end{array}%
\right) \\
&=&\left( 
\begin{array}{c}
-\sin t \\ 
\alpha \cos t%
\end{array}%
\begin{array}{c}
\alpha \cos t \\ 
-\sin t%
\end{array}%
\right) \\
&=&x^{2,0}.
\end{eqnarray*}%
Since $\alpha $ varies, $x^{2,0}$ is not left invariant.

Note also that one should think of $\left\{ \eta _{1},\eta _{2}\right\} $ as
defining a global distribution rather than as global vector fields. The fact
that $S^{2}$ is not parallelizable corresponds to the fact that $\gamma _{1}$
is not canonically determined by $\alpha .$ Consequently, any statement that
we make about a single unit $\gamma \in \mathrm{span}\left\{ \gamma
_{1},\gamma _{2}\right\} $ is valid for any $\gamma \in \mathrm{span}\left\{
\gamma _{1},\gamma _{2}\right\} .$ Similarly any statement about a single
unit $\eta \in \mathrm{span}\left\{ \eta _{1},\eta _{2}\right\} $ is valid
for any $\eta \in \mathrm{span}\left\{ \eta _{1},\eta _{2}\right\} ,$ and
any statement about a single $\vartheta \in \mathrm{span}\left\{ \vartheta
_{1},\vartheta _{2}\right\} $ is valid for any $\vartheta \in \mathrm{span}%
\left\{ \vartheta _{1},\vartheta _{2}\right\} .$

\section{Zero Curvatures of $\Sigma ^{7}$}

In this final review section we discuss the zero curvatures of $\left(
\Sigma ^{7},g_{\nu ,l}\right) .$ The description that we give is more
geometric than that of \cite{Wilh2}. We give a brief idea of why the zeros
occur, but for a full justification we combine \cite{Wilh2} and \cite{Tapp2}
with new computations of the zero curvatures when $t=0.$ These were not
given in \cite{Wilh2} because they were not needed. We give them here to
fully justify our description and also because they give a flavor of some of
the important issues of \cite{Wilh2}.

From Proposition 3.1, we see that a $0$--plane for $g_{\nu ,l}$ must have a
degenerate projection onto the tangent spaces to the orbits of all four $%
S^{3}$--actions, $A^{u},A^{d},A^{h_{1}},$ and $A^{h_{2}}.$

There is a vector field tangent to $Sp\left( 2\right) $ that is normal to
the orbits of all four actions. We call this field $\zeta .$When restricted
to an $S^{7}$--factor, $\zeta $ is the field that is normal to the $%
S^{3}\times S^{3}$s in the join decomposition $S^{7}=S^{3}\ast S^{3},$ that
corresponds to writing a point in $S^{7}$ as 
\begin{equation*}
\left( 
\begin{array}{c}
a \\ 
c%
\end{array}%
\right) \text{ with }a,c\in \mathbb{H}.
\end{equation*}%
$\zeta $ is of course in $\mathrm{span}\left\{ x^{2,0},y^{2,0}\right\} ,$
but the combination is quite complicated. 
\begin{equation*}
\zeta =\frac{\left( \sin 2t\cos 2\theta \right) \,x^{2,0}-\left( \sin
2\theta \right) \,y^{2,0}}{\sqrt{\sin ^{2}2t\cos ^{2}2\theta +\sin
^{2}2\theta }}.
\end{equation*}%
So $\zeta $ does not have much to do with our join decomposition $S^{4}=S_{%
\mathbb{R}}^{1}\ast S_{\func{Im}}^{2}.$ Rather it is the geodesic field that
is the gradient of the distance from the point where $\left( t,\theta
\right) =\left( 0,0\right) .$ In our coordinate system for $S^{4},$ the
antipodal point to $\left( t,\theta \right) =\left( 0,0\right) $ is $\left(
t,\theta \right) =\left( 0,\frac{\pi }{2}\right) .$ So $\zeta $ is the field
that is tangent to the meridians between these two points. Thus $\zeta $ is
multivalued at the two poles. This corresponds to the fact that our formula
for $\zeta $ is $\frac{0}{0}$ at these poles.

Unfortunately $\zeta $ is everywhere normal to the Gromoll-Meyer action.
Fortunately the vectors%
\begin{equation*}
ZV\equiv \left\{ U\in TSp\left( 2\right) |\mathrm{curv}_{b}\left( \zeta
,U\right) =0\right\}
\end{equation*}%
are typically not horizontal for the Gromoll-Meyer submersion $p_{2,-1}$.
However, from [Tapp2] we know that every time a vector $U$ is horizontal for 
$p_{2,-1},$ we get a zero plane in $\Sigma ^{7},$ even with respect to $%
g_{\nu ,l}.$

The projection to $S^{4}$ of the points in $\Sigma ^{7}$ that have zero
curvature planes containing $\zeta $ are

\begin{theorem}
\label{zeta zeros in S^4}The points in $S^{4}$ over which there is a
horizontal vector for $q_{2,-1}:Sp\left( 2\right) \longrightarrow \Sigma
^{7} $ that is in $ZV$ are the meridians emmanating from $\left( t,\theta
\right) =\left( 0,0\right) $ that make an angle that is $\leq \frac{\pi }{6}$
with the meridians that go from $\left( t,\theta \right) =\left( 0,0\right) $
through $S_{\func{Im}}^{2}.$
\end{theorem}

The set is therefore $4$--dimensional with a four dimensional complement. In
[Wilh2] it is described as the sublevel set $L\left( t,\theta \right) \leq
1, $ where $L:S^{4}\longrightarrow \mathbb{R}$ is 
\begin{equation*}
L(t,\theta )=\left\{ 
\begin{array}{cl}
\frac{2\cos \left( 2t\right) \sin \left( 2\theta \right) }{\sqrt{\sin
^{2}2\theta +\sin ^{2}2t\cos ^{2}2\theta }} & \text{if }\left( t,\theta
\right) \neq \left( 0,0\right) \text{ or }\left( 0,\frac{\pi }{2}\right) \\ 
0 & \text{if }\left( t,\theta \right) =\left( 0,0\right) \text{ or }\left( 0,%
\frac{\pi }{2}\right)%
\end{array}%
\right. .
\end{equation*}%
Combining this with the main theorem of \cite{Tapp2} and Proposition \ref%
{exceptional zeros} below gives us Theorem \ref{zeta zeros in S^4}.

Of course there can also be zero planes that do not contain $\zeta .$ Since $%
\zeta $ (generically) spans the orthogonal complement of the orbit of the $%
S^{3}\times S^{3}\times S^{3}\times S^{3}$ action, such planes necessarily
have a nondegenerate projection onto the tangent space to the entire orbit
of $S^{3}\times S^{3}\times S^{3}\times S^{3},$ but a degenerate projection
onto the orbit of each individual $S^{3}$--action. In addition, the plane
must have zero curvature for the biinvariant metric and be horizontal for
the Gromoll-Meyer submersion, it is not surprising that such planes are
fairly rare.

\begin{theorem}
\label{zeros in S^4} The set of points $\mathcal{Z}$ in $S^{4}$ over which
there is a $0$--plane in $\Sigma ^{7}$is the union of the points described
in Theorem \ref{zeta zeros in S^4} with the points where $\cos 2\theta =0.$
\end{theorem}

To get a quick idea of how these other zeros occur, we point out that the
horizontal vectors for $q_{2,-1}:Sp\left( 2\right) \longrightarrow \Sigma
^{7}$ that are also perpendicular to the orbits of $A^{h_{1}}\oplus
A^{h_{2}} $ are 
\begin{equation*}
\mathrm{span}\left\{ x^{2,0},y^{2,0}\right\}
\end{equation*}%
when $t>0$ and 
\begin{equation*}
\mathrm{span}\left\{ x^{2,0},y^{2,0},\eta _{1}^{2,0},\eta _{2}^{2,0}\right\}
\end{equation*}%
when $t=0.$

Since $\zeta \in \mathrm{span}\left\{ x^{2,0},y^{2,0}\right\} $, the issue
boils down to its complementary vector $\xi \equiv \alpha \zeta $ in $%
\mathrm{span}\left\{ x^{2,0},y^{2,0}\right\} .$ Fortunately $\xi $ does have
a projection onto the tangent space to the orbits of $A^{u}\oplus A^{d}.$
Combining this with the other requirements for zero planes it is argued in 
\cite{Wilh2}, that the points in $S^{4}$ over which there are $0$ planes are
those described in the previous theorem.

The actual zero planes have the form

\begin{theorem}
If $P$ is a plane with $0$ curvature at a point where $\left( t,\sin 2\theta
\right) \neq \left( 0,0\right) $ and $\cos 2\theta \neq 0$, then $P$ has the
form 
\begin{equation*}
P=\mathrm{span}\left\{ \zeta ,W\right\}
\end{equation*}%
where 
\begin{equation*}
W\in V_{1}\oplus V_{2}.
\end{equation*}%
If $\zeta $ has the form 
\begin{equation*}
\zeta =x^{2,0}\cos \varphi +y^{2,0}\sin \varphi ,
\end{equation*}%
then $W$ has the form 
\begin{equation*}
\cos \lambda \left( \frac{\mathfrak{v}}{\nu ^{2}},\frac{\mathfrak{v}}{\nu
^{2}}\right) +\sin \lambda \left( \frac{\ddot{\vartheta}_{1}}{\nu ^{2}},%
\frac{\ddot{\vartheta}_{1}}{\nu ^{2}}\cos \psi +\frac{\ddot{\vartheta}_{2}}{%
\nu ^{2}}\sin \psi \right)
\end{equation*}%
where 
\begin{equation*}
\psi =\pi -2\varphi ,
\end{equation*}

\noindent $\ddot{\vartheta}_{1},\ddot{\vartheta}_{2}\in \mathrm{span}\left\{
\vartheta _{1},\vartheta _{2}\right\} ,$ correspond to spherical
combinations $\ddot{\gamma}_{1},\ddot{\gamma}_{2}$ of $\left\{ \gamma
_{1},\gamma _{2}\right\} $ that satisfy $\alpha \ddot{\gamma}_{1}=\ddot{%
\gamma}_{2},$ and $\left( \cos \lambda ,\sin \lambda \right) $ is the point
in the first quadrant of $\mathbb{R}^{2}$ that is on the unit circle and on
the ellipse parameterized by\addtocounter{algorithm}{1} 
\begin{equation}
\sigma \longmapsto \left( \frac{\cos \sigma }{2},\frac{\sin \sigma }{L\left(
t,\theta \right) }\right) .  \label{ellipse}
\end{equation}

When $\cos 2\theta =0,$ there are zero planes of the form described above.
In addition there are zero planes of the form%
\begin{equation*}
P=\mathrm{span}\left\{ x^{2,0},W\right\}
\end{equation*}%
where 
\begin{equation*}
W=\left( \frac{\mathfrak{v}}{\nu ^{2}}\frac{1}{2}+\frac{\vartheta }{\nu ^{2}}%
\frac{\sqrt{3}}{2},\frac{\mathfrak{v}}{\nu ^{2}}\frac{1}{2}-\frac{\vartheta 
}{\nu ^{2}}\frac{\sqrt{3}}{2}\right) .
\end{equation*}
\end{theorem}

\begin{remark}
There is a further conjugacy condition for a vector of the form of $W$ to
actually be horizontal for $q_{2,-1}:\Sigma ^{7}\longrightarrow S^{4}.$
Because of this, in a given fiber of $\Sigma ^{7}\longrightarrow S^{4}$ over
a point in $\mathcal{Z}\subset S^{4}$ most points do not in fact have zero
curvatures, and at most points where there is a zero curvature, there is
just one zero curvature. None of these issues will be important for us, so
we will not review them.
\end{remark}

\begin{remark}
The unit circle and the ellipse in question do not intersect when $L\left(
t,\theta \right) >1.$ When this happens the corresponding $W$s are not
horizontal for the Gromoll-Meyer submersion.
\end{remark}

\subsection{Zero Curvatures at $t=0$}

When $t=0,$ all points have positive curvature except for certain points
with $\cos 2\theta \sin 2\theta =0.$ The lack of $0$--planes in $\Sigma ^{7}$
is caused by the zero planes of $Sp\left( 2\right) $ not intersecting the
horizontal distribution of the Gromoll--Meyer submersion. The reason for
this is the fact that the unit circle and the ellipse in (\ref{ellipse}) do
not intersect when $L\left( t,\theta \right) >1.$ So the corresponding $W$s
are not horizontal for the Gromoll-Meyer submersion. For example, if $t=0$
and $\sin 2\theta \neq 0,$ then $\zeta =-y^{2,0}$ and $L\left( 0,\theta
\right) =2.$ For $\mathrm{span}\left\{ y^{2,0},W\right\} $ to have $0$
curvature, with respect to $g_{\nu }$, $W$ must have the form 
\begin{equation*}
\cos \lambda \left( \mathfrak{v},\mathfrak{v}\right) +\sin \lambda \left(
\vartheta ,\vartheta \right) =\left( N_{1}\beta p,N_{2}\beta \right)
\end{equation*}%
for some purely imaginary $\beta \in S^{3}\subset \mathbb{H}.$

When $t=0,$ we have $V_{h}=V_{\tilde{h}}$ so none of the horizontal vectors 
\begin{equation*}
\left( N_{1}\beta p,-N_{2}\beta +p_{V_{h},V_{\tilde{h}}}^{-1}\left( \bar{p}%
\beta pN_{2}\right) \right)
\end{equation*}%
can have the required form 
\begin{equation*}
\left( N_{1}\beta p,N_{2}\beta \right) .
\end{equation*}

When $\left( t,\theta \right) =\left( 0,0\right) $ or $\left( 0,\frac{\pi }{2%
}\right) $, the definition of $\zeta $ is ambiguous$.$ The definition of $%
x^{2,0}$ is also ambiguous since the $\alpha $ coordinate is nonexistent at $%
t=0.$ In fact, the three vectors $x^{2,0},$ $\eta _{1}^{2,0},$ and $\eta
_{2}^{2,0}$ project under $p_{2,-1}\circ q_{2,-1}$ to a basis for the normal
space of $S_{\mathbb{R}}^{1}\subset S^{4}.$ Declaring that a particular
purely, imaginary unit quaternion is \textquotedblleft $\alpha $%
\textquotedblright\ amounts to declaring that a particular unit normal
vector to $S_{\mathbb{R}}^{1}$ is \textquotedblleft $x^{2,0}$%
\textquotedblright . This choice is somewhat irrelevant since, on the level
of $S^{4},$ the isometric action $A^{h_{2}}$ fixes $S_{\mathbb{R}}^{1}$ and
acts transitively on the normal space. Thus, to find $0$ curvatures when $%
\left( \sin 2t,\sin 2\theta \right) =\left( 0,0\right) ,$ we only need to
consider planes of the form%
\begin{equation*}
P=\mathrm{span}\left\{ z,W\right\}
\end{equation*}%
where $z\in \mathrm{span}\left\{ x^{,2,0},y^{2,0}\right\} $ and $W\in
V_{2,-1}.$ There will of course be other $0$--planes, but they are the
images of these under $A^{h_{2}}.$

Since $L\left( 0,\theta \right) \equiv 2,$ when $\theta \neq 0,\frac{\pi }{2}%
,$ there are no $0$--curvatures when $t=0,$ provided $\theta $ is not $0,%
\frac{\pi }{4},\frac{\pi }{2},$ or $\frac{3\pi }{4}.$ The details can be
found in \cite{Wilh2}, but the basic reason is contained in the remark
above, when $\theta \neq 0,\frac{\pi }{2},$ then $\zeta =y^{2,0},$ and the $%
W $s that together with $y$ form $0$ planes are not horizontal at $t=0.$

The structure of the $0$--planes when $\left( t,\theta \right) =\left(
0,0\right) $ or $\left( 0,\frac{\pi }{2}\right) $ was claimed in \cite[p. 556%
]{Wilh2} to be

\begin{proposition}
\label{exceptional zeros}Let 
\begin{equation*}
\zeta _{\varphi }=x^{2,0}\cos \varphi +y^{2,0}\sin \varphi
\end{equation*}%
for some $\varphi \in \left[ \frac{-\pi }{2},\frac{\pi }{2}\right] .$ If $%
\left( t,\theta \right) =\left( 0,0\right) $ or $\left( 0,\frac{\pi }{2}%
\right) $ and $\left\vert \sin \varphi \right\vert \leq \frac{1}{2}$, then
there are values of $p$ for which $\zeta _{\varphi }$ is in $0$--planes of
the form 
\begin{equation*}
P=\mathrm{span}\left\{ \zeta _{\varphi },W\right\}
\end{equation*}%
where 
\begin{equation*}
W=\cos \lambda \left( \mathfrak{v},\mathfrak{v}\right) +\sin \lambda \left(
\vartheta _{1},\vartheta _{1}\cos \psi +\vartheta _{2}\sin \psi \right)
\end{equation*}%
and 
\begin{equation*}
\psi =\pi -2\varphi .
\end{equation*}%
Any other $0$--plane is the image of one of these under $A^{h_{2}}.$
\end{proposition}

The details of this were not given in \cite{Wilh2}, since it was not crucial
to the goal of that paper. Since we will need to use it, we will prove it
here.

The value of $\cos \lambda $ is determined by $\varphi ;$ the relationship
can be inferred from our proof.

\begin{proof}
As explained in \cite{Wilh2} it is enough to consider planes of the form 
\begin{equation*}
P=\mathrm{span}\left\{ z,W\right\}
\end{equation*}%
where $z$ is horizontal for $p_{2,-1}:\Sigma ^{7}\longrightarrow S^{4}$ and $%
W\in V_{1}\oplus V_{2}$ is horizontal for $q_{2,-1}:Sp\left( 2\right)
\longrightarrow \Sigma ^{7}.$ Since $t=0,$ we can use the isometries $%
A^{h_{2}}$ to further reduce our consideration to planes with $z\in \mathrm{%
span}\left\{ x^{2,-1},y^{2,-1}\right\} .$ In other words we may replace $z$
with 
\begin{equation*}
\zeta _{\varphi }=x^{2,0}\cos \varphi +y^{2,0}\sin \varphi .
\end{equation*}%
Using Proposition 4.6 in \cite{Wilh2}, this then forces $W$ to have the form
listed in the statement. It only remains to check what $W$s are horizontal
for $q_{2,-1}$ when $t=0.$

We explained above that when $\zeta _{\varphi }=y^{2,0},$ the required $W$
is not horizontal.

When $\zeta _{\varphi }=x^{2,0},$ the required $W$ is 
\begin{equation*}
W=\left( \frac{\mathfrak{v}}{2}+\frac{\sqrt{3}}{2}\vartheta ,\frac{\mathfrak{%
v}}{2}-\frac{\sqrt{3}}{2}\vartheta \right) .
\end{equation*}%
We see that $W$ can be realized in the form 
\begin{equation*}
\left( N_{1}\beta p,-N_{2}\beta +p_{V_{h},V_{\tilde{h}}}^{-1}\left( \bar{p}%
\beta pN_{2}\right) \right)
\end{equation*}%
by choosing 
\begin{equation*}
\beta =\left( \frac{1}{2}\alpha +\frac{\sqrt{3}}{2}\gamma \right)
\end{equation*}%
and $p$ so that 
\begin{equation*}
\bar{p}\beta p=\alpha .
\end{equation*}%
Now we consider the general problem of realizing 
\begin{equation*}
W=\cos \lambda \left( \mathfrak{v},\mathfrak{v}\right) +\sin \lambda \left(
\vartheta _{1},\vartheta _{1}\cos \psi +\vartheta _{2}\sin \psi \right)
\end{equation*}%
in the form 
\begin{equation*}
\left( N_{1}\beta p,-N_{2}\beta +p_{V_{h},V_{\tilde{h}}}^{-1}\left( \bar{p}%
\beta pN_{2}\right) \right) .
\end{equation*}

The first coordinate, 
\begin{equation*}
\mathfrak{v}\cos \lambda +\sin \lambda \vartheta _{1},
\end{equation*}%
of $W$ forces us to set 
\begin{equation*}
\beta =\alpha \cos \lambda +\gamma _{1}\sin \lambda .
\end{equation*}

The question then becomes whether there is a choice of $p$ that will achieve
the desired second coordinate. The second coordinate of $W$ can be written %
\addtocounter{algorithm}{1}%
\begin{equation}
\mathfrak{v}\cos \lambda +\left( \vartheta _{1}\cos \psi +\vartheta _{2}\sin
\psi \right) \sin \lambda =N_{2}\alpha \cos \lambda +N_{2}\gamma \sin
\lambda .  \label{2nd coord of w}
\end{equation}%
for $\gamma =\gamma _{1}\cos \psi +\gamma _{2}\sin \psi .$ On the other hand
if we set 
\begin{equation*}
\bar{p}\beta p=\alpha \cos \sigma +\gamma ^{\prime }\sin \sigma
\end{equation*}%
then%
\begin{eqnarray*}
-N_{2}\beta +p_{V_{h},V_{\tilde{h}}}^{-1}\left( \bar{p}\beta pN_{2}\right)
&=&-N_{2}\left( \alpha \cos \lambda +\gamma _{1}\sin \lambda \right) +\left( 
\bar{p}\beta p\right) N_{2} \\
&=&-N_{2}\alpha \cos \lambda -N_{2}\gamma _{1}\sin \lambda +\left( \alpha
\cos \sigma +\gamma ^{\prime }\sin \sigma \right) N_{2} \\
&=&N_{2}\left( -\alpha \cos \lambda +\alpha \cos \sigma \right) +N_{2}\left(
-\gamma _{1}\sin \lambda +\gamma ^{\prime }\sin \sigma \right)
\end{eqnarray*}%
since $N_{2}$ is real and therefore commutes with all quaternions.

Equating this with \ref{2nd coord of w} gives us the equations%
\begin{eqnarray*}
\alpha \cos \lambda &=&-\alpha \cos \lambda +\alpha \cos \sigma , \\
\gamma \sin \lambda &=&-\gamma _{1}\sin \lambda +\gamma ^{\prime }\sin \sigma
\end{eqnarray*}

or 
\begin{eqnarray*}
\cos \sigma &=&2\cos \lambda , \\
\gamma ^{\prime }\sin \sigma &=&\left( \gamma +\gamma _{1}\right) \sin
\lambda .
\end{eqnarray*}%
We can always choose $p$ so that $\gamma ^{\prime }$ points in the direction
of $\gamma +\gamma _{1}.$ The issue is that since $\gamma +\gamma _{1}$ has
a variable length, sometimes there are solutions and sometimes there are
not. In fact, if we set 
\begin{equation*}
L=\left\vert \gamma +\gamma _{1}\right\vert ,
\end{equation*}%
then our equations become 
\begin{eqnarray*}
\cos \lambda &=&\frac{\cos \sigma }{2}, \\
\sin \lambda &=&\frac{\sin \sigma }{L}.
\end{eqnarray*}%
So the question becomes whether or not the unit circle $\left( \cos \lambda
,\sin \lambda \right) $ intersects the ellipse whose parametrization is 
\begin{equation*}
\sigma \longmapsto \left( \frac{\cos \sigma }{2},\frac{\sin \sigma }{L}%
\right) .
\end{equation*}%
Thus when $\left\vert L\right\vert \leq 1$ there are solutions and when $L>1$
there are no solutions. So it remains to analyze how $L$ depends on $\varphi
.$

Since 
\begin{eqnarray*}
\gamma &=&\gamma _{1}\cos \psi +\gamma _{2}\sin \psi , \\
\psi &=&\pi -2\varphi ,
\end{eqnarray*}%
we have%
\begin{eqnarray*}
\gamma &=&\gamma _{1}\cos \left( \pi -2\varphi \right) +\gamma _{2}\sin
\left( \pi -2\varphi \right) \\
&=&-\gamma _{1}\cos 2\varphi +\gamma _{2}\sin 2\varphi .
\end{eqnarray*}%
Thus 
\begin{eqnarray*}
L^{2} &=&\left\vert \gamma _{1}+\gamma \right\vert ^{2} \\
&=&1-2\cos 2\varphi +1.
\end{eqnarray*}%
So our condition, $L\leq 1$ for $0$ curvature is 
\begin{eqnarray*}
2-2\cos 2\varphi &\leq &1 \\
-2\cos 2\varphi &\leq &-1
\end{eqnarray*}%
or%
\begin{equation*}
\cos 2\varphi \geq \frac{1}{2}
\end{equation*}%
Keeping in mind that $\varphi \in \left[ -\frac{\pi }{2},\frac{\pi }{2}%
\right] ,$ we get 
\begin{equation*}
-\frac{\pi }{3}\leq 2\varphi \leq \frac{\pi }{3}
\end{equation*}%
or%
\begin{equation*}
-\frac{1}{2}\leq \sin \varphi \leq \frac{1}{2}.
\end{equation*}
\end{proof}

\section{Further Summary}

Having scaled the fibers of $\Sigma ^{7}\longrightarrow S^{4},$ we can get
pointwise positive curvature along any \emph{single} (formerly) flat torus
via a conformal change. The idea is that the Hessian of the conformal factor
needs to cancel the $s^{2}$ term in Equation \ref{curv formula}, %
\addtocounter{algorithm}{1}%
\begin{equation}
\mathrm{curv}_{g_{s}}\left( X,W\right) =-s^{2}\left( D_{X}\left( \left\vert
H_{w}\right\vert D_{X}\left\vert H_{w}\right\vert \right) \right)
+s^{4}\left( D_{X}\left\vert H_{w}\right\vert \right) ^{2}.
\label{curv formula-2}
\end{equation}%
Unfortunately, there is no conformal change that will produce pointwise
positive curvature along all of the tori simultaneously. The problem is that
only the component of $W$ that is horizontal for $\Sigma ^{7}\longrightarrow
S^{4}$ appears in our curvature formula. Along any one torus, the vector $%
H_{w}$ is a Killing field for the $SO\left( 3\right) $--action on $S^{4},$
but the precise Killing field, and more importantly the ratio $\frac{%
\left\vert H_{w}\right\vert }{\left\vert W\right\vert }$ varies from torus
to torus, so the size of the required Hessian varies as well.

This difficulty is overcome by using only a \textquotedblleft
partial\textquotedblright\ conformal change. The restriction of the metric
to the distribution $\mathrm{span}\left\{ \left( N\alpha p,N\alpha \right)
\right\} $ is not modified. The metric only changes on the orthogonal
complement of $\mathrm{span}\left\{ \left( N\alpha p,N\alpha \right)
\right\} .$ The details are carried out in Section 10.

A further difficulty is created by the fact that the pieces of the zero
locus with $L\left( t,\theta \right) \leq 1$ and $\cos 2\theta =0$ intersect
at certain points over points in $S^{4}$ where $t\geq \frac{\pi }{6}$ and $%
\cos 2\theta =0.$ A description of this intersection is given in \cite{Wilh2}%
, Theorem E(iv,v).

The difficulty this creates is that the natural choices of conformal factors
do not agree on this intersection.

To circumvent this difficulty, in Sections 7, 8, and the appendix we analyze
the effect on Equation \ref{curv formula-2} of running the $h_{2}$--Cheeger
perturbation for a long time. If $\nu $ is the parameter of this
perturbation, then it turns out that making $\nu $ small has the effect of
concentrating all of the terms on the right hand side of equation \ref{curv
formula-2}, $-s^{2}\left( D_{X}\left( \left\vert H_{w}\right\vert
D_{X}\left\vert H_{w}\right\vert \right) \right) +s^{4}\left(
D_{X}\left\vert H_{w}\right\vert \right) ^{2},$ around $t=0.$ We will make $%
\nu $ small enough so that we can choose a (partial) conformal factor that
is constant near the intersection of the two pieces of the zero locus, and
hence not have to worry about the conflict that the intersection creates.

The intersection of the two pieces of the zero locus, also creates a
notational conflict. To simplify the exposition we will henceforth write
explicitly only about the planes at points where $L\left( t,\theta \right)
\leq 1.$ With the obvious modifications in notation and a few
simplifications, the argument simultaneously will give us positive curvature
near the planes where $\cos 2\theta =0.$

Unfortunately, to really move the support of the partial conformal change
away from the intersection we have to make $\nu $ depend on $s.$ In the end
we will pick $\nu =O\left( s^{6/7}\right) .$ This means that our ultimate
metric is not obtained as an infinitesimal perturbation of any (known)
metric with nonnegative curvature. This fact will further complicate our
exposition. Before we can explain why, some further clarification is needed.

Imagine that we have a deformation in which all of the former zero curvature
planes, $\mathrm{span}\left\{ \zeta ,W\right\} ,$ are positively curved.
Next comes the daunting challenge of establishing that an entire
neighborhood (of uniform size) of these planes in the Grassmannian is
positively curved. We have to consider what happens when we move the base
point of our plane and also when we move the plane with out moving the base
point.

To deal with points that are close to, but not on the old zero locus, we
expand our definition of $W$ to include certain vectors in $TSp\left(
2\right) $ that are close to, but not on the old zero locus.

At points with $0$--curvature,

\begin{equation*}
W=\left( N_{1}\beta p,N_{2}\delta \right) ,
\end{equation*}%
is determined by the requirements that \textrm{curv}$_{\left( Sp\left(
2\right) ,g_{\nu ,l}\right) }\left( \zeta ,W\right) =0$ and that $W$ be
horizontal for the Gromoll-Meyer submersion. The points in $\left( \Sigma
^{7},g_{\nu ,l}\right) $ with positive curvature are images of points in $%
Sp\left( 2\right) $ at which no horizontal $W$ solves 
\begin{equation*}
\mathrm{curv}_{\nu ,l}\left( \zeta ,W\right) =0.
\end{equation*}

For the purpose of this discussion only, we require $\left\vert \beta
\right\vert =\left\vert \delta \right\vert =1$, and we let $Z^{4}$ be the
set of $\left( t,\theta \right) $ for which there is some zero plane in $%
\Sigma ^{7}.$ For $\left( t,\theta \right) \in Z^{4}$ the size of the $%
\gamma $--component of $\delta $ depends only on $\left( t,\theta \right) ,$
and not on $\left( \alpha ,p\right) .$ At points in $Sp\left( 2\right) $
with $\left( t,\theta \right) \in Z^{4},$ we let $W$ be any vector in $%
TSp\left( 2\right) $ proportional to such a $\left( N_{1}\beta p,N_{2}\delta
\right) ,$ with the size of the $\gamma $--component of $\delta $ determined
by $\left( t,\theta \right) $ and \textrm{curv}$_{\left( Sp\left( 2\right)
,g_{\nu ,l}\right) }\left( \zeta ,W\right) =0.$ Note that such $W$ are \emph{%
not} required to be horizontal for the Gromoll-Meyer submersion, there are
no such horizontal $W$s when $\left( \alpha ,p\right) $ is such that $\Sigma
^{7}$ is positively curved at $\left( t,\theta ,\alpha ,p\right) ,$ and $W$
is of course multivalued.

At points in $Sp\left( 2\right) $ with $\left( t,\theta \right) \notin Z^{4}$%
, we let $W$ be any vector of the form $\left( N_{1}\beta p,N_{2}\delta
\right) $ with \textrm{curv}$_{\left( Sp\left( 2\right) ,g_{\nu ,l}\right)
}\left( \zeta ,W\right) =0$ and $\beta ,\delta \in \mathrm{span}\left\{
\gamma _{1},\gamma _{2}\right\} .$ Of course, when $\left( t,\theta \right)
\notin Z^{4}$, $W$ is \emph{never }horizontal for the Gromoll--Meyer
submersion.

In all of our subsequent statements we assume that \textrm{span}$\left\{
\zeta ,W\right\} $ is any one of these planes, whether or not it corresponds
to a zero plane in $\Sigma ^{7}.$ In this way we will only have to worry
about deforming our planes within the fibers of the Grassmannian.

We get positive curvature on the Gromoll-Meyer sphere by proving

\begin{theorem}
\label{Positive Neighborhood}There is a neighborhood $N$ of the set of all $%
\left\{ \zeta ,W\right\} $ in the Grassmannian of $Sp\left( 2\right) $, a
choice of $\left( \nu ,l\right) $, and a deformation $g_{new}$ of $g_{\nu
,l},$ that is invariant under the Gromoll-Meyer action, so that 
\begin{equation*}
\mathrm{curv}_{g_{new}}|_{N}>0
\end{equation*}%
and $\mathrm{curv}_{g_{new}}>0$ on horizontal planes in the complement of $%
N. $ In particular, $\mathrm{curv}_{g_{new}}>0$ on all horizontal planes.
\end{theorem}

We can now explain why the fact that our deformation is not infinitesimal
will further complicate our exposition.

If our deformation is infinitesimal, then we only have to understand the
\textquotedblleft quadratic perturbations\textquotedblright\ of our planes
in the Grassmannian.

To explain what this means precisely, let $\left\{ g_{s}\right\} $ be a $%
C^{\infty }$ family of metrics on a compact manifold with $g_{0}$ having
nonnegative curvature, and let any zero curvature plane with respect to $%
g_{0}$ be represented by $\mathrm{span}\left\{ \zeta ,W\right\} .$ We
represent a general plane near $\mathrm{span}\left\{ \zeta ,W\right\} $ in
the form $P=\mathrm{span}\left\{ \zeta +\sigma z,W+\tau V\right\} $ where $%
z\perp \zeta ,$ $V\perp W,$ and $\sigma ,\tau \in \mathbb{R}.$ The curvature
is then a quartic polynomial 
\begin{equation*}
P\left( \sigma ,\tau \right) =\mathrm{curv}\left( \zeta +\sigma z,W+\tau
V\right)
\end{equation*}%
in $\sigma $ and $\tau $.

Let $R^{s}$ be the curvature tensor of $g_{s}$, let $R^{\mathrm{old}}$ be
the curvature tensor with respect to $g_{0},$ and let $R^{\mathrm{diff,s}%
}=R^{s}-R^{\mathrm{old}}.$ Let $P^{\mathrm{old}}$ and $P^{\mathrm{diff,s}}$
have the obvious meaning. It is not hard to see

\begin{lemma}
\label{abstract qudrat def}At all points for which $g_{0}$ has some $0$%
--curvatures, $g_{s}$ is positively curved for all sufficiently small $s$
provided for all $0$--planes, $\mathrm{span}\left\{ \zeta ,W\right\} ,$with
respect to $g_{0},$ 
\begin{eqnarray*}
\frac{\partial }{\partial s}\mathrm{curv}^{\mathrm{diff,s}}\left( \zeta
,W\right) |_{s=0} &=&0 \\
\frac{\partial ^{2}}{\partial s^{2}}\mathrm{curv}^{\mathrm{diff,s}}\left(
\zeta ,W\right) |_{s=0} &>&0,\text{ } \\
P^{\mathrm{old}}\left( \sigma ,\tau \right) &>&0,
\end{eqnarray*}%
for all $\left( \sigma ,\tau \right) \neq \left( 0,0\right) ,$ and %
\addtocounter{algorithm}{1}%
\begin{eqnarray}
&&P_{Q}\left( \sigma ,\tau \right) \equiv \mathrm{curv}^{\mathrm{diff,s}%
}\left( \zeta ,W\right) +2\sigma R^{\mathrm{diff,s}}\left( \zeta
,W,W,z\right) +2\tau R^{\mathrm{diff,s}}\left( W,\zeta ,\zeta ,V\right) 
\notag \\
&&+\sigma ^{2}\mathrm{curv}^{\mathrm{old}}\left( z,W\right) +2\sigma \tau %
\left[ R^{\mathrm{old}}\left( \zeta ,W,V,z\right) +R^{\mathrm{old}}\left(
\zeta ,V,W,z\right) \right]  \label{estimate} \\
&&+\tau ^{2}\mathrm{curv}^{\mathrm{old}}\left( \zeta ,V\right)  \notag \\
&>&0  \notag
\end{eqnarray}%
for all sufficiently small $s$ and all $\sigma ,\tau \in \mathbb{R}.$
\end{lemma}

\begin{remark}
In the abstract setting of this lemma, we can not guarantee that the metrics
become positively curved because we know nothing about points that are close
to, but not on the point wise $0$--curvature locus of $g_{0}.$ This is not a
concern for the Gromoll-Meyer sphere because we have explained how to extend 
$\mathrm{span}\left\{ \zeta ,W\right\} $ to a family of planes in $Sp\left(
2\right) $ that includes all points in a neighborhood of the point wise $0$%
--locus (and also includes planes that are not horizontal for $\Sigma
^{7}\longrightarrow S^{4}).$

It should also be emphasized that we never establish the hypotheses of this
Lemma for our deformation. This is because our deformation is not
infinitesimal. We have never the less included the result because it
suggests a reasonable frame work for our computations.
\end{remark}

\begin{proof}
Since we do not use this, we give only a sketch of the proof.

The idea is that all of the other terms of $P\left( \sigma ,\tau \right) $
are either positive, $0$ or too small to matter. Since $g_{0}$ is
nonnegatively curved, the constant and linear terms are $0$ when $s=0.$
Since 
\begin{equation*}
\left\vert R^{\mathrm{diff,s}}\right\vert =O\left( s\right)
\end{equation*}%
the quadratic, cubic, and quartic terms of $P^{\mathrm{diff,s}}$ are smaller
than 
\begin{equation*}
sO\left( \sigma ^{2}+\sigma \tau +\tau ^{2}+\sigma \tau ^{2}+\sigma ^{2}\tau
+\sigma ^{2}\tau ^{2}\right)
\end{equation*}%
and hence are smaller than the corresponding terms of $P^{\mathrm{old}},$ if 
$s$ is sufficiently small.

On the one hand, the minimum of \ref{estimate} occurs in the region where 
\begin{equation*}
\max \left\{ \sigma ,\tau \right\} =O\left( s\right) ,
\end{equation*}%
and the size of this minimum is $O\left( s^{2}\right) ,$ so in this region
the cubic, and quartic terms of $P$ are too small too matter. On the other
hand, when $\max \left\{ \sigma ,\tau \right\} >O\left( s\right) ,$ the
linear terms have order 
\begin{equation*}
O\left( s\right) \max \left\{ \sigma ,\tau \right\}
\end{equation*}%
and since $P^{\mathrm{old}}\left( \sigma ,\tau \right) >0$ our curvature has
order 
\begin{equation*}
\geq O\left( \sigma ^{2}+\tau ^{2}\right) ,
\end{equation*}%
so the linear terms are too small to matter.
\end{proof}

Since our deformation is not infinitesimal, we will need to understand the
full polynomial $P\left( \sigma ,\tau \right) .$ In fact, all of the
possible values of all of the possible $P\left( \sigma ,\tau \right) $s only
describe the curvatures of an open dense subset in the Grassmannian. The
curvatures of the complement of this open dense set are described by
quadratic sub-polynomials of the $P\left( \sigma ,\tau \right) $ that are
proportional to sums of quartic, cubic, and pure quadratic terms of the
various $P\left( \sigma ,\tau \right) $s.

We will establish positive curvature on the Gromoll-Meyer sphere by showing
that all of these polynomials and sub-polynomials are positive.

Since $\nu =O\left( s^{6/7}\right) ,$ we still have that $s$ is much smaller
than $\nu .$ Morally this means that even though our deformation is not
infinitesimal, it is still fairly short term. The upshot of this is that
many of the higher order coefficients of $P^{\mathrm{diff}}$ will be too
small to matter. Those that are large will turn out to be comparable (in a
favorable way) to terms in $P^{\mathrm{old}}.$ We carry this out in sections
12 and 13.

It turns out that the metric we have outlined thus far is not actually
positively curved. The problem is that we do not actually get inequality \ref%
{estimate} everywhere. To correct this problem we make a further
modification of the metric in section 6. We call this the \textquotedblleft
redistribution\textquotedblright\ perturbation, and the resulting metric is $%
g_{\nu ,re}.$

Finally, there is one further Cheeger deformation that we use that was not
used in the earlier papers. The diagonal of $U$ and $D,$ which we will call $%
\Delta \left( U,D\right) .$ The purpose of this final Cheeger deformation is
that coupled with the $h_{1}$--Cheeger deformation it will allow us to see
that any plane whose projection onto the vertical space is nondegenerate is
positively curved. Although none of the original zero planes have this
feature, this observation will still be useful, since it will allow us to
immediately see that many of the possible perturbations of \textrm{span}$%
\left\{ \zeta ,W\right\} $ are positively curved. Modulo an identification
this diagonal perturbation is also used in \cite{EschKer}.

The positively curved metric that we obtain can probably be constructed via
several orderings of our deformations. However, to make our construction
unambiguous, we will adopt the following order:

\begin{description}
\item[(1)] The $\left( h_{1}\oplus h_{2}\right) $--Cheeger deformation

\item[(2)] The redistribution, described in section 6.

\item[(3)] The $\left( U\oplus D\right) $--Cheeger deformation.

\item[(4)] The scaling of the fibers.

\item[(5)] The partial conformal change.

\item[(6)] The $\Delta \left( U,D\right) $ Cheeger deformation and a further 
$h_{1}$--deformation.
\end{description}

We will accordingly discuss the redistribution perturbation next. Although
this is the logical order, it is not entirely clear that this order of
exposition is optimal. The real need for the redistribution only becomes
clear after one has done the subsequent computations; moreover, the desired
change in the curvature is also only clear after further computations have
been carried out. The reader may therefore wish to skip the next section,
until its need becomes clear. We have written the rest of the paper in a
sufficiently abstract form so that with the exception of subsection 8.1 this
should be possible. The exceptional subsection concerns an effect of the
redistribution that is not discussed in section 6.

Since our deformation is fairly short term, we have divided our curvature
computations into to those required to prove \ref{estimate} and those
required to understand the higher order terms of $P\left( \sigma ,\tau
\right) .$ The part necessary to prove \ref{estimate} is Sections 6--11, by
the end of which we will have proven

\begin{lemma}
(Main Lemma) \label{main lemma}Let $g_{\nu ,re,l}$ be the metric obtained
after carrying out the deformations $1$--$3$ above. Let $g_{\mathrm{new}}$
be the metric obtained after carrying out the deformations $1$--$5$ above.
Set 
\begin{equation*}
R^{\mathrm{diff}}=R^{\mathrm{new}}-R^{\nu ,re,l}.
\end{equation*}%
Then for any choice of $V$ and $z$ as above with $z$ horizontal for $%
p_{2,-1}:\Sigma ^{7}\longrightarrow S^{4}$ and any $\sigma ,\tau \in \mathbb{%
R},$ 
\begin{eqnarray*}
&&P_{Q}\left( \sigma ,\tau \right) =\mathrm{curv}^{\mathrm{diff}}\left(
\zeta ,W\right) +2\sigma R^{\mathrm{diff}}\left( \zeta ,W,W,z\right) +2\tau
R^{\mathrm{diff}}\left( W,\zeta ,\zeta ,V\right) \\
&&+\sigma ^{2}\mathrm{curv}^{\nu ,re,l}\left( z,W\right) +2\sigma \tau \left[
R^{\nu ,re,l}\left( \zeta ,W,V,z\right) +R^{\nu ,re,l}\left( \zeta
,V,W,z\right) \right] \\
&&+\tau ^{2}\mathrm{curv}^{\nu ,re,l}\left( \zeta ,V\right) \\
&>&0.
\end{eqnarray*}
\end{lemma}

\textbf{Notation: }We denote the metrics obtained following deformations $1$%
--$5$ above $g_{\nu },$ $g_{\nu ,re},$ $g_{\nu ,re,l},$ $g_{s},$ and $%
g_{new} $ respectively. We let notation like $\mathrm{curv}^{\nu ,re,l}$ and 
$R^{s}$ have the obviously meaning.

We will not discuss the role of deformation 6, any further. It is a Cheeger
deformation, so its effect is well understood. In particular, it preserves
nonnegative and positive curvatures, and for us the purpose is that it
allows the a priori simplification of the polynomial $P\left( \sigma ,\tau
\right) $ that we discussed above, and was explained in detail in
Proposition \ref{Cheeger Reduction}.

The notation $O\left( s\right) $ will (as usual) stand for a quantity that
converges to $0$ faster than a fixed constant times $s.$ The notation $O$
will stand for a quantity that is too small to effect whether or not our
metric is positively curved.

\section{The Redistribution}

As we mentioned above the metric obtained by carrying out deformations (1)
and (3)--(6) described above is not positively curved. It is not possible to
fully explain why at this point, but as mentioned above, the Main Lemma does
not hold, in particular, there are choices of $\tau $ and $V$ so that%
\begin{equation}
P_{Q}\left( 0,\tau \right) =\mathrm{curv}^{\mathrm{diff}}\left( \zeta
,W\right) +2\tau R^{\mathrm{diff}}\left( W,\zeta ,\zeta ,V\right) +\tau ^{2}%
\mathrm{curv}^{\mathrm{old}}\left( \zeta ,V\right) <0.  \notag
\end{equation}

To fix this problem we discuss the redistribution deformation (2) here. The
idea is that certain (positive) curvatures of the type, $\mathrm{curv}^{%
\mathrm{old}}\left( \zeta ,V\right) ,$ are redistributed so that they become
larger near $t=0$ and relatively smaller away from $t=0.$ This is at least a
reasonable goal, since (as we'll see in section 8) $\mathrm{curv}^{\mathrm{%
diff}}\left( \zeta ,W\right) $ and $R^{\mathrm{diff}}\left( W,\zeta ,\zeta
,V\right) $ are both concentrated near $t=0.$

Within $V_{1}\oplus V_{2}$ there is a $3$--dimensional subdistribution $%
\mathcal{Z},$ that has zero curvature with $\zeta .$ $\mathcal{Z}^{\perp
}\subset V_{1}\oplus V_{2}$ is therefore a three dimensional
subdistribution, along which we redistribute the curvature with $\zeta $ by
warping the metric by a function $\varphi $ whose gradient is proportional
to $\zeta .$

We want to concentrate the curvature near $t=0,$ so we choose $\varphi $ to
be concave down near $t=0$ and concave up away from $t=0.$

More specifically, using $\varphi ^{\prime }$ for $D_{\zeta }\left( \varphi
\right) ,$ we choose $\varphi $ so that 
\begin{eqnarray*}
\varphi ^{\prime }\left( 0\right) &=&0 \\
-101 &<&\frac{\varphi ^{\prime \prime }}{\nu ^{2}}<-100\text{ on an interval
of size }O\left( \nu \right) \text{ near }t=0
\end{eqnarray*}%
and 
\begin{equation*}
10,000\nu ^{3}<\varphi ^{\prime \prime }<10,001\nu ^{3}\text{ on an interval
that looks like }\left[ O\left( \nu \right) ,\frac{1}{100}\right] .
\end{equation*}

For this section only, we call the metric obtained by doing only the Cheeger
perturbations (1), (3), and (6) described above $g_{\mathrm{old}},$ and we
call the metric obtained by doing deformations (1), (2), (3), and (6), $g_{%
\mathrm{new}}.$ Where by $(2)$ we mean, that we multiply the restriction of
the metric to $\mathcal{Z}^{\perp }$ by $\varphi ^{2},$ and do not change
the metric on the orthogonal complement of $\mathcal{Z}^{\perp }.$

Our choice of $\varphi ^{\prime \prime }$ allows us to also have 
\begin{eqnarray*}
\left\vert \varphi ^{\prime }\right\vert &\leq &O\left( 100\nu ^{3}\right) \\
\left\vert \varphi ^{2}-1\right\vert &\leq &O\left( 100\nu ^{3}\right) , \\
\varphi |_{\left[ O\left( \frac{1}{100}\right) ,\frac{\pi }{4}\right] }
&\equiv &1.
\end{eqnarray*}

Since we carry out this change on $Sp\left( 2\right) $ before some of our
Cheeger deformations we have to check that the resulting metric is still
invariant under the various $S^{3}$--actions. To see this, simply note that
they all leave $V_{1}\oplus V_{2}$ and $\zeta $ invariant. From this it
follows that they all leave $\mathcal{Z}$ and $\mathcal{Z}^{\perp }$
invariant, and hence they all leave $g_{\mathrm{new}}$--invariant.

\begin{remark}
The constants $100,$ $101,$ $10,000$, ect. really just symbolize large
constants that are independent of our choice of metric parameters. We have
not verified that our whole argument actually works with these particular
constants. This question is fairly subtle, but since it is merely academic
we have only checked that the argument works with \emph{some} fixed
constants playing the role of $100,$ $101,$ $10,000$, ect.
\end{remark}

It is not surprising that the effect of this change in metric is to
redistribute $\mathrm{curv}^{\mathrm{old}}\left( \zeta ,\mathcal{Z}^{\perp
}\right) $ toward $t=0.$ The fact that we can do this without changing other
curvatures in a substantial way, is an amazing fact, that makes our whole
argument work.

\begin{theorem}
\label{redistr thm}$g_{\mathrm{new}}$ induces a metric of nonnegative
curvature on $\Sigma ^{7}$ whose zero planes are identical to those of $g_{%
\mathrm{old}}.$ Moreover, for any $V\in \mathcal{Z}^{\perp }$ 
\begin{equation*}
\mathrm{curv}^{\mathrm{new}}\left( \zeta ,V\right) =\mathrm{curv}^{\mathrm{%
old}}\left( \zeta ,V\right) +\varphi \varphi ^{\prime \prime }\left\vert
V\right\vert _{\mathrm{old}}^{2},
\end{equation*}%
and all other curvatures satisfy%
\begin{equation*}
\mathrm{curv}^{\mathrm{new}}\left( z,u\right) \geq \mathrm{curv}^{\mathrm{old%
}}\left( z,u\right) +O\left( \nu \right) \mathrm{curv}^{\mathrm{old}}\left(
z,u\right) .
\end{equation*}
\end{theorem}

\begin{remark}
Please note that we are not asserting the existence of a new nonnegatively
curved metric on $Sp\left( 2\right) ,$ only on $\Sigma ^{7}.$ The difference
is that we have a tighter control on the pre-existing $0$--curvatures of $%
\Sigma ^{7}.$ The result is, nevertheless, surprising. For a quick
explanation of why it holds, we point to the extreme amount of rigidity
present. (Cf \cite{Tapp2}.) For example, since the distribution $\mathcal{Z}$
is parallel along $\zeta ,$ it follows that any vector $v$ tangent to $%
\mathcal{Z}^{\perp }$ can be extended to a field $V$ tangent to $\mathcal{Z}%
^{\perp }$ so that 
\begin{equation*}
\left( \nabla _{\zeta }V\right) ^{V_{1}\oplus V_{2}}=0.
\end{equation*}%
Within this section, we will call such a field \textquotedblleft vertically
parallel\textquotedblright . Along a curve tangent to $H,$ these fields look
like $\left( N_{1}\beta ,N_{2}\delta \right) ,$ and hence in the language of
Lie groups are the left invariant fields determined by $\left( 
\begin{array}{cc}
\beta & 0 \\ 
0 & \delta%
\end{array}%
\right) .$ In the directions tangent to $V_{1}\oplus V_{2}$, the splitting $%
\mathcal{Z}\oplus \mathcal{Z}^{\perp }$ is right invariant, but not left
invariant. So we will extend these \textquotedblleft vertically
parallel\textquotedblright\ to be right invariant along $V_{1}\oplus V_{2}.$
If $U$ is such a field and $Z$ is basic horizontal, then (as O'Neill
observed)%
\begin{equation*}
\left[ U,Z\right] ^{H}=0.
\end{equation*}%
With respect to the biinvariant metric we have 
\begin{equation*}
\nabla _{Z}U\in H,\text{ and }
\end{equation*}%
Since the orbits of $A^{h_{1}}\oplus A^{h_{2}}$ are totally geodesic, $%
\nabla _{U}Z$ is also in $H,$ so $\left[ U,Z\right] ^{V_{1}\oplus V_{2}}=0,$
and in fact 
\begin{equation*}
\left[ U,Z\right] =0.
\end{equation*}
\end{remark}

\begin{proposition}
\label{zeta, P deriv}For $P\in \mathcal{Z}^{\perp },$ vertically parallel
along $\zeta ,$%
\begin{equation*}
\nabla _{\zeta }^{\nu ,re}P=\varphi ^{2}\nabla _{\zeta }^{\nu }P+\frac{%
\varphi ^{\prime }}{\varphi }P
\end{equation*}%
\begin{equation*}
\nabla _{P}^{\nu ,re}P=\nabla _{P}^{\nu }P-\varphi \varphi ^{\prime
}\left\vert P\right\vert _{\nu }^{2}\zeta
\end{equation*}%
For $Z\in H$, perpendicular to $\zeta $ and basic horizontal for $%
h_{1}\oplus h_{2}$%
\begin{equation*}
\nabla _{P}^{\nu ,re}Z=\varphi ^{2}\nabla _{P}^{\nu }Z,
\end{equation*}%
and for $U\in H$ and basic horizontal for $h_{1}\oplus h_{2}$ and for $Z$
basic horizontal for $h_{1}\oplus h_{2}$ or for $Z\in \mathcal{Z}$ and
vertically parallel%
\begin{equation*}
\nabla _{U}^{\nu ,re}Z=\nabla _{U}^{\nu }Z.
\end{equation*}
\end{proposition}

\begin{proof}
For $P\in \mathcal{Z}^{\perp },$ and vertically parallel%
\begin{eqnarray*}
2\left\langle \nabla _{\zeta }^{\nu ,re}P,P\right\rangle _{\nu ,re}
&=&D_{\zeta }\left\langle P,P\right\rangle _{\nu ,re} \\
&=&2\varphi \varphi ^{\prime }\left\langle P,P\right\rangle _{\nu } \\
&=&2\left\langle \nabla _{\zeta }^{\nu }P,P\right\rangle _{\nu }+2\varphi
\varphi ^{\prime }\left\langle P,P\right\rangle _{\nu }.
\end{eqnarray*}

For $Q\in \mathcal{Z}^{\perp }$ vertically parallel (with respect to $g_{\nu
})$ and perpendicular to $P$%
\begin{eqnarray*}
2\left\langle \nabla _{\zeta }^{\nu ,re}P,Q\right\rangle _{\nu ,re}
&=&2\left\langle \nabla _{\zeta }^{\nu }P,Q\right\rangle _{\nu } \\
&=&0.
\end{eqnarray*}%
For $Q\in \mathcal{Z},$ vertically parallel, we also know that $\left[ \zeta
,Q\right] =\left[ \zeta ,P\right] =0.$ So 
\begin{eqnarray*}
2\left\langle \nabla _{\zeta }^{\nu ,re}P,Q\right\rangle _{\nu ,re}
&=&\left\langle \left[ \zeta ,P\right] ,Q\right\rangle _{\nu
,re}-\left\langle \left[ \zeta ,Q\right] ,P\right\rangle _{\nu ,re} \\
&=&0 \\
&=&2\left\langle \nabla _{\zeta }^{\nu }P,Q\right\rangle _{\nu ,re}
\end{eqnarray*}%
For $Z$ in the orthogonal complement $H$ of $V_{1}\oplus V_{2},$ and basic
horizontal for $h_{1}\oplus h_{2}$%
\begin{eqnarray*}
2\left\langle \nabla _{\zeta }^{\nu ,re}P,Z\right\rangle _{\nu ,re}
&=&-\left\langle \left[ \zeta ,Z\right] ,P\right\rangle _{\nu ,re} \\
&=&-\varphi ^{2}\left\langle \left[ \zeta ,Z\right] ,P\right\rangle _{\nu }
\\
&=&2\varphi ^{2}\left\langle \nabla _{\zeta }^{\nu }P,Z\right\rangle _{\nu
,re} \\
&=&2\varphi ^{2}\left\langle A_{\zeta }^{\nu \mathrm{,}h_{1}\oplus
h_{2}}P,Z\right\rangle _{\nu ,re}.
\end{eqnarray*}%
Combining equations gives us%
\begin{equation*}
\nabla _{\zeta }^{\nu ,re}P=\varphi ^{2}\nabla _{\zeta }^{\nu }P+\frac{%
\varphi ^{\prime }}{\varphi }P
\end{equation*}%
as claimed.%
\begin{eqnarray*}
\left\langle \nabla _{P}^{\nu ,re}P,\zeta \right\rangle _{\nu ,re}
&=&-\left\langle \nabla _{\zeta }^{\nu ,re}P,P\right\rangle _{\nu ,re} \\
&=&-\left\langle \nabla _{\zeta }^{\nu }P,P\right\rangle _{\nu }-\varphi
\varphi ^{\prime }\left\langle P,P\right\rangle _{\nu } \\
&=&\left\langle \nabla _{P}^{\nu }P,\zeta \right\rangle _{\nu ,re}-\varphi
\varphi ^{\prime }\left\langle P,P\right\rangle _{\nu }
\end{eqnarray*}%
For $Z\in H$ and perpendicular to $\zeta $%
\begin{eqnarray*}
2\left\langle \nabla _{P}^{\nu ,re}P,Z\right\rangle _{\nu ,re}
&=&-D_{Z}\left\langle P,P\right\rangle _{\nu ,re}+2\left\langle \left[ Z,P%
\right] ,P\right\rangle _{\nu ,re} \\
&=&2\varphi ^{2}\left\langle \nabla _{P}^{\nu }P,Z\right\rangle _{\nu }
\end{eqnarray*}%
However, since $\left\langle \nabla _{P}^{\nu }P,Z\right\rangle _{\nu }=0,$
we conclude that%
\begin{equation*}
2\left\langle \nabla _{P}^{\nu ,re}P,Z\right\rangle _{\nu ,re}=\left\langle
\nabla _{P}^{\nu }P,Z\right\rangle _{\nu }=0
\end{equation*}%
For $Q\in \mathcal{Z}^{\perp },$ 
\begin{eqnarray*}
2\left\langle \nabla _{P}^{\nu ,re}P,Q\right\rangle _{\nu ,re} &=&2\varphi
^{2}\left\langle \nabla _{P}^{\nu }P,Q\right\rangle _{\nu } \\
&=&2\left\langle \nabla _{P}^{\nu }P,Q\right\rangle _{\nu ,re} \\
&=&0
\end{eqnarray*}%
For $Q\in \mathcal{Z}$%
\begin{eqnarray*}
2\left\langle \nabla _{P}^{\nu ,re}P,Q\right\rangle _{\nu ,re}
&=&2\left\langle \left[ Q,P\right] ,P\right\rangle _{\nu ,re} \\
&=&2\varphi ^{2}\left\langle \left[ Q,P\right] ,P\right\rangle _{\nu } \\
&=&2\varphi ^{2}\left\langle \nabla _{P}^{\nu }P,Q\right\rangle _{\nu } \\
&=&2\varphi ^{2}\left\langle \nabla _{P}^{\nu }P,Q\right\rangle _{\nu ,re}
\end{eqnarray*}%
However, since $\left\langle \nabla _{P}^{\nu }P,Q\right\rangle _{\nu }=0,$
both sides are again $0.$ Combining equations we have 
\begin{equation*}
\nabla _{P}^{\nu ,re}P=\nabla _{P}^{\nu }P-\varphi \varphi ^{\prime
}\left\vert P\right\vert _{\nu }^{2}\zeta .
\end{equation*}%
For $Z,Y\in H,$ basic horizontal and $Z$ perpendicular to $\zeta $ 
\begin{eqnarray*}
2\left\langle \nabla _{P}^{\nu ,re}Z,Y\right\rangle _{\nu ,re}
&=&\left\langle \left[ Y,Z\right] ,P\right\rangle _{\nu ,re} \\
&=&\varphi ^{2}\left\langle \left[ Y,Z\right] ,P\right\rangle _{\nu } \\
&=&2\varphi ^{2}\left\langle \nabla _{P}^{\nu }Z,Y\right\rangle _{\nu ,re}
\end{eqnarray*}%
For $Z\in H$, perpendicular to $\zeta $ and for $Q\in \mathcal{Z}^{\perp }$ 
\begin{eqnarray*}
2\left\langle \nabla _{P}^{\nu ,re}Z,Q\right\rangle _{\nu ,re}
&=&D_{Z}\left\langle P,Q\right\rangle _{\nu ,re}-\left\langle \left[ Z,P%
\right] ,Q\right\rangle _{\nu ,re}-\left\langle \left[ Z,Q\right]
,P\right\rangle _{\nu ,re} \\
&=&2\varphi ^{2}\left\langle \nabla _{P}^{\nu }Z,Q\right\rangle _{\nu } \\
&=&2\left\langle \nabla _{P}^{\nu }Z,Q\right\rangle _{\nu ,re}
\end{eqnarray*}%
If $Q$ is chosen to be one of our vertically parallel fields, then all three
terms in the second expression are $0,$ so in fact 
\begin{equation*}
2\left\langle \nabla _{P}^{\nu ,re}Z,Q\right\rangle _{\nu ,re}=2\left\langle
\nabla _{P}^{\nu }Z,Q\right\rangle _{\nu ,re}=2\left\langle \nabla _{P}^{\nu
}Z,Q\right\rangle _{\nu }=0
\end{equation*}

Similarly for $Q\in \mathcal{Z}$ vertically parallel, we have 
\begin{equation*}
2\left\langle \nabla _{P}^{\nu ,re}Z,Q\right\rangle _{\nu ,re}=-\left\langle 
\left[ Z,P\right] ,Q\right\rangle _{\nu ,re}-\left\langle \left[ Z,Q\right]
,P\right\rangle _{\nu ,re}
\end{equation*}%
However, since both terms are $0$ we have 
\begin{equation*}
2\left\langle \nabla _{P}^{\nu ,re}Z,Q\right\rangle _{\nu ,re}=2\left\langle
\nabla _{P}^{\nu }Z,Q\right\rangle _{\nu ,re}=2\left\langle \nabla _{P}^{\nu
}Z,Q\right\rangle _{\nu }=0
\end{equation*}%
Combining equations we have 
\begin{equation*}
\nabla _{P}^{\nu ,re}Z=\varphi ^{2}\nabla _{P}^{\nu }Z.
\end{equation*}%
Finally, the last equation 
\begin{equation*}
\nabla _{U}^{\nu ,re}Z=\nabla _{U}^{\nu }Z,
\end{equation*}%
follows from the Koszul formula.
\end{proof}

\begin{proposition}
\label{curv z^pepr} For $P\in \mathcal{Z}^{\perp }$%
\begin{equation*}
R^{\nu ,re}\left( P,\zeta \right) \zeta =\varphi ^{2}R^{\nu }\left( P,\zeta
\right) \zeta -3\varphi \varphi ^{\prime }A_{\zeta }^{h_{1}\oplus h_{2}}P-%
\frac{\varphi ^{\prime \prime }}{\varphi }P
\end{equation*}%
\begin{equation*}
R^{\nu ,re}\left( \zeta ,P\right) P=\varphi ^{4}R^{\nu }\left( \zeta
,P\right) P-\left( \varphi \varphi ^{\prime \prime }\right) \left\vert
P\right\vert _{\nu }^{2}\zeta
\end{equation*}
\end{proposition}

\begin{proof}
For $P\in \mathcal{Z}^{\perp },$ vertically parallel, we know that $\left[
\zeta ,P\right] =0$. Since $\zeta $ is a geodesic field 
\begin{eqnarray*}
R^{\nu ,re}\left( P,\zeta \right) \zeta &=&-\nabla _{\zeta }^{\nu ,re}\nabla
_{P}^{\nu ,re}\zeta \\
&=&-\nabla _{\zeta }^{\nu ,re}\left( \varphi ^{2}\nabla _{\zeta }^{\nu }P+%
\frac{\varphi ^{\prime }}{\varphi }P\right) \\
&=&-2\varphi \varphi ^{\prime }A_{\zeta }^{h_{1}\oplus h_{2}}P-\varphi
^{2}\nabla _{\zeta }^{\nu ,re}\nabla _{\zeta }^{\nu }P-\frac{\varphi \varphi
^{\prime \prime }-\left( \varphi ^{\prime }\right) ^{2}}{\varphi ^{2}}P-%
\frac{\varphi ^{\prime }}{\varphi }\nabla _{\zeta }^{\nu ,re}P \\
&=&-2\varphi \varphi ^{\prime }A_{\zeta }^{h_{1}\oplus h_{2}}P-\varphi
^{2}\nabla _{\zeta }^{\nu ,re}\nabla _{\zeta }^{\nu }P-\frac{\varphi \varphi
^{\prime \prime }-\left( \varphi ^{\prime }\right) ^{2}}{\varphi ^{2}}P-%
\frac{\left( \varphi ^{\prime }\right) ^{2}}{\varphi ^{2}}P-\varphi ^{2}%
\frac{\varphi ^{\prime }}{\varphi }\nabla _{\zeta }^{\nu }P \\
&=&\varphi ^{2}R^{\nu }\left( P,\zeta \right) \zeta -3\varphi \varphi
^{\prime }A_{\zeta }^{h_{1}\oplus h_{2}}P-\frac{\varphi ^{\prime \prime }}{%
\varphi }P
\end{eqnarray*}%
\begin{eqnarray*}
R^{\nu ,re}\left( \zeta ,P\right) P &=&\nabla _{\zeta }^{\nu ,re}\nabla
_{P}^{\nu ,re}P-\nabla _{P}^{\nu ,re}\nabla _{\zeta }^{\nu ,re}P \\
&=&\nabla _{\zeta }^{\nu ,re}\left( \nabla _{P}^{\nu }P-\varphi \varphi
^{\prime }\left\vert P\right\vert _{\nu }^{2}\zeta \right) -\nabla _{P}^{\nu
,re}\left( \varphi ^{2}\nabla _{\zeta }^{\nu }P+\frac{\varphi ^{\prime }}{%
\varphi }P\right)
\end{eqnarray*}%
Since 
\begin{eqnarray*}
\nabla _{P}^{\nu }P &=&0 \\
\nabla _{\zeta }^{\nu ,re}\nabla _{P}^{\nu }P &=&\nabla _{\zeta }^{\nu
}\nabla _{P}^{\nu }P=0.
\end{eqnarray*}%
We use the third equation of the previous proposition to get 
\begin{equation*}
\nabla _{P}^{\nu ,re}\left( \varphi ^{2}\nabla _{\zeta }^{\nu }P\right)
=\varphi ^{4}\nabla _{P}^{\nu }\nabla _{\zeta }^{\nu }P.
\end{equation*}%
So%
\begin{eqnarray*}
R^{\nu ,re}\left( \zeta ,P\right) P &=&\varphi ^{4}R^{\nu ,re}\left( \zeta
,P\right) P-\nabla _{\zeta }^{\nu ,re}\left( \varphi \varphi ^{\prime
}\left\vert P\right\vert _{\nu }^{2}\zeta \right) -\left( \frac{\varphi
^{\prime }}{\varphi }\nabla _{P}^{\nu ,re}P\right) \\
&=&\varphi ^{4}R^{\nu ,re}\left( \zeta ,P\right) P-\left( \varphi \varphi
^{\prime \prime }+\left( \varphi ^{\prime }\right) ^{2}\right) \left\vert
P\right\vert _{\nu }^{2}\zeta +\frac{\varphi ^{\prime }}{\varphi }\varphi
\varphi ^{\prime }\left\vert P\right\vert _{\nu }^{2}\zeta \\
&=&\varphi ^{4}R^{\nu ,re}\left( \zeta ,P\right) P-\left( \varphi \varphi
^{\prime \prime }\right) \left\vert P\right\vert _{\nu }^{2}\zeta
\end{eqnarray*}
\end{proof}

\begin{proposition}
\label{zeta-W} For $W\in \mathcal{Z},$ and vertically parallel along $\zeta $
with respect to $g_{\nu }$ 
\begin{eqnarray*}
\nabla _{\zeta }^{\nu ,re}W &=&\nabla _{\zeta }^{\nu }W=0 \\
\nabla _{W}^{\nu ,re}W &=&\nabla _{W}^{\nu }W=0.
\end{eqnarray*}
\end{proposition}

\begin{proof}
For the first equation, the point is that for $V\in \mathcal{Z}^{\perp },$
and vertically parallel, 
\begin{equation*}
\left[ \zeta ,W\right] =\left[ \zeta ,V\right] =0.
\end{equation*}%
For the second equation, when we compute the inner product with $P\in 
\mathcal{Z}^{\perp },$ we extend $W$ and $P$ to be invariant under $%
A^{h_{1}}\oplus A^{h_{2}},$ so we can compute their Lie bracket as though
they are right invariant vector fields in $S^{3}.$ In particular, 
\begin{equation*}
\left\langle \left[ W,P\right] ,W\right\rangle _{\nu ,re}=\left\langle \left[
W,P\right] ,W\right\rangle _{\nu }=0
\end{equation*}%
so%
\begin{equation*}
\left\langle \nabla _{W}^{\nu ,re}W,P\right\rangle _{\nu ,re}=\left\langle
\nabla _{W}^{\nu ,re}W,P\right\rangle _{\nu }=0.
\end{equation*}
\end{proof}

\begin{proposition}
\label{original zeros}%
\begin{equation*}
R^{\nu ,re}\left( W,\zeta \right) \zeta =R^{\nu }\left( W,\zeta \right)
\zeta =0
\end{equation*}%
\begin{equation*}
R^{\nu ,re}\left( \zeta ,W\right) W=R^{\nu }\left( \zeta ,W\right) W=0.
\end{equation*}
\end{proposition}

\begin{proof}
\begin{equation*}
R^{\nu ,re}\left( W,\zeta \right) \zeta =-\nabla _{\zeta }^{\nu ,re}\nabla
_{W}^{\nu ,re}\zeta =0=R^{\nu }\left( W,\zeta \right) \zeta
\end{equation*}%
\begin{eqnarray*}
R^{\nu ,re}\left( \zeta ,W\right) W &=&\nabla _{\zeta }^{\nu ,re}\nabla
_{W}^{\nu ,re}W-\nabla _{W}^{\nu ,re}\nabla _{\zeta }^{\nu ,re}W \\
&=&0=R^{\nu }\left( \zeta ,W\right) W.
\end{eqnarray*}
\end{proof}

\begin{proposition}
\label{z-P in Z^perp}For $z\in H$ perpendicular to $\zeta $ and $P\in 
\mathcal{Z}^{\perp }$%
\begin{equation*}
R^{\nu ,re}\left( z,P\right) P=\varphi ^{4}R^{\nu }\left( z,P\right)
P-\varphi \varphi ^{\prime }\left\vert P\right\vert _{\nu }^{2}\left( \nabla
_{z}^{\nu }\zeta \right) ^{P,\perp },
\end{equation*}%
where the superscript $^{P,\perp }$ denotes the component perpendicular to $%
P.$%
\begin{equation*}
R^{\nu ,re}\left( P,z\right) z=\varphi ^{2}R^{\nu }\left( P,z\right) z+%
\mathrm{II}^{\zeta }\left( z,z\right) \frac{\varphi ^{\prime }}{\varphi }P,
\end{equation*}%
where 
\begin{equation*}
\mathrm{II}^{\zeta }\left( z,z\right) =\left\langle \nabla _{z}^{\nu
}z,\zeta \right\rangle .
\end{equation*}
\end{proposition}

Note that these give consistent answers for the sectional curvatures---%
\begin{eqnarray*}
\left\langle R^{\nu ,re}\left( z,P\right) P,z\right\rangle &=&\varphi
^{4}\left\langle R^{\nu }\left( z,P\right) P,z\right\rangle -\varphi \varphi
^{\prime }\left\vert P\right\vert _{\nu }^{2}\left\langle \nabla _{z}^{\nu
}\zeta ,z\right\rangle \\
\left\langle R^{\nu ,re}\left( P,z\right) z,P\right\rangle &=&\varphi
^{4}\left\langle R^{\nu }\left( P,z\right) z,P\right\rangle +\mathrm{II}%
^{\zeta }\left( z,z\right) \frac{\varphi ^{\prime }}{\varphi }\left\langle
P,P\right\rangle _{\nu ,re}
\end{eqnarray*}

\begin{proof}
Choose $P$ to be the vertically parallel extension, then%
\begin{eqnarray*}
R^{\nu ,re}\left( z,P\right) P &=&\nabla _{z}^{\nu ,re}\nabla _{P}^{\nu
,re}P-\nabla _{P}^{\nu ,re}\nabla _{z}^{\nu ,re}P \\
&=&\nabla _{z}^{\nu ,re}\left( \nabla _{P}^{\nu }P-\varphi \varphi ^{\prime
}\left\vert P\right\vert _{\nu }^{2}\zeta \right) -\nabla _{P}^{\nu
,re}\left( \varphi ^{2}\nabla _{P}^{\nu }z\right)
\end{eqnarray*}

Since $\nabla _{P}^{\nu }P=0$%
\begin{equation*}
\nabla _{z}^{\nu ,re}\left( \nabla _{P}^{\nu }P\right) =\nabla _{z}^{\nu
}\left( \nabla _{P}^{\nu }P\right) =0.
\end{equation*}%
Also%
\begin{eqnarray*}
\nabla _{P}^{\nu ,re}\left( \varphi ^{2}\nabla _{P}^{\nu }z\right)
&=&\varphi ^{4}\nabla _{P}^{\nu }\left( \nabla _{P}^{\nu }z\right) +\varphi
^{2}\frac{\varphi ^{\prime }}{\varphi }P\left\langle \nabla _{P}^{\nu
}z,\zeta \right\rangle _{\nu } \\
&=&\varphi ^{4}\nabla _{P}^{\nu }\left( \nabla _{P}^{\nu }z\right) +\varphi
\varphi ^{\prime }\left\langle \nabla _{P}^{\nu }z,\zeta \right\rangle _{\nu
}P
\end{eqnarray*}%
So%
\begin{equation*}
R^{\nu ,re}\left( z,P\right) P=\varphi ^{4}R^{\nu }\left( z,P\right)
P-\varphi \varphi ^{\prime }\left\vert P\right\vert _{\nu }^{2}\nabla
_{z}^{\nu ,re}\zeta -\varphi \varphi ^{\prime }\left\langle \nabla _{P}^{\nu
}z,\zeta \right\rangle _{\nu }P.
\end{equation*}

The last term on the right does not seem to be correct since it is
proportional to $P.$ The formula is nevertheless correct since this term
cancels with the $P$--component of the second term. Indeed%
\begin{eqnarray*}
&&-\varphi \varphi ^{\prime }\left\vert P\right\vert _{\nu }^{2}\left\langle
\nabla _{z}^{\nu ,re}\zeta ,\frac{P}{\left\vert P\right\vert _{\nu ,re}}%
\right\rangle _{\nu ,re}\frac{P}{\left\vert P\right\vert _{\nu ,re}}-\varphi
\varphi ^{\prime }\left\langle \nabla _{P}^{\nu }z,\zeta \right\rangle _{\nu
}P \\
&=&-\varphi \varphi ^{\prime }\left\langle \nabla _{z}^{\nu }\zeta
,P\right\rangle _{\nu }P-\varphi \varphi ^{\prime }\left\langle \nabla
_{P}^{\nu }z,\zeta \right\rangle _{\nu }P \\
&=&\varphi \varphi ^{\prime }\left\langle \zeta ,\nabla _{z}^{\nu
}P\right\rangle _{\nu }P-\varphi \varphi ^{\prime }\left\langle \nabla
_{P}^{\nu }z,\zeta \right\rangle _{\nu }P \\
&=&0
\end{eqnarray*}

So%
\begin{equation*}
R^{\nu ,re}\left( z,P\right) P=\varphi ^{4}R^{\nu }\left( z,P\right)
P-\varphi \varphi ^{\prime }\left\vert P\right\vert _{\nu }^{2}\left( \nabla
_{z}^{\nu }\zeta \right) ^{P,\perp }.
\end{equation*}%
as claimed.

Extend $z$ so that its basic horizontal and tangent to an intrinsic geodesic
for the metric spheres around $\left( t,\sin 2\theta \right) =\left(
0,0\right) .$ Then%
\begin{eqnarray*}
R^{\nu ,re}\left( P,z\right) z &=&\nabla _{P}^{\nu ,re}\nabla _{z}^{\nu
,re}z-\nabla _{z}^{\nu ,re}\nabla _{P}^{\nu ,re}z \\
&=&\nabla _{P}^{\nu ,re}\nabla _{z}^{\nu }z-\nabla _{z}^{\nu ,re}\varphi
^{2}\nabla _{P}^{\nu }z
\end{eqnarray*}%
Since $\nabla _{z}^{\nu }z$ is proportional to $\zeta $ write 
\begin{eqnarray*}
\nabla _{z}^{\nu }z &=&\mathrm{II}^{\zeta }\left( z,z\right) \zeta \\
\nabla _{P}^{\nu ,re}\nabla _{z}^{\nu }z &=&\mathrm{II}^{\zeta }\left(
z,z\right) \left( \varphi ^{2}\nabla _{\zeta }^{\nu }P+\frac{\varphi
^{\prime }}{\varphi }P\right) \\
&=&\varphi ^{2}\nabla _{P}^{\nu }\nabla _{z}^{\nu }z+\mathrm{II}^{\zeta
}\left( z,z\right) \frac{\varphi ^{\prime }}{\varphi }P
\end{eqnarray*}%
Since $\nabla _{P}^{\nu }z$ is horizontal and $z\perp \zeta $%
\begin{equation*}
\nabla _{z}^{\nu ,re}\varphi ^{2}\nabla _{P}^{\nu }z=\varphi ^{2}\nabla
_{z}^{\nu }\nabla _{P}^{\nu }z
\end{equation*}

So%
\begin{equation*}
R^{\nu ,re}\left( P,z\right) z=\varphi ^{2}R^{\nu }\left( P,z\right) z+%
\mathrm{II}^{\zeta }\left( z,z\right) \frac{\varphi ^{\prime }}{\varphi }P.
\end{equation*}
\end{proof}

\begin{proposition}
\label{covar bat}For $z\in H$ perpendicular to $\zeta $ and for $W\in 
\mathcal{Z},$ vertically parallel with respect to $g_{\nu }$%
\begin{equation*}
\nabla _{z}^{\nu ,re}W=\nabla _{z}^{\nu }W\in H\cap \left( \mathrm{span}%
\left\{ \zeta \right\} \right) ^{\perp }.
\end{equation*}%
If $P$ and $Q$ are in $\mathcal{Z}^{\perp }$ and invariant under $%
A^{h_{1}}\oplus A^{h_{2}}$ and $W$ and $U$ are in $\mathcal{Z}$ and
invariant under $A^{h_{1}}\oplus A^{h_{2}},$ then%
\begin{equation*}
\left( \nabla _{W}^{\nu ,re}P\right) ^{H}=\left( \nabla _{W}^{\nu }P\right)
^{H}=0,
\end{equation*}%
\begin{equation*}
\left( \nabla _{P}^{\nu ,re}Q\right) ^{V_{1}\oplus V_{2}}=\left( \nabla
_{P}^{\nu }Q\right) ^{V_{1}\oplus V_{2}},
\end{equation*}%
\begin{equation*}
\left( \nabla _{W}^{\nu ,re}U\right) ^{V_{1}\oplus V_{2}}=\left( \nabla
_{W}^{\nu }U\right) ^{V_{1}\oplus V_{2}},
\end{equation*}%
\begin{equation*}
\left( \nabla _{P}^{\nu ,re}W\right) ^{\mathcal{Z}^{\perp }}=\varphi
^{-2}\left( \nabla _{P}^{\nu }W\right) ^{\mathcal{Z}^{\perp }},
\end{equation*}%
\begin{equation*}
\left( \nabla _{W}^{\nu ,re}P\right) ^{\mathcal{Z}}=\varphi ^{2}\left(
\nabla _{W}^{\nu }P\right) ^{\mathcal{Z}},
\end{equation*}%
\begin{eqnarray*}
\left( \nabla _{P}^{\nu ,re}W\right) ^{\mathcal{Z}} &=&O\left( 1-\varphi
^{2}\right) \left( \nabla _{P}^{\nu }W\right) ^{\mathcal{Z}}, \\
\left( \nabla _{W}^{\nu ,re}P\right) ^{\mathcal{Z}^{\perp }} &=&O\left(
1-\varphi ^{2}\right) \left( \nabla _{W}^{\nu }P\right) ^{\mathcal{Z}^{\perp
}}
\end{eqnarray*}
\end{proposition}

\begin{proof}
For $U\in \mathcal{Z}^{\perp }$ vertically parallel and $z$ basic horizontal
all terms in the Koszul formula for $\left\langle \nabla
_{z}W,U\right\rangle $ are $0$ with respect to both metrics. For $U$
perpendicular to $\mathcal{Z}^{\perp },$ all terms in the Koszul formula for 
$\left\langle \nabla _{z}^{\nu ,re}W,U\right\rangle $ are the same for both
metrics, so%
\begin{equation*}
\nabla _{z}^{\nu ,re}W=\nabla _{z}^{\nu }W\in H\cap \left( \mathrm{span}%
\left\{ \zeta \right\} \right) ^{\perp }.
\end{equation*}

If $Z\in H$ is basic horizontal, then all terms in the Koszul formulas for 
\begin{equation*}
\left\langle \nabla _{W}^{\nu ,re}P,Z\right\rangle _{\nu ,re}\text{ and }%
\left\langle \nabla _{W}^{\nu }P,Z\right\rangle _{\nu }
\end{equation*}%
vanish, so $\left( \nabla _{W}^{\nu ,re}P\right) ^{H}=\left( \nabla
_{W}^{\nu }P\right) ^{H}=0.$

If $P$ and $Q$ are in $\mathcal{Z}^{\perp }$ and $W$ is in $\mathcal{Z}$ and
all three fields are invariant under $A^{h_{1}}\oplus A^{h_{1}},$ then%
\begin{eqnarray*}
2\left\langle \nabla _{P}^{\nu ,re}Q,W\right\rangle _{\nu ,re}
&=&\left\langle \left[ P,Q\right] ,W\right\rangle _{\nu ,re}-\left\langle %
\left[ Q,W\right] ,P\right\rangle _{\nu ,re}+\left\langle \left[ W,P\right]
,Q\right\rangle _{\nu ,re} \\
&=&\left\langle \left[ P,Q\right] ,W\right\rangle _{\nu }-\varphi
^{2}\left\langle \left[ Q,W\right] ,P\right\rangle _{\nu }+\varphi
^{2}\left\langle \left[ W,P\right] ,Q\right\rangle _{\nu }
\end{eqnarray*}%
We can compute these Lie brackets as though the fields are right invariant
fields in $S^{3}$, so%
\begin{equation*}
-\varphi ^{2}\left\langle \left[ Q,W\right] ,P\right\rangle _{\nu }+\varphi
^{2}\left\langle \left[ W,P\right] ,Q\right\rangle _{\nu }=0
\end{equation*}%
and 
\begin{eqnarray*}
2\left\langle \nabla _{P}^{\nu ,re}Q,W\right\rangle &=&\left\langle \left[
P,Q\right] ,W\right\rangle _{\nu } \\
&=&2\left\langle \nabla _{P}^{\nu }Q,W\right\rangle _{\nu ,re}
\end{eqnarray*}

If $V$ is also in $\mathcal{Z}^{\perp },$ then 
\begin{equation*}
\left\langle \nabla _{P}^{\nu ,re}Q,V\right\rangle _{\nu ,re}=\varphi
^{2}\left\langle \nabla _{P}^{\nu }Q,V\right\rangle _{\nu }=\left\langle
\nabla _{P}^{\nu }Q,V\right\rangle _{\nu ,re}
\end{equation*}%
So 
\begin{equation*}
\left( \nabla _{P}^{\nu ,re}Q\right) ^{V_{1}\oplus V_{2}}=\left( \nabla
_{P}^{\nu }Q\right) ^{V_{1}\oplus V_{2}}
\end{equation*}%
as claimed.

A similar argument give us 
\begin{equation*}
\left( \nabla _{W}^{\nu ,re}U\right) ^{V_{1}\oplus V_{2}}=\left( \nabla
_{W}^{\nu }U\right) ^{V_{1}\oplus V_{2}}
\end{equation*}

Now suppose $Q$ is in $\mathcal{Z}^{\perp }$ and invariant under $%
A^{h_{1}}\oplus A^{h_{2}}.$ Since $\mathcal{Z}$ and $\mathcal{Z}^{\perp }$
are invariant under $A^{h_{1}}\oplus A^{h_{2}}$ 
\begin{equation*}
\left\langle \nabla _{P}^{\nu ,re}W,Q\right\rangle _{\nu ,re}=-\left\langle
W,\nabla _{P}^{\nu ,re}Q\right\rangle _{\nu ,re}=-\left\langle W,\nabla
_{P}^{\nu }Q\right\rangle _{\nu ,re}=\left\langle \nabla _{P}^{\nu
}W,Q\right\rangle _{\nu }=\varphi ^{-2}\left\langle \nabla _{P}^{\nu
}W,Q\right\rangle _{\nu ,re}
\end{equation*}%
proving the fifth equation.

Similarly, if $U$ is in $\mathcal{Z}$ and invariant under $A^{h_{1}}\oplus
A^{h_{2}},$ 
\begin{equation*}
\left\langle \nabla _{W}^{\nu ,re}P,U\right\rangle _{\nu ,re}=-\left\langle
P,\nabla _{W}^{\nu ,re}U\right\rangle _{\nu ,re}=-\left\langle P,\nabla
_{W}^{\nu }U\right\rangle _{\nu ,re}=-\varphi ^{2}\left\langle P,\nabla
_{W}^{\nu }U\right\rangle _{\nu }=\varphi ^{2}\left\langle \nabla _{W}^{\nu
}P,U\right\rangle _{\nu ,re}
\end{equation*}%
proving the sixth equation.

The last two equations have similar proofs. The Koszul formulas only have
Lie Bracket terms, only we must compare terms with multiplied by $\varphi
^{2}$ with terms with no $\varphi ^{2}$. This leads us to get only the
approximate answers that we have asserted.
\end{proof}

\begin{proposition}
\label{fubar}For $z\in H$ perpendicular to $\zeta $ and for $W\in \mathcal{Z}%
,$%
\begin{equation*}
R^{\nu ,re}\left( z,W\right) W=R^{\nu }\left( z,W\right) W\in H\cap \left( 
\mathrm{span}\left\{ \zeta \right\} \right) ^{\perp }
\end{equation*}%
\begin{equation*}
R^{\nu ,re}\left( W,z\right) z=R^{\nu }\left( W,z\right) z,
\end{equation*}%
\begin{equation*}
\left[ R^{\nu ,re}\left( W,\zeta \right) z\right] ^{H}=\left[ R^{\nu }\left(
W,\zeta \right) z\right] ^{H}.
\end{equation*}%
\begin{equation*}
\left[ R^{\nu ,re}\left( W,\zeta \right) z\right] ^{\mathcal{Z}}=\varphi ^{2}%
\left[ R^{\nu }\left( W,\zeta \right) z\right] ^{\mathcal{Z}}.
\end{equation*}%
\begin{equation*}
\left\vert \left[ R^{\nu ,re}\left( W,\zeta \right) z\right] ^{\mathcal{Z}%
^{\perp }}\right\vert =O\left( 1-\varphi ^{2}\right) \left\vert z\right\vert
\left\vert W\right\vert .
\end{equation*}
\end{proposition}

\begin{proof}
Choose $z$ to be basic horizontal and $W$ to be vertically parallel, then%
\begin{equation*}
R^{\nu ,re}\left( z,W\right) W=\nabla _{z}^{\nu ,re}\nabla _{W}^{\nu
,re}W-\nabla _{W}^{\nu ,re}\nabla _{z}^{\nu ,re}W
\end{equation*}%
Since 
\begin{eqnarray*}
\nabla _{W}^{\nu ,re}W &=&\nabla _{W}^{\nu }W=0, \\
\nabla _{z}^{\nu ,re}\nabla _{W}^{\nu ,re}W &=&\nabla _{z}^{\nu }\nabla
_{W}^{\nu }W=0.
\end{eqnarray*}%
On the other hand, using the previous proposition twice we have 
\begin{equation*}
\nabla _{W}^{\nu ,re}\nabla _{z}^{\nu ,re}W=\nabla _{W}^{\nu }\nabla
_{z}^{\nu }W\in H\cap \left( \mathrm{span}\left\{ \zeta \right\} \right)
^{\perp }
\end{equation*}%
So%
\begin{equation*}
R^{\nu ,re}\left( z,W\right) W=R^{\nu }\left( z,W\right) W\in H\cap \left( 
\mathrm{span}\left\{ \zeta \right\} \right) ^{\perp }
\end{equation*}%
Choose $z$ to be basic horizontal and $W$ to be vertically parallel, then%
\begin{equation*}
R^{\nu ,re}\left( W,z\right) z=\nabla _{W}^{\nu ,re}\nabla _{z}^{\nu
,re}z-\nabla _{z}^{\nu ,re}\nabla _{W}^{\nu ,re}z
\end{equation*}%
Since 
\begin{equation*}
\nabla _{z}^{\nu ,re}z=\nabla _{z}^{\nu }z\in H,
\end{equation*}%
\begin{equation*}
\nabla _{W}^{\nu ,re}\nabla _{z}^{\nu ,re}z=\nabla _{W}^{\nu ,re}\nabla
_{z}^{\nu }z=\nabla _{W}^{\nu }\nabla _{z}^{\nu }z
\end{equation*}%
where the last equality follows from the previous proposition and
Proposition \ref{zeta-W}. As before we have 
\begin{equation*}
\nabla _{W}^{\nu ,re}z=\nabla _{W}^{\nu }z\in H\cap \left( \mathrm{span}%
\left\{ \zeta \right\} \right) ^{\perp }.
\end{equation*}%
So%
\begin{equation*}
\nabla _{z}^{\nu ,re}\nabla _{W}^{\nu ,re}z=\nabla _{z}^{\nu }\nabla
_{W}^{\nu }z.
\end{equation*}%
So 
\begin{equation*}
R^{\nu ,re}\left( W,z\right) z=R^{\nu }\left( W,z\right) z.
\end{equation*}

To prove the final three equations we note that since $\left[ W,\zeta \right]
=0,$%
\begin{equation*}
R^{\nu ,re}\left( W,\zeta \right) z=\nabla _{W}^{\nu ,re}\nabla _{\zeta
}^{\nu ,re}z-\nabla _{\zeta }^{\nu ,re}\nabla _{W}^{\nu ,re}z.
\end{equation*}

Since 
\begin{eqnarray*}
\nabla _{W}^{\nu ,re}z &=&\nabla _{W}^{\nu }z\in H\cap \left( \mathrm{span}%
\left\{ \zeta \right\} \right) ^{\perp } \\
\nabla _{\zeta }^{\nu ,re}\nabla _{W}^{\nu ,re}z &=&\nabla _{\zeta }^{\nu
}\nabla _{W}^{\nu }z.
\end{eqnarray*}%
On the other hand 
\begin{equation*}
\left( \nabla _{\zeta }^{\nu ,re}z\right) ^{H}
\end{equation*}%
is basic horizontal, so 
\begin{equation*}
\nabla _{W}^{\nu ,re}\left( \nabla _{\zeta }^{\nu ,re}z\right) ^{H}=\nabla
_{W}^{\nu }\left( \nabla _{\zeta }^{\nu }z\right) ^{H}\in H\cap \left( 
\mathrm{span}\left\{ \zeta \right\} \right) ^{\perp }.
\end{equation*}

Since $\nabla _{W}^{\nu ,re}\left( \nabla _{\zeta }^{\nu ,re}z\right)
^{V_{1}\oplus V_{2}}$ and $\nabla _{W}^{\nu }\left( \nabla _{\zeta }^{\nu
}z\right) ^{V_{1}\oplus V_{2}}$ are both in $V_{1}\oplus V_{2},$ it follows
that 
\begin{equation*}
\left( \nabla _{W}^{\nu ,re}\nabla _{\zeta }^{\nu ,re}z\right) ^{H}=\nabla
_{W}^{\nu ,re}\left( \nabla _{\zeta }^{\nu ,re}z\right) ^{H}=\nabla
_{W}^{\nu }\left( \nabla _{\zeta }^{\nu }z\right) ^{H}=\left( \nabla
_{W}^{\nu }\nabla _{\zeta }^{\nu }z\right) ^{H}.
\end{equation*}%
So 
\begin{equation*}
\left[ R^{\nu ,re}\left( W,\zeta \right) z\right] ^{H}=\left[ R^{\nu }\left(
W,\zeta \right) z\right] ^{H}.
\end{equation*}

Since 
\begin{equation*}
\left( \nabla _{\zeta }^{\nu ,re}z\right) ^{V_{1}\oplus V_{2}}\in \mathcal{Z}%
^{\perp },
\end{equation*}%
it follows from the previous proposition that 
\begin{equation*}
\left( \nabla _{W}^{\nu ,re}\nabla _{\zeta }^{\nu ,re}z\right) ^{\mathcal{Z}%
}=\varphi ^{2}\left( \nabla _{W}^{\nu }\nabla _{\zeta }^{\nu }z\right) ^{%
\mathcal{Z}}.
\end{equation*}%
We also have $\nabla _{\zeta }^{\nu ,re}\nabla _{W}^{\nu ,re}z=\nabla
_{\zeta }^{\nu }\nabla _{W}^{\nu }z.$ However, since $\nabla _{W}^{\nu }z\in
H\cap \left( \mathrm{span}\left\{ \zeta \right\} \right) ^{\perp },$ we have 
$\left( \nabla _{\zeta }^{\nu }\nabla _{W}^{\nu }z\right) ^{\mathcal{Z}}=0.$
So 
\begin{equation*}
\left[ R^{\nu ,re}\left( W,\zeta \right) z\right] ^{\mathcal{Z}}=\varphi ^{2}%
\left[ R^{\nu }\left( W,\zeta \right) z\right] ^{\mathcal{Z}}.
\end{equation*}%
On the other hand, we just have 
\begin{equation*}
\left( \nabla _{W}^{\nu ,re}\nabla _{\zeta }^{\nu ,re}z\right) ^{\mathcal{Z}%
^{\perp }}=O\left( 1-\varphi ^{2}\right) \left( \nabla _{W}^{\nu }\nabla
_{\zeta }^{\nu }z\right) ^{\mathcal{Z}^{\perp }}
\end{equation*}

Combining this with 
\begin{equation*}
\nabla _{\zeta }^{\nu ,re}\nabla _{W}^{\nu ,re}z=\nabla _{\zeta }^{\nu
}\nabla _{W}^{\nu }z
\end{equation*}%
\begin{equation*}
\left[ R^{\nu }\left( W,\zeta \right) z\right] ^{\mathcal{Z}^{\perp }}=0,
\end{equation*}%
we have 
\begin{equation*}
\left\vert \left[ R^{\nu ,re}\left( W,\zeta \right) z\right] ^{\mathcal{Z}%
^{\perp }}\right\vert =O\left( 1-\varphi ^{2}\right) \left\vert z\right\vert
\left\vert W\right\vert .
\end{equation*}
\end{proof}

A very similar argument gives us

\begin{proposition}
\begin{equation*}
\left[ R^{\nu ,re}\left( W,z\right) \zeta \right] ^{H}=\left[ R^{\nu }\left(
W,z\right) \zeta \right] ^{H}.
\end{equation*}%
\begin{equation*}
\left[ R^{\nu ,re}\left( W,z\right) \zeta \right] ^{\mathcal{Z}}=\varphi ^{2}%
\left[ R^{\nu }\left( W,z\right) \zeta \right] ^{\mathcal{Z}}.
\end{equation*}%
\begin{equation*}
\left\vert \left[ R^{\nu ,re}\left( W,z\right) \zeta \right] ^{\mathcal{Z}%
^{\perp }}\right\vert =O\left( 1-\varphi ^{2}\right) \left\vert z\right\vert
\left\vert W\right\vert .
\end{equation*}
\end{proposition}

\begin{proposition}
For $U,V,Q\in \mathcal{Z}\cup \mathcal{Z}^{\perp }$ and mutually
perpendicular%
\begin{equation*}
\left( R^{\nu ,re}\left( U,V\right) Q\right) ^{H}=O\left( \varphi ^{\prime
}\right) \left\vert U\right\vert \left\vert V\right\vert \left\vert
Q\right\vert \zeta .
\end{equation*}
\end{proposition}

\begin{proof}
Extend all three vectors in be invariant under $A^{h_{1}}\oplus A^{h_{2}}.$
We have%
\begin{equation*}
\left( R^{\nu ,re}\left( U,V\right) Q\right) ^{H}=\left( \nabla _{U}^{\nu
,re}\nabla _{V}^{\nu ,re}Q-\nabla _{V}^{\nu ,re}\nabla _{U}^{\nu
,re}Q-\nabla _{\left[ U,V\right] }^{\nu ,re}Q\right) ^{H}.
\end{equation*}%
Our covariant derivative computations and our hypothesis about the three
vectors being mutually perpendicular give us that the $H$--components of
each of $\nabla _{V}^{\nu ,re}Q,$ $\nabla _{U}^{\nu ,re}Q,$ and $\left[ U,V%
\right] $ are $0.$ Therefore using Propositions \ref{zeta, P deriv}, \ref%
{zeta-W}, and \ref{covar bat} we have 
\begin{eqnarray*}
\left( \nabla _{U}^{\nu ,re}\nabla _{V}^{\nu ,re}Q\right) ^{H}
&=&\left\langle U^{\mathcal{Z}^{\perp }},\nabla _{V}^{\nu ,re}Q\right\rangle 
\frac{\varphi ^{\prime }}{\varphi }\zeta , \\
\left( \nabla _{V}^{\nu ,re}\nabla _{U}^{\nu ,re}Q\right) ^{H}
&=&\left\langle V^{\mathcal{Z}^{\perp }},\nabla _{U}^{\nu ,re}Q\right\rangle 
\frac{\varphi ^{\prime }}{\varphi }\zeta ,\text{ and } \\
\left( \nabla _{\left[ U,V\right] }^{\nu ,re}Q\right) ^{H} &=&\left\langle
Q^{\mathcal{Z}^{\perp }},\left[ U,V\right] \right\rangle \frac{\varphi
^{\prime }}{\varphi }\zeta .
\end{eqnarray*}%
So 
\begin{equation*}
\left( R^{\nu ,re}\left( U,V\right) Q\right) ^{H}=O\left( \varphi ^{\prime
}\right) \left\vert U\right\vert \left\vert V\right\vert \left\vert
Q\right\vert \zeta
\end{equation*}%
as claimed.
\end{proof}

When all four vectors are tangent to $\mathcal{Z}$ and $\mathcal{Z}^{\perp }$
we have

\begin{proposition}
For $u,v,w,z\in \mathcal{Z}\oplus \mathcal{Z}^{\perp },$%
\begin{equation*}
R^{\nu ,re}\left( u,v,w,z\right) =O\left( 1-\varphi ^{2}\right) R^{\nu
}\left( u,v,w,z\right) +O\left( 1-\varphi ^{2}\right) \left\vert
u\right\vert \left\vert v\right\vert \left\vert w\right\vert \left\vert
z\right\vert
\end{equation*}%
and%
\begin{equation*}
R^{\nu ,re}\left( u,w,w,u\right) =O\left( 1-\varphi ^{2}\right) R^{\nu
}\left( u,w,w,u\right)
\end{equation*}
\end{proposition}

\begin{proof}
If $U,W,$ and $Z$ are in either $\mathcal{Z}$ or $\mathcal{Z}^{\perp }$ and
invariant under $A^{h_{1}}\oplus A^{h_{2}},$ then in the Koszul formula for $%
2\left\langle \nabla _{U}^{\nu ,re}W,Z\right\rangle ,$ the derivative terms
vanish The new Lie bracket terms can differ from the old ones by a
multiplicative factor of $O\left( 1-\varphi ^{2}\right) .$ Applying this
principle several times yields the result.
\end{proof}

Finally mimicking the proof of O'Neill's horizontal curvature equation we
have

\begin{proposition}
If $x,y,z,$ and $u$ are in $H,$then 
\begin{equation*}
R^{\nu ,re}\left( x,y,z,u\right) =O\left( 1-\varphi ^{2}\right) R^{\nu
}\left( x,y,z,u\right)
\end{equation*}
\end{proposition}

To complete the proof of Theorem \ref{redistr thm} it remains to establish
the assertion about nonnegative curvature.

A plane that is perpendicular to either $\zeta $ or $W$ is positively
curved, since such planes were uniformly positively curved before the
redistribution, and the redistribution has a small effect on curvatures.

A plane that is not perpendicular to $\zeta $ and not perpendicular to $W$
has the form $P=\mathrm{span}\left\{ \zeta +\sigma z,W+\tau V\right\} $.
Because of the Cheeger deformation (6) we may assume that $z$ is in the
horizontal space for the Gromoll-Meyer submersion $Sp\left( 2\right)
\longrightarrow S^{4}.$

Our curvature is a quartic polynomial 
\begin{equation*}
P\left( \sigma ,\tau \right) =R\left( \zeta +\sigma z,W+\tau V,W+\tau
V,\zeta +\sigma z\right) .
\end{equation*}%
We have seen that the constant and linear terms vanish with respect to $g_{%
\mathrm{new}}.$ So our polynomial is 
\begin{eqnarray*}
P\left( \sigma ,\tau \right) &=&\sigma ^{2}R^{\mathrm{new}}\left(
z,W,W,z\right) +2\sigma \tau R^{\mathrm{new}}\left( \zeta ,W,V,z\right)
+2\sigma \tau R^{\mathrm{new}}\left( \zeta ,V,W,z\right) +\tau ^{2}R^{%
\mathrm{new}}\left( \zeta ,V,V,\zeta \right) \\
&&+2\sigma ^{2}\tau R^{\mathrm{new}}\left( z,W,V,z\right) +2\sigma \tau
^{2}R^{\mathrm{new}}\left( \zeta ,V,V,z\right) +\sigma ^{2}\tau ^{2}R^{%
\mathrm{new}}\left( z,V,V,z\right)
\end{eqnarray*}%
combining our curvature computations with the fact 
\begin{eqnarray*}
1-\varphi ^{2} &=&O\left( 100\nu ^{3}\right) \\
\varphi ^{\prime } &=&O\left( 100\nu ^{3}\right) \\
-\varphi ^{\prime \prime } &\geq &-\frac{\nu ^{2}}{100}
\end{eqnarray*}%
gives us that \addtocounter{algorithm}{1} 
\begin{equation}
P\left( \sigma ,\tau \right) \geq \left( 1-O\left( \nu ^{3}\right) \right)
P^{\mathrm{old}}\left( \sigma ,\tau \right) -\frac{\tau ^{2}\nu ^{2}}{100}R^{%
\mathrm{old}}\left( \zeta ,V,V,\zeta \right) +Q\left( \sigma ,\tau \right) .
\label{Redistr Poly}
\end{equation}%
Here $Q\left( \sigma ,\tau \right) $ is a quartic polynomial that looks like 
\begin{equation*}
Q\left( \sigma ,\tau \right) =C_{\sigma \tau }\sigma \tau +C_{\sigma
^{2}\tau }\sigma ^{2}\tau +C_{\sigma \tau ^{2}}\sigma \tau ^{2},
\end{equation*}%
whose coefficients $C_{\sigma \tau },C_{\sigma ^{2}\tau },$ and $C_{\sigma
\tau ^{2}}$ satisfy 
\begin{eqnarray*}
C_{\sigma \tau } &\leq &O\left( \nu \right) \sqrt{R^{\mathrm{new}}\left(
z,W,W,z\right) }\sqrt{R^{\mathrm{new}}\left( \zeta ,V,V,\zeta \right) } \\
C_{\sigma ^{2}\tau } &\leq &O\left( \nu \right) \sqrt{R^{\mathrm{new}}\left(
z,W,W,z\right) }\sqrt{R^{\mathrm{new}}\left( z,V,V,z\right) } \\
C_{\sigma \tau ^{2}} &\leq &O\left( \nu \right) \sqrt{R^{\mathrm{new}}\left(
\zeta ,V,V,\zeta \right) }\sqrt{R^{\mathrm{new}}\left( z,V,V,z\right) }.
\end{eqnarray*}

\medskip These estimates imply that we can replace $Q\left( \sigma ,\tau
\right) $ in \ref{Redistr Poly} with $O.$ For example, the quadratic 
\begin{eqnarray*}
&&\sigma ^{2}R^{\mathrm{new}}\left( z,W,W,z\right) +\sigma \tau C_{\sigma
\tau }\sigma \tau +\tau ^{2}R^{\mathrm{new}}\left( \zeta ,V,V,\zeta \right)
\\
&\geq &\sigma ^{2}\left( R^{\mathrm{new}}\left( z,W,W,z\right) -\frac{%
O\left( \nu ^{2}\right) R^{\mathrm{new}}\left( z,W,W,z\right) R^{\mathrm{new}%
}\left( \zeta ,V,V,\zeta \right) }{R^{\mathrm{new}}\left( \zeta ,V,V,\zeta
\right) }\right) \\
&\geq &\sigma ^{2}\left( R^{\mathrm{new}}\left( z,W,W,z\right) -O\left( \nu
^{2}\right) R^{\mathrm{new}}\left( z,W,W,z\right) \right) \\
&=&\sigma ^{2}\left( R^{\mathrm{new}}\left( z,W,W,z\right) +O\right)
\end{eqnarray*}

Similar arguments allow us to drop the $C_{\sigma ^{2}\tau }\sigma ^{2}\tau $
and $C_{\sigma \tau ^{2}}\sigma \tau ^{2}$ terms of $Q\left( \sigma ,\tau
\right) $. (Cf Theorem \ref{R^diff, small}). So \ref{Redistr Poly} becomes 
\begin{equation}
P\left( \sigma ,\tau \right) \geq \left( 1-O\left( \nu ^{3}\right) \right)
P^{\mathrm{old}}\left( \sigma ,\tau \right) -\frac{\tau ^{2}\nu ^{2}}{100}R^{%
\mathrm{old}}\left( \zeta ,V,V,\zeta \right) +O.  \label{Redistr poly 2}
\end{equation}

We have an inequality instead of an equality because in many cases the
curvature is much bigger. For example from Proposition \ref{curv z^pepr} we
have that for $P\in \mathcal{Z}^{\perp }$ 
\begin{eqnarray*}
\left\langle R^{\mathrm{new}}\left( P,\zeta \right) \zeta ,P\right\rangle
&\geq &\varphi ^{4}\left\langle R^{\mathrm{old}}\left( P,\zeta \right) \zeta
,P\right\rangle -\left( \varphi \varphi ^{\prime \prime }\right) \left\vert
P\right\vert _{\mathrm{old}}^{2} \\
&\geq &\varphi ^{4}\left\langle R^{\mathrm{old}}\left( P,\zeta \right) \zeta
,P\right\rangle -\frac{\nu ^{2}}{100}\left\vert P\right\vert _{\mathrm{old}%
}^{2}
\end{eqnarray*}%
but in many places this curvature is larger. Similarly from Proposition \ref%
{z-P in Z^perp} we have that for $z\in H$, perpendicular to $\zeta $ and for 
$P\in \mathcal{Z}^{\perp }$ 
\begin{equation*}
\left\langle R^{\mathrm{new}}\left( P,z\right) z,P\right\rangle =\varphi
^{4}\left\langle R^{\mathrm{old}}\left( z,P\right) P,z\right\rangle +\mathrm{%
II}^{\zeta }\left( z,z\right) \frac{\varphi ^{\prime }}{\varphi }%
\left\langle P,P\right\rangle _{\mathrm{new}}
\end{equation*}%
The extra term here is nonnegative since both $\mathrm{II}^{\zeta }\left(
z,z\right) $ and $\varphi ^{\prime }$ are nonpositive.

The theorem follows from inequality \ref{Redistr poly 2}.

\section{The Warping function induced by $Sp\left( 2\right) $}

As promised, in the next two sections we analyze the effect on Equation \ref%
{curv formula-2} of running the $h_{2}$--Cheeger perturbation for a long
time. If $\nu $ is the parameter of this perturbation, then we will show
that making $\nu $ small has the effect of concentrating all of the terms on
the right hand side of equation \ref{curv formula-2}, 
\begin{equation*}
\mathrm{curv}\left( \zeta ,W\right) =-s^{2}\left( D_{\zeta }\left(
\left\vert H_{w}\right\vert D_{\zeta }\left\vert H_{w}\right\vert \right)
\right) +s^{4}\left( D_{\zeta }\left\vert H_{w}\right\vert \right) ^{2},
\end{equation*}%
around $t=0.$ (In the Gromoll-Meyer sphere $\zeta $ plays the role of $X.$)

The advantage of doing this is that it will allow us to choose our
\textquotedblleft partial\textquotedblright\ conformal factor so that it is
constant away from $t=0,$ thus avoiding an analysis of how the partial\
conformal change effects the intersection of the two pieces of the zero
curvature locus.

Along any integral curve of $\zeta ,$ $\left\vert H_{w}\right\vert $ is the
length of a Killing field of our $SO\left( 3\right) $--action on $S^{4}.$
Since the principal orbits of this action on $S^{4}$ are two spheres and the
action on these two spheres is standard, these two spheres are round.

So that our geometry is more easily comparable to the standard round $S^{4},$
we look at the Killing fields 
\begin{equation*}
\left( 0,\frac{\vartheta }{2}\right)
\end{equation*}%
on $Sp\left( 2\right) $ and we set 
\begin{equation*}
\psi =\left\vert \left( 0,\frac{\vartheta }{2}\right) ^{\mathrm{horiz}%
}\right\vert .
\end{equation*}

To understand the geometric meaning of $\psi ,$ think of the join
decomposition described in the remark after Proposition \ref{hopf symmetries}%
,%
\begin{equation*}
S^{4}=S_{\mathbb{R}}^{1}\ast S_{\func{Im}}^{2}.
\end{equation*}%
The $S^{2}$s of the join decomposition are the principal orbits of the $%
SO\left( 3\right) $--action and the intrinsic metric on them is $\psi ^{2}$
times the unit metric.

Along any integral curve of $\zeta ,$ $H_{w}$ is a constant multiple of $%
\left( 0,\frac{\vartheta }{2}\right) ^{\mathrm{horiz}}$ we call this
multiple $w_{h},$ so 
\begin{eqnarray*}
H_{w} &=&w_{h}\left( 0,\frac{\vartheta }{2}\right) ^{\mathrm{horiz}} \\
\left\vert H_{w}\right\vert &=&w_{h}\psi ,\text{ and} \\
w_{h} &=&O\left( \frac{1}{\nu ^{2}}\right) .
\end{eqnarray*}

\begin{remark}
The exact value of $w_{h}$ depends on which integral curve of $\zeta $ we
are on. The variation can be seen by noticing how $\sin \lambda $ varies in
Proposition \ref{exceptional zeros}. It is for precisely this reason that we
cannot use a regular conformal change to even out the curvature.
\end{remark}

Since $\left\vert \left( 0,\frac{\vartheta }{2}\right) \right\vert =\frac{%
\nu }{2}$ and $\psi =\left\vert \left( 0,\frac{\vartheta }{2}\right) ^{%
\mathrm{horiz}}\right\vert ,$ it is not hard to see that the effect of the $%
h_{2}$--Cheeger perturbation on the geometry of $S^{4}$ is to shrink the $%
S^{2}$s. More precisely the $S^{3}$s that are the join of $S_{\func{Im}}^{2}$
and any $S^{0}\subset S_{\mathbb{R}}^{1}$ become very thin \textquotedblleft
cigars\textquotedblright . Unfortunately this coarse description is not
sufficient for our purposes, since we need to understand the derivatives and
second derivatives of $\psi .$

We will prove in subsection 8.1 that the redistribution described in the
previous section has a minimal effect on $\psi .$ Once this is established,
it will be enough to know the effect of the two Cheeger parameters $\nu $
and $l.$ For now we just focus on this.

When we want to emphasize the dependence of $\psi $ on $\nu $ and $l$ we
will write, $\psi _{\nu ,l}$.

To find $\psi _{\nu ,l}$ we recall that the horizontal vectors that project
to the $S^{2}$s look like%
\begin{equation*}
\left( \cos 2t\right) \eta ^{2,0}=\left( \left( \cos 2t\right) \eta ,\left(
\cos 2t\right) \eta +\sin 2t\frac{\vartheta }{\nu ^{2}}\right) ,
\end{equation*}%
here as always, the notational convention on page \pageref{notational
convention copy(1)} is in effect. So 
\begin{eqnarray*}
\psi _{\nu ,l} &=&\frac{1}{\left\vert \left( \cos 2t\right) \eta
^{2,0}\right\vert _{\nu ,l}}\left\langle \left( 0,\frac{\vartheta }{2}%
\right) ,\left( \cos 2t\right) \eta ^{2,0}\right\rangle \\
&=&\frac{1}{2}\frac{\sin 2t}{\left\vert \left( \cos 2t\right) \eta
^{2,0}\right\vert _{\nu ,l}}.
\end{eqnarray*}%
Using the formulas for the projections of $\eta ^{2,0}$ onto the orbits of $%
A^{u}\times A^{d}$ from \cite{Wilh2} we have

\begin{equation*}
\left\vert \left( \cos 2t\right) \eta ^{2,0}\right\vert _{\nu ,l}^{2}=\cos
^{2}2t+\frac{\sin ^{2}2t}{\nu ^{2}}+\frac{1}{2l^{2}}\left( 1-\cos ^{2}2t\cos
^{2}2\theta \right)
\end{equation*}%
and

\begin{proposition}
\begin{eqnarray*}
\frac{\partial }{\partial t}\psi _{\nu ,l} &=&\frac{\left( 1+\frac{1}{2l^{2}}%
\sin ^{2}2\theta \right) \cos 2t}{\left\vert \left( \cos 2t\right) \eta
^{2,0}\right\vert _{\nu ,l}^{3}} \\
&=&\frac{\left\vert x^{2,0}\right\vert _{\nu ,l}^{2}\cos 2t}{\left\vert
\left( \cos 2t\right) \eta ^{2,0}\right\vert _{\nu ,l}^{3}} \\
\frac{\partial }{\partial \theta }\psi _{\nu ,l} &=&-\frac{1}{4l^{2}}\frac{%
\sin 2t\cos ^{2}2t\sin 4\theta }{\left\vert \left( \cos 2t\right) \eta
^{2,0}\right\vert _{\nu ,l}^{3}}
\end{eqnarray*}%
\begin{eqnarray*}
\frac{\partial ^{2}}{\partial t^{2}}\psi _{\nu ,l} &=&-\left\vert
x^{2,0}\right\vert _{\nu ,l}^{2}\frac{\sin 2t}{\left\vert \left( \cos
2t\right) \eta ^{2,0}\right\vert _{\nu ,l}^{5}}\left( -4\left\vert
x^{2,0}\right\vert _{\nu ,l}^{2}\cos ^{2}2t+\frac{2}{\nu _{l}^{2}}+4\left( 
\frac{1}{\nu _{l}^{2}}\right) \cos ^{2}2t\right) \\
\frac{\partial }{\partial \theta }\frac{\partial }{\partial t}\psi _{\nu ,l}
&=&\frac{\cos 2t\sin 4\theta }{l^{2}\left\vert \left( \cos 2t\right) \eta
^{2,0}\right\vert _{\nu ,l}^{5}}\left( -\frac{1}{2}\left\vert
x^{2,0}\right\vert _{\nu ,l}^{2}\cos ^{2}2t+\frac{1}{\nu _{l}^{2}}\sin
^{2}2t\right)
\end{eqnarray*}%
\begin{equation*}
\frac{\partial ^{2}}{\partial \theta ^{2}}\psi _{\nu ,l}=-\frac{\sin 2t\cos
^{2}2t}{l^{2}}\frac{\cos 4\theta \left( \left\vert x^{2,0}\right\vert _{\nu
,l}^{2}\cos ^{2}2t+\frac{1}{\nu _{l}^{2}}\sin ^{2}2t\right) }{\left\vert
\left( \cos 2t\right) \eta ^{2,0}\right\vert _{\nu ,l}^{5}}+\frac{3}{2}\frac{%
\sin 2t\cos ^{4}2t}{4l^{4}}\frac{\sin ^{2}4\theta }{\left\vert \left( \cos
2t\right) \eta ^{2,0}\right\vert _{\nu ,l}^{5}}\text{ }
\end{equation*}
\end{proposition}

The computations are long, but straightforward. Since the results are not
qualitatively surprising, we have deferred giving the details until the
appendix.

\section{Concentrated Curvature Near $t=0$}

Plugging $\zeta =X$ and $\left\vert H_{w}\right\vert =w_{h}\psi $ into \ref%
{curv formula} gives us

\begin{equation*}
\mathrm{curv}_{g_{s}}\left( \zeta ,W\right) =-s^{2}w_{h}^{2}\left( D_{\zeta
}\left( \psi D_{\zeta }\psi \right) \right) +w_{h}^{2}s^{4}\left( D_{\zeta
}\psi \right) ^{2}.
\end{equation*}%
If $z$ is the parameter of an integral curve of $\zeta ,$ then the leading
order, total derivative term, $-s^{2}w_{h}^{2}\left( D_{\zeta }\left( \psi
D_{\zeta }\psi \right) \right) ,$ is negative near $z=0$, positive for large
enough $z,$ and integrates to $0$. The effect of the $\nu $ perturbation is
to concentrate this region of negativity, and the bulk of the region of
positivity near $z=0.$ Before proving this we need

\begin{proposition}
Let $n$ be the normalized gradient field for $\mathrm{dist}\left( S_{\mathbb{%
R}}^{1},\cdot \right) $ on $S^{4}$ with respect to $g_{\nu ,l}.$ If 
\begin{equation*}
\zeta =n\cos \varphi +y^{2,0}\sin \varphi ,
\end{equation*}%
then%
\begin{eqnarray*}
D_{\zeta }\left( \cos \varphi \right) &=&O\left( t\right) \\
D_{\zeta }\left( \sin \varphi \right) &=&O\left( t\right) .
\end{eqnarray*}
\end{proposition}

\begin{proof}
Let $c_{\zeta }$ be an integral curve of $\zeta $ starting at $\left(
t,\theta \right) =\left( 0,0\right) .$ Consider the triangles, $\triangle
_{\theta }$ whose sides are the geodesic with $t=0,$ $c_{\zeta },$ and the
various geodesics that are integral curves of $n$ starting at $\left(
t,\theta \right) =\left( 0,\theta \right) .$

Let $\varphi _{0}$ be the angle between $\dot{c}_{\zeta }\left( 0\right) $
and $n.$ Then the interior angles of $\triangle _{\theta }$ are $\frac{\pi }{%
2},\frac{\pi }{2}-\varphi _{0},$ and $\varphi .$ So 
\begin{equation*}
\varphi =\varphi _{0}+\text{\textrm{angle--excess}}\left( \triangle _{\theta
}\right) .
\end{equation*}%
Since $\mathrm{area}\left( \triangle _{\theta }\right) =O\left( \theta
^{2}\right) ,$ the result follows.
\end{proof}

\begin{proposition}
\label{concentrated curvature}For $t>\frac{\nu }{2}$%
\begin{equation*}
-s^{2}w_{h}^{2}\left( D_{\zeta }\left( \psi D_{\zeta }\psi \right) \right) >0
\end{equation*}%
and 
\begin{equation*}
\mathrm{curv}_{s}\left( \zeta ,W\right) |_{\left[ O\left( c^{3/4}\right) ,%
\frac{\pi }{4}\right] }\leq \int_{\gamma _{\zeta }}\mathrm{curv}_{s}\left(
\zeta ,W\right)
\end{equation*}%
provided $c\nu =s^{6/7}$ and $l=O\left( \nu ^{\frac{1}{3}}\right) .$
\end{proposition}

\begin{remark}
Together these inequalities imply that all of the negative curvature of $%
g_{s}$ occurs on the interval $\left[ 0,\nu \right] $ and the bulk of the
positive curvature occurs on $\left[ \nu ,O\left( c\right) \right] .$ In
particular, $g_{s}$ is positively curved for $t>\frac{\nu }{\sqrt{8}}$ and
our partial warping can be carried out on $\left[ 0,O\left( c\right) \right]
.$
\end{remark}

\begin{remark}
Our proof relies on the computations of the various derivatives of $\psi $
that are stated in previous section and proven in the Appendix. They are
done in the Appendix with respect to the metric $g_{\nu ,l},$ while to
justify this proposition we will need to know them with respect to $g_{\nu
,re,l}.$ So technically this proposition is about an (as yet) undiscussed
metric $g_{\nu ,l,s}$. I.e. the metric obtained by scaling the fibers of $%
Sp\left( 2\right) \longrightarrow S^{4}$ after performing the Cheeger
deformation $A^{u}\times A^{d}\times A^{h_{1}}\times A^{h_{2}},$ but with
out performing the redistribution. We will show in Subsection 8.1 (at the
end of this section) that the effect of the redistribution on the various
derivatives of $\psi $ is sufficiently small so that this proposition
remains valid for the actual metric $g_{s}.$
\end{remark}

\begin{proof}
From the previous section we have 
\begin{eqnarray*}
\frac{\partial }{\partial t}\psi _{\nu ,l} &=&\frac{\left\vert
x^{2,0}\right\vert _{\nu ,l}^{2}\cos 2t}{\left\vert \left( \cos 2t\right)
\eta ^{2,0}\right\vert _{\nu ,l}^{3}} \\
\frac{\partial }{\partial \theta }\psi _{\nu ,l} &=&-\frac{1}{4l^{2}}\frac{%
\sin 2t\cos ^{2}2t\sin 4\theta }{\left\vert \left( \cos 2t\right) \eta
^{2,0}\right\vert _{\nu ,l}^{3}}
\end{eqnarray*}%
Since the $\frac{\partial }{\partial \theta }$--direction is a linear
combination of the vectors $y^{2,0}$ and $\left( -\mathfrak{v},\mathfrak{v}%
\right) $ and $D_{\left( -\mathfrak{v},\mathfrak{v}\right) }\psi _{\nu
,l}=0, $ we get 
\begin{equation*}
D_{y^{2,0}}\psi _{\nu ,l}=-\frac{1}{4l^{2}}\frac{\sin 2t\cos ^{3}2t\sin
4\theta }{\left\vert \left( \cos 2t\right) \eta ^{2,0}\right\vert _{\nu
,l}^{3}},
\end{equation*}%
where the extra factor of $\cos 2t$ is $\left\langle \frac{\partial }{%
\partial \theta },y^{2,0}\right\rangle .$ So if 
\begin{equation*}
\zeta =n\cos \varphi +y^{2,0}\sin \varphi
\end{equation*}%
\begin{equation*}
D_{\zeta }\psi _{\nu ,l}=\frac{\left\vert x^{2,0}\right\vert _{\nu
,l}^{2}\cos 2t}{\left\vert \left( \cos 2t\right) \eta ^{2,0}\right\vert
_{\nu ,l}^{3}}\cos \varphi -\frac{1}{4l^{2}}\frac{\sin 2t\cos ^{3}2t\sin
4\theta }{\left\vert \left( \cos 2t\right) \eta ^{2,0}\right\vert _{\nu
,l}^{3}}\sin \varphi
\end{equation*}%
So%
\begin{equation*}
\left( D_{\zeta }\left( \psi _{\nu ,l}\right) \right) ^{2}\leq 2\left( \frac{%
\left\vert x^{2,0}\right\vert _{\nu ,l}^{4}\cos ^{2}2t}{\left\vert \left(
\cos 2t\right) \eta ^{2,0}\right\vert _{\nu ,l}^{6}}\right) \cos ^{2}\varphi
+\frac{\sin ^{2}2t\cos ^{6}2t\sin ^{2}4\theta }{8\left\vert \left( \cos
2t\right) \eta ^{2,0}\right\vert _{\nu ,l}^{6}l^{4}}\sin ^{2}\varphi .
\end{equation*}

We can also get an explicit formula for $-\psi _{\nu ,l}D_{\zeta }D_{\zeta
}\left( \psi _{\nu ,l}\right) ,$ but its quite complicated, so its easier to
estimate it. First notice that erasing various $A$--tensors we have 
\begin{equation*}
-\frac{D_{\zeta }D_{\zeta }\left( \psi _{\nu ,l}\right) }{\psi _{\nu ,l}}%
\geq \mathrm{curv}_{g_{\nu }}\left( \zeta ,\eta _{u}^{2,0}\right) ,\text{ }
\end{equation*}%
where $\eta _{u}^{2,0}=\frac{\eta ^{2,0}}{\left\vert \eta ^{2,0}\right\vert }
$. So%
\begin{eqnarray*}
-\psi _{\nu ,l}D_{\zeta }D_{\zeta }\left( \psi _{\nu ,l}\right) &\geq &\psi
_{\nu ,l}^{2}\mathrm{curv}_{g_{\nu }}\left( \zeta ,\eta _{u}^{2,0}\right) \\
&=&\frac{\psi _{\nu ,l}^{2}}{\left\vert \cos 2t\eta ^{2,0}\right\vert _{\nu
,l}^{2}}\left( \cos ^{2}2t+\frac{1}{2}\sin ^{2}2t\right) .
\end{eqnarray*}%
So to determine where the total derivative is positive, it suffices to solve 
\begin{equation*}
\psi _{\nu ,l}^{2}\left( \cos ^{2}2t\right) \geq 2\left( \frac{\left\vert
x^{2,0}\right\vert _{\nu ,l}^{4}\cos ^{2}2t}{\left\vert \left( \cos
2t\right) \eta ^{2,0}\right\vert _{\nu ,l}^{4}}\right) \cos ^{2}\varphi +%
\frac{\sin ^{2}2t\cos ^{6}2t\sin ^{2}4\theta }{8\left\vert \left( \cos
2t\right) \eta ^{2,0}\right\vert _{\nu ,l}^{4}l^{4}}\sin ^{2}\varphi
\end{equation*}%
or%
\begin{equation*}
\frac{\sin ^{2}2t}{4\left\vert \left( \cos 2t\right) \eta ^{2,0}\right\vert
_{\nu ,l}^{2}}\geq 2\left( \frac{\left\vert x^{2,0}\right\vert _{\nu ,l}^{4}%
}{\left\vert \left( \cos 2t\right) \eta ^{2,0}\right\vert _{\nu ,l}^{4}}%
\right) \cos ^{2}\varphi +\frac{\sin ^{2}2t\cos ^{4}2t\sin ^{2}4\theta }{%
8\left\vert \left( \cos 2t\right) \eta ^{2,0}\right\vert _{\nu ,l}^{4}l^{4}}%
\sin ^{2}\varphi
\end{equation*}%
or%
\begin{equation*}
\frac{\sin ^{2}2t}{4}\geq 2\left( \frac{\left\vert x^{2,0}\right\vert _{\nu
,l}^{4}}{\left\vert \left( \cos 2t\right) \eta ^{2,0}\right\vert _{\nu
,l}^{2}}\right) \cos ^{2}\varphi +\frac{\sin ^{2}2t\cos ^{4}2t\sin
^{2}4\theta }{8\left\vert \left( \cos 2t\right) \eta ^{2,0}\right\vert _{\nu
,l}^{2}l^{4}}\sin ^{2}\varphi
\end{equation*}

Since $l=O\left( \nu ^{1/3}\right) ,$ and on the integral curves of $\zeta $
in the former $0$--locus, $\sin 4\theta =O\left( \sin 2\theta \right)
=O\left( \sin 2t\right) ,$ and from the appendix we have 
\begin{eqnarray*}
\left\vert \cos 2t\eta ^{2,0}\right\vert _{\nu ,l}^{2} &=&1+\frac{\sin
^{2}2\theta }{2l^{2}}+\left( \frac{1}{\nu ^{2}}+\frac{1}{2l^{2}}-\left( 1+%
\frac{\sin ^{2}2\theta }{2l^{2}}\right) \right) \sin ^{2}2t \\
&\geq &1+\frac{\sin ^{2}2t}{\nu ^{2}}+\frac{\sin ^{2}2t}{2l^{2}},
\end{eqnarray*}%
the last term and the $\left\vert x^{2,0}\right\vert _{\nu ,l}^{4}$ factor
on the first term can be ignored. So (with a minor adjustment) our
inequality is%
\begin{equation*}
\frac{\sin ^{2}2t}{4}\geq 2\left( \frac{1}{\left\vert \left( \cos 2t\right)
\eta ^{2,0}\right\vert _{\nu ,l}^{2}}\right)
\end{equation*}%
or 
\begin{equation*}
t^{2}\geq \frac{2}{1+\frac{\sin ^{2}2t}{\nu ^{2}}}=\frac{2\nu ^{2}}{\nu
^{2}+\sin ^{2}2t},
\end{equation*}%
which happens when $t\geq O\left( \nu ^{1/2}\right) ,$ which is not good
enough for our purposes.

However, assuming that $t\leq \nu ^{1/2}$ allows us to greatly simplify our
estimates for $-\psi _{\nu ,l}D_{\zeta }D_{\zeta }\left( \psi _{\nu
,l}\right) .$ Indeed starting with 
\begin{equation*}
\zeta =n\cos \varphi +y\sin \varphi
\end{equation*}%
we have%
\begin{eqnarray*}
D_{\zeta }D_{\zeta }\psi _{\nu ,l} &=&\cos ^{2}\varphi \frac{\partial ^{2}}{%
\partial t^{2}}\psi _{\nu ,l}+2\cos \varphi \sin \varphi \cos 2t\frac{%
\partial }{\partial \theta }\frac{\partial }{\partial t}\psi _{\nu ,l}+\sin
^{2}\varphi \cos ^{2}2t\frac{\partial ^{2}}{\partial \theta ^{2}}\psi _{\nu
,l} \\
&&+\frac{\left\vert \hat{x}^{2,0}\right\vert _{\nu ,l}^{2}\cos 2t}{%
\left\vert \left( \cos 2t\right) \hat{\eta}^{2,0}\right\vert _{\nu ,l}^{3}}%
\left( D_{\zeta }\cos \varphi \right) -\frac{1}{4l^{2}}\frac{\sin 2t\cos
^{3}2t\cos 2\theta \sin 2\theta }{\left\vert \left( \cos 2t\right) \hat{\eta}%
^{2,0}\right\vert _{\nu ,l}^{3}}\left( D_{\zeta }\sin \varphi \right) .
\end{eqnarray*}%
When we consider our formulas for $\frac{\partial ^{2}}{\partial t^{2}}\psi
_{\nu ,l},$ $\frac{\partial }{\partial \theta }\frac{\partial }{\partial t}%
\psi _{\nu ,l},$ and $\frac{\partial ^{2}}{\partial \theta ^{2}}\psi _{\nu
,l}$ from the appendix, and the fact that $\left( D_{\zeta }\sin \varphi
\right) =O\left( t\right) ,$ we see that the second, third and last terms
are dominated by the first term when $t\leq O\left( \nu ^{1/2}\right) .$

The fourth term is positive (in $-D_{\zeta }D_{\zeta }\left( \psi _{\nu
,l}\right) ),$ so dropping it gives us that for $t\leq \nu ^{1/2}$%
\begin{eqnarray*}
-D_{\zeta }D_{\zeta }\psi _{\nu ,l} &\geq &-\frac{9}{10}\cos ^{2}\varphi 
\frac{\partial ^{2}}{\partial t^{2}}\psi _{\nu ,l} \\
&\geq &\left\vert x^{2,0}\right\vert _{\nu ,l}^{2}\left( \cos ^{2}\varphi
\right) \frac{\sin 2t}{\left\vert \left( \cos 2t\right) \eta
^{2,0}\right\vert _{\nu ,l}^{5}}\frac{5}{\nu _{l}^{2}}.
\end{eqnarray*}

Similarly, when $t\leq \nu ^{1/2}$ we have that the second term in our
estimate for $\left( D_{\zeta }\left( \psi _{\nu ,l}\right) \right) ^{2}$ is
overwhelmed by the first. So

\begin{equation*}
\left( D_{\zeta }\left( \psi _{\nu ,l}\right) \right) ^{2}\leq 2\left( \frac{%
\left\vert x^{2,0}\right\vert _{\nu ,l}^{4}\cos ^{2}2t}{\left\vert \left(
\cos 2t\right) \eta ^{2,0}\right\vert _{\nu ,l}^{6}}\right) \cos ^{2}\varphi
.
\end{equation*}%
Thus the total derivative is positive when%
\begin{equation*}
\psi _{\nu ,l}\left\vert x^{2,0}\right\vert _{\nu ,l}^{2}\left( \cos
^{2}\varphi \right) \frac{\sin 2t}{\left\vert \left( \cos 2t\right) \eta
^{2,0}\right\vert _{\nu ,l}^{5}}\frac{5}{\nu _{l}^{2}}\geq 2\left( \frac{%
\left\vert x^{2,0}\right\vert _{\nu ,l}^{4}\cos ^{2}2t}{\left\vert \left(
\cos 2t\right) \eta ^{2,0}\right\vert _{\nu ,l}^{6}}\right) \cos ^{2}\varphi
.
\end{equation*}%
Since $\psi _{\nu ,l}=\frac{\sin 2t}{2\left\vert \left( \cos 2t\right) \eta
^{2,0}\right\vert },$ this is equivalent to 
\begin{eqnarray*}
\frac{1}{2}\sin ^{2}2t\frac{5}{\nu _{l}^{2}} &\geq &2\left( \left\vert
x^{2,0}\right\vert _{\nu ,l}^{2}\cos ^{2}2t\right) \text{ or } \\
\sin ^{2}2t &\geq &\nu _{l}^{2}\text{ or} \\
4t^{2} &\geq &\nu _{l}^{2}
\end{eqnarray*}%
so its enough to have 
\begin{equation*}
t\geq \frac{1}{2}\nu _{l}.
\end{equation*}

To prove the integral inequality we first note that 
\begin{eqnarray*}
\left( D_{\zeta }\psi _{\nu ,l}\right) ^{2} &\geq &\left( \frac{1+\frac{\sin
^{2}2\theta }{l^{2}}}{2\left( \cos 2t+\frac{\sin ^{2}2t}{\nu ^{2}}\right)
^{3}}\right) \\
&\geq &\frac{1}{16}\text{ for }t\in \left[ 0,\frac{\nu }{2}\right]
\end{eqnarray*}%
So%
\begin{eqnarray*}
\int_{\gamma _{\zeta }}\mathrm{curv}_{s}\left( \zeta ,W\right)
&=&\int_{\gamma _{\zeta }}w_{h}^{2}s^{4}\left( D_{\zeta }\psi \right) ^{2} \\
&\geq &O\left( w_{h}^{2}s^{4}\nu \right) .
\end{eqnarray*}%
On the other hand, we note that for $t>\nu $%
\begin{equation*}
\left\vert \mathrm{curv}_{s}\left( \zeta ,W\right) \right\vert \leq
2s^{2}w_{h}^{2}\left\vert \psi _{\nu ,l}D_{\zeta }D_{\zeta }\left( \psi
_{\nu ,l}\right) \right\vert
\end{equation*}%
So we have to find the interval where 
\begin{equation*}
2s^{2}w_{h}^{2}\left\vert \psi _{\nu ,l}D_{\zeta }D_{\zeta }\left( \psi
_{\nu ,l}\right) \right\vert \leq O\left( w_{h}^{2}s^{4}\nu \right) ,
\end{equation*}%
or%
\begin{equation*}
\left\vert \psi _{\nu ,l}D_{\zeta }D_{\zeta }\left( \psi _{\nu ,l}\right)
\right\vert \leq O\left( s^{2}\nu \right) ,
\end{equation*}

Since%
\begin{eqnarray*}
D_{\zeta }D_{\zeta }\psi _{\nu ,l} &=&\cos ^{2}\varphi \frac{\partial ^{2}}{%
\partial t^{2}}\psi _{\nu ,l}+2\cos \varphi \sin \varphi \cos 2t\frac{%
\partial }{\partial \theta }\frac{\partial }{\partial t}\psi _{\nu ,l}+\sin
^{2}\varphi \cos ^{2}2t\frac{\partial ^{2}}{\partial \theta ^{2}}\psi _{\nu
,l} \\
&&+\frac{\left\vert \hat{x}^{2,0}\right\vert _{\nu ,l}^{2}\cos 2t}{%
\left\vert \left( \cos 2t\right) \eta ^{2,0}\right\vert _{\nu ,l}^{3}}\left(
D_{\zeta }\cos \varphi \right) -\frac{1}{4l^{2}}\frac{\sin 2t\cos ^{3}2t\cos
2\theta \sin 2\theta }{\left\vert \left( \cos 2t\right) \eta
^{2,0}\right\vert _{\nu ,l}^{3}}\left( D_{\zeta }\sin \varphi \right) ,
\end{eqnarray*}%
we can use our formulas for $\left\vert \left( \cos 2t\right) \eta
^{2,0}\right\vert _{\nu ,l}^{2}$ and the second derivatives of $\psi $ and
from the appendix to get a formula for $\psi _{\nu ,l}D_{\zeta }D_{\zeta
}\left( \psi _{\nu ,l}\right) .$ So the only unknown quantities in this
(complicated) formula are $\left( D_{\zeta }\cos \varphi \right) $ and $%
\left( D_{\zeta }\sin \varphi \right) ,$ whose order is $O\left( 1\right) .$
The important point is that for generic $t,$ the largest terms in this
formula for $\psi _{\nu ,l}D_{\zeta }D_{\zeta }\left( \psi _{\nu ,l}\right) $
are of order $\frac{\nu ^{4}}{l^{2}}.$ So we have that for sufficiently
large $t$%
\begin{equation*}
\left\vert \psi _{\nu ,l}D_{\zeta }D_{\zeta }\left( \psi _{\nu ,l}\right)
\right\vert =O\left( \frac{\nu ^{4}}{l^{2}}\right)
\end{equation*}%
using $l=O\left( \nu ^{1/3}\right) $ and $\nu =O\left( s^{6/7}\right) $ we
then get for $t$ sufficiently large 
\begin{eqnarray*}
\left\vert \psi _{\nu ,l}D_{\zeta }D_{\zeta }\left( \psi _{\nu ,l}\right)
\right\vert &\leq &O\left( \frac{\nu ^{4}}{\nu ^{2/3}}\right) \\
&=&O\left( \nu \nu ^{7/3}\right) \\
&=&O\left( \nu \left( s^{6/7}\right) ^{7/3}\right) \\
&=&O\left( \nu s^{2}\right)
\end{eqnarray*}%
as desired.

The interval where this holds is $\left[ O\left( c\right) ,\frac{\pi }{4}%
\right] ,$ where $c$ is the constant so that $c\nu =s^{6/7}.$
\end{proof}

Before leaving the subject of derivatives of $\psi $ we establish the
following estimate, which will be used in Section 11.

\begin{lemma}
\label{Derivatives} \addtocounter{algorithm}{1}%
\begin{equation}
\left\vert \frac{\psi }{D_{\zeta }D_{\zeta }\psi }\left[ D_{\zeta }\left(
\psi D_{\zeta }\psi \right) \right] \right\vert \leq \frac{\nu _{l}^{2}}{4}.
\label{Peter's est}
\end{equation}
\end{lemma}

\begin{remark}
Since 
\begin{equation*}
D_{\zeta }\left( \psi D_{\zeta }\psi \right) =\psi D_{\zeta }D_{\zeta }\psi
+\left( D_{\zeta }\psi \right) ^{2}
\end{equation*}%
and the two terms have opposite sign, it suffices to show 
\begin{equation*}
\frac{\psi }{D_{\zeta }D_{\zeta }\psi }\max \left\{ \psi D_{\zeta }D_{\zeta
}\psi ,\left( D_{\zeta }\psi \right) ^{2}\right\} \leq \frac{\nu _{l}^{2}}{4}%
.
\end{equation*}%
Since we prove this stronger estimate, we doubt that $\frac{1}{4}$ is the
optimal constant in \ref{Peter's est}; it is, nevertheless, sufficient for
our purposes.
\end{remark}

\begin{proof}
We have 
\begin{equation*}
\frac{\psi }{D_{\zeta }D_{\zeta }\psi }\psi D_{\zeta }D_{\zeta }\psi =\psi
^{2},
\end{equation*}%
and%
\begin{eqnarray*}
\psi ^{2} &=&\frac{1}{4}\frac{\sin ^{2}2t}{\left\vert \left( \cos 2t\right)
\eta ^{2,0}\right\vert _{\nu ,l}^{2}} \\
&=&\frac{1}{4}\frac{\sin ^{2}2t}{\left( \left\vert x^{2,0}\right\vert _{\nu
,l}^{2}\cos ^{2}2t+\frac{1}{\nu _{l}^{2}}\sin ^{2}2t\right) } \\
&=&\frac{\nu _{l}^{2}}{4}\frac{\sin ^{2}2t}{\left( \nu _{l}^{2}\left\vert
x^{2,0}\right\vert _{\nu ,l}^{2}\cos ^{2}2t+\sin ^{2}2t\right) } \\
&\leq &\frac{\nu _{l}^{2}}{4}.
\end{eqnarray*}

We saw above that 
\begin{equation*}
\psi D_{\zeta }D_{\zeta }\psi \geq \left( D_{\zeta }\psi \right) ^{2}
\end{equation*}%
when $t>\frac{\nu _{l}}{2},$ so we only have to establish 
\begin{equation*}
\frac{\psi }{D_{\zeta }D_{\zeta }\psi }\left( D_{\zeta }\psi \right)
^{2}\leq \frac{\nu _{l}^{2}}{4}
\end{equation*}%
when $t<\frac{\nu _{l}}{2}.$

We saw in the previous proof that for $t<\frac{\nu _{l}}{2},$%
\begin{equation*}
\left\vert D_{\zeta }D_{\zeta }\left( \psi _{\nu ,l}\right) \right\vert \geq
\left\vert x^{2,0}\right\vert _{\nu ,l}^{2}\left( \cos ^{2}\varphi \right) 
\frac{\sin 2t}{\left\vert \left( \cos 2t\right) \eta ^{2,0}\right\vert _{\nu
,l}^{5}}\frac{5}{\nu _{l}^{2}}
\end{equation*}

Similarly we have%
\begin{equation*}
\left( D_{\zeta }\left( \psi _{\nu ,l}\right) \right) ^{2}\leq 1.1\left( 
\frac{\left\vert \hat{x}^{2,0}\right\vert _{\nu ,l}^{4}}{\left\vert \left(
\cos 2t\right) \hat{\eta}^{2,0}\right\vert _{\nu ,l}^{6}}\right) \cos
^{2}\varphi
\end{equation*}%
for $t<\frac{\nu _{l}}{2}.$

So for $t<\frac{\nu _{l}}{2},$ 
\begin{eqnarray*}
&&\left\vert \frac{\psi }{D_{\zeta }D_{\zeta }\psi }\left( D_{\zeta }\left(
\psi _{\nu ,l}\right) \right) ^{2}\right\vert \\
&\leq &\frac{\frac{1}{2}\frac{\sin 2t}{\left\vert \left( \cos 2t\right) \hat{%
\eta}^{2,0}\right\vert _{\nu ,l}}\left[ 1.1\left( \frac{\left\vert \hat{x}%
^{2,0}\right\vert _{\nu ,l}^{4}}{\left\vert \left( \cos 2t\right) \hat{\eta}%
^{2,0}\right\vert _{\nu ,l}^{6}}\right) \cos ^{2}\varphi .\right] }{%
\left\vert \hat{x}^{2,0}\right\vert _{\nu ,l}^{2}\frac{\sin 2t}{\left\vert
\left( \cos 2t\right) \hat{\eta}^{2,0}\right\vert _{\nu ,l}^{5}}\left( \frac{%
5}{\nu _{l}^{2}}\right) \cos ^{2}\varphi } \\
&\leq &\frac{1.1\left\vert \hat{x}^{2,0}\right\vert _{\nu ,l}^{2}\nu _{l}^{2}%
}{10\left\vert \left( \cos 2t\right) \hat{\eta}^{2,0}\right\vert _{\nu
,l}^{2}} \\
&\leq &\frac{\left\vert \hat{x}^{2,0}\right\vert _{\nu ,l}^{2}\nu _{l}^{2}}{%
5\left\vert \left( \cos 2t\right) \hat{\eta}^{2,0}\right\vert _{\nu ,l}^{2}}
\\
&\leq &\frac{\nu _{l}^{2}}{4},
\end{eqnarray*}%
as desired.
\end{proof}

\subsection{Effect of Redistribution on $\protect\psi $}

\begin{proposition}
Proposition \ref{concentrated curvature} remains true after the
redistribution.
\end{proposition}

\begin{proof}
First we get a formula for $\psi $ after the redistribution in terms of $%
\psi $ before the redistribution. In other words, we will compare $\psi
_{\nu ,l}$ and $\psi _{\nu ,re,l}.$ For this proof only we call $\psi _{\nu
,l},$ $\psi _{\mathrm{old}},$ and all other quantities that are computed
with respect to $g_{\nu ,l}$ will have an \textquotedblleft
old\textquotedblright\ sub or superscript attached.

All of our derivatives of $\psi $ in this proof will be in the $\zeta $%
--direction so we write $\psi ^{\prime }$ for $D_{\zeta }\psi .$

Keeping in mind that $\psi _{\nu ,re,l}$ is the length of the horizontal
part of the Killing field $\left( 0,\frac{1}{2}\vartheta \right) ,$ we see
that we just need to compute the inner product of $\left( 0,\frac{1}{2}%
\vartheta \right) $ with the appropriate horizontal vector. Motivated by our
computations of Cheeger perturbations we see that in fact 
\begin{equation*}
\psi _{\nu ,re,l}=\frac{1}{2}\frac{\sin 2t}{\left\vert \cos 2t\tilde{\eta}%
^{2,0}\right\vert _{\nu ,re,l}}
\end{equation*}%
where $\tilde{\eta}^{2,0}$ is in the $\gamma $--part of the horizontal
space. More specifically 
\begin{equation*}
\cos 2t\tilde{\eta}^{2,0}=\cos 2t\eta ^{2,0}+\frac{\left( 1-\varphi
^{2}\right) }{\varphi ^{2}}\left( \cos 2t\eta ^{2,0}\right) ^{\mathcal{Z}%
^{\perp }}
\end{equation*}

Since the redistribution occurs before the $\left( U,D\right) $--Cheeger
perturbation, the computation of $\left( \cos 2t\eta ^{2,0}\right) ^{%
\mathcal{Z}^{\perp }},$ can be viewed as happening with respect to the
metric with $l=\infty ,$ or more formally it happens within the $Sp\left(
2\right) $--factor of $\left( S^{3}\right) ^{2}\times Sp\left( 2\right) ,$
where the product metric is the one that gives the $\left( U,D\right) $%
--Cheeger deformation.

To compute $\left( \cos 2t\eta ^{2,0}\right) ^{\mathcal{Z}^{\perp }}$ we
need its direction within $\mathcal{Z}^{\perp }.$ This direction looks like 
\begin{equation*}
\frac{1}{\sqrt{2}}\left( \frac{\vartheta _{3}}{\nu },\frac{\vartheta }{\nu }%
\right) ,
\end{equation*}%
there is a relationship between $\vartheta _{3}$ and $\vartheta ,$ but it
will not be important here.

So 
\begin{eqnarray*}
\left\vert \left( \cos 2t\eta ^{2,0}\right) ^{\mathcal{Z}^{\perp
}}\right\vert &=&\left\vert \left\langle \frac{1}{\sqrt{2}}\left( \frac{%
\vartheta _{3}}{\nu },\frac{\vartheta }{\nu }\right) ,\left( 0,\sin 2t\frac{%
\vartheta }{\nu ^{2}}\right) \right\rangle _{\nu }\frac{1}{\sqrt{2}}\left( 
\frac{\vartheta _{3}}{\nu },\frac{\vartheta }{\nu }\right) \right\vert \\
&=&\left\vert \frac{1}{2}\frac{\sin 2t}{\nu }\left( \frac{\vartheta _{3}}{%
\nu },\frac{\vartheta }{\nu }\right) \right\vert \\
&=&\frac{1}{2}\frac{\sin 2t}{\nu }
\end{eqnarray*}%
and%
\begin{eqnarray*}
\psi _{\nu ,re,l}^{2} &=&\frac{1}{4}\frac{\sin ^{2}2t}{\left\vert \cos 2t%
\tilde{\eta}^{2,0}\right\vert _{\nu ,re,l}^{2}} \\
&=&\frac{1}{4}\frac{\sin ^{2}2t}{\left\vert \cos 2t\eta ^{2,0}\right\vert _{%
\mathrm{old}}^{2}+2\frac{\left( 1-\varphi ^{2}\right) }{\varphi ^{2}}%
\left\langle \cos 2t\eta ^{2,0},\left( \cos 2t\eta ^{2,0}\right) ^{\mathcal{Z%
}^{\perp }}\right\rangle +\left\vert \frac{\left( 1-\varphi ^{2}\right) }{%
\varphi ^{2}}\left( \cos 2t\eta ^{2,0}\right) ^{\mathcal{Z}^{\perp
}}\right\vert ^{2}} \\
&=&\frac{1}{4}\frac{\sin ^{2}2t}{\left\vert \cos 2t\eta ^{2,0}\right\vert _{%
\mathrm{old}}^{2}+\frac{1}{2}\frac{\left( 1-\varphi ^{2}\right) }{\varphi
^{2}}\frac{\sin ^{2}2t}{\nu ^{2}}+\frac{\left( 1-\varphi ^{2}\right) ^{2}}{%
4\varphi ^{4}}\frac{\sin ^{2}2t}{\nu ^{2}}} \\
&=&\frac{\sin ^{2}2t}{4\left\vert \cos 2t\eta ^{2,0}\right\vert _{\mathrm{old%
}}^{2}}\frac{1}{1+\frac{1}{2}\frac{\sin ^{2}2t}{\nu ^{2}\left\vert \cos
2t\eta ^{2,0}\right\vert _{\mathrm{old}}^{2}}\frac{\left( 1-\varphi
^{2}\right) }{\varphi ^{2}}+\frac{\left( 1-\varphi ^{2}\right) ^{2}}{%
4\varphi ^{4}}\frac{\sin ^{2}2t}{\nu ^{2}\left\vert \cos 2t\eta
^{2,0}\right\vert _{\mathrm{old}}^{2}}} \\
&=&\frac{\sin ^{2}2t}{4\left\vert \cos 2t\eta ^{2,0}\right\vert _{\mathrm{old%
}}^{2}}\frac{1}{1+2\frac{\psi _{\mathrm{old}}^{2}}{\nu ^{2}}\left( \frac{%
\left( 1-\varphi ^{2}\right) }{\varphi ^{2}}+\frac{\left( 1-\varphi
^{2}\right) ^{2}}{2\varphi ^{4}}\right) } \\
&=&\psi _{\mathrm{old}}^{2}\left( 1-2\frac{\psi _{\mathrm{old}}^{2}}{\nu ^{2}%
}\left( 1-\varphi ^{2}\right) \right) +O
\end{eqnarray*}%
Since 
\begin{equation*}
\left( 1-\varphi ^{2}\right) =O\left( \nu ^{3}\right)
\end{equation*}%
We have 
\begin{equation*}
\psi _{\nu ,re,l}^{2}=\psi _{\mathrm{old}}^{2}+O
\end{equation*}%
and%
\begin{equation*}
\left( \psi _{\nu ,re,l}^{2}\right) ^{\prime }=\left( \psi _{\mathrm{old}%
}^{2}\right) ^{\prime }+8\frac{\psi _{\mathrm{old}}^{3}}{\nu ^{2}}\psi _{%
\mathrm{old}}^{\prime }\left( \varphi ^{2}-1\right) +4\frac{\psi _{\mathrm{%
old}}^{4}}{\nu ^{2}}\varphi \varphi ^{\prime }+O
\end{equation*}%
Since we also have 
\begin{equation*}
\left( \psi _{\nu ,re,l}^{2}\right) ^{\prime }=2\psi _{\nu ,re,l}\psi _{\nu
,re,l}^{\prime }
\end{equation*}%
We get%
\begin{equation*}
\psi _{\nu ,re,l}^{\prime }=\frac{\frac{1}{2}\left( \psi _{\mathrm{old}%
}^{2}\right) ^{\prime }+4\frac{\psi _{\mathrm{old}}^{3}}{\nu ^{2}}\psi _{%
\mathrm{old}}^{\prime }\left( \varphi ^{2}-1\right) +2\frac{\psi _{\mathrm{%
old}}^{4}}{\nu ^{2}}\varphi \varphi ^{\prime }}{\psi _{\nu ,re,l}}+O.
\end{equation*}%
Using $\psi _{\nu ,re,l}^{2}=\psi _{\mathrm{old}}^{2}+O,$ this becomes 
\begin{equation*}
\psi _{\nu ,re,l}^{\prime }=\psi _{\mathrm{old}}^{\prime }+4\frac{\psi _{%
\mathrm{old}}^{2}}{\nu ^{2}}\psi _{\mathrm{old}}^{\prime }\left( \varphi
^{2}-1\right) +2\frac{\psi _{\mathrm{old}}^{3}}{\nu ^{2}}\varphi \varphi
^{\prime }+O
\end{equation*}%
Since 
\begin{eqnarray*}
\varphi ^{\prime } &=&O\left( 100\nu ^{3}\right) \\
\varphi ^{2}-1 &=&O\left( 100\nu ^{3}\right)
\end{eqnarray*}%
and 
\begin{equation*}
\psi _{\mathrm{old}}^{\prime }\geq O\left( \nu ^{3}\right) \cos 2t
\end{equation*}%
we get 
\begin{equation*}
\psi _{\nu ,re,l}^{\prime }=\psi _{\mathrm{old}}^{\prime }+O.
\end{equation*}

It is impossible to get a similar formula for $\left( \psi _{\nu
,re,l}^{2}\right) ^{\prime \prime }$ in terms of $\left( \psi _{\mathrm{old}%
}^{2}\right) ^{\prime \prime },$ since $\left( \psi _{\mathrm{old}%
}^{2}\right) ^{\prime \prime }$ has a $0$ around $O\left( \nu \right) .$
Instead we will show that the difference $\left\vert \left( \psi _{\nu
,re,l}^{2}\right) ^{\prime \prime }-\left( \psi _{\mathrm{old}}^{2}\right)
^{\prime \prime }\right\vert $ is pointwise much smaller than $\max \left\{
\left( \psi _{\mathrm{old}}^{\prime }\right) ^{2},\left\vert \psi _{\mathrm{%
old}}\psi _{\mathrm{old}}^{\prime \prime }\right\vert \right\} .$

Combining this with our estimate $\psi _{\mathrm{redistr}}^{\prime }=\psi _{%
\mathrm{old}}^{\prime }+O$ gives us the proposition.

Starting with%
\begin{equation*}
\left( \psi _{\nu ,re,l}^{2}\right) ^{\prime }=\left( \psi _{\mathrm{old}%
}^{2}\right) ^{\prime }+8\frac{\psi _{\mathrm{old}}^{3}}{\nu ^{2}}\psi _{%
\mathrm{old}}^{\prime }\left( \varphi ^{2}-1\right) +4\frac{\psi _{\mathrm{%
old}}^{4}}{\nu ^{2}}\varphi \varphi ^{\prime }+O
\end{equation*}%
we have%
\begin{eqnarray*}
\left( \psi _{\nu ,re,l}^{2}\right) ^{\prime \prime } &=&\left( \psi _{%
\mathrm{old}}^{2}\right) ^{\prime \prime }+24\frac{\psi _{\mathrm{old}}^{2}}{%
\nu ^{2}}\left( \psi _{\mathrm{old}}^{\prime }\right) ^{2}\left( \varphi
^{2}-1\right) +8\frac{\psi _{\mathrm{old}}^{3}}{\nu ^{2}}\psi _{\mathrm{old}%
}^{\prime \prime }\left( \varphi ^{2}-1\right) +32\frac{\psi _{\mathrm{old}%
}^{3}}{\nu ^{2}}\psi _{\mathrm{old}}^{\prime }\varphi \varphi ^{\prime } \\
&&+4\frac{\psi _{\mathrm{old}}^{4}}{\nu ^{2}}\left( \varphi ^{\prime
}\right) ^{2}+4\frac{\psi _{\mathrm{old}}^{4}}{\nu ^{2}}\varphi \varphi
^{\prime \prime }
\end{eqnarray*}%
The second term is everywhere much smaller than $\left( \psi _{\mathrm{old}%
}^{\prime }\right) ^{2}.$ Similarly we can bound the third term by%
\begin{equation*}
\left\vert 8\frac{\psi _{\mathrm{old}}^{3}}{\nu ^{2}}\psi _{\mathrm{old}%
}^{\prime \prime }\left( \varphi ^{2}-1\right) \right\vert \leq \left\vert
800\nu \psi _{\mathrm{old}}^{3}\psi _{\mathrm{old}}^{\prime \prime
}\right\vert
\end{equation*}%
which is much smaller than $\psi _{\mathrm{old}}\psi _{\mathrm{old}}^{\prime
\prime }.$ The fourth term is%
\begin{equation*}
\left\vert 32\frac{\psi _{\mathrm{old}}^{3}}{\nu ^{2}}\psi _{\mathrm{old}%
}^{\prime }\varphi \varphi ^{\prime }\right\vert \leq 3200\nu \psi _{\mathrm{%
old}}^{3}\psi _{\mathrm{old}}^{\prime }
\end{equation*}%
and hence is much smaller than $\left( \psi _{\mathrm{old}}^{\prime }\right)
^{2}$ in the region where $t\leq O\left( c\right) $ that matters.

The fifth term is smaller than $O\left( \nu ^{8}\right) $ and $0$ at $t=0$
and hence smaller than both $\left( \psi _{\mathrm{old}}^{\prime }\right)
^{2}$ and $\psi _{\mathrm{old}}\psi _{\mathrm{old}}^{\prime \prime }$
everywhere $t\leq O\left( c\right) $.

The last term 
\begin{equation*}
\left\vert 4\frac{\psi _{\mathrm{old}}^{4}}{\nu ^{2}}\varphi \varphi
^{\prime \prime }\right\vert \leq 400\psi _{\mathrm{old}}^{4}
\end{equation*}%
and hence is smaller than both $\left( \psi _{\mathrm{old}}^{\prime }\right)
^{2}$ and $\psi _{\mathrm{old}}\psi _{\mathrm{old}}^{\prime \prime }$ on $%
\left( 0,100\nu \right) .$ On the other hand, on $\left( 50\nu ,\frac{\pi }{4%
}\right) ,$ 
\begin{equation*}
\left\vert 4\frac{\psi _{\mathrm{old}}^{4}}{\nu ^{2}}\varphi \varphi
^{\prime \prime }\right\vert \leq 40,000\nu \psi _{\mathrm{old}}^{4}
\end{equation*}%
and hence is much smaller than $\psi _{\mathrm{old}}\psi _{\mathrm{old}%
}^{\prime \prime }.$
\end{proof}

\section{Concrete A--Tensor Estimates}

In this section we refine our formulas for the two key $\left( 1,3\right) $%
--curvature tensors%
\begin{eqnarray*}
&&R^{s}\left( \zeta ,W\right) W\text{ and } \\
&&R^{s}\left( W,\zeta \right) \zeta
\end{eqnarray*}%
after the fibers are shrunk. We have to go beyond the abstract situation of
section 1, to compute the iterated $A$--tensors of $\Sigma
^{7}\longrightarrow S^{4}.$ Substituting 
\begin{eqnarray*}
\zeta &=&X \\
w_{h}k_{\gamma } &=&H_{w}
\end{eqnarray*}%
into Lemma \ref{Abstract (1,3) tensors} we have

\begin{lemma}
\label{1,3 tensors} 
\begin{equation*}
R^{g_{s}}\left( W,\zeta \right) \zeta =-s^{2}w_{h}\left( \frac{D_{\zeta
}D_{\zeta }\psi }{\psi }\right) k_{\gamma }-s^{2}\left[ w_{h}\frac{D_{\zeta
}\psi }{\psi }A_{\zeta }k_{\gamma }\right]
\end{equation*}%
\begin{equation*}
\left( R^{g_{s}}\left( \zeta ,W\right) W\right) ^{\mathcal{H}%
}=-s^{2}w_{h}^{2}\psi \nabla _{\zeta }\left( \mathrm{grad\,}\psi \right)
-\left( 1-s^{2}\right) s^{2}w_{h}\frac{D_{\zeta }\psi }{\psi }A_{k_{\gamma
}}W^{\mathcal{V}}
\end{equation*}
\end{lemma}

The possibilities for the iterated $A$--tensors in the curvature formulas
above are a bit daunting. We can nevertheless get estimates. First let $%
\left( V_{1}\oplus V_{2}\right) ^{GM}$ denote the intersection $V_{1}\oplus
V_{2}$ with the horizontal space for the Gromoll--Meyer submersion $%
q_{2,-1}:Sp\left( 2\right) \longrightarrow \Sigma ^{7},$ and let $V_{2,-1}$
be the horizontal lift to $TSp\left( 2\right) $ of the vertical space of $%
p_{2,-1}:\Sigma ^{7}\longrightarrow S^{4}.$ Then away from $t=\frac{\pi }{4}$%
, the orthogonal projection onto the vertical space $V_{2,-1}$ restricts to
an isomorphism $p_{\mathrm{orthog}}:\left( V_{1}\oplus V_{2}\right)
^{GM}\longrightarrow V_{2,-1}.$ Therefore the following lemma will give us
all of the data that we need.

\begin{lemma}
\label{A--tensor estimate}Let $\mathrm{II}$ denote the second fundamental
form of the $S^{2}$s in $S^{4},$ and let $S$ denote the shape operator.

For $U\in V_{1}\oplus V_{2}$, extend $U$ to be a Killing field for the $%
\left( h_{1}\oplus h_{2}\right) $--action. Then for $z\in \mathrm{span}%
\left\{ x^{2,0},y^{2,0}\right\} ,$ and $k_{\gamma }=\psi \eta _{u}^{2,0},$
with $\left\vert \eta _{u}^{2,0}\right\vert =1$ 
\begin{eqnarray*}
A_{z}U^{\mathcal{V}} &=&\left( \nabla _{z}^{\nu ,re,l}U\right) ^{\mathcal{H}%
}-S_{z}\left( U^{\mathcal{H}}\right) , \\
A_{k_{\gamma }}U^{\mathcal{V}} &=&\frac{\psi }{\left\vert \cos 2t\eta
^{2,0}\right\vert }\left( \nabla _{\left( \eta ,\eta \right) }^{\nu
,re,l}U\right) ^{\mathcal{H}}-\mathrm{II}\left( k_{\gamma },U^{\mathcal{H}%
}\right) +4\frac{\psi ^{3}}{\nu ^{3}}\left\vert U^{\alpha }\right\vert
_{h_{2}}\eta _{u,4}^{2,0}+O
\end{eqnarray*}%
where $\eta _{u,4}^{2,0}$ is the vector in \textrm{span}$\left\{ \eta
_{u,1}^{2,0},\eta _{u,2}^{2,0}\right\} $ that is perpendicular to $k_{\gamma
}$ and $U^{\alpha }$ denotes the $\alpha $--part of $U.$
\end{lemma}

\begin{proof}
To prove the first equation extend $U$ to be a Killing field for the $%
V_{1}\oplus V_{2}$ action. Then 
\begin{eqnarray*}
A_{z}U^{\mathcal{V}} &=&\left[ \nabla _{z}^{\nu ,re,l}\left( U-U^{\mathcal{H}%
}\right) \right] ^{\mathcal{H}} \\
&=&\left( \nabla _{z}^{\nu ,re,l}U\right) ^{\mathcal{H}}-\left( \nabla
_{z}^{\nu ,re,l}U^{\mathcal{H}}\right) ^{\mathcal{H}}.
\end{eqnarray*}%
Since $U^{\mathcal{H}}$ is a Killing field for the $h_{2}$--action on $%
S^{4}, $ if we extend $z$ to be a constant linear combination of $x^{2,0}$
and $y^{2,0},$ then $\left( \left[ z,U^{\mathcal{H}}\right] \right) ^{%
\mathcal{H}}=0.$ So 
\begin{equation*}
A_{z}U^{\mathcal{V}}=\left( \nabla _{z}^{\nu ,re,l}U\right) ^{\mathcal{H}%
}-S_{z}\left( U^{\mathcal{H}}\right)
\end{equation*}%
as claimed.

For the second equation we again extend $U$ to be a Killing field for the $%
V_{1}\oplus V_{2}$ action. As before 
\begin{equation*}
A_{k_{\gamma }}U^{\mathcal{V}}=\left( \nabla _{k_{\gamma }}^{\nu
,re,l}\left( U-U^{\mathcal{H}}\right) \right) ^{\mathcal{H}}
\end{equation*}%
Now 
\begin{eqnarray*}
\left( \nabla _{k_{\gamma }}^{\nu ,re,l}U\right) ^{\mathcal{H}} &=&\psi
\left( \nabla _{\eta _{u}^{2,0}}^{\nu ,re,l}\left( U\right) \right) ^{%
\mathcal{H}} \\
&=&\frac{\psi }{\left\vert \cos 2t\eta ^{2,0}\right\vert }\left( \nabla
_{\left( \eta ,\eta \right) }^{\nu ,re,l}U\right) ^{\mathcal{H}}+\left(
0,V\right) ^{\mathcal{H}}
\end{eqnarray*}%
where we have split $\eta _{u}^{2,0}$ into its horizontal and vertical parts
for $h_{1}\oplus h_{2}.$ Thus $V$ is a vector tangent to the $h_{2}$ orbits
and perpendicular to $U.$ It comes from differentiating $U$ in the direction
of the $V_{2}$--part of $k_{\gamma }.$ Since we are taking the horizontal
part of $V,$ only the $\alpha $--component of $U$ makes a contribution.
Since $k_{\gamma }=\psi \eta _{u}^{2,0},$ and the $V_{2}$--part of $\eta
_{u}^{2,0}$ is $\left( 0,2\psi \frac{\vartheta }{\nu _{2}^{2}}\right) ,$ we
have 
\begin{equation*}
\left( 0,V\right) ^{\mathcal{H}}=\psi \left( \nabla _{\left( 0,2\psi \frac{%
\vartheta }{\nu _{2}^{2}}\right) }^{\nu ,re,l}\left( 0,U^{\alpha }\right)
\right) ^{\mathcal{H}}.
\end{equation*}

If, for example, $\left( 0,U^{\alpha }\right) =\left( 0,\frac{N\alpha }{\nu
^{2}}\right) ,$ then%
\begin{eqnarray*}
\left( 0,V\right) &=&2\psi ^{2}\left( 0,\frac{N\gamma _{4}}{\nu ^{4}}\right)
,\text{ and } \\
\left\vert \left( 0,V\right) ^{\mathcal{H}}\right\vert &=&\left\vert
2\left\langle \psi ^{2}\left( 0,\frac{N\gamma _{4}}{\nu ^{4}}\right) ,\eta
_{u,4}^{2,0}\right\rangle \right\vert
\end{eqnarray*}%
where $\left( 0,\frac{N\gamma _{4}}{\nu ^{4}}\right) $ and $\eta
_{u,4}^{2,0} $ are perpendicular to $k_{\gamma }.$ Thus 
\begin{eqnarray*}
\left\vert \left( 0,V\right) ^{\mathcal{H}}\right\vert &=&4\frac{\psi ^{3}}{%
\nu ^{4}} \\
&=&4\frac{\psi ^{3}}{\nu ^{3}}\left\vert U^{\alpha }\right\vert _{h_{2}},
\end{eqnarray*}%
and 
\begin{equation*}
\left( 0,V\right) ^{\mathcal{H}}=\left( 4\frac{\psi ^{3}}{\nu ^{3}}%
\left\vert U^{\alpha }\right\vert _{h_{2}}\right) \eta _{u,4}^{2,0}+O
\end{equation*}

The \textquotedblleft $O$\textquotedblright\ is present because we did not
take the effect of the $\left( U,D\right) $--deformation into account. The
computation is very similar, but since $l=O\left( \nu ^{1/3}\right) ,$ the
terms we get do not play a significant role.

For the other term, since $U$ is a Killing field for the $h_{2}$--action 
\begin{equation*}
\left[ \nabla _{k_{\gamma }}^{\nu ,re,l}U^{\mathcal{H}}\right] ^{\mathcal{H}%
}=\mathrm{II}\left( k_{\gamma },U^{\mathcal{H}}\right) .
\end{equation*}%
So combining equations yields the claim.
\end{proof}

Combining the previous two results gives us

\begin{proposition}
\label{1,3 tensors after s} For $U\in \mathcal{H}_{p_{2,-1}},$%
\begin{equation*}
\left\langle R^{s}\left( W,\zeta \right) \zeta ,U\right\rangle
=-s^{2}w_{h}\left( \frac{D_{\zeta }D_{\zeta }\psi }{\psi }\right)
\left\langle k_{\gamma },U\right\rangle
\end{equation*}%
For $U\in V_{1}\oplus V_{2},U$ extend $U$ to be a Killing field for the $%
\left( h_{1}\oplus h_{2}\right) $--action. Then%
\begin{eqnarray*}
\left\langle R^{s}\left( W,\zeta \right) \zeta ,U\right\rangle
&=&-s^{2}w_{h}\left( \frac{D_{\zeta }D_{\zeta }\psi }{\psi }\right)
\left\langle k_{\gamma },U\right\rangle -s^{2}\left( 1-s^{2}\right) w_{h}%
\frac{D_{\zeta }\psi }{\psi }\left\langle k_{\gamma },S_{\zeta }\left( U^{%
\mathcal{H}}\right) \right\rangle \\
&&+s^{2}\left( 1-s^{2}\right) w_{h}\frac{D_{\zeta }\psi }{\psi }\left\langle
k_{\gamma },\nabla _{\zeta }^{\nu ,re,l}U\right\rangle .
\end{eqnarray*}%
Let $\eta _{u,W}^{2,0}$ be the unit vector in $\mathrm{span}\left\{ \eta
_{u,1}^{2,0},\eta _{u,2}^{2,0}\right\} $ that is proportional to the
projection of $W$ onto $\mathrm{span}\left\{ \eta _{u,1}^{2,0},\eta
_{u,2}^{2,0}\right\} ,$and let $\eta _{u,W^{\perp }}^{2,0}$ be perpendicular
to $\eta _{u,W}^{2,0}.$ Then%
\begin{eqnarray*}
\left( R^{s}\left( \zeta ,W\right) W\right) ^{\mathcal{H}}
&=&-s^{2}w_{h}^{2}\nabla _{\zeta }\left( \psi \mathrm{grad\,}\psi \right)
+s^{4}w_{h}^{2}\left( D_{\zeta }\psi \right) \left( \mathrm{grad\,}\psi
\right) + \\
&&-s^{2}w_{h}\frac{D_{\zeta }\psi }{\left\vert \cos 2t\eta ^{2,0}\right\vert 
}\left( \nabla _{\left( \eta ,\eta \right) }^{\nu ,re,l}W\right) ^{\mathcal{H%
}}+4w_{h}s^{2}\left[ D_{\zeta }\psi \right] \frac{\psi ^{2}}{\nu ^{3}}%
\left\vert W^{\alpha }\right\vert _{h_{2}}\eta _{u,W^{\perp }}^{2,0}+O
\end{eqnarray*}
\end{proposition}

\begin{proof}
From Lemma \ref{1,3 tensors} we have

\begin{equation*}
R^{g_{s}}\left( W,\zeta \right) \zeta =-s^{2}w_{h}\left( \frac{D_{\zeta
}D_{\zeta }\psi }{\psi }\right) k_{\gamma }-s^{2}\left[ w_{h}\frac{D_{\zeta
}\psi }{\psi }A_{\zeta }k_{\gamma }\right]
\end{equation*}%
So for $U\in \mathcal{H}_{p_{2,-1}},$%
\begin{equation*}
\left\langle R^{s}\left( W,\zeta \right) \zeta ,U\right\rangle
=-s^{2}w_{h}\left( \frac{D_{\zeta }D_{\zeta }\psi }{\psi }\right)
\left\langle k_{\gamma },U\right\rangle ,
\end{equation*}%
and for $U\in V_{1}\oplus V_{2}$%
\begin{eqnarray*}
\left\langle R^{s}\left( W,\zeta \right) \zeta ,U\right\rangle _{s}
&=&-s^{2}w_{h}\left( \frac{D_{\zeta }D_{\zeta }\psi }{\psi }\right)
\left\langle k_{\gamma },U\right\rangle _{s}-s^{2}w_{h}\frac{D_{\zeta }\psi 
}{\psi }\left\langle A_{\zeta }k_{\gamma },U\right\rangle _{s} \\
&=&-s^{2}w_{h}\left( \frac{D_{\zeta }D_{\zeta }\psi }{\psi }\right)
\left\langle k_{\gamma },U\right\rangle _{\nu ,re,l}+s^{2}\left(
1-s^{2}\right) w_{h}\frac{D_{\zeta }\psi }{\psi }\left\langle k_{\gamma
},A_{\zeta }U^{\mathcal{V}}\right\rangle _{\nu ,re,l}
\end{eqnarray*}%
Applying Lemma \ref{A--tensor estimate}%
\begin{eqnarray*}
\left\langle R^{s}\left( W,\zeta \right) \zeta ,U\right\rangle
&=&-s^{2}w_{h}\left( \frac{D_{\zeta }D_{\zeta }\psi }{\psi }\right)
\left\langle k_{\gamma },U\right\rangle -s^{2}\left( 1-s^{2}\right) w_{h}%
\frac{D_{\zeta }\psi }{\psi }\left\langle k_{\gamma },S_{\zeta }\left( U^{%
\mathcal{H}}\right) \right\rangle \\
&&+s^{2}\left( 1-s^{2}\right) w_{h}\frac{D_{\zeta }\psi }{\psi }\left\langle
k_{\gamma },\nabla _{\zeta }^{\nu ,re,l}U\right\rangle
\end{eqnarray*}

From Lemma \ref{1,3 tensors}%
\begin{equation*}
\left( R^{g_{s}}\left( \zeta ,W\right) W\right) ^{\mathcal{H}%
}=-s^{2}w_{h}^{2}\psi \nabla _{\zeta }\left( \mathrm{grad\,}\psi \right)
-\left( 1-s^{2}\right) s^{2}w_{h}\frac{D_{\zeta }\psi }{\psi }A_{k_{\gamma
}}W^{\mathcal{V}}
\end{equation*}%
Applying Lemma \ref{A--tensor estimate}%
\begin{eqnarray*}
\left( R^{s}\left( \zeta ,W\right) W\right) ^{\mathcal{H}}
&=&-s^{2}w_{h}^{2}\psi \nabla _{\zeta }\left( \mathrm{grad\,}\psi \right)
+\left( 1-s^{2}\right) s^{2}w_{h}\frac{D_{\zeta }\psi }{\psi }\mathrm{II}%
\left( k_{\gamma },W^{\mathcal{H}}\right) \\
&&-\left( 1-s^{2}\right) s^{2}w_{h}\frac{D_{\zeta }\psi }{\psi }\frac{\psi }{%
\left\vert \cos 2t\eta ^{2,0}\right\vert }\left( \nabla _{\left( \eta ,\eta
\right) }^{\nu ,re,l}W\right) ^{\mathcal{H}}+ \\
&&+4w_{h}s^{2}\left( 1-s^{2}\right) \frac{D_{\zeta }\psi }{\psi }\frac{\psi
^{3}}{\nu ^{3}}\left\vert W^{\alpha }\right\vert _{h_{2}}\eta _{u,W^{\perp
}}^{2,0}+O
\end{eqnarray*}

\noindent where $\eta _{u,W^{\perp }}^{2,0}$ is the unit vector in $\mathrm{%
span}\left\{ \eta _{1,u}^{2,0},\eta _{2,u}^{2,0}\right\} $ that is
perpendicular to $W^{\mathcal{H}}.$ Thus%
\begin{eqnarray*}
\left( R^{s}\left( \zeta ,W\right) W\right) ^{\mathcal{H}}
&=&-s^{2}w_{h}^{2}\psi \nabla _{\zeta }\left( \mathrm{grad\,}\psi \right)
-\left( 1-s^{2}\right) s^{2}w_{h}^{2}\left( D_{\zeta }\psi \right) \left( 
\mathrm{grad\,}\psi \right) + \\
&&-\frac{s^{2}w_{h}D_{\zeta }\psi }{\left\vert \cos 2t\eta ^{2,0}\right\vert 
}\left( \nabla _{\left( \eta ,\eta \right) }^{\nu ,re,l}W\right) ^{\mathcal{H%
}}+4w_{h}s^{2}\left[ D_{\zeta }\psi \right] \frac{\psi ^{2}}{\nu ^{3}}%
\left\vert W^{\alpha }\right\vert _{h_{2}}\eta _{u,W^{\perp }}^{2,0}+O \\
&=&-s^{2}w_{h}^{2}\nabla _{\zeta }\left( \psi \mathrm{grad\,}\psi \right)
+s^{4}w_{h}^{2}\left( D_{\zeta }\psi \right) \left( \mathrm{grad\,}\psi
\right) + \\
&&-\frac{s^{2}w_{h}D_{\zeta }\psi }{\left\vert \cos 2t\eta ^{2,0}\right\vert 
}\left( \nabla _{\left( \eta ,\eta \right) }^{\nu ,re,l}W\right) ^{\mathcal{H%
}}+4w_{h}s^{2}\left[ D_{\zeta }\psi \right] \frac{\psi ^{2}}{\nu ^{3}}%
\left\vert W^{\alpha }\right\vert _{h_{2}}\eta _{u,W^{\perp }}^{2,0}+O
\end{eqnarray*}
\end{proof}

\begin{corollary}
\label{curv after s}%
\begin{equation*}
\left\langle \left( R^{g_{s}}\left( \zeta ,W\right) W\right) ,\zeta
\right\rangle =-\left( s^{2}w_{h}^{2}\right) D_{\zeta }\left( \psi D_{\zeta
}\psi \right) +s^{4}w_{h}^{2}\left( D_{\zeta }\psi \right) ^{2}
\end{equation*}
\end{corollary}

\begin{proof}
For redundancy we compute $\left\langle \left( R^{g_{s}}\left( \zeta
,W\right) W\right) ,\zeta \right\rangle $ twice, using each of the last two
formulas of the previous proposition. Since%
\begin{equation*}
\nabla _{\zeta }^{\mathrm{redistr}}W\equiv 0,
\end{equation*}%
the second formula gives us 
\begin{eqnarray*}
\left\langle R^{\mathrm{redistr}}\left( W,\zeta \right) \zeta
,W\right\rangle &=&-s^{2}w_{h}\left( \frac{D_{\zeta }D_{\zeta }\psi }{\psi }%
\right) \left\langle k_{\gamma },W\right\rangle -s^{2}\left( 1-s^{2}\right)
w_{h}\frac{D_{\zeta }\psi }{\psi }\left\langle k_{\gamma },S_{\zeta }\left(
W^{\mathcal{H}}\right) \right\rangle \\
&=&-s^{2}w_{h}^{2}\psi D_{\zeta }D_{\zeta }\psi -s^{2}\left( 1-s^{2}\right)
w_{h}^{2}\left( D_{\zeta }\psi \right) ^{2} \\
&=&-s^{2}w_{h}^{2}\left( \psi D_{\zeta }D_{\zeta }\psi +\left( D_{\zeta
}\psi \right) ^{2}\right) +s^{4}w_{h}^{2}\left( D_{\zeta }\psi \right) ^{2}
\\
&=&-\left( s^{2}w_{h}^{2}\right) D_{\zeta }\left( \psi D\zeta \psi \right)
+s^{4}w_{h}^{2}\left( D_{\zeta }\psi \right) ^{2}
\end{eqnarray*}%
Computing the other way we get 
\begin{eqnarray*}
\left\langle R^{g_{s}}\left( \zeta ,W\right) W,\zeta \right\rangle
&=&-s^{2}w_{h}^{2}\left\langle \nabla _{\zeta }\left( \psi \mathrm{grad\,}%
\psi \right) ,\zeta \right\rangle +s^{4}w_{h}^{2}\left( D_{\zeta }\psi
\right) ^{2} \\
&=&-s^{2}w_{h}^{2}\left( \left( D_{\zeta }\psi \right) ^{2}+\psi
\left\langle \nabla _{\zeta }\mathrm{grad\,}\psi ,\zeta \right\rangle
\right) +s^{4}w_{h}^{2}\left( D_{\zeta }\psi \right) ^{2} \\
&=&-s^{2}w_{h}^{2}\left( \psi D_{\zeta }D_{\zeta }\psi +\left( D_{\zeta
}\psi \right) ^{2}\right) +s^{4}w_{h}^{2}\left( D_{\zeta }\psi \right) ^{2}
\\
&=&-\left( s^{2}w_{h}^{2}\right) D_{\zeta }\left( \psi D\zeta \psi \right)
+s^{4}w_{h}^{2}\left( D_{\zeta }\psi \right) ^{2}
\end{eqnarray*}
\end{proof}

In the remainder of this section we record the effect of the $s$%
--deformation on some key covariant derivatives that we will need later.

\begin{proposition}
\label{nabla^s_W W}%
\begin{equation*}
\nabla _{W}^{s}W=-s^{2}w_{h}^{2}\psi \mathrm{grad}\text{ }\psi ,
\end{equation*}%
\begin{eqnarray*}
\nabla _{W}^{s}W^{\gamma } &=&-s^{2}w_{h}^{2}\psi \mathrm{grad}\text{ }\psi ,%
\text{ and} \\
\nabla _{W^{\gamma }}^{s}W^{\gamma } &=&-s^{2}w_{h}^{2}\psi \mathrm{grad}%
\text{ }\psi
\end{eqnarray*}%
where $W^{\gamma }$ is the $\gamma $--part of $W.$
\end{proposition}

\begin{proof}
Since $W$ is a Killing field on $Sp\left( 2\right) $%
\begin{eqnarray*}
\left\langle \nabla _{W}^{s}W,Z\right\rangle &=&-\left\langle \nabla
_{Z}^{s}W,W\right\rangle \\
&=&-\frac{1}{2}D_{Z}\left\langle W,W\right\rangle \\
&=&-\frac{1}{2}D_{Z}\left\vert W\right\vert ^{2} \\
&=&-\left\vert W\right\vert D_{Z}\left\vert W\right\vert \\
&=&-\left\langle \left\vert W\right\vert _{s}\mathrm{grad}\left\vert
W\right\vert _{s},Z\right\rangle
\end{eqnarray*}%
Since%
\begin{eqnarray*}
\left\vert W\right\vert &=&\sqrt{\left( 1-s^{2}\right) \left\vert
W^{v}\right\vert _{\nu ,re,l}^{2}+\left\vert W^{h}\right\vert _{\nu
,re,l}^{2}},\text{ and } \\
&&\left\vert W\right\vert _{\nu ,re,l}\text{ is constant}
\end{eqnarray*}%
\begin{eqnarray*}
D_{Z}\left\vert W\right\vert _{s} &=&\frac{1}{2}\left( \left( 1-s^{2}\right)
\left\vert W^{v}\right\vert _{\nu ,re,l}^{2}+\left\vert W^{h}\right\vert
_{\nu ,re,l}^{2}\right) ^{-1/2}\left( -s^{2}D_{Z}\left\vert W^{v}\right\vert
_{\nu ,re,l}^{2}\right) \\
&=&\frac{1}{2}\frac{s^{2}}{\left\vert W\right\vert _{s}}D_{Z}\left(
\left\vert W^{h}\right\vert _{\nu ,re,l}^{2}\right) \\
&=&\frac{s^{2}}{\left\vert W\right\vert _{s}}\left\vert W^{h}\right\vert
_{\nu ,re,l}D_{Z}\left( \left\vert W^{h}\right\vert _{\nu ,re,l}\right) \\
&=&\frac{s^{2}}{\left\vert W\right\vert _{s}}w_{h}\psi D_{Z}\left( w_{h}\psi
\right) \\
&=&\frac{s^{2}}{\left\vert W\right\vert _{s}}w_{h}^{2}\psi D_{Z}\left( \psi
\right)
\end{eqnarray*}

Thus%
\begin{eqnarray*}
\left\langle \nabla _{W}^{s}W,Z\right\rangle &=&-\left\vert W\right\vert
_{s}D_{Z}\left\vert W\right\vert _{s} \\
&=&-s^{2}w_{h}^{2}\psi D_{Z}\left( \psi \right) \\
&=&-s^{2}w_{h}^{2}\left\langle Z,\psi \text{ }\mathrm{grad}\text{ }\psi
\right\rangle
\end{eqnarray*}%
So%
\begin{equation*}
\nabla _{W}^{s}W=-s^{2}w_{h}^{2}\psi \mathrm{grad}\text{ }\psi
\end{equation*}%
as claimed.

Since $W^{\gamma }$ is also a Killing field we have 
\begin{eqnarray*}
\left\langle \nabla _{W}^{s}W^{\gamma },Z\right\rangle &=&-\left\langle
\nabla _{Z}^{s}W^{\gamma },W\right\rangle \\
&=&-\frac{1}{2}D_{Z}\left\langle W^{\gamma },W\right\rangle .
\end{eqnarray*}

But $D_{Z}\left\langle W^{\gamma },W\right\rangle =D_{Z}\left\langle
W,W\right\rangle $ so $\nabla _{W}^{s}W^{\gamma }=-s^{2}w_{h}^{2}\psi 
\mathrm{grad}$ $\psi .$ Similarly $\nabla _{W^{\gamma }}^{s}W^{\gamma
}=-s^{2}w_{h}^{2}\psi \mathrm{grad}$ $\psi $
\end{proof}

\begin{proposition}
\label{nabla^s_zeta W}%
\begin{equation*}
\nabla _{\zeta }^{s}W=\nabla _{\zeta }^{s}W^{\gamma }=s^{2}\frac{D_{\zeta
}\psi }{\psi }H_{w}\text{.}
\end{equation*}
\end{proposition}

\begin{proof}
For any vertical field $U$ with respect to $p_{2,-1}:\Sigma \longrightarrow
S^{4}$ we have $D_{U}\left\langle W,\zeta \right\rangle =D_{W}\left\langle
U,\zeta \right\rangle =\left\langle \left[ U,W\right] ,\zeta \right\rangle =%
\left[ U,\zeta \right] ^{horiz}=0.$ So the Koszul formula gives us 
\begin{eqnarray*}
2\left\langle \nabla _{\zeta }^{s}W,U\right\rangle _{s} &=&D_{\zeta
}\left\langle W,U\right\rangle _{s}+\left\langle \left[ \zeta ,W\right]
,U\right\rangle _{s}+\left\langle \left[ U,\zeta \right] ,W\right\rangle _{s}
\\
&=&2\left( 1-s^{2}\right) \left\langle \nabla _{\zeta }^{\nu
,re,l}W,U\right\rangle _{\nu ,re,l} \\
&=&0.
\end{eqnarray*}%
Breaking $W$ into its horizontal and vertical parts we have 
\begin{eqnarray*}
0 &=&\left( \nabla _{\zeta }^{\nu ,re,l}W\right) ^{\mathcal{H}} \\
&=&\left( \nabla _{\zeta }^{\nu ,re,l}V\right) ^{\mathcal{H}}+\left( \nabla
_{\zeta }^{\nu ,re,l}H_{w}\right) ^{\mathcal{H}}
\end{eqnarray*}%
On the one hand, $\left( \nabla _{\zeta }^{\nu ,re,l}H_{w}\right) ^{\mathcal{%
H}}=\left( \nabla _{\zeta }^{s}H_{w}\right) ^{\mathcal{H}}.$ On the other
hand, for any basic horizontal field $Z$ 
\begin{eqnarray*}
2\left\langle \nabla _{\zeta }^{s}V,Z\right\rangle _{s} &=&-\left\langle 
\left[ \zeta ,Z\right] ,V\right\rangle _{s} \\
&=&-\left( 1-s^{2}\right) \left\langle \left[ \zeta ,Z\right]
,V\right\rangle _{0} \\
&=&\left( 1-s^{2}\right) 2\left\langle \nabla _{\zeta }^{\nu
,re,l}V,Z\right\rangle _{s}
\end{eqnarray*}%
So%
\begin{equation*}
\left( \nabla _{\zeta }^{s}V\right) ^{\mathcal{H}}=\left( 1-s^{2}\right)
\left( \nabla _{\zeta }^{\nu ,re,l}V\right) ^{\mathcal{H}}
\end{equation*}%
and 
\begin{eqnarray*}
\nabla _{\zeta }^{s}W &=&\left( \nabla _{\zeta }^{s}W\right) ^{\mathcal{H}}
\\
&=&\left( \nabla _{\zeta }^{s}V\right) ^{\mathcal{H}}+\left( \nabla _{\zeta
}^{s}H_{w}\right) ^{\mathcal{H}} \\
&=&\left( \nabla _{\zeta }^{\nu ,re,l}V+\nabla _{\zeta }^{\nu
,re,l}H_{w}\right) ^{\mathcal{H}}-s^{2}\left( \nabla _{\zeta }^{\nu
,re,l}V\right) ^{\mathcal{H}} \\
&=&-s^{2}\left( \nabla _{\zeta }^{\nu ,re,l}V\right) ^{\mathcal{H}} \\
&=&s^{2}\frac{D_{\zeta }\left\vert H_{w}\right\vert }{\left\vert
H_{w}\right\vert }H_{w} \\
&=&s^{2}\frac{D_{\zeta }\psi }{\psi }H_{w}
\end{eqnarray*}%
where for the next to last equality we have used Lemma \ref{Abstract
A--tensor}. A similar argument gives us $\nabla _{\zeta }^{s}W^{\gamma
}=s^{2}\frac{D_{\zeta }\psi }{\psi }H_{w}.$
\end{proof}

\section{Partial Conformal Change}

Having carried out deformations (1)--(4), we have apparently made things
worse. Indeed, from Corollary \ref{curv after s}, we see that near $t=0,$
some of the planes that used to have $0$--curvature now have negative
curvature. The ray of hope is that, as we discussed in section 1, the
integral of the curvatures over the old zero locus is now positive. In this
section, we will even it out to make it positive everywhere. The metric that
we obtain is in fact positively curved; however, after this section we will
only know that it is positively curved along the former zero locus. In the
final three sections we check that the curvature is positive everywhere.

Consider the $1$--dimensional subdistribution 
\begin{equation*}
\Delta \left( \alpha \right) =\mathrm{span}\left\{ \left( N\alpha p,N\alpha
\right) \right\} .
\end{equation*}%
We change the metric on $Sp\left( 2\right) $ by multiplying the restriction
to the orthogonal complement of $\Delta \left( \alpha \right) $ by a
function $e^{2f}.$ We leave $\Delta \left( \alpha \right) $ and its
orthogonal complement perpendicular to each other, and we leave the metric
restricted to $\Delta \left( \alpha \right) $ unchanged.

In each $S^{7}$--factor of $Sp\left( 2\right) \subset S^{7}\times S^{7},$
our distribution $\Delta \left( \alpha \right) $ is the intersection of the
vertical spaces of the two Hopf fibrations $h$ and $\tilde{h}.$ Since the
two Hopf actions are by symmetries of each other, our distribution $\Delta
\left( \alpha \right) $ is invariant under the Gromoll-Meyer action of $%
\left( S^{3}\times S^{3}\right) $ on $Sp\left( 2\right) ,$ and also under
the symmetry action of $S^{3}.$ So our new metric will be invariant under
all of these actions. In particular, it induces a metric on $\Sigma ^{7}.$

Our notational convention of writing vectors before the $\left( A^{u}\times
A^{d}\right) $--Cheeger deformation doesn't matter much when we talk about $%
\Delta \left( \alpha \right) ,$ since its invariant under the
\textquotedblleft Cheeger parameterization\textquotedblright . On the other
hand, the orthogonal complement of $\Delta \left( \alpha \right) $ is not
invariant, and we continue with our convention of page \pageref{notational
convention copy(1)}.

We choose 
\begin{equation*}
f=C-\frac{s^{2}}{2\nu ^{2}}\psi ^{2}+E,
\end{equation*}%
where $C$ is a constant that is a little larger than $1$ and $E$ is a
function $Sp\left( 2\right) \longrightarrow \mathbb{R}$ that is much smaller
than $\frac{s^{2}}{\nu ^{2}}\psi _{\nu ,l}^{2}$ in the $C^{2}$--topology.
The function $E$ has the form 
\begin{equation*}
E=I\circ \mathrm{dist}_{S^{4}}\left( \left( 0,0\right) ,\cdot \right) \circ
p_{GM}
\end{equation*}%
where 
\begin{equation*}
p_{GM}:Sp\left( 2\right) \longrightarrow S^{4}
\end{equation*}%
is the Gromoll-Meyer submersion, $\left( 0,0\right) $ one of the two points
in $S^{4}$ with $\left( \sin 2t,\sin 2\theta \right) =\left( 0,0\right) ,$
and $I:\mathbb{R}\longrightarrow \mathbb{R}$ is a function that satisfies 
\begin{eqnarray*}
I^{\prime }\left( 0\right) &=&0, \\
I^{\prime }|_{\left[ O\left( c\right) ,\frac{\pi }{4}\right] } &\equiv &0, \\
I^{\prime \prime } &=&O\left( \frac{s^{4}}{\nu ^{2}}\right) .
\end{eqnarray*}%
Thus%
\begin{eqnarray*}
\mathrm{grad\,}f &=&-\frac{s^{2}}{\nu ^{2}}\psi \mathrm{grad\,}\psi +\mathrm{%
grad\,}E \\
&=&-\frac{s^{2}}{\nu ^{2}}\psi \mathrm{grad\,}\psi +I^{\prime }\zeta .
\end{eqnarray*}

\begin{remark}
There is a minor problem with our partial conformal change. Our
distribution, $\Delta \left( \alpha \right) ,$ is three dimensional at $t=0,$
and one dimensional everywhere else. We circumvent this by having our
conformal change be a standard conformal change in a very, very small
neighborhood of $t=0,$ and then flattening out the $\Delta \left( \alpha
\right) $ portion. Since we can do this on an arbitrarily small neighborhood
of $t=0,$ the effect on curvatures can be made to be irrelevant.
\end{remark}

\begin{lemma}
\label{conf-perp to Delta(alpha)} Let $\nabla ^{\mathrm{old}}$ and $\nabla ^{%
\mathrm{new}}$ denote the covariant derivative before and after the partial
conformal change. If $x,y$ are fields that are orthogonal to $\Delta \left(
\alpha \right) $, then\addtocounter{algorithm}{1}%
\begin{eqnarray}
\nabla _{x}^{\mathrm{new}}y &=&O\left( e^{2f}-1\right) \left( \nabla _{x}^{%
\mathrm{old}}y\right) ^{\Delta \left( \alpha \right) }
\label{partial conf covar} \\
&&+\left( \nabla _{x}^{\mathrm{old}}y\right) ^{\Delta \left( \alpha \right)
,\perp }+\left( D_{x}f\right) y+\left( D_{y}f\right) x-\left\langle
x,y\right\rangle \nabla f,  \notag
\end{eqnarray}%
where the superscripts $^{\Delta \left( \alpha \right) }$ and $^{\Delta
\left( \alpha \right) ,\perp }$ denote the components tangent and
perpendicular to $\Delta \left( \alpha \right) .$
\end{lemma}

\begin{proof}
If we replace the first two terms on the right hand side of equation \ref%
{partial conf covar} with $\nabla _{x}^{\mathrm{old}}y,$ then we get the
formula for the covariant derivative after an actual conformal change. It
can be found in exercise 5a on page 90 in [Pet]. The three derivative terms
come from the three derivative terms in the Koszul formula.

When we test $\nabla _{x}^{\mathrm{new}}y$ by taking its inner product with
a vector in the orthogonal complement of $\Delta \left( \alpha \right) ,$
the Koszul formula looks precisely like the one for a standard conformal
change, and so we certainly have that the component of $\nabla _{x}^{\mathrm{%
new}}y$ that is perpendicular to $\Delta \left( \alpha \right) $ is given by %
\ref{partial conf covar}.

Finding the component tangent to $\Delta \left( \alpha \right) $ takes more
care. The important point is that there is no standard field that is tangent
to $\Delta \left( \alpha \right) .$ Indeed, \textquotedblleft $\alpha $%
\textquotedblright\ changes in the directions \textrm{span}$\left\{ \left(
\eta _{1},\eta _{1}\right) ,\left( \eta _{2},\eta _{2}\right) \right\} .$ So
even though we can compute the precise formula for the $\Delta \left( \alpha
\right) $--component in many cases, we can't get a general formula that is
much better than equation \ref{partial conf covar}.
\end{proof}

To deal with covariant derivatives involving vectors in $\Delta \left(
\alpha \right) $ we prove

\begin{lemma}

\begin{description}
\item[(i)] For $x$ and $U$ fields with 
\begin{equation*}
x\in H\cup V_{1}\oplus V_{2}\text{ and }U\in \mathrm{span}\left\{ \left(
N\alpha p,N\alpha \right) \right\}
\end{equation*}%
\begin{equation*}
\nabla _{x}^{\mathrm{new}}U=O\left( e^{2f}-1\right) \nabla _{x}^{\mathrm{old}%
}U,\text{ and}
\end{equation*}%
\begin{equation*}
\nabla _{U}^{\mathrm{new}}x=O\left( e^{2f}-1\right) \nabla _{U}^{\mathrm{old}%
}x.\text{ }
\end{equation*}

\item[(ii)] For $U=\left( N\alpha p,N\alpha \right) $ 
\begin{equation*}
\nabla _{U}^{\mathrm{new}}U=\nabla _{U}^{\mathrm{old}}U.
\end{equation*}
\end{description}
\end{lemma}

\begin{proof}
Since at least one of our fields is in $\Delta \left( \alpha \right) ,$ the
three derivative terms from equation \ref{partial conf covar} are not
present. For (i) the three Lie bracket terms of the Koszul formula can be a
bit complicated, so again we can't get general formulas that are much better
than the two we have asserted.

For (ii) the key point is that for $Z$ perpendicular to $\Delta \left(
\alpha \right) $, the Koszul formula gives us%
\begin{eqnarray*}
2\left\langle \nabla _{U}^{\mathrm{new}}U,Z\right\rangle _{\mathrm{new}}
&=&-D_{Z}\left\langle U,U\right\rangle _{\mathrm{new}}+2\left\langle \left[
Z,U\right] ,U\right\rangle _{\mathrm{new}} \\
&=&-D_{Z}\left\langle U,U\right\rangle _{\mathrm{old}}+2\left\langle \left[
Z,U\right] ,U\right\rangle _{\mathrm{old}} \\
&=&2\left\langle \nabla _{U}^{\mathrm{old}}U,Z\right\rangle _{\mathrm{old}}
\end{eqnarray*}%
Similarly $\left\langle \nabla _{U}^{\mathrm{new}}U,U\right\rangle _{\mathrm{%
new}}=\left\langle \nabla _{U}^{\mathrm{old}}U,U\right\rangle _{\mathrm{old}%
}.$
\end{proof}

For us the really important curvatures are 
\begin{eqnarray*}
&&\left( R\left( \zeta ,W\right) W\right) ^{\mathcal{H}}\text{ and } \\
&&R\left( W,\zeta \right) \zeta .
\end{eqnarray*}%
Fortunately we can get precise formulas for the required covariant
derivatives.

Note that $W$ is typically neither tangent nor perpendicular to $\Delta
\left( \alpha \right) .$ We let $W^{\gamma }$ denote the component of $W$
that is perpendicular to $\Delta \left( \alpha \right) .$ With this we have

\begin{lemma}
\begin{eqnarray*}
\nabla _{W}^{\mathrm{new}}W &=&\nabla _{W}^{\mathrm{old}}W-\left\langle
W^{\gamma },W^{\gamma }\right\rangle \nabla f \\
\nabla _{\zeta }^{\mathrm{new}}W &=&\nabla _{\zeta }^{\mathrm{old}}W+\left(
D_{\zeta }f\right) W^{\gamma },\text{ and } \\
\nabla _{\zeta }^{\mathrm{new}}\zeta &=&\nabla _{\zeta }^{\mathrm{old}}\zeta
+2\left( D_{\zeta }f\right) \zeta -\nabla f.
\end{eqnarray*}
\end{lemma}

\begin{remark}
In other words, if we replace $W$ with $W^{\gamma }$ then the formulas for
the three covariant derivatives are precisely the same as that of a standard
conformal change.
\end{remark}

\begin{proof}
If $Z$ is any standard field that is either $\zeta ,W,$ or initially
perpendicular to $\mathrm{span}\left\{ \zeta ,W\right\} ,$ then all three
Lie bracket terms in the three Koszul formulas for 
\begin{equation*}
\left\langle \nabla _{W}^{\mathrm{new}}W,Z\right\rangle ,\text{ }%
\left\langle \nabla _{\zeta }^{\mathrm{new}}W,Z\right\rangle ,\text{ and }%
\left\langle \nabla _{\zeta }^{\mathrm{new}}\zeta ,Z\right\rangle
\end{equation*}%
vanish. So the only change in the Koszul formula comes from the three
derivative terms, and only the $\gamma $--component of $W$ effects these
terms.
\end{proof}

To get the two key curvature formulas, we will also have to check the $%
\Delta \left( \alpha \right) $--components of the various iterated covariant
derivatives. Since $\Delta \left( \alpha \right) $ is contained in the
vertical space of $Sp\left( 2\right) \longrightarrow S^{4},$ we do not need
to worry about the $\Delta \left( \alpha \right) $--component of 
\begin{equation*}
\left( R\left( \zeta ,W\right) W\right) ^{\mathcal{H}}
\end{equation*}%
Thus it suffices to check the following.

\begin{proposition}
Before and after the partial conformal change the $\Delta \left( \alpha
\right) $--components of 
\begin{eqnarray*}
&&\nabla _{\zeta }\nabla _{\zeta }^{\mathrm{new}}W,\text{ and } \\
&&\nabla _{W}\nabla _{\zeta }^{\mathrm{new}}\zeta
\end{eqnarray*}%
are $0.$
\end{proposition}

\begin{proof}
The bottom line is that all of the Lie Bracket terms in all of the relevant
Koszul formulas are $0.$ Because of the importance of the result we check
this.

Let $\mathcal{V}$ be a unit field in \textrm{span}$\left\{ \Delta \left(
\alpha \right) \right\} $. Since the partial conformal change occurs after
the $\left( U,D\right) $--Cheeger deformations, we will have to consider all
of these computations as occurring in $\left( S^{3}\right) ^{2}\times
Sp\left( 2\right) .$

For $\left\langle \nabla _{\zeta }\nabla _{\zeta }^{\mathrm{new}}W,\mathcal{V%
}\right\rangle $ we first note that $\nabla _{\zeta }^{\mathrm{new}}W=\nabla
_{\zeta }^{\mathrm{old}}W+\left( D_{\zeta }f\right) W^{\gamma }$ and $\nabla
_{\zeta }^{\mathrm{old}}W\in \mathrm{span}\left\{ H_{w}\right\} .$ Next we
point out that in both the $Sp\left( 2\right) $ and the $\left( S^{3}\right)
^{2}$--factors, $\left[ \zeta ,\mathcal{V}\right] =0.$ It remains to compute
each of 
\begin{eqnarray*}
&&\left\langle \left[ \zeta ,H_{w}\right] ,\mathcal{V}\right\rangle , \\
&&\left\langle \left[ \mathcal{V},H_{w}\right] ,\zeta \right\rangle , \\
&&\left\langle \left[ \zeta ,W^{\gamma }\right] ,\mathcal{V}\right\rangle ,
\\
&&\left\langle \left[ \mathcal{V},W^{\gamma }\right] ,\zeta \right\rangle
\end{eqnarray*}%
These are all $0$ in both the $Sp\left( 2\right) $ and the $\left(
S^{3}\right) ^{2}$--components because in each case one of the vectors in
the inner product is an $\alpha $--vector and one of the vectors is a $%
\gamma $--vector.

For $\left\langle \nabla _{W}\nabla _{\zeta }^{\mathrm{new}}\zeta ,\mathcal{V%
}\right\rangle ,$ we note that 
\begin{equation*}
\nabla _{\zeta }^{\mathrm{new}}\zeta =\nabla _{\zeta }^{\mathrm{old}}\zeta
+2\left( D_{\zeta }f\right) \zeta -\nabla f
\end{equation*}%
and $\nabla _{\zeta }^{\mathrm{old}}\zeta =0.$ The terms%
\begin{eqnarray*}
&&\left\langle \left[ \mathcal{V},W\right] ,\zeta \right\rangle \text{ and }
\\
&&\left\langle \left[ \mathcal{V},W\right] ,\nabla f\right\rangle
\end{eqnarray*}%
are $0$ since $\left[ W,\mathcal{V}\right] $ is a $\gamma $--vector and both 
$\zeta $ and $\nabla f$ are $\alpha $--vectors.

The computations that gave us these $0$--planes in the first place yield
that each of 
\begin{eqnarray*}
&&\left[ W,\zeta \right] , \\
&&\left[ \mathcal{V},\zeta \right] , \\
&&\left[ \mathcal{V},\nabla f\right]
\end{eqnarray*}%
is $0.$

The inner product%
\begin{equation*}
\left\langle \left[ W,\nabla f\right] ,\mathcal{V}\right\rangle
\end{equation*}%
is $0$ in the $Sp\left( 2\right) $ factor since $W$ is vertically parallel.
In the $\left( S^{3}\right) ^{2}$--factor, we point out that $\left[
W,\nabla f\right] $ is a $\gamma $--vector so 
\begin{equation*}
\left\langle \left[ W,\nabla f\right] ,\mathcal{V}\right\rangle =0.
\end{equation*}
\end{proof}

Combining the previous two Lemmas we see that our two key curvature tensors%
\begin{eqnarray*}
&&R^{\mathrm{new}}\left( W,\zeta \right) \zeta \text{ and} \\
&&\left( R^{\mathrm{new}}\left( \zeta ,W\right) W\right) ^{\mathcal{H}}
\end{eqnarray*}%
are obtained from $R^{\mathrm{old}},$ from the familiar conformal change
formulas (cf exercise 5B on page 90 in \cite{Pet}) with $W$ replaced by $%
W^{\gamma }.$

\begin{proposition}
\label{conf curv}For any vector $U$ 
\begin{eqnarray*}
e^{-2f}\left\langle R^{\mathrm{new}}\left( W,\zeta \right) \zeta
,U\right\rangle &=&\left\langle R^{\mathrm{old}}\left( W,\zeta \right) \zeta
,U\right\rangle \\
&&-g\left( W^{\gamma },U\right) \mathrm{Hess}_{f}\left( \zeta ,\zeta \right)
-g\left( \zeta ,\zeta \right) \mathrm{Hess}_{f}\left( W^{\gamma },U\right)
+g\left( \zeta ,U\right) \mathrm{Hess}_{f}\left( W^{\gamma },\zeta \right) \\
&&+g\left( W^{\gamma },U\right) D_{\zeta }fD_{\zeta }f-g\left( \zeta ,\zeta
\right) g\left( W^{\gamma },U\right) \left\vert \mathrm{grad}f\right\vert
^{2}
\end{eqnarray*}%
For any vector $Z\in H^{2,-1}$%
\begin{eqnarray*}
&&e^{-2f}\left\langle R^{\mathrm{new}}\left( \zeta ,W\right)
W,Z\right\rangle =\left\langle R^{\mathrm{old}}\left( \zeta ,W\right)
W,Z\right\rangle \\
&&-g\left( \zeta ,Z\right) \mathrm{Hess}_{f}\left( W^{\gamma },W^{\gamma
}\right) -g\left( W^{\gamma },W^{\gamma }\right) \mathrm{Hess}_{f}\left(
\zeta ,Z\right) +g\left( W^{\gamma },Z\right) \mathrm{Hess}_{f}\left( \zeta
,W^{\gamma }\right) \\
&&+g\left( W^{\gamma },W^{\gamma }\right) D_{\zeta }fD_{Z}f-g\left(
W^{\gamma },W^{\gamma }\right) g\left( \zeta ,Z\right) \left\vert \mathrm{%
grad}f\right\vert ^{2}
\end{eqnarray*}
\end{proposition}

Since our deformation is not infinitesimal, this result is not enough. By
combining our first two lemmas on the covariant derivatives of the almost
conformal change we have

\begin{proposition}
\label{arbitrary conf change}For arbitrary $X,Y,Z,$ and $U$%
\begin{eqnarray*}
e^{-2f}R^{\mathrm{new}}\left( X,Y,Z,U\right) &=&R^{\mathrm{old}}\left(
X,Y,Z,U\right) \\
&&-g\left( X,U\right) \mathrm{Hess}_{f}\left( Y,Z\right) -g\left( Y,Z\right) 
\mathrm{Hess}_{f}\left( X,U\right) \\
&&+g\left( X,Z\right) \mathrm{Hess}_{f}\left( Y,U\right) +g\left( Y,U\right) 
\mathrm{Hess}_{f}\left( X,Z\right) \\
&&+g\left( X,U\right) D_{Y}fD_{Z}f+g\left( Y,Z\right) D_{X}fD_{U}f \\
&&-g\left( Y,U\right) D_{X}fD_{Z}f-g\left( X,Z\right) D_{Y}fD_{U}f \\
&&-g\left( Y,Z\right) g\left( X,U\right) \left\vert \mathrm{grad}%
f\right\vert ^{2}+g\left( X,Z\right) g\left( Y,U\right) \left\vert \mathrm{%
grad}f\right\vert ^{2}, \\
&&+O\left( e^{2f}-1,\left\vert \mathrm{grad}f\right\vert \right) \mathrm{%
\max }\left\{ R^{\mathrm{old}}\left( X,Y,Z,U\right) ,\left\vert X\right\vert
\left\vert Y\right\vert \left\vert Z\right\vert \left\vert U\right\vert
\right\}
\end{eqnarray*}
\end{proposition}

To evaluate curvatures we need to compute the Hessian of $f.$ Recall that $%
\xi $ is the vector in $\mathrm{span}\left\{ x^{2,0},y^{2,0}\right\} $ that
is perpendicular to $\zeta .$ Some of the formulas below are redundant. We
include the redundancy for later convenience.

\begin{proposition}
\label{Hessians}%
\begin{equation*}
\mathrm{Hess}_{f}\left( \zeta ,\zeta \right) =-\frac{s^{2}}{\nu ^{2}}%
D_{\zeta }\left( \psi D_{\zeta }\psi \right) +I^{\prime \prime }
\end{equation*}%
\begin{equation*}
\mathrm{Hess}_{f}\left( \zeta ,\xi \right) =\frac{s^{2}}{\nu ^{2}}\left(
D_{\zeta }\left( \psi \right) D_{\xi }\left( \psi \right) +\psi D_{\zeta }%
\left[ D_{\xi }\left( \psi \right) \right] -\psi D_{\xi }\left( \psi \right)
O\left( \frac{t}{l^{2}}\right) \right)
\end{equation*}%
\begin{equation*}
\mathrm{Hess}_{f}\left( \zeta ,y^{2,0}\right) =-\frac{s^{2}}{\nu ^{2}}\left(
D_{\zeta }\left( \psi \right) D_{y^{2,0}}\left( \psi \right) +\psi D_{\zeta
}D_{y^{2,0}}\left( \psi \right) -\psi \left\vert \mathrm{grad\,}\psi
\right\vert O\left( \frac{t}{l^{2}}\right) \right) +I^{\prime \prime
}\left\langle \zeta ,y^{2,0}\right\rangle
\end{equation*}%
\begin{equation*}
\nabla _{W^{\gamma }}^{s}\mathrm{grad}f=-\frac{s^{4}}{\nu ^{2}}\left\vert 
\mathrm{grad\,}\psi \right\vert ^{2}H_{w}+\frac{s^{2}}{\nu ^{2}}\psi \left(
D_{\xi }\psi \right) \nabla _{W}^{\nu ,re,l}\xi +O
\end{equation*}%
\begin{equation*}
\mathrm{Hess}_{f}\left( W^{\gamma },W^{\gamma }\right) =-s^{4}\frac{w_{h}^{2}%
}{\nu ^{2}}\psi ^{2}\left\vert \mathrm{grad\,}\psi \right\vert ^{2}+O
\end{equation*}
\end{proposition}

\begin{proof}
Since 
\begin{equation*}
\mathrm{grad\,}f=-\frac{s^{2}}{\nu ^{2}}\psi \mathrm{grad\,}\psi +I^{\prime
}\zeta .
\end{equation*}

we have 
\begin{eqnarray*}
\mathrm{Hess}_{f}\left( \zeta ,\zeta \right) &=&-\frac{s^{2}}{\nu ^{2}}%
\left\langle \nabla _{\zeta }\left( \psi \mathrm{grad\,}\psi \right) ,\zeta
\right\rangle +\left\langle \nabla _{\zeta }\left( I^{\prime }\zeta \right)
,\zeta \right\rangle \\
&=&-\frac{s^{2}}{\nu ^{2}}\left( \left( D_{\zeta }\psi \right) ^{2}+\psi
\left\langle \nabla _{\zeta }\left( \mathrm{grad}\psi \right) ,\zeta
\right\rangle \right) +I^{\prime \prime } \\
&=&-\frac{s^{2}}{\nu ^{2}}\left( \left( D_{\zeta }\psi \right) ^{2}+\psi
D_{\zeta }D_{\zeta }\psi \right) +I^{\prime \prime } \\
&=&-\frac{s^{2}}{\nu ^{2}}D_{\zeta }\left( \psi D_{\zeta }\psi \right)
+I^{\prime \prime }
\end{eqnarray*}%
\begin{eqnarray*}
\mathrm{Hess}_{f}\left( \zeta ,\xi \right) &=&\left\langle \nabla _{\zeta }%
\mathrm{grad}f,\xi \right\rangle \\
&=&-\frac{s^{2}}{\nu ^{2}}\left\langle \nabla _{\zeta }\left( \psi \mathrm{%
grad\,}\psi \right) ,\xi \right\rangle +\left\langle \nabla _{\zeta }\left(
I^{\prime }\zeta \right) ,\xi \right\rangle \\
&=&-\frac{s^{2}}{\nu ^{2}}\left( D_{\zeta }\left( \psi \right) \left\langle 
\mathrm{grad\,}\psi ,\xi \right\rangle +\psi \left\langle \nabla _{\zeta
}\left( \mathrm{grad\,}\psi \right) ,\xi \right\rangle \right) \\
&=&-\frac{s^{2}}{\nu ^{2}}\left( D_{\zeta }\left( \psi \right) D_{\xi
}\left( \psi \right) +\psi D_{\zeta }\left\langle \left( \mathrm{grad\,}\psi
\right) ,\xi \right\rangle -\psi \left\langle \left( \mathrm{grad\,}\psi
\right) ,\nabla _{\zeta }\xi \right\rangle \right) \\
&=&-\frac{s^{2}}{\nu ^{2}}\left( D_{\zeta }\left( \psi \right) D_{\xi
}\left( \psi \right) +\psi D_{\zeta }D_{\xi }\left( \psi \right) -\psi
\left\langle \left( \mathrm{grad\,}\psi \right) ,\xi \right\rangle O\left( 
\frac{t}{l^{2}}\right) \right) \\
&=&-\frac{s^{2}}{\nu ^{2}}\left( D_{\zeta }\left( \psi \right) D_{\xi
}\left( \psi \right) +\psi D_{\zeta }\left[ D_{\xi }\left( \psi \right) %
\right] -\psi D_{\xi }\left( \psi \right) O\left( \frac{t}{l^{2}}\right)
\right)
\end{eqnarray*}%
\begin{eqnarray*}
\mathrm{Hess}_{f}\left( \zeta ,y^{2,0}\right) &=&\left\langle \nabla _{\zeta
}\mathrm{grad}f,y^{2,0}\right\rangle \\
&=&-\frac{s^{2}}{\nu ^{2}}\left\langle \nabla _{\zeta }\left( \psi \mathrm{%
grad\,}\psi \right) ,y^{2,0}\right\rangle +\left\langle \nabla _{\zeta
}\left( I^{\prime }\zeta \right) ,y^{2,0}\right\rangle \\
&=&-\frac{s^{2}}{\nu ^{2}}\left( D_{\zeta }\left( \psi \right) \left\langle 
\mathrm{grad\,}\psi ,y^{2,0}\right\rangle +\psi \left\langle \nabla _{\zeta
}\left( \mathrm{grad\,}\psi \right) ,y^{2,0}\right\rangle \right) +O \\
&=&-\frac{s^{2}}{\nu ^{2}}\left( D_{\zeta }\left( \psi \right)
D_{y^{2,0}}\left( \psi \right) +\psi D_{\zeta }\left\langle \left( \mathrm{%
grad\,}\psi \right) ,y^{2,0}\right\rangle -\psi \left\langle \left( \mathrm{%
grad\,}\psi \right) ,\nabla _{\zeta }y^{2,0}\right\rangle \right) +O
\end{eqnarray*}%
To evaluate the next to last term%
\begin{eqnarray*}
\left\langle \mathrm{grad\,}\psi ,\nabla _{\zeta }y^{2,0}\right\rangle
&=&\cos \varphi \left\langle \mathrm{grad\,}\psi ,\nabla
_{x^{20}}y^{2,0}\right\rangle +\sin \varphi \left\langle \mathrm{grad\,}\psi
,\nabla _{y^{20}}y^{2,0}\right\rangle \\
&=&\left\vert \mathrm{grad\,}\psi \right\vert O\left( \frac{t}{l^{2}}\right)
\end{eqnarray*}%
So%
\begin{equation*}
\mathrm{Hess}_{f}\left( \zeta ,y^{2,0}\right) =-\frac{s^{2}}{\nu ^{2}}\left(
D_{\zeta }\left( \psi \right) D_{y^{2,0}}\left( \psi \right) +\psi D_{\zeta
}D_{y^{2,0}}\left( \psi \right) -\psi \left\vert \mathrm{grad\,}\psi
\right\vert O\left( \frac{t}{l^{2}}\right) \right) +O
\end{equation*}%
To find $\nabla _{W^{\gamma }}\mathrm{grad}f$ we note that since $\mathrm{%
grad}f\in \mathrm{span}\left\{ x^{2,0},y^{2,0}\right\} ,$ and $\nabla
_{W^{\gamma }}^{\nu ,re,l}\zeta =0,$ we should think of $\mathrm{grad}f$ as
a linear combination of $\zeta $ and $\xi .$ Since this combination is
constant in the $W$ direction we have 
\begin{eqnarray*}
\nabla _{W^{\gamma }}^{\nu ,re,l}\mathrm{grad}f &=&\left\langle \mathrm{grad}%
f,\xi \right\rangle \nabla _{W^{\gamma }}^{\nu ,re,l}\xi \\
&=&\frac{s^{2}}{\nu ^{2}}\psi \left( D_{\xi }\psi \right) \nabla _{W^{\gamma
}}^{\nu ,re,l}\xi
\end{eqnarray*}

We proved in Proposition \ref{nabla^s_zeta W} that 
\begin{eqnarray*}
\nabla _{\zeta }^{s}W^{\gamma } &=&s^{2}\frac{D_{\zeta }\psi }{\psi }H_{w} \\
&=&\nabla _{\zeta }^{\nu ,re,l}W^{\gamma }+s^{2}\frac{D_{\zeta }\psi }{\psi }%
H_{w}.
\end{eqnarray*}%
A similar argument gives us 
\begin{equation*}
\nabla _{W^{\gamma }}^{s}\mathrm{grad}f=\nabla _{W^{\gamma }}^{\nu ,re,l}%
\mathrm{grad}f+s^{2}\frac{D_{\mathrm{grad}f}\psi }{\psi }H_{w}
\end{equation*}%
Substituting we get 
\begin{equation*}
\nabla _{W^{\gamma }}^{s}\mathrm{grad}f=\frac{s^{2}}{\nu ^{2}}\psi \left(
D_{\xi }\psi \right) \nabla _{W^{\gamma }}^{\nu ,re,l}\xi +s^{2}\frac{%
\left\langle \mathrm{grad}\,f,\mathrm{grad}\,\psi \right\rangle }{\psi }H_{w}
\end{equation*}

Since $\mathrm{grad}\,f=-\frac{s^{2}}{\nu ^{2}}\psi \mathrm{grad\,}\psi
+I^{\prime }\zeta $ and $I^{\prime }=O\left( \frac{s^{4}}{\nu ^{2}}\right) ,$
we get 
\begin{equation*}
\nabla _{W^{\gamma }}^{s}\mathrm{grad}f=\frac{s^{2}}{\nu ^{2}}\psi \left(
D_{\xi }\psi \right) \nabla _{W}^{\nu ,re,l}\xi -\frac{s^{4}}{\nu ^{2}}%
\left\vert \mathrm{grad\,}\psi \right\vert ^{2}H_{w}+O
\end{equation*}

as claimed.

For redundancy we compute%
\begin{eqnarray*}
-\mathrm{Hess}_{f}\left( W^{\gamma },W^{\gamma }\right) &=&-\left\langle
\nabla _{W^{\gamma }}\mathrm{grad}f,W^{\gamma }\right\rangle \\
&=&+\left\langle \mathrm{grad}f,\nabla _{W^{\gamma }}W^{\gamma }\right\rangle
\\
&=&\left\langle -\frac{s^{2}}{\nu ^{2}}\psi \mathrm{grad\,}\psi
,-s^{2}w_{h}^{2}\psi \mathrm{grad}\text{ }\psi \right\rangle +\left\langle
-I^{\prime }\zeta ,-s^{2}w_{h}^{2}\psi \mathrm{grad}\text{ }\psi
\right\rangle \\
&=&s^{4}\frac{w_{h}^{2}}{\nu ^{2}}\psi ^{2}\left\vert \mathrm{grad\,}\psi
\right\vert ^{2}+O
\end{eqnarray*}
\end{proof}

We can now compute $\mathrm{curv}\left( \zeta ,W\right) $

\begin{proposition}
\label{new curv} 
\begin{equation*}
e^{-2f}\left\langle R^{\mathrm{new}}\left( \zeta ,W\right) W,\zeta
\right\rangle _{\mathrm{new}}=s^{4}w_{h}^{2}\left( D_{\zeta }\psi \right)
^{2}+s^{4}\frac{w_{h}^{2}}{\nu ^{2}}\psi ^{2}\left\langle \mathrm{grad\,}%
\psi ,\zeta \right\rangle ^{2}+\iota +O,
\end{equation*}%
where 
\begin{equation*}
\iota \equiv -\left\vert W^{\gamma }\right\vert ^{2}I^{\prime \prime }.
\end{equation*}%
In particular, we can choose $\iota $ so that the zero planes with respect
to $g_{\nu ,l}$ have positive curvature with respect to $g_{new}.$
\end{proposition}

\begin{proof}
Our partial conformal change formula gives us 
\begin{eqnarray*}
e^{-2f}\left\langle R^{\mathrm{new}}\left( \zeta ,W\right) W,\zeta
\right\rangle _{\mathrm{new}} &=&\left\langle R^{s}\left( \zeta ,W\right)
W,\zeta \right\rangle _{s}-\mathrm{Hess}_{f}\left( \zeta ,\zeta \right)
\left\vert W^{\gamma }\right\vert ^{2}-\mathrm{Hess}_{f}\left( W^{\gamma
},W^{\gamma }\right) \left\vert \zeta \right\vert ^{2} \\
&&+\left( D_{\zeta }f\right) ^{2}\left\vert W^{\gamma }\right\vert
^{2}-\left\vert \nabla f\right\vert ^{2}\left\vert W^{\gamma }\right\vert
^{2}\left\vert \zeta \right\vert ^{2}
\end{eqnarray*}%
To evaluate this we combine%
\begin{equation*}
\left\langle \left( R^{s}\left( \zeta ,W\right) W\right) ,\zeta
\right\rangle _{s}=-\left( s^{2}w_{h}^{2}\right) D_{\zeta }\left( \psi
D_{\zeta }\psi \right) +s^{4}w_{h}^{2}\left( D_{\zeta }\psi \right) ^{2}+O
\end{equation*}%
\begin{eqnarray*}
\left\vert W^{\gamma }\right\vert ^{2}\mathrm{Hess}_{f}\left( \zeta ,\zeta
\right) &=&-\left\vert W^{\gamma }\right\vert ^{2}\frac{s^{2}}{\nu ^{2}}%
D_{\zeta }\left( \psi D_{\zeta }\psi \right) +\left\vert W^{\gamma
}\right\vert ^{2}I^{\prime \prime } \\
&=&-\nu ^{2}w_{h}^{2}\frac{s^{2}}{\nu ^{2}}D_{\zeta }\left( \psi D_{\zeta
}\psi \right) -\iota \\
&=&-w_{h}^{2}s^{2}D_{\zeta }\left( \psi D_{\zeta }\psi \right) -\iota
\end{eqnarray*}%
\begin{equation*}
\mathrm{Hess}_{f}\left( W^{\gamma },W^{\gamma }\right) =-s^{4}\frac{w_{h}^{2}%
}{\nu ^{2}}\psi ^{2}\left\vert \mathrm{grad\,}\psi \right\vert ^{2}+O
\end{equation*}%
and%
\begin{eqnarray*}
-\left\vert W^{\gamma }\right\vert ^{2}\left\vert \mathrm{grad}f\right\vert
^{2}+\left\vert W^{\gamma }\right\vert ^{2}\left( D_{\zeta }f\right) ^{2}
&=&-\left\vert W^{\gamma }\right\vert ^{2}\frac{s^{4}}{\nu ^{4}}\psi
^{2}\left\vert \mathrm{grad\,}\psi \right\vert ^{2}+\left\vert W^{\gamma
}\right\vert ^{2}\frac{s^{4}}{\nu ^{4}}\psi ^{2}\left\langle \mathrm{grad\,}%
\psi ,\zeta \right\rangle ^{2}+O \\
&=&-s^{4}\frac{w_{h}^{2}}{\nu ^{2}}\psi ^{2}\left\vert \mathrm{grad\,}\psi
\right\vert ^{2}+s^{4}\frac{w_{h}^{2}}{\nu ^{2}}\psi ^{2}\left\langle 
\mathrm{grad\,}\psi ,\zeta \right\rangle ^{2}+O,
\end{eqnarray*}%
to get%
\begin{equation*}
e^{-2f}\left\langle R^{\mathrm{new}}\left( \zeta ,W\right) W,\zeta
\right\rangle _{\mathrm{new}}=s^{4}w_{h}^{2}\left( D_{\zeta }\psi \right)
^{2}+s^{4}\frac{w_{h}^{2}}{\nu ^{2}}\psi ^{2}\left\langle \mathrm{grad\,}%
\psi ,\zeta \right\rangle ^{2}+\iota +O
\end{equation*}%
as desired.
\end{proof}

\begin{proposition}
\label{R(zeta, W, W, xi)}%
\begin{equation*}
e^{-2f}R^{\mathrm{new}}\left( \zeta ,W,W,\xi \right) =-s^{2}w_{h}\frac{%
D_{\zeta }\psi }{\left\vert \cos 2t\eta ^{2,0}\right\vert }\left\langle
\nabla _{\left( \eta ,\eta \right) }^{\nu ,re,l}W,\xi \right\rangle +O
\end{equation*}
\end{proposition}

\begin{proof}
From Proposition \ref{conf curv} we have%
\begin{equation*}
e^{-2f}R^{\mathrm{new}}\left( \zeta ,W,W,\xi \right) =R^{\mathrm{old}}\left(
\zeta ,W,W,\xi \right) -g\left( W^{\gamma },W^{\gamma }\right) \mathrm{Hess}%
_{f}\left( \zeta ,\xi \right) +g\left( W^{\gamma },W^{\gamma }\right)
D_{\zeta }fD_{\xi }f
\end{equation*}%
From Proposition \ref{1,3 tensors after s} we have%
\begin{eqnarray*}
\left\langle \left( R^{\mathrm{old}}\left( \zeta ,W\right) W\right) ^{%
\mathcal{H}},\xi \right\rangle &=&-s^{2}w_{h}^{2}\left\langle \nabla _{\zeta
}\left( \psi \mathrm{grad\,}\psi \right) ,\xi \right\rangle
+s^{4}w_{h}^{2}\left( D_{\zeta }\psi \right) \left\langle \mathrm{grad\,}%
\psi ,\xi \right\rangle \\
&&-s^{2}w_{h}\frac{D_{\zeta }\psi }{\left\vert \cos 2t\eta ^{2,0}\right\vert 
}\left\langle \nabla _{\left( \eta ,\eta \right) }^{\nu ,re,l}W,\xi
\right\rangle
\end{eqnarray*}%
Since%
\begin{equation*}
\mathrm{grad\,}f=-\frac{s^{2}}{\nu ^{2}}\psi \mathrm{grad\,}\psi +I^{\prime
}\zeta ,
\end{equation*}%
\begin{eqnarray*}
e^{-2f}R^{\mathrm{new}}\left( \zeta ,W,W,\xi \right)
&=&-s^{2}w_{h}^{2}\left\langle \nabla _{\zeta }\left( \psi \mathrm{grad\,}%
\psi \right) ,\xi \right\rangle +s^{4}w_{h}^{2}\left( D_{\zeta }\psi \right)
\left( \left\langle \mathrm{grad\,}\psi ,\xi \right\rangle \right) \\
&&-s^{2}w_{h}\frac{D_{\zeta }\psi }{\left\vert \cos 2t\eta ^{2,0}\right\vert 
}\left\langle \nabla _{\left( \eta ,\eta \right) }^{\nu ,re,l}W,\xi
\right\rangle \\
&&+\nu ^{2}w_{h}^{2}\frac{s^{2}}{\nu ^{2}}\left\langle \nabla _{\zeta
}\left( \psi \mathrm{grad\,}\psi \right) ,\xi \right\rangle +g\left(
W^{\gamma },W^{\gamma }\right) D_{\zeta }fD_{\xi }f+O \\
&=&s^{4}w_{h}^{2}\left( D_{\zeta }\psi \right) \left( D_{\xi }\psi \right)
+\nu ^{2}w_{h}^{2}\frac{s^{4}}{\nu ^{4}}\psi ^{2}\left( D_{\zeta }\psi
\right) \left( D_{\xi }\psi \right) \\
&&-s^{2}w_{h}\frac{D_{\zeta }\psi }{\left\vert \cos 2t\eta ^{2,0}\right\vert 
}\left\langle \nabla _{\left( \eta ,\eta \right) }^{\nu ,re,l}W,\xi
\right\rangle +O \\
&=&-s^{2}w_{h}\frac{D_{\zeta }\psi }{\left\vert \cos 2t\eta
^{2,0}\right\vert }\left\langle \nabla _{\left( \eta ,\eta \right) }^{\nu
,re,l}W,\xi \right\rangle +O
\end{eqnarray*}
\end{proof}

Let $\eta _{u,W}^{2,0}$ be the unit vector in $\mathrm{span}\left\{ \eta
_{u,1}^{2,0},\eta _{u,2}^{2,0}\right\} $ that is proportional to the
projection of $W$ onto $\mathrm{span}\left\{ \eta _{u,1}^{2,0},\eta
_{u,2}^{2,0}\right\} ,$and let $\eta _{u,W^{\perp }}^{2,0}$ be perpendicular
to $\eta _{u,W}^{2,0}.$

\begin{proposition}
\label{R( W, zeta, zeta , eta_W)}%
\begin{equation*}
e^{-2f}R^{\mathrm{new}}\left( W,\zeta ,\zeta ,\eta _{u,W}^{2,0}\right)
=-s^{2}w_{h}\left( D_{\zeta }D_{\zeta }\psi \right) +w_{h}\psi \frac{s^{2}}{%
\nu ^{2}}D_{\zeta }\left( \psi D_{\zeta }\psi \right) +O
\end{equation*}
\end{proposition}

\begin{proof}
Indeed for $U=\eta _{u,W}^{2,0}$ we have%
\begin{eqnarray*}
e^{-2f}R^{\mathrm{new}}\left( W,\zeta ,\zeta ,\eta _{u,W}^{2,0}\right) &=&R^{%
\mathrm{old}}\left( W,\zeta ,\zeta ,\eta _{u,W}^{2,0}\right) \\
&&-\left\langle W,\eta _{u,W}^{2,0}\right\rangle \mathrm{Hess}_{f}\left(
\zeta ,\zeta \right) -\mathrm{Hess}_{f}\left( W,\eta _{u,W}^{2,0}\right) \\
&&+\left\langle W,\eta _{u,W}^{2,0}\right\rangle \left( D_{\zeta }f\right)
^{2}+ \\
&&-\left\langle W,\eta _{u,W}^{2,0}\right\rangle \left\vert \mathrm{grad}%
f\right\vert ^{2}.
\end{eqnarray*}%
Using Propositions \ref{1,3 tensors after s} and \ref{Hessians} this becomes%
\begin{eqnarray*}
e^{-2f}R^{\mathrm{new}}\left( W,\zeta ,\zeta ,\eta _{u,W}^{2,0}\right)
&=&-s^{2}w_{h}\left( \frac{D_{\zeta }D_{\zeta }\psi }{\psi }\right)
\left\langle k_{\gamma },\eta _{u,W}^{2,0}\right\rangle \\
&&+\left\langle W,\eta _{u,W}^{2,0}\right\rangle \frac{s^{2}}{\nu ^{2}}%
D_{\zeta }\left( \psi D_{\zeta }\psi \right) \\
&&+\frac{s^{4}}{\nu ^{2}}\left\vert \mathrm{grad\,}\psi \right\vert
^{2}\left\langle H_{w},\eta _{u,W}^{2,0}\right\rangle +O\left( \frac{s^{2}}{%
\nu ^{2}}\psi \left( D_{\xi }\psi \right) \right) \\
&&+O\left( \frac{s^{4}}{\nu ^{4}}\left\vert \mathrm{grad\,}\psi \right\vert
^{2}\psi ^{2}\right) \left\langle H_{w},\eta _{u,W}^{2,0}\right\rangle +O
\end{eqnarray*}%
So%
\begin{eqnarray*}
e^{-2f}R^{\mathrm{new}}\left( W,\zeta ,\zeta ,\eta _{u,W}^{2,0}\right)
&=&-s^{2}w_{h}\left( D_{\zeta }D_{\zeta }\psi \right) +w_{h}\psi \frac{s^{2}%
}{\nu ^{2}}D_{\zeta }\left( \psi D_{\zeta }\psi \right) \\
&&+\frac{s^{4}}{\nu ^{2}}w_{h}\left\vert \mathrm{grad\,}\psi \right\vert
^{2}\psi +O\left( s^{2}w_{h}\psi \left( D_{\xi }\psi \right) \right)
+O\left( \frac{s^{4}}{\nu ^{4}}w_{h}\left\vert \mathrm{grad\,}\psi
\right\vert ^{2}\psi ^{3}\right) +O \\
&=&-s^{2}w_{h}\left( D_{\zeta }D_{\zeta }\psi \right) +w_{h}\psi \frac{s^{2}%
}{\nu ^{2}}D_{\zeta }\left( \psi D_{\zeta }\psi \right) +O
\end{eqnarray*}
\end{proof}

\begin{proposition}
\label{R(zeta, W,W, eta_W perp)}%
\begin{equation*}
e^{-2f}\left\langle R^{\mathrm{new}}\left( \zeta ,W\right) W,\eta
_{u,W^{\perp }}^{2,0}\right\rangle =4w_{h}s^{2}D_{\zeta }\psi \frac{\psi ^{2}%
}{\nu ^{3}}\left\vert W_{\alpha }\right\vert _{h_{2}}+O
\end{equation*}
\end{proposition}

\begin{proof}
The partial conformal change has no effect here. So this is just what comes
from Proposition \ref{1,3 tensors after s}.
\end{proof}

\begin{proposition}
For $U$ perpendicular to \textrm{span}$\left\{ W,\eta _{u,W}^{2,0}\right\} .$

\begin{description}
\item[(i)] If $U\in \mathcal{H}_{p_{2,-1}}$%
\begin{equation*}
\left\langle R^{\mathrm{old}}\left( W,\zeta \right) \zeta ,U\right\rangle =0
\end{equation*}

\item[(ii)] If $U\in V_{1}\oplus V_{2},$%
\begin{equation*}
\left\langle R^{\mathrm{old}}\left( W,\zeta \right) \zeta ,U\right\rangle
=s^{2}w_{h}D_{\zeta }\psi \left\langle \eta _{u}^{2,0},\nabla _{\zeta }^{\nu
,re,l}U\right\rangle +O
\end{equation*}
\end{description}
\end{proposition}

\begin{proof}
For $U\in \mathcal{H}_{p_{2,-1}},$%
\begin{equation*}
\left\langle R^{\mathrm{old}}\left( W,\zeta \right) \zeta ,U\right\rangle
=-s^{2}w_{h}\left( \frac{D_{\zeta }D_{\zeta }\psi }{\psi }\right)
\left\langle k_{\gamma },U\right\rangle ,
\end{equation*}%
and this is $0,$ if $U$ is also perpendicular to \textrm{span}$\left\{
W,\eta _{u,W}^{2,0}\right\} .$

For $U\in V_{1}\oplus V_{2},$ extend $U$ to be a Killing field for the $%
\left( h_{1}\oplus h_{2}\right) $--action. Then%
\begin{eqnarray*}
\left\langle R^{\mathrm{old}}\left( W,\zeta \right) \zeta ,U\right\rangle
&=&-s^{2}w_{h}\left( \frac{D_{\zeta }D_{\zeta }\psi }{\psi }\right)
\left\langle k_{\gamma },U\right\rangle -s^{2}\left( 1-s^{2}\right) w_{h}%
\frac{D_{\zeta }\psi }{\psi }\left\langle k_{\gamma },S_{\zeta }\left( U^{%
\mathcal{H}}\right) \right\rangle \\
&&+s^{2}\left( 1-s^{2}\right) w_{h}\frac{D_{\zeta }\psi }{\psi }\left\langle
k_{\gamma },\nabla _{\zeta }^{\nu ,re,l}U\right\rangle \\
&=&s^{2}w_{h}D_{\zeta }\psi \left\langle \eta _{u}^{2,0},\nabla _{\zeta
}^{\nu ,re,l}U\right\rangle +O
\end{eqnarray*}%
since $U$ is also perpendicular to \textrm{span}$\left\{ W,\eta
_{u,W}^{2,0}\right\} .$
\end{proof}

\begin{corollary}
\label{R(W,zeta, zeta, U)}For $U$ perpendicular to \textrm{span}$\left\{
W,\eta _{u,W}^{2,0}\right\} $

\begin{description}
\item[(i)] If $U\in \mathcal{H}_{p_{2,-1}}$%
\begin{equation*}
\left\langle R^{\mathrm{new}}\left( W,\zeta \right) \zeta ,U\right\rangle =O
\end{equation*}

\item[(ii)] If $U\in V_{1}\oplus V_{2}$ and a Killing field for the $\left(
h_{1}\oplus h_{2}\right) $--action%
\begin{equation*}
\left\langle R^{\mathrm{new}}\left( W,\zeta \right) \zeta ,U\right\rangle
=-e^{2f}s^{2}w_{h}D_{\zeta }\psi \left\langle \eta _{u}^{2,0},\nabla _{\zeta
}^{\nu ,re,l}U\right\rangle +O
\end{equation*}
\end{description}
\end{corollary}

\begin{proof}
The partial conformal change does contribute some nonzero terms here, but
they are too small to matter.
\end{proof}

\section{Quadratic Perturbations of Planes}

Having established that the planes $\mathrm{span}\left\{ \zeta ,W\right\} $
are now positively curved, we are left with the daunting problem of
establishing that an entire neighborhood of these planes in the Grassmannian
is positively curved. I.e. proving Theorem \ref{Positive Neighborhood}. Our
first task will be to prove the main lemma (\ref{main lemma}), which we do
in this section.

Accordingly, we represent a general plane near $\mathrm{span}\left\{ \zeta
,W\right\} $ in the form $P=\mathrm{span}\left\{ \zeta +\sigma z,W+\tau
V\right\} $ where $z\perp \zeta ,$ $V\perp W$ . The curvature is then a
quartic polynomial 
\begin{equation*}
P\left( \sigma ,\tau \right) =\mathrm{curv}\left( \zeta +\sigma z,W+\tau
V\right)
\end{equation*}%
in $\sigma $ and $\tau $. As we mentioned in section 5, running the Cheeger
perturbations by $h_{1}$ and $\Delta \left( U,D\right) $ for a long time,
allows us to reduce to the case $z\in H_{p_{2,-1}}.$

Our first task is to analyze the \textquotedblleft quadratic
perturbation\textquotedblright , i.e. to prove the main lemma, that is we
will show that for all $\sigma ,\tau \in \mathbb{R}$ and for all possible
choice of $z$ and $V,$ 
\begin{eqnarray*}
&&P_{Q}\left( \sigma ,\tau \right) =\mathrm{curv}^{\mathrm{diff}}\left(
\zeta ,W\right) +2\sigma R^{\mathrm{diff}}\left( \zeta ,W,W,z\right) +2\tau
R^{\mathrm{diff}}\left( W,\zeta ,\zeta ,V\right) \\
&&+\sigma ^{2}\mathrm{curv}^{\nu ,re,l}\left( z,W\right) +2\sigma \tau \left[
R^{\nu ,re,l}\left( \zeta ,W,V,z\right) +R^{\nu ,re,l}\left( \zeta
,V,W,z\right) \right] \\
&&+\tau ^{2}\mathrm{curv}^{\nu ,re,l}\left( \zeta ,V\right) \\
&>&0,
\end{eqnarray*}%
where 
\begin{eqnarray*}
R^{\mathrm{diff}} &=&R^{\mathrm{new}}-R^{\nu ,re,l}\text{ and} \\
\mathrm{curv}^{\mathrm{diff}} &=&\mathrm{curv}^{\mathrm{new}}-\mathrm{curv}%
^{\nu ,re,l}.
\end{eqnarray*}

Because of the $e^{2f}$--factor in the partial conformal change curvature
formulas, we will ultimately want this to also hold with all of the $\left(
\nu ,re,l\right) $--curvature terms multiplied by $e^{2f}.$ This is actually
easier to prove, and is in fact what we will do. Because $e^{2f}$ is pretty
close to $1,$ our argument also gives the main lemma, but this is just an
academic point.

We have already established that $\mathrm{curv}^{\mathrm{diff}}\left( \zeta
,W\right) =\mathrm{curv}^{\mathrm{new}}\left( \zeta ,W\right) >0.$ By
combining

\begin{itemize}
\item $P^{\nu ,re,l}\left( \sigma ,\tau \right) >0$ for all $\sigma ,\tau
\in \mathbb{R}$, and

\item The constant and linear terms of $P^{\nu ,re,l}$ are $0$
\end{itemize}

we see that 
\begin{equation*}
\sigma ^{2}\mathrm{curv}^{\nu ,re,l}\left( z,W\right) +2\sigma \tau \left[
R^{\nu ,re,l}\left( \zeta ,W,V,z\right) +R^{\nu ,re,l}\left( \zeta
,V,W,z\right) \right] +\tau ^{2}\mathrm{curv}^{\nu ,re,l}\left( \zeta
,V\right) >0
\end{equation*}%
for all $\sigma ,\tau \in \mathbb{R}.$

Therefore we only need to focus on the cases where the two linear
coefficients $R^{\mathrm{diff}}\left( \zeta ,W,W,z\right) $ and $R^{\mathrm{%
diff}}\left( W,\zeta ,\zeta ,V\right) $ are large enough so that they could
possibly cause a negative curvature. By combining our formulas for the
curvature of the partial conformal change with Propositions \ref{1,3 tensors
after s}, \ref{Hessians}, \ref{R(zeta, W, W, xi)}, \ref{R( W, zeta, zeta ,
eta_W)}, \ref{R(zeta, W,W, eta_W perp)}, and Corollary \ref{R(W,zeta, zeta,
U)} we see that these are

\begin{itemize}
\item $V=$ $U\in V_{1}\oplus V_{2}$ is perpendicular to \textrm{span}$%
\left\{ W,\eta _{u,W}^{2,0}\right\} .$

\item $z=\xi $

\item $V=\eta _{u,W}^{2,0}$

\item $z=\eta _{u,W^{\perp }}^{2,0}.$
\end{itemize}

In the first two cases we will show that the linear terms are not even close
to being large enough to create negative curvature. Because this turns out
to be the case, to dispense with the first two possibilities, it will be
enough to consider just the single variable quadratics corresponding to the
perturbations \textrm{span}$\left\{ \zeta ,W+\tau U\right\} $ and \textrm{%
span}$\left\{ \zeta +\sigma \xi ,W\right\} .$

In the first case we consider the single variable quadratic polynomial 
\begin{equation*}
P\left( \tau \right) =\mathrm{curv}^{\mathrm{diff}}\left( \zeta ,W\right)
+2\tau R^{\mathrm{diff}}\left( W,\zeta ,\zeta ,U\right) +\tau ^{2}e^{2f}%
\mathrm{curv}^{\nu ,re,l}\left( \zeta ,U\right) .
\end{equation*}%
The minimum of this quadratic polynomial is 
\begin{equation*}
\mathrm{curv}^{\mathrm{new}}\left( \zeta ,W\right) -\frac{\left\langle R^{%
\mathrm{diff}}\left( W,\zeta \right) \zeta ,U\right\rangle ^{2}}{e^{2f}%
\mathrm{curv}^{\nu ,re,l}\left( \zeta ,U\right) }
\end{equation*}%
Combining Proposition \ref{new curv} and Corollary \ref{R(W,zeta, zeta, U)}
we get that \addtocounter{algorithm}{1}

\begin{equation}
P\left( \tau \right) \geq e^{2f}\left( s^{4}w_{h}^{2}\left( D_{\zeta }\psi
\right) ^{2}+s^{4}\frac{w_{h}^{2}}{\nu ^{2}}\psi ^{2}\left\langle \mathrm{%
grad\,}\psi ,\zeta \right\rangle ^{2}+\iota \right)
-e^{2f}s^{4}w_{h}^{2}\left( D_{\zeta }\psi \right) ^{2}\frac{\left\langle
\eta _{u}^{2,0},\nabla _{\zeta }^{\nu ,re,l}U\right\rangle ^{2}}{\mathrm{curv%
}^{\nu ,re,l}\left( \zeta ,U\right) }  \label{P(tau)}
\end{equation}

Using Theorem \ref{redistr thm} we will prove

\begin{proposition}
\label{death by redistr} For any constant $c>O\left( \nu \right) $ , there a
choice of metric $g_{\nu ,re,l}$ so that with respect to $g_{\nu ,re,l}$%
\begin{equation*}
\int_{\mu }\left( D_{\zeta }\psi \right) ^{2}\frac{\left\langle \eta
_{u}^{2,0},\nabla _{\zeta }^{\nu ,re,l}U\right\rangle ^{2}}{\mathrm{curv}%
^{\nu ,re,l}\left( \zeta ,U\right) }\leq c\int_{\mu }\left( D_{\zeta }\psi
\right) ^{2},
\end{equation*}%
where $\mu $ is any of the geodesics of length $\frac{\pi }{4},$ tangent to $%
\zeta $ along the old zero locus, starting over either of the two points in $%
S^{4}$ with $\left( t,\sin 2\theta \right) =\left( 0,0\right) .$

Moreover, for any constant $c>O\left( \nu \right) $ , there is a choice of $%
g_{\nu ,re,l}$ and a choice of $\iota $ so that with respect to $g_{\nu
,re,l}$\addtocounter{algorithm}{1}%
\begin{equation}
c\left[ e^{2f}\left( s^{4}w_{h}^{2}\left( D_{\zeta }\psi \right) ^{2}+s^{4}%
\frac{w_{h}^{2}}{\nu ^{2}}\psi ^{2}\left\langle \mathrm{grad\,}\psi ,\zeta
\right\rangle ^{2}+\iota \right) \right] \geq e^{2f}s^{4}w_{h}^{2}\left(
D_{\zeta }\psi \right) ^{2}\frac{\left\langle \eta _{u}^{2,0},\nabla _{\zeta
}^{\nu ,re,l}U\right\rangle ^{2}}{\mathrm{curv}^{\nu ,re,l}\left( \zeta
,U\right) }.  \label{irrelv cross term}
\end{equation}%
In particular, $P\left( \tau \right) >0.$
\end{proposition}

\begin{remark}
At this point we can begin to appreciate the need for the redistribution
metric. It allows us to make the negative term in \ref{P(tau)} as small as
we like. It will become clear after we have considered the case when $V=\eta
_{u,W}^{2,0}$ that without this redistribution there would in fact be some
negative curvatures.
\end{remark}

\begin{proof}
From Proposition \ref{zeta, P deriv} we see that the redistribution has very
little effect on $\left\langle \eta _{u}^{2,0},\nabla _{\zeta
}U\right\rangle ^{2}.$ To compute this quantity with respect to $g_{\nu ,l},$
we must again consider $Sp\left( 2\right) \times \left( S^{3}\right) ^{2}.$
So for the purpose of this proof we suspend the notational convention on
page \pageref{notational convention copy(1)}, and revert to the
\textquotedblleft $\ \widehat{}\ $\textquotedblright\ notation for
discussing Cheeger deformations.%
\begin{equation*}
\left\langle \hat{\eta}_{u}^{2,0},\nabla _{\zeta }^{\nu ,l}\hat{U}%
\right\rangle ^{2}=\left\langle \eta _{u}^{2,0},\nabla _{\zeta }^{\nu
}U\right\rangle _{Sp\left( 2\right) }^{2}+\left\langle \left( \hat{\eta}%
_{u}^{2,0}\right) ^{\left( S^{3}\right) ^{2}},\nabla _{\zeta }\hat{U}%
\right\rangle _{\left( S^{3}\right) ^{2}}^{2}.
\end{equation*}%
The $Sp\left( 2\right) $ derivative is given by quaternion multiplication
and lives in the orthogonal complement $H$ of $V_{1}\oplus V_{2}.$ So 
\begin{eqnarray*}
\left\langle \eta _{u}^{2,0},\nabla _{\zeta }^{\nu }U\right\rangle
_{Sp\left( 2\right) }^{2} &=&\frac{\cos ^{2}2t}{\left\vert \cos 2t\eta
^{2,0}\right\vert ^{2}}\left\langle \left( \eta ,\eta \right) ,\nabla
_{\zeta }^{\nu }U\right\rangle _{\nu }^{2} \\
&\leq &\frac{1}{\left\vert \cos 2t\eta ^{2,0}\right\vert ^{2}}\mathrm{curv}%
^{\nu }\left( \zeta ,U\right)
\end{eqnarray*}

Our estimates for the $\left( S^{3}\right) ^{2}$--portion will be efficient,
but not optimal. First notice that if $\left\vert U\right\vert _{\nu
}=O\left( \frac{1}{\nu }\right) ,$ then%
\begin{equation*}
\left\vert \left( \nabla _{\zeta }\hat{U}\right) ^{S^{3}}\right\vert
_{l}=O\left( \frac{1}{l}\right) ,
\end{equation*}%
since 
\begin{equation*}
\left\vert \left( \hat{\eta}_{u}^{2,0}\right) ^{\left( S^{3}\right)
^{2}}\right\vert _{l}=\frac{1}{\left\vert \cos 2t\eta ^{2,0}\right\vert ^{2}}%
O\left( \frac{t}{l}\right)
\end{equation*}%
we get 
\begin{equation*}
\left\langle \left( \hat{\eta}_{u}^{2,0}\right) ^{\left( S^{3}\right)
^{2}},\nabla _{\zeta }\hat{U}\right\rangle _{\left( S^{3}\right)
^{2}}^{2}\leq \frac{1}{\left\vert \cos 2t\eta ^{2,0}\right\vert ^{2}}O\left( 
\frac{t^{2}}{l^{2}}\right)
\end{equation*}

combining estimates we have 
\begin{eqnarray*}
\left\langle \hat{\eta}_{u}^{2,0},\nabla _{\zeta }^{\nu ,l}\hat{U}%
\right\rangle ^{2} &\leq &\frac{1}{\left\vert \cos 2t\eta ^{2,0}\right\vert
^{2}}\left( \mathrm{curv}^{\nu }\left( \zeta ,U\right) +O\left( \frac{t^{2}}{%
l^{2}}\right) \right) \\
&\leq &1.1\mathrm{curv}^{\nu ,l}\left( \zeta ,U\right)
\end{eqnarray*}

From Proposition \ref{zeta, P deriv} we see that (with an irrelevant
adjustment), this estimate also holds with $\left\langle \hat{\eta}%
_{u}^{2,0},\nabla _{\zeta }^{\nu ,l}\hat{U}\right\rangle ^{2}$ replaced by $%
\left\langle \hat{\eta}_{u}^{2,0},\nabla _{\zeta }^{\nu ,re,l}\hat{U}%
\right\rangle ^{2}.$ On the other hand, we see from Theorem \ref{redistr thm}
and O'Neill's horizontal curvature equation that%
\begin{equation*}
\mathrm{curv}^{\nu ,re,l}\left( \zeta ,U\right) \geq \mathrm{curv}^{\nu
}\left( \zeta ,U\right) +\varphi \varphi ^{\prime \prime }\left\vert
U\right\vert _{\nu }^{2},
\end{equation*}

Since $\left( D_{\zeta }\psi \right) ^{2}$ is concentrated on a set that
looks like $\left[ 0,O\left( \nu \right) \right] $, we can redistribute the
ratio 
\begin{equation*}
\frac{\left\langle \eta _{u}^{2,0},\nabla _{\zeta }^{\nu
,re,l}U\right\rangle ^{2}}{e^{2f}\mathrm{curv}^{\nu ,re,l}\left( \zeta
,U\right) }
\end{equation*}%
so that it is small where $\left( D_{\zeta }\psi \right) ^{2}$ is large, and
large where $\left( D_{\zeta }\psi \right) ^{2}$ is small. The choice of $%
\varphi ^{\prime \prime }$ that we made at the beginning of section 6 will
give us the desired integral inequality (perhaps with an adjustment of the
constants $100,10,000,$ $\ldots ).$

To get \ref{irrelv cross term} we combine the integral inequality with the
fact that we have yet to impose any conditions on $\iota $ except, 
\begin{eqnarray*}
\iota &=&O\left( s^{4}w_{h}^{2}\right) ,\text{ and } \\
\int_{\mu }\iota &=&0.
\end{eqnarray*}%
To get \ref{irrelv cross term} we must now require that $\iota $ be positive
and (relatively large) on a region that looks like $\left[ \mathrm{const\,}%
\nu ,\frac{\pi }{4}\right] .$ The quantity $\mathrm{const\,}\nu $ is near
the inflection point of the redistribution function $\varphi .$
\end{proof}

In the case where $z=\xi $ we will again see that linear term is
overwhelmingly dominated. Consider the quadratic%
\begin{equation*}
P\left( \sigma \right) =\mathrm{curv}^{\mathrm{diff}}\left( \zeta ,W\right)
+2\sigma R^{\mathrm{diff}}\left( \zeta ,W,W,\xi \right) +\sigma ^{2}e^{2f}%
\mathrm{curv}^{\mathrm{old}}\left( \xi ,W\right)
\end{equation*}%
The minimum is 
\begin{equation*}
\mathrm{curv}^{\mathrm{diff}}\left( \zeta ,W\right) -\frac{R^{\mathrm{diff}%
}\left( \zeta ,W,W,\xi \right) ^{2}}{e^{2f}\mathrm{curv}^{\mathrm{old}%
}\left( \xi ,W\right) }
\end{equation*}%
Combining Propositions \ref{new curv} and \ref{R(zeta, W, W, xi)} we see
that this is \addtocounter{algorithm}{1} 
\begin{eqnarray}
&&e^{2f}\left( s^{4}w_{h}^{2}\left( D_{\zeta }\psi \right) ^{2}+s^{4}\frac{%
w_{h}^{2}}{\nu ^{2}}\psi ^{2}\left\langle \mathrm{grad\,}\psi ,\zeta
\right\rangle ^{2}+\iota \right) -e^{2f}\frac{s^{4}w_{h}^{2}\left( D_{\zeta
}\psi \right) ^{2}\left\langle \nabla _{\left( \eta ,\eta \right) }^{\nu
,re,l}W,\xi \right\rangle ^{2}}{\left\vert \cos 2t\eta ^{2,0}\right\vert ^{2}%
\mathrm{curv}^{\nu ,re,l}\left( \xi ,W\right) }.  \notag \\
&&  \label{quad min xi}
\end{eqnarray}

Since%
\begin{equation*}
\mathrm{curv}^{\nu ,re,l}\left( \xi ,W\right) \geq \left\langle \nabla
_{\left( \eta ,\eta \right) }^{\nu }W,\xi \right\rangle _{\nu }^{2}+O\left( 
\frac{t^{2}}{l^{6}}\right) ,
\end{equation*}%
\begin{equation*}
\left\langle \nabla _{\left( \eta ,\eta \right) }^{\nu ,re,l}W,\xi
\right\rangle ^{2}\leq \left\langle \nabla _{\left( \eta ,\eta \right)
}^{\nu }W,\xi \right\rangle _{\nu }^{2}+O\left( \frac{t^{2}}{l^{4}}\right)
+O\left( \frac{t^{4}}{l^{8}}\right) ,
\end{equation*}%
\begin{equation*}
\left\langle \nabla _{\left( \eta ,\eta \right) }^{\nu }W,\xi \right\rangle
_{\nu }^{2}\leq 1,
\end{equation*}%
\begin{eqnarray*}
\frac{1}{\left\vert \cos 2t\eta ^{2,0}\right\vert ^{2}} &\leq &\frac{1}{%
\left( \cos ^{2}2t+\frac{\sin ^{2}2t}{\nu ^{2}}\right) } \\
&=&\frac{\nu ^{2}}{\left( \nu ^{2}\cos ^{2}2t+\sin ^{2}2t\right) },
\end{eqnarray*}%
and 
\begin{equation*}
l=O\left( \nu ^{1/3}\right)
\end{equation*}%
we see that the negative term in \ref{quad min xi} is much smaller than the
positive term, provided the constant $c$ so that $l=c\nu ^{1/3}$ is
relatively large.

In the final two cases, $V=\eta _{u,W}^{2,0}$ and $z=\eta _{u,W^{\perp
}}^{2,0},$ the linear terms can be a substantial fraction of the total, so
we will have to be more careful. In particular, we will have to consider the
entire polynomial $P_{Q}\left( \sigma ,\tau \right) .$ We start by analyzing
the two mixed quadratic coefficients 
\begin{equation*}
R^{\nu ,re,l}\left( \zeta ,W,\eta _{u,W}^{2,0},\eta _{u,W^{\perp
}}^{2,0}\right) ,R^{\nu ,re,l}\left( \zeta ,\eta _{u,W}^{2,0},W,\eta
_{u,W^{\perp }}^{2,0}\right) .
\end{equation*}%
First notice that they are $0$ if our only deformations of the biinvariant
metric are the $h_{1}$ and $h_{2}$ Cheeger perturbations. We track the
effect of the $U$ and $D$ perturbations by considering the corresponding
submersion $S^{3}\times S^{3}\times Sp\left( 2\right) \longrightarrow
Sp\left( 2\right) .$ As we have observed the components of $R^{\nu
,re,l}\left( \zeta ,W,\eta _{u,W}^{2,0},\eta _{u,W^{\perp }}^{2,0}\right) $
and $R^{\nu ,re,l}\left( \zeta ,\eta _{u,W}^{2,0},W,\eta _{u,W^{\perp
}}^{2,0}\right) $ that come from the $Sp\left( 2\right) $--factor of $\left(
S^{3}\right) ^{2}\times Sp\left( 2\right) $ are $0.$ For similar reasons the
components that come from the $S^{3}$--factor are $0.$ The $A$--tensors of $%
S^{3}\times S^{3}\times Sp\left( 2\right) \longrightarrow Sp\left( 2\right) $
and $Sp\left( 2\right) \longrightarrow \Sigma ^{7}$ might make a nonzero
contribution, but its contribution to the curvature of the entire plane 
\begin{equation*}
\mathrm{curv}^{\nu ,re,l}\left( \zeta +\sigma \eta _{u,W^{\perp
}}^{2,0},W+\tau \eta _{u,W}^{2,0}\right)
\end{equation*}%
is nonnegative so we may drop it. (As long as we drop it from all
curvatures.) Finally we saw in section 6 that the redistribution deformation
only has a large effect on curvatures that have $\zeta $ in two variables.
So in the end we see that these two mixed terms are too small to matter.

Although this simplifies matters considerably, we still have to consider the
rest of $P_{Q}\left( \sigma ,\tau \right) $ as a whole. More specifically we
have to verify that 
\begin{equation*}
\mathrm{curv}^{\mathrm{diff}}\left( \zeta ,W\right) -\frac{R^{\mathrm{diff}%
}\left( W,\zeta ,\zeta ,\eta _{u,W}^{2,0}\right) ^{2}}{e^{2f}\mathrm{curv}%
^{\nu ,re,l}\left( \zeta ,\eta _{u,W}^{2,0}\right) }-\frac{R^{\mathrm{diff}%
}\left( \zeta ,W,W,\eta _{u,W^{\perp }}^{2,0}\right) ^{2}}{e^{2f}\mathrm{curv%
}^{\nu ,re,l}\left( \eta _{u,W^{\perp }}^{2,0},W\right) }>0.
\end{equation*}

To simplify the exposition we compute the sum of the first two terms and
then the last term. Using Propositions \ref{new curv} and \ref{R( W, zeta,
zeta , eta_W)} and the fact that 
\begin{equation*}
\mathrm{curv}^{\nu ,re,l}\left( \zeta ,\eta _{u,W}^{2,0}\right) =-\frac{%
D_{\zeta }D_{\zeta }\psi }{\psi }+O,
\end{equation*}%
we find 
\begin{equation*}
\mathrm{curv}^{\mathrm{diff}}\left( \zeta ,W\right) -\frac{R^{\mathrm{diff}%
}\left( W,\zeta ,\zeta ,\eta _{u,W}^{2,0}\right) ^{2}}{e^{2f}\mathrm{curv}%
^{\nu ,re,l}\left( \zeta ,\eta _{u,W}^{2,0}\right) }+O
\end{equation*}%
\begin{eqnarray*}
&=&e^{2f}\left( s^{4}w_{h}^{2}\left( D_{\zeta }\psi \right) ^{2}+s^{4}\frac{%
w_{h}^{2}}{\nu ^{2}}\psi ^{2}\left\langle \mathrm{grad\,}\psi ,\zeta
\right\rangle ^{2}+\iota \right) -\frac{e^{4f}\left( -s^{2}w_{h}\left(
D_{\zeta }D_{\zeta }\psi \right) +w_{h}\psi \frac{s^{2}}{\nu ^{2}}D_{\zeta
}\left( \psi D_{\zeta }\psi \right) \right) ^{2}}{-e^{2f}\frac{D_{\zeta
}D_{\zeta }\psi }{\psi }} \\
&=&e^{2f}s^{4}w_{h}^{2}\left( \left( D_{\zeta }\psi \right) ^{2}+\frac{\psi
^{2}}{\nu ^{2}}\left\langle \mathrm{grad\,}\psi ,\zeta \right\rangle
^{2}\right) +e^{2f}\iota \\
&&+e^{2f}s^{4}w_{h}^{2}\left[ \frac{\psi }{D_{\zeta }D_{\zeta }\psi }\left(
\left( D_{\zeta }D_{\zeta }\psi \right) ^{2}-\frac{2}{\nu ^{2}}\left(
D_{\zeta }D_{\zeta }\psi \right) \psi D_{\zeta }\left( \psi D_{\zeta }\psi
\right) +\frac{\psi ^{2}}{\nu ^{4}}\left[ D_{\zeta }\left( \psi D_{\zeta
}\psi \right) \right] ^{2}\right) \right]
\end{eqnarray*}%
\begin{eqnarray*}
&=&e^{2f}s^{4}w_{h}^{2}\left( \left( D_{\zeta }\psi \right) ^{2}+\frac{\psi
^{2}}{\nu ^{2}}\left\langle \mathrm{grad\,}\psi ,\zeta \right\rangle
^{2}\right) +e^{2f}\iota \\
&&+e^{2f}s^{4}w_{h}^{2}\left( \psi \left( D_{\zeta }D_{\zeta }\psi \right) -2%
\frac{\psi ^{2}}{\nu ^{2}}D_{\zeta }\left( \psi D_{\zeta }\psi \right) +%
\frac{\psi }{D_{\zeta }D_{\zeta }\psi }\frac{\psi ^{2}}{\nu ^{4}}\left[
D_{\zeta }\left( \psi D_{\zeta }\psi \right) \right] ^{2}\right) \\
&=&e^{2f}s^{4}w_{h}^{2}\left( D_{\zeta }\left( \psi D_{\zeta }\psi \right) +%
\frac{\psi ^{2}}{\nu ^{2}}\left( D_{\zeta }\psi \right) ^{2}-2\frac{\psi ^{2}%
}{\nu ^{2}}D_{\zeta }\left( \psi D_{\zeta }\psi \right) +\frac{\psi }{%
D_{\zeta }D_{\zeta }\psi }\frac{\psi ^{2}}{\nu ^{4}}\left[ D_{\zeta }\left(
\psi D_{\zeta }\psi \right) \right] ^{2}\right) +e^{2f}\iota
\end{eqnarray*}

The integral the first term is $0.$ The integral of the second term is
positive and the integral of the third term is positive as well, since the
total derivative is positive where $\psi $ is small and negative where $\psi 
$ is larger.

The next to last term has a negative integral, but in Lemma \ref{Derivatives}
we showed

\begin{equation*}
\left\vert \frac{\psi }{D_{\zeta }D_{\zeta }\psi }\left[ D_{\zeta }\left(
\psi D_{\zeta }\psi \right) \right] \right\vert \leq \frac{\nu ^{2}}{4}.
\end{equation*}%
so%
\begin{equation*}
\left\vert \frac{\psi }{D_{\zeta }D_{\zeta }\psi }\frac{\psi ^{2}}{\nu ^{4}}%
\left[ D_{\zeta }\left( \psi D_{\zeta }\psi \right) \right] ^{2}\right\vert
\leq \frac{\psi ^{2}}{4\nu ^{2}}\left\vert D_{\zeta }\left( \psi D_{\zeta
}\psi \right) \right\vert
\end{equation*}%
This is an eighth of the third term so\addtocounter{algorithm}{1} 
\begin{eqnarray}
&&\mathrm{curv}^{\mathrm{diff}}\left( \zeta ,W\right) -\frac{R^{\mathrm{diff}%
}\left( W,\zeta ,\zeta ,\eta _{u,W}^{2,0}\right) ^{2}}{e^{2f}\mathrm{curv}%
^{\nu ,re,l}\left( \zeta ,\eta _{u,W}^{2,0}\right) }+O  \notag \\
&\geq &e^{2f}s^{4}w_{h}^{2}\left( D_{\zeta }\left( \psi D_{\zeta }\psi
\right) +\frac{\psi ^{2}}{\nu ^{2}}\left( D_{\zeta }\psi \right) ^{2}-2\frac{%
\psi ^{2}}{\nu ^{2}}D_{\zeta }\left( \psi D_{\zeta }\psi \right) +\frac{\psi
^{2}}{4\nu ^{2}}\left\vert D_{\zeta }\left( \psi D_{\zeta }\psi \right)
\right\vert \right) +e^{2f}\iota  \notag \\
&=&e^{2f}s^{4}w_{h}^{2}\left( D_{\zeta }\left( \psi D_{\zeta }\psi \right) +%
\frac{\psi ^{2}}{\nu ^{2}}\left( D_{\zeta }\psi \right) ^{2}-\frac{7}{4}%
\frac{\psi ^{2}}{\nu ^{2}}D_{\zeta }\left( \psi D_{\zeta }\psi \right)
\right) +e^{2f}\iota  \notag \\
&&  \label{eta-min-2}
\end{eqnarray}

Finally using Proposition \ref{R(zeta, W,W, eta_W perp)} 
\begin{eqnarray*}
\frac{R^{\mathrm{diff}}\left( \zeta ,W,W,\eta _{u,W^{\perp }}^{2,0}\right)
^{2}}{e^{2f}\mathrm{curv}^{\nu ,re,l}\left( \eta _{u,W^{\perp
}}^{2,0},W\right) } &=&s^{4}w_{h}^{2}\left( D_{\zeta }\psi \right) ^{2}\frac{%
e^{4f}\left( 4\frac{\psi ^{2}}{\nu ^{3}}\left\vert W_{\alpha }\right\vert
\right) ^{2}}{e^{2f}\mathrm{curv}^{\nu ,re,l}\left( \eta _{u,W^{\perp
}}^{2,0},W\right) } \\
&\leq &e^{2f}s^{4}w_{h}^{2}\left( D_{\zeta }\psi \right) ^{2}\frac{\left( 4%
\frac{\psi ^{2}}{\nu ^{3}}\left\vert W_{\alpha }\right\vert \right) ^{2}}{%
\left\vert \cos 2t\eta ^{2,0}\right\vert ^{-2}+4\frac{\psi ^{2}}{\nu ^{6}}},
\end{eqnarray*}%
where the factor of $4$ in the denominator comes from the fact that $\psi =%
\frac{1}{2}\frac{\sin 2t}{\left\vert \cos 2t\eta ^{2,0}\right\vert }.$ Since 
$\left\vert W_{\alpha }\right\vert ^{2}\leq \frac{1}{4\nu ^{2}},$ we get
another factor of $4.$ So 
\begin{equation*}
\frac{R^{\mathrm{diff}}\left( \zeta ,W,W,\eta _{u,W^{\perp }}^{2,0}\right)
^{2}}{e^{2f}\mathrm{curv}^{\nu ,re,l}\left( \eta _{u,W^{\perp
}}^{2,0},W\right) }\leq e^{2f}s^{4}w_{h}^{2}\left( D_{\zeta }\psi \right)
^{2}\frac{\psi ^{2}}{\nu ^{2}}
\end{equation*}%
Combining the displays we get 
\begin{eqnarray*}
&&\mathrm{curv}^{\mathrm{diff}}\left( \zeta ,W\right) -\frac{R^{\mathrm{diff}%
}\left( W,\zeta ,\zeta ,\eta _{u,W}^{2,0}\right) ^{2}}{e^{2f}\mathrm{curv}%
^{\nu ,re,l}\left( \zeta ,\eta _{u,W}^{2,0}\right) }-\frac{R^{\mathrm{diff}%
}\left( \zeta ,W,W,\eta _{u,W^{\perp }}^{2,0}\right) ^{2}}{e^{2f}\mathrm{curv%
}^{\nu ,re,l}\left( \eta _{u,W^{\perp }}^{2,0},W\right) }\geq \\
&&e^{2f}s^{4}w_{h}^{2}\left( D_{\zeta }\left( \psi D_{\zeta }\psi \right) +%
\frac{\psi ^{2}}{\nu ^{2}}\left( D_{\zeta }\psi \right) ^{2}-\frac{7}{4}%
\frac{\psi ^{2}}{\nu ^{2}}D_{\zeta }\left( \psi D_{\zeta }\psi \right) -%
\frac{\psi ^{2}}{\nu ^{2}}\left( D_{\zeta }\psi \right) ^{2}\right)
+e^{2f}\iota +O \\
&=&e^{2f}s^{4}w_{h}^{2}\left( D_{\zeta }\left( \psi D_{\zeta }\psi \right) -%
\frac{7}{4}\frac{\psi ^{2}}{\nu ^{2}}D_{\zeta }\left[ \psi D_{\zeta }\psi %
\right] \right) +\iota e^{2f}+O.
\end{eqnarray*}%
\medskip

So we can choose $\iota $ so that the right hand side is point wise
positive. With some moments of reflection we see that this choice of $\iota $
can be consistent with the choice required for the proof of Proposition \ref%
{death by redistr}.

\begin{remark}
With a careful review of the estimates in this section one can appreciate
the necessity of the redistribution. Indeed without the redistribution, we
can't do much better in Proposition \ref{death by redistr} than 
\begin{equation*}
\frac{\left\langle \eta _{u}^{2,0},\nabla _{\zeta }^{\nu
,re,l}U\right\rangle ^{2}}{\mathrm{curv}^{\nu ,re,l}\left( \zeta ,U\right) }%
\leq 1.
\end{equation*}%
Tracing through the rest of our estimates one can then see that there can be
a vector in $V\in \mathrm{span}\left\{ V_{1}\oplus V_{2},\eta
_{u,W}^{2,0}\right\} $ so that the single variable polynomial 
\begin{equation*}
P\left( \tau \right) =\mathrm{curv}\left( \zeta ,W+\tau V\right)
\end{equation*}
has some negative values. To see this one must also observe that the
integrals of%
\begin{equation*}
\frac{\psi ^{2}}{\nu ^{2}}\left( \psi D_{\zeta }D_{\zeta }\psi \right) \text{
and }\frac{\psi ^{2}}{\nu ^{2}}D_{\zeta }\left[ \psi D_{\zeta }\psi \right]
\end{equation*}%
are something like $O\left( \frac{1}{100}\right) $ times the integral of 
\begin{equation*}
\psi D_{\zeta }D_{\zeta }\psi .
\end{equation*}
\end{remark}

\section{Higher Order Terms}

To prove that the Gromoll-Meyer sphere is now positively curved it remains
to show that the higher order terms in the curvature polynomial%
\begin{equation*}
P\left( \sigma ,\tau \right) =\mathrm{curv}\left( \zeta +\sigma z,W+\tau
V\right) ,
\end{equation*}%
do not change enough under our deformations to create a nonpositive
curvature.

Recall that it is enough to consider the case when $z\in H^{2,-1}\mathrm{.}$
For computational convenience, we choose $z$ and $V$ so that their
components in $\mathrm{span}\left\{ x^{2,0},y^{2,0}\right\} $ are
proportional to $y^{2,0}.$ In addition, we choose $V$ so that its component
in $V_{2}$ is perpendicular to the $\gamma $-part of $W.$ We further assume
that $z$ and $V$ are normalized so that they are spherical combinations of
our standard vectors.

The curvature of $P$ is a quartic polynomial 
\begin{equation*}
P\left( \sigma ,\tau \right) =R\left( \zeta +\sigma z,W+\tau V,W+\tau
V,\zeta +\sigma z\right)
\end{equation*}%
in $\sigma $ and $\tau $.

In addition we must verify the positivity of the quadratic subpolynomials%
\begin{eqnarray*}
Q_{\zeta }\left( \sigma \right)  &=&\mathrm{curv}\left( \zeta +\sigma
z,V\right) \text{ and } \\
Q_{W}\left( \tau \right)  &=&\mathrm{curv}\left( z,W+\tau V\right) .
\end{eqnarray*}

We let $\varkappa :\mathbb{R}_{+}\longrightarrow \mathbb{R}_{+}$ stand for a
function so that $\lim_{s\rightarrow 0}\varkappa \left( s\right) =0.$

Set%
\begin{eqnarray*}
&&H^{\mathrm{diff}}\left( \sigma ,\tau \right) \equiv \tau ^{2}R^{\mathrm{%
diff}}\left( \zeta ,V,V,\zeta \right) +2\sigma \tau R^{\mathrm{diff}}\left(
\zeta ,W,V,z\right) +2\sigma \tau R^{\mathrm{diff}}\left( \zeta
,V,W,z\right) +\sigma ^{2}R^{\mathrm{diff}}\left( z,W,W,z\right) \\
&&+2\sigma \tau ^{2}R^{\mathrm{diff}}\left( \zeta ,V,V,z\right) +2\sigma
^{2}\tau R^{\mathrm{diff}}\left( z,W,V,z\right) +\sigma ^{2}\tau ^{2}\text{ }%
R^{\mathrm{diff}}\left( z,V,V,z\right) ,
\end{eqnarray*}%
and let $P^{\nu ,re,l}\left( \sigma ,\tau \right) $ be the curvature
polynomial for $g_{\nu ,re,l}.$

\begin{theorem}
\label{R^diff, small} To verify that $P\left( \sigma ,\tau \right) >0$, $%
Q_{\zeta }\left( \sigma \right) >0,$ and $Q_{W}\left( \tau \right) >0$ for
all $\sigma ,\tau ,$ we may ignore

\begin{description}
\item[(a)] Any term in a coefficient of $H^{\mathrm{diff}}$ that is smaller
than $\varkappa \left( s\right) $ times the corresponding coefficient of $%
P^{\nu ,re,l}.$

\item[(b)] Any term in the $\left( \sigma \tau \right) $--coefficient of $H^{%
\mathrm{diff}}$ that is smaller than 
\begin{equation*}
\chi \left( s\right) \sqrt{\mathrm{curv}^{\nu ,re,l}\left( \zeta ,V\right) }%
\sqrt{\mathrm{curv}^{\nu ,re,l}\left( z,W\right) }
\end{equation*}

\item[(c)] Any term in the $\left( \sigma ^{2}\tau \right) $--coefficient of 
$H^{\mathrm{diff}}$ that is smaller than%
\begin{equation*}
\chi \left( s\right) \sqrt{\mathrm{curv}^{\nu ,re,l}\left( z,W\right) }\sqrt{%
\mathrm{curv}^{\nu ,re,l}\left( z,V\right) }
\end{equation*}

\item[(d)] Any term in the $\left( \sigma \tau ^{2}\right) $--coefficient of 
$H^{\mathrm{diff}}$ that is smaller than%
\begin{equation*}
\chi \left( s\right) \sqrt{\mathrm{curv}^{\nu ,re,l}\left( \zeta ,V\right) }%
\sqrt{\mathrm{curv}^{\nu ,re,l}\left( z,V\right) }.
\end{equation*}
\end{description}
\end{theorem}

\begin{proof}
Part (a) follows from the main lemma and the fact that $P^{\nu ,re,l}\left(
\sigma ,\tau \right) >0,$ $Q_{\zeta }^{\nu ,re,l}\left( \sigma \right) >0,$
and $Q_{W}^{\nu ,re,l}\left( \tau \right) >0$ for all $\sigma ,\tau >0.$

To prove part (b) we fix $\tau $ and view the $\tau ^{2}$ and $\sigma ^{2}$
terms of $P^{\nu ,re,l}$ together with the term in the $\left( \sigma \tau
\right) $--coefficient that is smaller than%
\begin{equation*}
\chi \left( s\right) \sqrt{\mathrm{curv}^{\nu ,re,l}\left( \zeta ,V\right) }%
\sqrt{\mathrm{curv}^{\nu ,re,l}\left( z,W\right) }.
\end{equation*}%
as a quadratic in $\sigma .$ The minimum is smaller than 
\begin{eqnarray*}
&&\tau ^{2}\left( \mathrm{curv}^{\nu ,re,l}\left( \zeta ,V\right) -\chi
\left( s\right) ^{2}\frac{\left( \sqrt{\mathrm{curv}^{\nu ,re,l}\left( \zeta
,V\right) }\sqrt{\mathrm{curv}^{\nu ,re,l}\left( z,W\right) }\right) ^{2}}{%
\mathrm{curv}^{\nu ,re,l}\left( z,W\right) }\right) \\
&=&\tau ^{2}\left( \mathrm{curv}^{\nu ,re,l}\left( \zeta ,V\right) -\chi
\left( s\right) ^{2}\mathrm{curv}^{\nu ,re,l}\left( \zeta ,V\right) \right)
\\
&=&\tau ^{2}\mathrm{curv}^{\nu ,re,l}\left( \zeta ,V\right) -O.
\end{eqnarray*}%
Parts (c) and (d) are proven with similar arguments. For part (c), we
dominate the portion of the $\left( \sigma ^{2}\tau \right) $--coefficient
of $H^{\mathrm{diff}}$ in question with the $\sigma ^{2}$ and $\sigma
^{2}\tau ^{2}$--coefficients of $P^{\nu ,re,l}\left( \sigma ,\tau \right) .$
For part (d), we dominate the portion of the $\left( \sigma \tau ^{2}\right) 
$--coefficient of $H^{\mathrm{diff}}$ in question with the $\tau ^{2}$ and $%
\sigma ^{2}\tau ^{2}$--coefficients of $P^{\nu ,re,l}\left( \sigma ,\tau
\right) .$

We do not have to consider the $Q_{\zeta }\left( \sigma \right) $s and $%
Q_{W}\left( \tau \right) $s for part (b). The proofs of parts (c) and (d)
for the $Q_{\zeta }\left( \sigma \right) $s and $Q_{W}\left( \tau \right) $s
are essentially the same as the proofs for $P\left( \sigma ,\tau \right) .$
\end{proof}

\begin{remark}
Since many of the possible coefficients of $P^{\nu ,re,l}$ can be large,
many of the terms that this theorem allows us to ignore are in fact large.
Its just that their effect is swamped by certain terms of $P^{\nu ,re,l}$.
\end{remark}

We let 
\begin{equation*}
R^{\mathrm{diff,big}}
\end{equation*}%
denote the terms of $R^{\mathrm{diff}}$ that can not be thrown out using the
previous theorem.

\begin{theorem}
\label{higher order curvatures}If $z$ and $V$ are as above and normalized as
in our standard basis, then $\mathrm{curv}^{\mathrm{diff,big}}\left(
z,V\right) $ and $\mathrm{curv}^{\mathrm{diff,big}}\left( z,W\right) ,$ are
nonnegative and 
\begin{equation*}
\left\vert R^{\mathrm{diff,big}}\left( z,V,W,z\right) \right\vert \leq \sqrt{%
\mathrm{curv}^{\mathrm{diff,big}}\left( z,V\right) }\sqrt{\mathrm{curv}^{%
\mathrm{diff,big}}\left( z,W\right) }
\end{equation*}%
All other coefficients of $R^{\mathrm{diff,big}}$ are $0,$ unless our
perturbation bivector $\left( z,V\right) $ has a nonzero inner product with
either the case when $z=y^{2,0}$ and $V=\eta _{u,W}^{2,0}$ or with the case
when $z=\eta _{u,W}^{2,0}$ and $V=y^{2,0}.$
\end{theorem}

\begin{theorem}
\label{higher order curvatures-2}If $z=y^{2,0}$ and $V=\eta _{u,W}^{2,0}$ or 
$z=\eta _{u,W}^{2,0}$ and $V=y^{2,0},$ then 
\begin{eqnarray*}
\left\langle R^{\mathrm{diff}}\left( W,y^{2,0}\right) y^{2,0},W\right\rangle
&=&e^{2f}s^{2}w_{h}^{2}\psi ^{2}\mathrm{curv}^{S^{4}}\left( y^{2,0},\eta
_{u,W}^{2,0}\right) +\varkappa \left( s\right) \\
\left\vert R^{\mathrm{diff,big}}\left( \zeta ,\eta
_{u,W}^{2,0},W,y^{2,0}\right) \right\vert &=&\left\vert R^{\mathrm{diff,big}%
}\left( \zeta ,W,\eta _{u,W}^{2,0},y^{2,0}\right) \right\vert \\
&=&e^{2f}s^{2}w_{h}\psi \mathrm{curv}^{S^{4}}\left( y^{2,0},\eta
_{u,W}^{2,0}\right) \left\langle y^{2,0},\zeta \right\rangle +\varkappa
\left( s\right) \\
\left\vert R^{\mathrm{diff,big}}\left( W,y^{2,0}y^{2,0},\eta
_{u,W}^{2,0}\right) \right\vert &=&e^{2f}s^{2}w_{h}\psi \mathrm{curv}%
^{S^{4}}\left( y^{2,0},\eta _{u,W}^{2,0}\right) +\varkappa \left( s\right)
\end{eqnarray*}%
and all other coefficients of $R^{\mathrm{diff,big}}$ are $0$.
\end{theorem}

Before discussing the proofs, we show how these two theorems gives us that $%
P\left( \sigma ,\tau \right) >0,Q_{\zeta }\left( \sigma \right) >0,$ and $%
Q_{W}\left( \tau \right) >0$ for all $\sigma ,\tau \in \mathbb{R},$ and
hence that the Gromoll-Meyer sphere is positively curved. The proofs that $%
Q_{\zeta }\left( \sigma \right) >0,$ and $Q_{W}\left( \tau \right) >0$ are
strictly contained in the proof that $P\left( \sigma ,\tau \right) >0,$so we
only write out the details that $P\left( \sigma ,\tau \right) >0.$

We discuss the case of Theorem \ref{higher order curvatures} and then those
of Theorem \ref{higher order curvatures-2}.

From our proof of the main lemma, we have that in the case of Theorem \ref%
{higher order curvatures} 
\begin{eqnarray*}
P\left( \sigma ,\tau \right) &\geq &O\left( s^{4}w_{h}^{2}\nu \right)
+P^{\nu ,re,l}\left( \sigma ,\tau \right) +\sigma ^{2}\mathrm{curv}^{\mathrm{%
diff,big}}\left( z,W\right) \\
&&+2\sigma ^{2}\tau \sqrt{\mathrm{curv}^{\mathrm{diff,big}}\left( z,V\right) 
}\sqrt{\mathrm{curv}^{\mathrm{diff,big}}\left( z,W\right) }+\sigma ^{2}\tau
^{2}\mathrm{curv}^{\mathrm{diff,big}}\left( z,V\right) +O.
\end{eqnarray*}

The sum 
\begin{eqnarray*}
&&\sigma ^{2}\mathrm{curv}^{\mathrm{diff,big}}\left( z,W\right) +2\sigma
^{2}\tau \sqrt{\mathrm{curv}^{\mathrm{diff,big}}\left( z,V\right) }\sqrt{%
\mathrm{curv}^{\mathrm{diff,big}}\left( z,W\right) }+\sigma ^{2}\tau ^{2}%
\mathrm{curv}^{\mathrm{diff,big}}\left( z,V\right) \\
&=&\sigma ^{2}\left( \sqrt{\mathrm{curv}^{\mathrm{diff,big}}\left(
z,W\right) }+\tau \sqrt{\mathrm{curv}^{\mathrm{diff,big}}\left( z,V\right) }%
\right) ^{2}
\end{eqnarray*}%
is nonnegative so we may drop it.

Thus

\begin{equation*}
P\left( \sigma ,\tau \right) \geq O\left( s^{4}w_{h}^{2}\nu \right) +P^{\nu
,re,l}\left( \sigma ,\tau \right) +O,
\end{equation*}%
and hence is positive.

In the case of Theorem \ref{higher order curvatures-2}, when $z=y^{2,0}$ and 
$V=\eta _{u,W}^{2,0}$%
\begin{eqnarray*}
P\left( \sigma ,\tau \right) &\geq &O\left( s^{4}w_{h}^{2}\nu \right)
+P^{\nu ,re,l}\left( \sigma ,\tau \right) +\sigma ^{2}\mathrm{curv}^{\mathrm{%
diff,big}}\left( y^{2,0},W\right) \\
&&+2\sigma \tau R^{\mathrm{diff,big}}\left( \zeta ,W,\eta
_{u}^{2,0},y^{2,0}\right) +2\sigma \tau R^{\mathrm{diff,big}}\left( \zeta
,\eta _{u}^{2,0},W,y^{2,0}\right) + \\
&&+2\sigma ^{2}\tau R^{\mathrm{diff,big}}\left( y^{2,0},W,\eta
_{u}^{2,0},y^{2,0}\right) +2\sigma \tau ^{2}R^{\mathrm{diff,big}}\left(
\zeta ,\eta _{u}^{2,0},\eta _{u}^{2,0},y^{2,0}\right) +O.
\end{eqnarray*}%
Plugging in our curvature estimates we get 
\begin{eqnarray*}
P\left( \sigma ,\tau \right) &\geq &O\left( s^{4}w_{h}^{2}\nu \right)
+P^{\nu ,re,l}\left( \sigma ,\tau \right) +\sigma
^{2}e^{2f}s^{2}w_{h}^{2}\psi ^{2}\mathrm{curv}^{S^{4}}\left( y^{2,0},\eta
_{u,W}^{2,0}\right) \\
&&+4\sigma \tau e^{2f}s^{2}w_{h}\psi \mathrm{curv}^{S^{4}}\left(
y^{2,0},\eta _{u,W}^{2,0}\right) \left\langle y^{2,0},\zeta \right\rangle \\
&&+2\sigma ^{2}\tau e^{2f}s^{2}w_{h}\psi \mathrm{curv}^{S^{4}}\left(
y^{2,0},\eta _{u,W}^{2,0}\right)
\end{eqnarray*}

For fixed $\tau ,$ we can view the $\sigma ^{2}$ and $\tau ^{2}$ terms of $%
P^{\nu ,re,l}\left( \sigma ,\tau \right) $ together with 
\begin{equation*}
\sigma ^{2}e^{2f}s^{2}w_{h}^{2}\psi ^{2}\mathrm{curv}^{S^{4}}\left(
y^{2,0},\eta _{u,W}^{2,0}\right) \text{ and }4\sigma \tau
e^{2f}s^{2}w_{h}\psi \mathrm{curv}^{S^{4}}\left( y^{2,0},\eta
_{u,W}^{2,0}\right) \left\langle y^{2,0},\zeta \right\rangle
\end{equation*}%
as a quadratic in $\sigma .$ Since $\left\vert \left\langle y^{2,0},\zeta
\right\rangle \right\vert \leq \frac{1}{2}+O\left( t\right) ,$ the minimum
is 
\begin{eqnarray*}
&&\tau ^{2}\left( \mathrm{curv}^{\nu ,re,l}\left( \zeta ,\eta
_{u}^{2,0}\right) -\frac{\left[ s^{2}w_{h}\psi \mathrm{curv}^{S^{4}}\left(
y^{2,0},\eta _{u,W}^{2,0}\right) \right] ^{2}}{\mathrm{curv}^{\nu
,re,l}\left( y^{2,0},W\right) +s^{2}w_{h}^{2}\psi ^{2}\mathrm{curv}%
^{S^{4}}\left( y^{2,0},\eta _{u,W}^{2,0}\right) }\right) +O \\
&\geq &\tau ^{2}\mathrm{curv}^{\nu ,re,l}\left( \zeta ,\eta
_{u}^{2,0}\right) +O.
\end{eqnarray*}%
Thus we may replace the mixed quadratic $\sigma \tau $ term with $O,$ and
our estimate becomes 
\begin{eqnarray*}
P\left( \sigma ,\tau \right) &\geq &O\left( s^{4}w_{h}^{2}\nu \right)
+P^{\nu ,re,l}\left( \sigma ,\tau \right) +\sigma
^{2}e^{2f}s^{2}w_{h}^{2}\psi ^{2}\mathrm{curv}^{S^{4}}\left( y^{2,0},\eta
_{u,W}^{2,0}\right) \\
&&+2\sigma ^{2}\tau e^{2f}s^{2}w_{h}\psi \mathrm{curv}^{S^{4}}\left(
y^{2,0},\eta _{u,W}^{2,0}\right) +O.
\end{eqnarray*}%
For fixed $\sigma $, we view the $\sigma ^{2}e^{2f}s^{2}w_{h}^{2}\psi ^{2}%
\mathrm{curv}^{S^{4}}\left( y^{2,0},\eta _{u,W}^{2,0}\right) $ term, the $%
2\sigma ^{2}\tau e^{2f}s^{2}w_{h}\psi \mathrm{curv}^{S^{4}}\left(
y^{2,0},\eta _{u,W}^{2,0}\right) $ term and the $\sigma ^{2}\tau ^{2}$ term
of $P^{\nu ,re,l}\left( \sigma ,\tau \right) $ as a quadratic in $\tau .$
The minimum is 
\begin{equation*}
\sigma ^{2}e^{2f}\left( s^{2}w_{h}^{2}\psi ^{2}\mathrm{curv}^{S^{4}}\left(
y^{2,0},\eta _{u,W}^{2,0}\right) -s^{4}w_{h}^{2}\psi ^{2}\mathrm{curv}%
^{S^{4}}\left( y^{2,0},\eta _{u,W}^{2,0}\right) \right) +O=\sigma
^{2}s^{2}w_{h}^{2}\psi ^{2}\mathrm{curv}^{S^{4}}\left( y^{2,0},\eta
_{u,W}^{2,0}\right) +O
\end{equation*}%
So we again have 
\begin{eqnarray*}
P\left( \sigma ,\tau \right) &\geq &O\left( s^{4}w_{h}^{2}\nu \right)
+P^{\nu ,re,l}\left( \sigma ,\tau \right) +O \\
&>&0.
\end{eqnarray*}

Finally, in the exceptional case when $z=\eta _{u,W}^{2,0}$ and $V=y^{2,0}$
we plug in our curvature estimates and get 
\begin{eqnarray*}
&&P\left( \sigma ,\tau \right) \equiv O\left( s^{4}w_{h}^{2}\nu \right)
+P^{\nu ,re,l}\left( \sigma ,\tau \right) \\
&&+2\sigma \tau e^{2f}s^{2}w_{h}\psi \mathrm{curv}^{S^{4}}\left(
y^{2,0},\eta _{u,W}^{2,0}\right) \left\langle y^{2,0},\zeta \right\rangle +O
\end{eqnarray*}%
When $t\geq O\left( \nu ^{1/2}\right) ,$ the $\left( \sigma \tau \right) $%
--term is dominated by the $\sigma ^{2}\mathrm{curv}\left( \eta
_{u,W}^{2,0},W\right) $ and $\tau ^{2}\mathrm{curv}\left( y^{2,0},\zeta
\right) $ terms of $P^{\nu ,re,l}\left( \sigma ,\tau \right) .$ So we may
assume that $t\leq O\left( \nu ^{1/2}\right) .$

In this case, we view the $\sigma \tau ,$ and $\sigma ^{2}\tau ^{2}$ terms
of $P\left( \sigma ,\tau \right) $ as a quadratic in $\sigma \tau .$ The
minimum of this quadratic is 
\begin{equation*}
-e^{2f}\frac{\left( 2s^{2}w_{h}\psi \mathrm{curv}^{S^{4}}\left( y^{2,0},\eta
_{u,W}^{2,0}\right) \left\langle y^{2,0},\zeta \right\rangle \right) ^{2}}{4%
\mathrm{curv}^{S^{4}}\left( y^{2,0},\eta _{u,W}^{2,0}\right) }+O
\end{equation*}%
Since $\left\vert \left\langle y^{2,0},\zeta \right\rangle \right\vert \leq 
\frac{1}{2}+O\left( t\right) $ our minimum is 
\begin{eqnarray}
&\geq &-e^{2f}s^{4}w_{h}^{2}\psi ^{2}\left( \frac{1}{4}\mathrm{curv}%
^{S^{4}}\left( y^{2,0},\eta _{u,W}^{2,0}\right) +\mathrm{curv}^{S^{4}}\left(
y^{2,0},\eta _{u,W}^{2,0}\right) t\right)  \notag \\
&=&-e^{2f}\frac{1}{4}s^{4}w_{h}^{2}\psi ^{2}\left[ \left( \frac{D_{\zeta
}D_{\zeta }\psi }{\psi }+\frac{t^{2}\psi ^{2}}{l^{6}}\right) \right] +O 
\notag \\
&=&-e^{2f}\frac{1}{4}s^{4}w_{h}^{2}\psi D_{\zeta }D_{\zeta }\psi +O
\label{sigma-tau-min}
\end{eqnarray}

This is of the order of our constant coefficient%
\begin{equation*}
\mathrm{curv}^{new}\left( \zeta ,W\right) =e^{2f}s^{4}w_{h}^{2}\left(
D_{\zeta }\psi \right) ^{2}\left[ 1+\frac{\psi ^{2}}{\nu ^{2}}\right]
+e^{2f}\iota .
\end{equation*}%
So we will have to be careful here.

Notice that the minimum occurs when 
\begin{equation*}
\sigma \tau =O\left( s^{2}w_{h}\psi \right) ,
\end{equation*}%
and we have not used the two positive quadratic terms 
\begin{equation*}
\sigma ^{2}\mathrm{curv}\left( \eta _{u,W}^{2,0},W\right) +\tau ^{2}\mathrm{%
curv}\left( y^{2,0},\zeta \right) .
\end{equation*}%
It will be sufficient to show that near the minimum this sum is much larger
than $O\left( s^{4}w_{h}^{2}\nu \right) .$We will actually show that this
holds except for $t\in \left[ 0,s^{2}w_{h}\nu \right] .$ We will then argue
that with a very minor adjustment in $\iota ,$ we can easily dominate the
negative term \ref{sigma-tau-min} on the exceptional region.

Thus we have positive curvature except possibly if 
\begin{equation*}
\sigma ^{2}\mathrm{curv}\left( \eta _{u,W}^{2,0},W\right) \leq O\left(
s^{4}w_{h}^{2}\nu \right)
\end{equation*}%
or 
\begin{eqnarray*}
\sigma ^{2}\left( 1+\frac{\psi ^{2}}{l^{6}}\right) &\leq &O\left(
s^{4}w_{h}^{2}\nu \right) \text{ or} \\
\sigma ^{2} &\leq &\frac{O\left( s^{4}w_{h}^{2}\nu \right) }{\left( 1+\frac{%
\psi ^{2}}{l^{6}}\right) }\text{ or} \\
\frac{1}{\sigma } &\geq &\frac{\sqrt{1+\frac{\psi ^{2}}{l^{6}}}}{O\left(
s^{2}w_{h}\nu ^{1/2}\right) }
\end{eqnarray*}%
Since we also have%
\begin{equation*}
\sigma \tau =O\left( s^{2}w_{h}\psi \right) ,
\end{equation*}%
we get 
\begin{eqnarray*}
\tau &\geq &\frac{O\left( s^{2}w_{h}\psi \right) }{\sigma } \\
&\geq &O\left( s^{2}w_{h}\psi \right) \frac{\sqrt{1+\frac{\psi ^{2}}{l^{6}}}%
}{O\left( s^{2}w_{h}\nu ^{1/2}\right) } \\
&\geq &O\left( \frac{\psi }{\nu ^{1/2}}\right) \sqrt{1+\frac{\psi ^{2}}{l^{6}%
}}
\end{eqnarray*}%
Thus our quadratic term 
\begin{equation*}
\tau ^{2}\mathrm{curv}\left( y^{2,0},\zeta \right) \geq O\left( \frac{\psi
^{2}}{\nu }\right) \left( 1+\frac{\psi ^{2}}{l^{6}}\right) .
\end{equation*}%
This is much larger than $O\left( s^{4}w_{h}^{2}\nu \right) ,$ except if 
\begin{eqnarray*}
\psi ^{2} &\leq &O\left( s^{4}w_{h}^{2}\nu ^{2}\right) ,\text{ or} \\
\psi &\leq &O\left( s^{2}w_{h}\nu \right)
\end{eqnarray*}%
Since 
\begin{equation*}
\frac{\partial }{\partial t}\psi =O\left( 1\right) ,
\end{equation*}%
on $\left[ 0,O\left( \nu \right) \right] ,$ the exceptional region is when $%
t\in \left[ 0,O\left( s^{2}w_{h}\nu \right) \right] .$

\medskip On this region, we see from Proposition \ref{seceond derivatives}
that 
\begin{eqnarray*}
\left\vert \psi D_{\zeta }D_{\zeta }\psi \right\vert &\leq &7\frac{\sin
^{2}2t}{\nu ^{2}} \\
&\leq &\frac{O\left( s^{2}w_{h}\nu \right) ^{2}}{\nu ^{2}} \\
&=&O\left( s^{4}w_{h}^{2}\right)
\end{eqnarray*}%
So the absolute value of our minimum in (\ref{sigma-tau-min}) is 
\begin{equation*}
\leq O\left( s^{4}w_{h}^{2}\right) ^{2}=O\left( s^{8}w_{h}^{4}\right)
\end{equation*}%
and the integral of our minimum over this exceptional region is 
\begin{equation*}
O\left( s^{10}w_{h}^{5}\nu \right) .
\end{equation*}%
So with an extremely small adjustment to $\iota ,$ we can dominate the
negative term \ref{sigma-tau-min} even on this exceptional region.

\section{Higher order computations}

In this section we prove Theorems \ref{higher order curvatures} and \ref%
{higher order curvatures-2}, and so (modulo the appendix) complete the proof
that the Gromoll-Meyer sphere admits positive curvature. To do this we think
of the lift of $T\Sigma ^{7}$ to $TSp\left( 2\right) $ as split into 
\begin{equation*}
\mathrm{span}\left\{ \zeta \right\} \oplus \mathrm{span}\left\{
y^{2,0}\right\} \oplus \mathrm{span}\left\{ \eta _{u,1}^{2,0},\eta
_{u,2}^{2,0}\right\} \oplus \mathrm{span}\left\{ W\right\} \oplus \mathrm{%
span}\left\{ \left( V_{1}\oplus V_{2}\right) ^{\perp ,W}\right\} .
\end{equation*}%
Since $z\in H_{2,-1},$ it can only be in either the second or the third
factor, whereas the perturbation vector $V$ can be in any but the $\zeta $
or $W$ factors.

We divide our computations accordingly. So we have five cases to consider 
\begin{eqnarray*}
z &=&y^{2,0},\text{ }V\in \left( V_{1}\oplus V_{2}\right) ^{\perp ,W} \\
z,V &\in &\mathrm{span}\left\{ \eta _{u,1}^{2,0},\eta _{u,2}^{2,0}\right\} \\
z &=&y^{2,0},\text{ }V\in \mathrm{span}\left\{ \eta _{1}^{2,0},\eta
_{2}^{2,0}\right\} , \\
z &\in &\mathrm{span}\left\{ \eta _{1}^{2,0},\eta _{2}^{2,0}\right\} ,\text{ 
}V=y^{2,0} \\
z &\in &\mathrm{span}\left\{ \eta _{1}^{2,0},\eta _{2}^{2,0}\right\} ,\text{ 
}V\in \left( V_{1}\oplus V_{2}\right) ^{\perp ,W}.
\end{eqnarray*}%
Some sectional curvature terms occur in more than one of these cases. So to
simplify the accounting we handle the possible sectional curvatures in the
first subsection. These are those that occur as quadratic or quartic
coefficients of $P^{\mathrm{diff,big}}\left( \sigma ,\tau \right) $ in each
of these five cases.

We also need the $\sigma \tau ,$ $\sigma \tau ^{2},$ and $\sigma ^{2}\tau $
coefficients of $P^{\mathrm{diff,big}}\left( \sigma ,\tau \right) .$ These
are computed on a case by case basis in the last four subsections. (The
third and fourth case are handled as one in the third subsection.)

\subsection{Sectional Curvatures}

Letting $\mathcal{V}$ be a vector in $\left( V_{1}\oplus V_{2}\right)
^{\perp ,W},$ the (unnormalized) sectional curvatures that we need are 
\begin{eqnarray*}
&&\mathrm{curv}^{\mathrm{diff}}\left( \zeta ,\mathcal{V}\right) ,\mathrm{curv%
}^{\mathrm{diff}}\left( \zeta ,\eta _{u}^{2,0}\right) ,\mathrm{curv}^{%
\mathrm{diff}}\left( \zeta ,y^{2,0}\right) ,\mathrm{curv}^{\mathrm{diff}%
}\left( W,\eta _{u}^{2,0}\right) ,\mathrm{curv}^{\mathrm{diff}}\left(
W,y^{2,0}\right) , \\
&&\mathrm{curv}^{\mathrm{diff}}\left( \eta _{u}^{2,0},\mathcal{V}\right) ,%
\mathrm{curv}^{\mathrm{diff}}\left( \eta _{u}^{2,0},y^{2,0}\right) ,\mathrm{%
curv}^{\mathrm{diff}}\left( \mathcal{V},y^{2,0}\right) ,\mathrm{curv}^{%
\mathrm{diff}}\left( \eta _{u,1}^{2,0},\eta _{u,2}^{2,0}\right) .
\end{eqnarray*}%
The $A$--tensor term in \ref{Detlef equation} does not appear in the
curvatures of two horizontal vectors, and the $s^{2}\mathrm{curv}^{S^{4}}$
term and all of the terms of the partial conformal change are also small on
these planes. Thus

\begin{proposition}
The curvatures%
\begin{equation*}
\mathrm{curv}^{\mathrm{diff,big}}\left( \zeta ,y^{2,0}\right) ,\mathrm{curv}%
^{\mathrm{diff,big}}\left( \zeta ,\eta _{u}^{2,0}\right) ,\mathrm{curv}^{%
\mathrm{diff,big}}\left( \eta _{u}^{2,0},y^{2,0}\right) ,\mathrm{curv}^{%
\mathrm{diff,big}}\left( \eta _{u,1}^{2,0},\eta _{u,2}^{2,0}\right)
\end{equation*}%
are $0.$
\end{proposition}

\begin{proposition}
For $\mathcal{V}\in \left( V_{1}\oplus V_{2}\right) ^{\perp ,W}$%
\begin{eqnarray*}
\mathrm{curv}^{\mathrm{diff,big}}\left( W,y^{2,0}\right)
&=&s^{2}w_{h}^{2}\psi ^{2}\mathrm{curv}^{S^{4}}\left( y^{2,0},\eta
_{u,W}^{2,0}\right) +O, \\
\mathrm{curv}^{\mathrm{diff,big}}\left( \mathcal{V},y^{2,0}\right)
&=&s^{2}v_{h}^{2}\psi ^{2}\mathrm{curv}^{S^{4}}\left( y^{2,0},\eta
_{u,V}^{2,0}\right)
\end{eqnarray*}
\end{proposition}

\begin{proof}
The iterated $A$--tensor term is small because $S_{y}\left( k_{\gamma
}\right) $ is small. Similarly, the partial conformal change is small
because the various $y$--derivatives of $\psi $ are small. The $S^{4}$--term
gives the leading contribution so%
\begin{eqnarray*}
\mathrm{curv}^{\mathrm{diff,big}}\left( W,y^{2,0}\right)
&=&s^{2}w_{h}^{2}\psi ^{2}\mathrm{curv}^{S^{4}}\left( y^{2,0},\eta
_{u,W}^{2,0}\right) +O\text{ and} \\
\mathrm{curv}^{\mathrm{diff,big}}\left( \mathcal{V},y^{2,0}\right)
&=&s^{2}w_{h}^{2}\psi ^{2}\mathrm{curv}^{S^{4}}\left( y^{2,0},\eta
_{u,V}^{2,0}\right) +O.
\end{eqnarray*}
\end{proof}

\begin{proposition}
\begin{equation*}
\mathrm{curv}^{\mathrm{diff,big}}\left( \zeta ,\mathcal{V}\right) =0
\end{equation*}
\end{proposition}

\begin{proof}
This computation looks like the computation of \textrm{curv}$\left( \zeta
,W\right) .$ The $A$--tensor, $S^{4}$--curvature, and $-\left\vert
V\right\vert ^{2}\mathrm{hess}_{f}\left( \zeta ,\zeta \right) $ terms can
all be large, but to leading order they cancel each other out.
\end{proof}

\begin{proposition}
\label{diff big sec}For $V\in V_{1}\oplus V_{2}$%
\begin{equation*}
R^{\mathrm{diff,big}}\left( \eta _{u}^{2,0},W,W,\eta _{u}^{2,0}\right) \geq
e^{2f}w_{h}^{2}s^{2}\left\vert \mathrm{grad\,}\psi \right\vert ^{2}\left(
1-\left\langle \eta _{u,W}^{2,0},\eta _{u}^{2,0}\right\rangle ^{2}\right)
\end{equation*}%
\begin{equation*}
R^{\mathrm{diff,big}}\left( \eta _{u}^{2,0},V,V,\eta _{u}^{2,0}\right) \geq
e^{2f}w_{h}^{2}s^{2}\left\vert \mathrm{grad\,}\psi \right\vert ^{2}\left(
1-\left\langle \eta _{u,V}^{2,0},\eta _{u}^{2,0}\right\rangle ^{2}\right)
\end{equation*}
\end{proposition}

\begin{proof}
The two inequalities have similar proofs, so we just focus on the first. 
\begin{equation*}
R^{s}\left( \eta _{u}^{2,0},W,W,\eta _{u}^{2,0}\right) =R^{\nu ,re,l}\left(
\eta _{u}^{2,0},W,W,\eta _{u}^{2,0}\right) +s^{2}R^{S^{4}}\left( \eta
_{u}^{2,0},H,H,\eta _{u}^{2,0}\right) -s^{2}\left( 1-s^{2}\right) \left\vert
A_{\eta _{u}^{2,0}}W^{v}\right\vert ^{2}
\end{equation*}%
We have 
\begin{eqnarray*}
A_{\eta _{u}^{2,0}}W^{v} &=&\frac{1}{\left\vert \cos 2t\eta
^{2,0}\right\vert }\left( \nabla _{\left( \eta ,\eta \right) }^{\nu
,re,l}W\right) ^{\mathcal{H}}-\mathrm{II}\left( \eta _{u}^{2,0},W^{\mathcal{H%
}}\right) +4\frac{\psi ^{2}}{\nu ^{3}}\left\vert W_{\alpha }\right\vert
\left( \eta _{u}^{2,0}\right) ^{\perp } \\
&=&\frac{O\left( 1+\frac{t}{l^{2}}+\frac{t^{2}}{l^{6}}\right) }{\left\vert
\cos 2t\eta ^{2,0}\right\vert }+w_{h}\mathrm{grad}\psi \left\langle \eta
_{u,W}^{2,0},\eta _{u}^{2,0}\right\rangle +4\frac{\psi ^{2}}{\nu ^{3}}%
\left\vert W_{\alpha }\right\vert _{h_{2}}\left( \eta _{u}^{2,0}\right)
^{\perp }
\end{eqnarray*}%
So 
\begin{equation*}
s^{2}\left\vert A_{\eta _{u}^{2,0}}W^{v}\right\vert ^{2}=s^{2}O\left( \frac{%
\psi ^{4}}{\nu ^{6}}\left\vert W_{\alpha }\right\vert _{h_{2}}^{2}\right)
+s^{2}w_{h}^{2}\left\vert \mathrm{grad}\psi \right\vert ^{2}\left\langle
\eta _{u,W}^{2,0},\eta _{u}^{2,0}\right\rangle ^{2}+O
\end{equation*}%
Since 
\begin{equation*}
\mathrm{curv}\left( \eta _{u}^{2,0},W\right) \geq \frac{\psi ^{2}}{\nu ^{4}}%
\left\vert W_{\alpha }\right\vert _{h_{2}}^{2},
\end{equation*}%
we can bound the first term of the $A$--tensor by 
\begin{equation*}
s^{2}\frac{\psi ^{2}}{\nu ^{2}}\mathrm{curv}\left( \eta _{u}^{2,0},W\right)
\geq s^{2}O\left( \frac{\psi ^{4}}{\nu ^{6}}\left\vert W_{\alpha
}\right\vert _{h_{2}}^{2}\right) .
\end{equation*}

The second term in our expression for $s^{2}\left\vert A_{\eta
_{u}^{2,0}}W^{v}\right\vert ^{2}$ compares well with the $-\left\vert
W\right\vert ^{2}\mathrm{Hess}_{f}\left( \eta _{u}^{2,0},\eta
_{u}^{2,0}\right) $ term from the partial conformal change indeed 
\begin{eqnarray*}
-\left\vert W\right\vert ^{2}\mathrm{Hess}_{f}\left( \eta _{u}^{2,0},\eta
_{u}^{2,0}\right) &=&-\left\vert W\right\vert ^{2}\left\langle \nabla _{\eta
_{u}^{2,0}}\mathrm{grad}f,\eta _{u}^{2,0}\right\rangle \\
&=&\left\vert W\right\vert ^{2}\frac{s^{2}}{\nu ^{2}}\left\langle \nabla
_{\eta _{u}^{2,0}}\left( \psi \mathrm{grad\,}\psi \right) ,\eta
_{u}^{2,0}\right\rangle +O \\
&=&-\left\vert W\right\vert ^{2}\frac{s^{2}}{\nu ^{2}}\left\langle \left(
\psi \mathrm{grad\,}\psi \right) ,\nabla _{\eta _{u}^{2,0}}\eta
_{u}^{2,0}\right\rangle +O \\
&=&\left\vert W\right\vert ^{2}\frac{s^{2}}{\nu ^{2}}\left\vert \mathrm{%
grad\,}\psi \right\vert ^{2}+O \\
&=&w_{h}^{2}s^{2}\left\vert \mathrm{grad\,}\psi \right\vert ^{2}+O,
\end{eqnarray*}%
So combining displays we have 
\begin{eqnarray*}
R^{\mathrm{new}}\left( \eta _{u}^{2,0},W,W,\eta _{u}^{2,0}\right) &\geq
&R^{\nu ,re,l}\left( \eta _{u}^{2,0},W,W,\eta _{u}^{2,0}\right)
+s^{2}R^{S^{4}}\left( \eta _{u}^{2,0},H,H,\eta _{u}^{2,0}\right) -s^{2}%
\mathrm{curv}\left( \eta _{u}^{2,0},W\right) \\
&&+w_{h}^{2}s^{2}\left\vert \mathrm{grad\,}\psi \right\vert ^{2}\left(
1-\left\langle \eta _{u,W}^{2,0},\eta _{u}^{2,0}\right\rangle ^{2}\right) +O
\end{eqnarray*}%
so 
\begin{equation*}
R^{\mathrm{diff,big}}\left( \eta _{u}^{2,0},W,W,\eta _{u}^{2,0}\right) \geq
w_{h}^{2}s^{2}\left\vert \mathrm{grad\,}\psi \right\vert ^{2}\left(
1-\left\langle \eta _{u,W}^{2,0},\eta _{u}^{2,0}\right\rangle ^{2}\right)
\end{equation*}%
as claimed.
\end{proof}

\subsection{$z=y^{2,0},$ $V\in V_{1}\oplus V_{2}$}

\begin{proposition}
If $V$ is in $\left( V_{1}\oplus V_{2}\right) $ and the $h_{2}$--part of $V$
is perpendicular to $W_{\gamma },$ then 
\begin{eqnarray*}
\left\vert \left\langle R^{\mathrm{diff}}\left( W,\zeta \right)
y^{2,0},V\right\rangle \right\vert &\leq &D_{y^{2,0}}\left( \psi \right)
O\left( s^{2}w_{h}\right) +D_{\zeta }\left( \psi \right) O\left(
s^{2}v_{h}\right) +O\left( \frac{s^{2}}{l^{3}}\right) \\
&\leq &\varkappa \left( s\right) \\
\left\vert \left\langle R^{\mathrm{diff}}\left( W,y^{2,0}\right) \zeta
,V\right\rangle \right\vert &\leq &D_{\zeta }\left( \psi \right) O\left(
s^{2}w_{h}\right) \leq \varkappa \left( s\right)
\end{eqnarray*}%
and%
\begin{eqnarray*}
\left\vert \left\langle R^{\mathrm{diff}}\left( W,y^{2,0}\right)
y^{2,0},V\right\rangle \right\vert &\leq &D_{y^{2,0}}\left( \psi \right)
O\left( s^{2}\left( w_{h}+v_{h}\right) \right) +O\left( \frac{s^{2}}{l^{3}}%
\right) \\
&=&\varkappa \left( s\right)
\end{eqnarray*}%
\begin{equation*}
\left\vert \left\langle R^{\mathrm{diff}}\left( V,y^{2,0}\right) \zeta
,V\right\rangle \right\vert \leq \varkappa \left( s\right)
\end{equation*}%
In particular, for all four curvatures $R^{\mathrm{diff,big}}=0.$
\end{proposition}

\begin{proof}
To find the effect of shrinking the fibers we use equations \ref{Detlef
equation} and get

\begin{eqnarray*}
R^{s}\left( W,\zeta \right) y^{2,0} &=&R^{s}\left( W^{\mathcal{V}},\zeta
\right) y^{2,0}+R^{s}\left( W^{\mathcal{H}},\zeta \right) y^{2,0} \\
&=&\left( 1-s^{2}\right) R^{\nu ,re,l}(W^{\mathcal{V}},\zeta
)y^{2,0}+s^{2}\left( R^{\nu ,re,l}(W^{\mathcal{V}},\zeta )y^{2,0}\right) ^{%
\mathcal{V}}+s^{2}A_{\zeta }A_{y^{2,0}}W^{\mathcal{V}} \\
&&+\left( 1-s^{2}\right) R^{\nu ,re,l}(W^{\mathcal{H}},\zeta
)y^{2,0}+s^{2}\left( R^{\nu ,re,l}(W^{\mathcal{H}},\zeta )y^{2,0}\right) ^{%
\mathcal{V}}+s^{2}R^{S^{4}}(W^{\mathcal{H}},\zeta )y^{2,0} \\
&=&\left( 1-s^{2}\right) \left( R^{\nu ,re,l}\left( W,\zeta \right)
y^{2,0}\right) ^{\mathcal{H}}+\left( R^{\nu ,re,l}\left( W,\zeta \right)
y^{2,0}\right) ^{\mathcal{V}} \\
&&+s^{2}A_{\zeta }A_{y^{2,0}}W^{\mathcal{V}}+s^{2}R^{S^{4}}(W^{\mathcal{H}%
},\zeta )y^{2,0}
\end{eqnarray*}%
Similarly%
\begin{eqnarray*}
R^{s}\left( W,y^{2,0}\right) \zeta &=&\left( 1-s^{2}\right) \left( R^{\nu
,re,l}\left( W,y^{2,0}\right) \zeta \right) ^{\mathcal{H}}+\left( R^{\nu
,re,l}\left( W,y^{2,0}\right) \zeta \right) ^{\mathcal{V}} \\
&&+s^{2}A_{y^{2,0}}A_{\zeta }W^{\mathcal{V}}+s^{2}R^{S^{4}}(W^{\mathcal{H}%
},y^{2,0})\zeta ,\text{ and}
\end{eqnarray*}%
\begin{eqnarray*}
R^{s}\left( W,y^{2,0}\right) y^{2,0} &=&\left( 1-s^{2}\right) \left( R^{\nu
,re,l}\left( W,y^{2,0}\right) y^{2,0}\right) ^{\mathcal{H}}+\left( R^{\nu
,re,l}\left( W,y^{2,0}\right) y^{2,0}\right) ^{\mathcal{V}} \\
&&+s^{2}A_{y^{2,0}}A_{y^{2,0}}W^{\mathcal{V}}+s^{2}R^{S^{4}}\left( W^{%
\mathcal{H}},y^{2,0}\right) y^{2,0}.
\end{eqnarray*}%
Since $A^{h_{2}}$--induces an $SO\left( 3\right) $--action on $S^{4}$ that
is standard on the $S_{\func{Im}}^{2}$s and leaves $\zeta $ and $y^{2,0}$
invariant, the restriction to $TS_{\func{Im}}^{2}$ of the compositions of
orthogonal projection to $TS_{\func{Im}}^{2}$ with any of $R^{S^{4}}\left(
\cdot ,y^{2,0}\right) y^{2,0},$ $R^{S^{4}}(\cdot ,\zeta )y^{2,0},$ or $%
R^{S^{4}}(\cdot ,y^{2,0})\zeta $ are homotheties. In particular, for $V$ in $%
\left( V_{1}\oplus V_{2}\right) \cap H^{GM}$ with the $h_{2}$--part of $V$
perpendicular to $W_{\gamma },$ we have%
\begin{eqnarray*}
\left\langle R^{S^{4}}\left( W^{\mathcal{H}},y^{2,0}\right) y^{2,0},V^{%
\mathcal{H}}\right\rangle &=&0, \\
\left\langle R^{S^{4}}(W^{\mathcal{H}},\zeta )y^{2,0},V^{\mathcal{H}%
}\right\rangle &=&0,\text{ and } \\
\left\langle R^{S^{4}}(W^{\mathcal{H}},y^{2,0})\zeta ,V^{\mathcal{H}%
}\right\rangle &=&0.
\end{eqnarray*}

For $V$ in $\left( V_{1}\oplus V_{2}\right) $ we use Lemma \ref{A--tensor
estimate} to see that%
\begin{eqnarray*}
\left\vert s^{2}\left\langle A_{\zeta }A_{y^{2,0}}W^{\mathcal{V}%
},V\right\rangle \right\vert &=&\left\vert s^{2}\left\langle A_{y^{2,0}}W^{%
\mathcal{V}},A_{\zeta }V\right\rangle \right\vert \\
&=&\left\vert s^{2}\left\langle \left( \nabla _{y^{2,0}}^{\nu ,re,l}W\right)
^{\mathcal{H}}-S_{y^{2,0}}\left( W^{\mathcal{H}}\right) ,\left( \nabla
_{\zeta }^{\nu ,re,l}V\right) ^{\mathcal{H}}-S_{\zeta }\left( V^{\mathcal{H}%
}\right) \right\rangle \right\vert \\
&=&\left\vert s^{2}\left\langle \left( \nabla _{y^{2,0}}^{\nu ,re,l}W\right)
^{\mathcal{H}}-w_{h}D_{y^{2,0}}\left( \psi \right) \frac{k_{\gamma ,W}}{\psi 
},\left( \nabla _{\zeta }^{\nu ,re,l}V\right) ^{\mathcal{H}}-v_{h}D_{\zeta
}\left( \psi \right) \frac{k_{\gamma ,V}}{\psi }\right\rangle \right\vert
\end{eqnarray*}

Since the $\gamma $--part of $V$ is perpendicular to the $\gamma $--part of $%
W,$%
\begin{eqnarray*}
\left\vert s^{2}\left\langle A_{\zeta }A_{y^{2,0}}W^{\mathcal{V}%
},V\right\rangle \right\vert &\leq &s^{2}\left( \left\vert \left\langle
\left( \nabla _{y^{2,0}}^{\nu ,re,l}W\right) ^{\mathcal{H}},v_{h}D_{\zeta
}\left( \psi \right) \frac{k_{\gamma ,V}}{\psi }\right\rangle \right\vert
+\left\vert \left\langle w_{h}D_{y^{2,0}}\left( \psi \right) \frac{k_{\gamma
,W}}{\psi },\left( \nabla _{\zeta }^{\nu ,re,l}V\right) ^{\mathcal{H}%
}\right\rangle \right\vert \right) \\
&&+s^{2}\left\vert \left\langle \left( \nabla _{y^{2,0}}^{\nu ,re,l}W\right)
^{\mathcal{H}},\left( \nabla _{\zeta }^{\nu ,re,l}V\right) ^{\mathcal{H}%
}\right\rangle \right\vert \\
&\leq &D_{y^{2,0}}\left( \psi \right) O\left( s^{2}w_{h}\right) +D_{\zeta
}\left( \psi \right) O\left( s^{2}v_{h}\right) +O\left( \frac{s^{2}}{l^{4}}%
\right) \\
&\leq &\varkappa \left( s\right)
\end{eqnarray*}

Similarly, 
\begin{eqnarray*}
\left\vert s^{2}\left\langle A_{y^{2,0}}A_{\zeta }W^{\mathcal{V}%
},V\right\rangle \right\vert &\leq &D_{\zeta }\left( \psi \right) O\left(
s^{2}w_{h}\right) +O \\
&\leq &\varkappa \left( s\right) ,\text{ and}
\end{eqnarray*}%
\begin{eqnarray*}
\left\vert s^{2}\left\langle A_{y^{2,0}}A_{y^{2,0}}W^{\mathcal{V}%
},V\right\rangle \right\vert &\leq &D_{y^{2,0}}\left( \psi \right) O\left(
s^{2}w_{h}\right) +D_{y^{2,0}}\left( \psi \right) O\left( s^{2}v_{h}\right)
+O \\
&\leq &\varkappa \left( s\right) .
\end{eqnarray*}

(There are fewer terms in the estimate for $\left\vert s^{2}\left\langle
A_{y^{2,0}}A_{\zeta }W^{\mathcal{V}},V\right\rangle \right\vert $ since $%
\nabla _{\zeta }^{\nu ,re,l}W=0.)$

These three $A$--tensor inequalities give the first three inequalities after
the fibers have been shrunken.

Combining this with our partial conformal change and Hessian formulas yields
the first three results.

The final curvature is also small, but this fact is much subtler.

The $A$--tensor part give us 
\begin{eqnarray*}
s^{2}\left\langle A_{y^{2,0}}A_{\zeta }V^{\mathcal{V}},V\right\rangle
&=&-s^{2}\left\langle A_{\zeta }V^{\mathcal{V}},A_{y^{2,0}}V^{\mathcal{V}%
}\right\rangle \\
&=&-s^{2}\left\langle \left( \nabla _{\zeta }^{\nu ,re,l}V\right) ^{\mathcal{%
H}}-S_{\zeta }\left( V^{\mathcal{H}}\right) ,\left( \nabla _{y^{2,0}}^{\nu
,re,l}V\right) ^{\mathcal{H}}-S_{y^{2,0}}\left( V^{\mathcal{H}}\right)
\right\rangle \\
&=&-s^{2}\left\langle \left( \nabla _{\zeta }^{\nu ,re,l}V\right) ^{\mathcal{%
H}}-v_{h}D_{\zeta }\left( \psi \right) \frac{k_{\gamma ,V}}{\psi },\left(
\nabla _{y}^{\nu ,re,l}V\right) ^{\mathcal{H}}-v_{h}D_{y^{2,0}}\left( \psi
\right) \frac{k_{\gamma ,V}}{\psi }\right\rangle \\
&=&-s^{2}v_{h}^{2}D_{\zeta }\left( \psi \right) D_{y^{2,0}}\left( \psi
\right) -s^{2}v_{h}\left( D_{\zeta }\left( \psi \right) +D_{y^{2,0}}\left(
\psi \right) \right) +O \\
&=&-s^{2}v_{h}^{2}D_{\zeta }\left( \psi \right) D_{y^{2,0}}\left( \psi
\right) +O
\end{eqnarray*}%
The $S^{4}$--curvature gives us 
\begin{eqnarray*}
s^{2}R^{S^{4}}\left( \zeta ,V^{horiz},V^{horiz},y^{2,0}\right)
&=&-s^{2}\left\vert V^{horiz}\right\vert \left\langle \nabla _{\zeta }%
\mathrm{\func{grad}}\left\vert V^{horiz}\right\vert ,y^{2,0}\right\rangle \\
&=&-s^{2}v_{h}^{2}\psi \left\langle \nabla _{\zeta }\mathrm{\func{grad}}\psi
,y^{2,0}\right\rangle
\end{eqnarray*}%
Adding we get 
\begin{eqnarray*}
R^{\mathrm{diff,}s}\left( \zeta ,V^{horiz},V^{horiz},y^{2,0}\right)
&=&-s^{2}v_{h}^{2}\left\langle \nabla _{\zeta }\psi \mathrm{\func{grad}}\psi
,y^{2,0}\right\rangle \\
&=&\left\vert V^{\gamma }\right\vert ^{2}\mathrm{hess}\left( \zeta
,y^{2,0}\right) +O
\end{eqnarray*}

So this cancels with a hessian term from the partial conformal change. The
other terms of the partial conformal change are small, so the result follows.
\end{proof}

\subsection{$z,V\in \mathrm{span}\left\{ \protect\eta _{u,1}^{2,0},\protect%
\eta _{u,2}^{2,0}\right\} $}

\begin{proposition}
If $z,V\in \mathrm{span}\left\{ \eta _{u,1}^{2,0},\eta _{u,2}^{2,0}\right\} $
and $\left\vert z\right\vert =\left\vert V\right\vert =1$, then 
\begin{equation*}
R^{\mathrm{diff}}\left( \zeta ,W,V,z\right) \text{, }R^{\mathrm{diff}}\left(
\zeta ,V,W,z\right) ,\text{\textrm{\ }and }R^{\mathrm{diff}}\left( \zeta
,V,V,z\right)
\end{equation*}%
are all $0.$%
\begin{equation*}
R^{\mathrm{diff}}\left( z,V,W,z\right) =\left\langle V,W\right\rangle
\left\langle z,\eta _{u,W^{\perp }}^{2,0}\right\rangle \left( s^{2}\mathrm{%
curv}^{S^{4}}\left( \eta _{W^{\perp }}^{2,0},\eta _{W}^{2,0}\right) -\frac{%
s^{2}}{\nu ^{2}}\left\vert \mathrm{\func{grad}}\psi \right\vert ^{2}\right)
+O
\end{equation*}
\end{proposition}

\begin{remark}
For generic $t,$ \textrm{curv}$^{\nu ,re,l}\left( \zeta ,\eta
_{u}^{2,0}\right) =O\left( \nu ^{2}\right) ,$ so it is important to have
pretty tight estimates $R^{\mathrm{diff}}\left( \zeta ,W,V,z\right) $, $R^{%
\mathrm{diff}}\left( \zeta ,V,W,z\right) ,$\textrm{\ }and $R^{\mathrm{diff}%
}\left( \zeta ,V,V,z\right) .$
\end{remark}

\begin{proof}
In all four cases the $A$--tensor term is $0$ because at least three of the
vectors are horizontal.

In the first three cases all other terms of $R^{\mathrm{diff}}$ are also $0.$

The $S^{4}$--curvature term is $0$ because of the fact that three of the
vectors are in $\mathrm{span}\left\{ \eta _{u,1}^{2,0},\eta
_{u,2}^{2,0}\right\} $ and one of the vectors is $\zeta .$ Hessian terms are
all $0$ because the hessian of $\zeta $ with each of the other three vectors
is $0$ and $\zeta $ is perpendicular to each of the other three vectors. The
derivative terms of the partial conformal change are all $0$ because the
directional derivative of $f$ in each of the directions $W,V,$ and $z$ is $0$
and because $\zeta $ is perpendicular to each of the other three vectors.

The error component 
\begin{equation*}
+O\left( e^{2f}-1,\left\vert \mathrm{grad}f\right\vert \right) \mathrm{\max }%
\left\{ R^{\mathrm{old}}\left( X,Y,Z,U\right) ,\left\vert X\right\vert
\left\vert Y\right\vert \left\vert Z\right\vert \left\vert U\right\vert
\right\}
\end{equation*}%
of Proposition \ref{arbitrary conf change} is also $0.$

This is because 
\begin{eqnarray*}
&&\text{the Lie brackets of all of }z,V,\text{ and }W\text{ with }\zeta 
\text{ have no }\Delta \left( \alpha \right) -\text{component,} \\
&&\text{the Lie brackets of }\left( N\alpha p,N\alpha \right) \text{ with
each of }z,V,\text{ and }W\text{ have no }\zeta \text{--component, and} \\
&&\text{the Lie bracket of }\zeta \text{ and }\left( N\alpha p,N\alpha
\right) \text{ is }0.
\end{eqnarray*}

For the last curvature, we note that only the components of $z$ that are
perpendicular to $V$ and $W$ can make a contribution.

The first term comes from the $S^{4}$--curvature via the $s$--perturbation
and the second term comes from

\begin{eqnarray*}
-\left\langle V,W\right\rangle \mathrm{hess}\left( z,z\right)
&=&-\left\langle V,W\right\rangle \mathrm{hess}\left( \eta _{u,W^{\perp
}}^{2,0},\eta _{u,W^{\perp }}^{2,0}\right) \\
&=&-\left\langle V,W\right\rangle \frac{s^{2}}{\nu ^{2}}\left\vert \mathrm{%
\func{grad}}\psi \right\vert \psi \frac{\left\vert \mathrm{\func{grad}}\psi
\right\vert }{\psi }+O \\
&=&-\left\langle V,W\right\rangle \frac{s^{2}}{\nu ^{2}}\left\vert \mathrm{%
\func{grad}}\psi \right\vert ^{2}+O
\end{eqnarray*}%
There are other nonzero terms that come from the partial conformal change,
but they are much smaller.
\end{proof}

Since 
\begin{eqnarray*}
\mathrm{curv}\left( \eta _{W^{\perp }}^{2,0},W\right) &=&\frac{1}{\left\vert
\cos 2t\eta ^{2,0}\right\vert ^{2}}+\frac{\psi ^{2}}{\nu ^{6}} \\
\mathrm{curv}\left( \eta _{W^{\perp }}^{2,0},\eta _{W}^{2,0}\right) &=&\frac{%
1}{\left\vert \cos 2t\eta ^{2,0}\right\vert ^{4}}+\frac{\psi ^{4}}{\nu ^{6}}
\end{eqnarray*}%
we have in any case that

\begin{proposition}
If $z,V\in \mathrm{span}\left\{ \eta _{u,1}^{2,0},\eta _{u,2}^{2,0}\right\} $%
, then 
\begin{equation*}
R^{\mathrm{diff,big}}\left( \zeta ,W,V,z\right) =R^{\mathrm{diff,big}}\left(
\zeta ,V,W,z\right) =R^{\mathrm{diff,big}}\left( \zeta ,V,V,z\right) =R^{%
\mathrm{diff,big}}\left( z,V,W,z\right) =0
\end{equation*}
\end{proposition}

\subsection{$z=y^{2,0},$ $V\in \mathrm{span}\left\{ \protect\eta _{1}^{2,0},%
\protect\eta _{2}^{2,0}\right\} $ or $z\in \mathrm{span}\left\{ \protect\eta %
_{1}^{2,0},\protect\eta _{2}^{2,0}\right\} ,$ $V=y^{2,0}$}

\begin{proposition}
\label{z=y V= eta} 
\begin{eqnarray*}
R^{\mathrm{diff,big}}\left( \zeta ,\eta _{u}^{2,0},W,y^{2,0}\right) &=&R^{%
\mathrm{diff,big}}\left( \zeta ,W,\eta _{u}^{2,0},y^{2,0}\right) \\
&=&s^{2}w_{h}\psi \mathrm{curv}^{S^{4}}\left( y^{2,0},\eta _{u}^{2,0}\right)
\left\langle y^{2,0},\zeta \right\rangle \left\langle \eta _{u,W}^{2,0},\eta
_{u}^{2,0}\right\rangle +O \\
R^{\mathrm{diff,big}}\left( \zeta ,y^{2,0},W,\eta _{u}^{2,0}\right) &=&0 \\
\left\langle R^{\mathrm{diff,big}}\left( W,y^{2,0}\right) y^{2,0},\eta
_{u}^{2,0}\right\rangle &=&s^{2}w_{h}\psi \mathrm{curv}^{S^{4}}\left(
y^{2,0},\eta _{u,W}^{2,0}\right) \left\langle \eta _{u,W}^{2,0},\eta
_{u}^{2,0}\right\rangle \\
\left\langle R^{\mathrm{diff}}\left( \eta _{u}^{2,0},\zeta \right)
y^{2,0},\eta _{u}^{2,0}\right\rangle &=&\left\langle R^{\mathrm{diff}}\left(
W,\eta _{u}^{2,0},\eta _{u}^{2,0}\right) ,y^{2,0}\right\rangle =\left\langle
R^{\mathrm{diff}}\left( y^{2,0},\zeta ,\eta _{u}^{2,0}\right)
,y^{2,0}\right\rangle =0
\end{eqnarray*}
\end{proposition}

\begin{proof}
Each of the curvatures involves at most one vector that is not horizontal,
so the $A$--tensor contribution from the $s$--perturbation is $0.$ For the
first two curvatures, the $S^{4}$ term gives us 
\begin{eqnarray*}
\left\langle R^{\mathrm{s}}\left( W,\zeta \right) y^{2,0},\eta
_{u}^{2,0}\right\rangle &=&\left\langle R^{\mathrm{s}}\left(
W,y^{2,0}\right) \zeta ,\eta _{u}^{2,0}\right\rangle \\
&=&s^{2}\left\langle W,\eta _{u}^{2,0}\right\rangle \mathrm{curv}%
^{S^{4}}\left( y^{2,0},\eta _{u}^{2,0}\right) \left\langle y^{2,0},\zeta
\right\rangle \\
&=&s^{2}w_{h}\psi \mathrm{curv}^{S^{4}}\left( y^{2,0},\eta _{u}^{2,0}\right)
\left\langle y^{2,0},\zeta \right\rangle \left\langle \eta _{u,W}^{2,0},\eta
_{u}^{2,0}\right\rangle .
\end{eqnarray*}%
For the third curvature $R^{S^{4}}\left( \zeta ,y^{2,0},W,\eta
_{u}^{2,0}\right) =0.$ So $R^{s}\left( \zeta ,y^{2,0},W,\eta
_{u}^{2,0}\right) =0.$ Similarly%
\begin{eqnarray*}
\left\langle R^{\mathrm{s}}\left( W,y^{2,0}\right) y^{2,0},\eta
_{u}^{2,0}\right\rangle &=&\left\langle R^{\nu ,re,l}\left( W,y^{2,0}\right)
y^{2,0},\eta _{u}^{2,0}\right\rangle +s^{2}w_{h}\psi \mathrm{curv}%
^{S^{4}}\left( y^{2,0},\eta _{u}^{2,0}\right) \left\langle \eta
_{u,W}^{2,0},\eta _{u}^{2,0}\right\rangle ,\text{ and } \\
\left\langle R^{\mathrm{s}}\left( \eta _{u}^{2,0},\zeta \right) y^{2,0},\eta
_{u}^{2,0}\right\rangle &=&s^{2}\mathrm{curv}^{S^{4}}\left( y^{2,0},\eta
_{u}^{2,0}\right) \left\langle y^{2,0},\zeta \right\rangle =O,\text{ and } \\
\left\langle R^{\mathrm{s}}\left( W,\eta _{u}^{2,0},\eta _{u}^{2,0}\right)
,y^{2,0}\right\rangle &=&\left\langle R^{\mathrm{s}}\left( y^{2,0},\zeta
,\eta _{u}^{2,0}\right) ,y^{2,0}\right\rangle =0
\end{eqnarray*}%
Combining these computations with our partial conformal change and Hessian
formulas yields the result.
\end{proof}

\subsection{$z\in \mathrm{span}\left\{ \protect\eta _{1}^{2,0},\protect\eta %
_{2}^{2,0}\right\} ,$ $V\in V_{1}\oplus V_{2}$}

\begin{proposition}
For $V\in V_{1}\oplus V_{2}$ with $V_{2}$--component perpendicular to the $%
\gamma $--part of $W$ and normalized so that $\left\vert V\right\vert
_{h_{2}}=O\left( \frac{1}{\nu }\right) $ 
\begin{equation*}
\left\vert \left\langle R^{s}\left( W,\zeta \right) \eta
_{u}^{2,0},V\right\rangle \right\vert \leq s^{2}v_{h}\left\vert \mathrm{grad}%
\psi \right\vert \frac{\psi }{\nu }\sqrt{\mathrm{curv}^{\nu ,l}\left( \eta
_{u}^{2,0},W_{\alpha }\right) }+O
\end{equation*}%
\begin{equation*}
\left\vert \left\langle R^{s}\left( W,\eta _{u}^{2,0}\right) \zeta
,V\right\rangle \right\vert \leq s^{2}w_{h}D_{\zeta }\left[ \psi \right]
O\left( \frac{\psi }{\nu }\sqrt{\mathrm{curv}^{\nu ,l}\left( \eta
_{u}^{2,0},W\right) }\right) +O
\end{equation*}
\end{proposition}

\begin{proof}
\begin{eqnarray*}
R^{s}\left( W,\zeta \right) \eta _{u}^{2,0} &=&R^{s}\left( W^{\mathcal{V}%
},\zeta \right) \eta _{u}^{2,0}+R^{s}\left( W^{\mathcal{H}},\zeta \right)
\eta _{u}^{2,0} \\
&=&\left( 1-s^{2}\right) R^{\nu ,re,l}(W^{\mathcal{V}},\zeta )\eta
_{u}^{2,0}+s^{2}\left( R^{\nu ,re,l}(W^{\mathcal{V}},\zeta )\eta
_{u}^{2,0}\right) ^{\mathcal{V}}+s^{2}A_{\zeta }A_{\eta _{u}^{2,0}}W^{%
\mathcal{V}} \\
&&+\left( 1-s^{2}\right) R^{\nu ,re,l}(W^{\mathcal{H}},\zeta )\eta
_{u}^{2,0}+s^{2}\left( R^{\nu ,re,l}(W^{\mathcal{H}},\zeta )\eta
_{u}^{2,0}\right) ^{\mathcal{V}}+s^{2}R^{S^{4}}(W^{\mathcal{H}},\zeta )\eta
_{u}^{2,0} \\
&=&\left( 1-s^{2}\right) \left( R^{\nu ,re,l}\left( W,\zeta \right) \eta
_{u}^{2,0}\right) ^{\mathcal{H}}+\left( R^{\nu ,re,l}\left( W,\zeta \right)
\eta _{u}^{2,0}\right) ^{\mathcal{V}} \\
&&+s^{2}A_{\zeta }A_{\eta _{u}^{2,0}}W^{\mathcal{V}}+s^{2}R^{S^{4}}(W^{%
\mathcal{H}},\zeta )\eta _{u}^{2,0}
\end{eqnarray*}

As before we have 
\begin{equation*}
\left\langle R^{S^{4}}(W^{\mathcal{H}},\zeta )\eta _{u}^{2,0},V\right\rangle
=0.
\end{equation*}%
\begin{eqnarray*}
A_{\eta _{u}^{2,0}}W^{\mathcal{V}} &=&-\mathrm{II}\left( \eta _{u}^{2,0},W^{%
\mathcal{H}}\right) +\frac{1}{\left\vert \cos 2t\eta ^{2,0}\right\vert }%
\left( \nabla _{\left( \eta ,\eta \right) }^{\nu ,re,l}W\right) ^{\mathcal{H}%
}+4\frac{\psi ^{2}}{\nu ^{3}}\left\vert W_{\alpha }\right\vert
_{h_{2}}\left( \eta _{u}^{2,0}\right) ^{\perp } \\
&=&w_{h}\mathrm{grad}\psi \left\langle \eta _{u}^{2,0},\frac{W^{\mathcal{H}}%
}{\left\vert W^{\mathcal{H}}\right\vert }\right\rangle +\frac{1}{\left\vert
\cos 2t\eta ^{2,0}\right\vert }\left( \nabla _{\left( \eta ,\eta \right)
}^{\nu ,re,l}W\right) ^{\mathcal{H}}+4\frac{\psi ^{2}}{\nu ^{3}}\left\vert
W_{\alpha }\right\vert _{h_{2}}\left( \eta _{u}^{2,0}\right) ^{\perp }
\end{eqnarray*}%
where $\left( \eta _{u}^{2,0}\right) ^{\perp }$ is the spherical combination
of \textrm{span}$\left\{ \eta _{u,1}^{2,0},\eta _{u,2}^{2,0}\right\} $
that's perpendicular to $\eta _{u}^{2,0}.$

To estimate the last term note 
\begin{eqnarray*}
\mathrm{curv}^{\nu ,re,l}\left( W,\eta _{u}^{2,0}\right) &\geq &4\frac{\psi
^{2}}{\nu ^{4}}\left\vert W_{\alpha }\right\vert _{h_{2}}^{2}, \\
2\frac{\psi }{\nu }\sqrt{\mathrm{curv}^{\nu ,re,l}\left( W,\eta
_{u}^{2,0}\right) } &\geq &\frac{\psi }{\nu }4\frac{\psi }{\nu ^{2}}%
\left\vert W_{\alpha }\right\vert _{h_{2}} \\
&=&4\frac{\psi ^{2}}{\nu ^{3}}\left\vert W_{\alpha }\right\vert _{h_{2}}.
\end{eqnarray*}%
We estimate the middle term as 
\begin{equation*}
\frac{1}{\left\vert \cos 2t\eta ^{2,0}\right\vert }\left\vert \left( \nabla
_{\left( \eta ,\eta \right) }^{\nu ,re,l}W\right) ^{\mathcal{H}}\right\vert =%
\frac{1}{\left\vert \cos 2t\eta ^{2,0}\right\vert }O\left( 1+\frac{t}{l^{2}}+%
\frac{t^{2}}{l^{4}}\right) .
\end{equation*}%
The $\frac{t}{l^{2}}$ and $\frac{t^{2}}{l^{4}}$ terms come from
differentiating the $S^{3}$--factor of $\hat{W}$ in $\left( S^{3}\right)
^{2}\times Sp\left( 2\right) .$ The $\frac{t}{l^{2}}$ comes from the
derivative in the $Sp\left( 2\right) $ direction. The factor of $t,$ is
present because we are taking the horizontal part of the answer, and the
entire horizontal space is perpendicular to the orbits of the $\left(
U,D\right) $--action when $\left( \sin 2t,\sin 2\theta \right) =\left(
0,0\right) .$ The $\frac{t^{2}}{l^{4}}$--factor comes from taking the
derivative in the $S^{3}$--direction. The extra factor of $t$ comes the fact
that $\left( \eta ,\eta \right) $ is perpendicular to the orbits of the $%
\left( U,D\right) $--action when $\left( \sin 2t,\sin 2\theta \right)
=\left( 0,0\right) .$

On the other hand,%
\begin{equation*}
A_{\zeta }V^{\mathcal{V}}=-v_{h}\frac{D_{\zeta }\left[ \psi \right] }{\psi }%
k_{V}+\left( \nabla _{\zeta }^{\nu ,re,l}V\right) ^{\mathcal{H}}.
\end{equation*}%
and if $V$ has the usual normalization, then%
\begin{equation*}
\left\vert \left( \nabla _{\zeta }^{\nu ,re,l}V\right) ^{\mathcal{H}%
}\right\vert =O\left( 1+\frac{t}{l^{2}}\right) .
\end{equation*}%
So%
\begin{eqnarray*}
&&\left\vert s^{2}\left\langle A_{\zeta }A_{\eta _{u}^{2,0}}W^{\mathcal{V}%
},V^{\mathcal{V}}\right\rangle \right\vert \\
&\leq &2s^{2}v_{h}\left\vert \mathrm{grad}\psi \right\vert \frac{\psi }{\nu }%
\sqrt{\mathrm{curv}^{\nu ,re,l}\left( W,\eta _{u}^{2,0}\right) } \\
&&+s^{2}v_{h}\frac{D_{\zeta }\left[ \psi \right] }{\left\vert \cos 2t\eta
^{2,0}\right\vert ^{2}}O\left( 1+\frac{t}{l^{2}}+\frac{t^{2}}{l^{4}}\right)
\\
&&+s^{2}w_{h}\mathrm{grad}\psi O\left( 1+\frac{t}{l^{2}}\right) +\frac{s^{2}%
}{\left\vert \cos 2t\eta ^{2,0}\right\vert }O\left( 1+\frac{t}{l^{2}}\right)
O\left( 1+\frac{t}{l^{2}}+\frac{t^{2}}{l^{4}}\right) \\
&&+2\frac{s^{2}}{\left\vert \cos 2t\eta ^{2,0}\right\vert }O\left( 1+\frac{t%
}{l^{2}}\right) \frac{\psi }{\nu }\sqrt{\mathrm{curv}^{\nu ,re,l}\left(
W,\eta _{u}^{2,0}\right) } \\
&\leq &s^{2}v_{h}\left\vert \mathrm{grad}\psi \right\vert \frac{\psi }{\nu }%
\sqrt{\mathrm{curv}^{\nu ,re,l}\left( W,\eta _{u}^{2,0}\right) }+\varkappa
\left( s\right) +O
\end{eqnarray*}

It follows that 
\begin{equation*}
\left\vert \left\langle R^{s}\left( W,\zeta \right) \eta
_{u}^{2,0},V\right\rangle \right\vert \leq s^{2}v_{h}\left\vert \mathrm{grad}%
\psi \right\vert \frac{\psi }{\nu }\sqrt{\mathrm{curv}^{\nu ,re,l}\left(
\eta _{u}^{2,0},W_{\alpha }\right) }+O
\end{equation*}

Similarly, since 
\begin{equation*}
A_{\zeta }W^{\mathcal{V}}=-w_{h}\frac{D_{\zeta }\left[ \psi \right] }{\psi }%
k_{\gamma }
\end{equation*}%
and%
\begin{eqnarray*}
A_{\eta _{u}^{2,0}}V^{\mathcal{V}} &=&\frac{1}{\left\vert \cos 2t\eta
^{2,0}\right\vert }\left( \nabla _{\left( \eta ,\eta \right) }^{\nu
,re,l}V\right) ^{\mathcal{H}}-\mathrm{II}\left( \eta _{u}^{2,0},V^{\mathcal{H%
}}\right) +4\frac{\psi ^{2}}{\nu ^{3}}\left\vert V_{\alpha }\right\vert
_{h_{2}}\left( \eta _{u}^{2,0}\right) ^{\perp } \\
&=&\frac{1}{\left\vert \cos 2t\eta ^{2,0}\right\vert }O\left( 1+\frac{t}{%
l^{2}}+\frac{t^{2}}{l^{4}}\right) +v_{h}\mathrm{grad}\psi \left\langle \eta
_{u}^{2,0},\frac{V^{\mathcal{H}}}{\left\vert V^{\mathcal{H}}\right\vert }%
\right\rangle +4\frac{\psi ^{2}}{\nu ^{3}}\left\vert V_{\alpha }\right\vert
_{h_{2}}\left( \eta _{u}^{2,0}\right) ^{\perp }
\end{eqnarray*}%
where $\left( \eta _{u}^{2,0}\right) ^{\perp }$ is the spherical combination
of \textrm{span}$\left\{ \eta _{u,1}^{2,0},\eta _{u,2}^{2,0}\right\} $
that's perpendicular to $\eta _{u}^{2,0}.$

It follows that 
\begin{eqnarray*}
&&\left\vert s^{2}\left\langle A_{\eta _{u}^{2,0}}A_{\zeta }W^{\mathcal{V}%
},V\right\rangle \right\vert \\
&=&\left\vert s^{2}w_{h}D_{\zeta }\left[ \psi \right] \left( \frac{O\left( 1+%
\frac{t}{l^{2}}+\frac{t^{2}}{l^{4}}\right) }{\left\vert \cos 2t\eta
^{2,0}\right\vert }+4\frac{\psi ^{2}}{\nu ^{3}}\left\vert V_{\alpha
}\right\vert _{h_{2}}\left\langle \eta _{u,W}^{2,0},\left( \eta
_{u}^{2,0}\right) ^{\perp }\right\rangle \right) \right\vert
\end{eqnarray*}

The first term is too small to matter. It is natural to control the second
term in terms of $\mathrm{curv}^{\nu ,re,l}\left( V^{h_{2},\alpha },\eta
_{u}^{2,0}\right) ,$ however, since this a mixed quadratic term, it is much
nicer subsequently if we can control it in terms of $\mathrm{curv}^{\nu
,re,l}\left( W,\eta _{u}^{2,0}\right) .$ To do this we need to use our
normalization $\left\vert V\right\vert _{h_{2}}=O\left( \frac{1}{\nu }%
\right) .$

Since 
\begin{equation*}
\mathrm{curv}^{\nu ,re,l}\left( W,\eta _{u}^{2,0}\right) \geq 4\frac{\psi
^{2}}{\nu ^{6}}\left\langle \eta _{u,W}^{2,0},\left( \eta _{u}^{2,0}\right)
^{\perp }\right\rangle ,
\end{equation*}%
we have 
\begin{eqnarray*}
\frac{\psi }{\nu }\sqrt{\mathrm{curv}^{\nu ,re,l}\left( W,\eta
_{u}^{2,0}\right) } &\geq &\frac{\psi }{\nu }2\left( \frac{\psi }{\nu ^{3}}%
\left\langle \eta _{u,W}^{2,0},\left( \eta _{u}^{2,0}\right) ^{\perp
}\right\rangle ^{1/2}\right) \\
&=&O\left( \frac{\psi ^{2}}{\nu ^{3}}\left\vert V_{\alpha }\right\vert
_{h_{2}}\right) \left\langle \eta _{u,W}^{2,0},\left( \eta _{u}^{2,0}\right)
^{\perp }\right\rangle ^{1/2}
\end{eqnarray*}

It follows that 
\begin{equation*}
\left\vert \left\langle R^{s}\left( W,\eta _{u}^{2,0}\right) \zeta
,V\right\rangle \right\vert \leq s^{2}w_{h}D_{\zeta }\left[ \psi \right]
O\left( \frac{\psi }{\nu }\sqrt{\mathrm{curv}^{\nu ,re,l}\left( W,\eta
_{u}^{2,0}\right) }\right) +O
\end{equation*}
\end{proof}

\begin{corollary}
\begin{equation*}
\left\vert \left\langle R^{\mathrm{diff,big}}\left( W,\zeta \right) \eta
_{u}^{2,0},V\right\rangle \right\vert =0
\end{equation*}%
\begin{equation*}
\left\vert \left\langle R^{\mathrm{diff,big}}\left( W,\eta _{u}^{2,0}\right)
\zeta ,V\right\rangle \right\vert =0.
\end{equation*}
\end{corollary}

\begin{proposition}
For $V\in V_{1}\oplus V_{2}$ with the $V_{2}$ part of $V$ perpendicular to
the $\gamma $--part of $W$%
\begin{equation*}
\left\vert R^{\mathrm{diff,big}}\left( \zeta ,V,V,\eta _{u}^{2,0}\right)
\right\vert =0
\end{equation*}%
\begin{equation*}
\left\vert R^{\mathrm{diff,big}}\left( \eta _{u}^{2,0},W,V,\eta
_{u}^{2,0}\right) \right\vert \leq \sqrt{\mathrm{curv}^{\mathrm{diff,big}%
}\left( \eta _{u}^{2,0},V\right) }\sqrt{\mathrm{curv}^{\mathrm{diff,big}%
}\left( \eta _{u}^{2,0},W\right) }
\end{equation*}
\end{proposition}

\begin{proof}
We will use%
\begin{equation*}
A_{\zeta }V=-v_{h}\frac{D_{\zeta }\left( \psi \right) }{\psi }k_{\gamma
,V}+\left( \nabla _{\zeta }^{\nu ,re,l}V\right) ^{\mathcal{H}}
\end{equation*}%
and%
\begin{equation*}
A_{\eta _{u}^{2,0}}V=\frac{1}{\left\vert \cos 2t\eta ^{2,0}\right\vert }%
\left( \nabla _{\left( \eta ,\eta \right) }^{\nu ,re,l}V\right) ^{\mathcal{H}%
}-\mathrm{II}\left( \eta _{u}^{2,0},V^{\mathcal{H}}\right) +O\left( \frac{%
\psi ^{2}}{\nu ^{3}}\left\vert V_{\alpha }\right\vert \right) \left( \eta
_{u}^{2,0}\right) ^{\perp }.
\end{equation*}%
If $V$ has the usual normalization, and we estimate as in the previous proof
we get%
\begin{equation*}
\left( \nabla _{\zeta }^{\nu ,re,l}V\right) ^{\mathcal{H}}=O\left( 1+\frac{t%
}{l^{2}}\right) \text{ and}
\end{equation*}%
\begin{equation*}
\frac{1}{\left\vert \cos 2t\eta ^{2,0}\right\vert }\left\vert \left( \nabla
_{\left( \eta ,\eta \right) }^{\nu ,re,l}V\right) ^{\mathcal{H}}\right\vert =%
\frac{1}{\left\vert \cos 2t\eta ^{2,0}\right\vert }O\left( 1+\frac{t}{l^{2}}+%
\frac{t^{2}}{l^{4}}\right) .
\end{equation*}%
So%
\begin{eqnarray*}
&&\left\vert s^{2}\left\langle A_{\zeta }V_{\gamma },A_{\eta
_{u}^{2,0}}V_{\gamma }\right\rangle \right\vert \\
&\leq &s^{2}v_{h}D_{\zeta }\left( \psi \right) \frac{O\left( 1+\frac{t}{l^{2}%
}+\frac{t^{2}}{l^{4}}\right) }{\left\vert \cos 2t\eta ^{2,0}\right\vert }%
+s^{2}v_{h}D_{\zeta }\left( \psi \right) O\left( \frac{\psi ^{2}}{\nu ^{3}}%
\left\vert V_{\alpha }\right\vert \right) \left\langle \eta
_{u,V}^{2,0},\left( \eta _{u}^{2,0}\right) ^{\perp }\right\rangle \\
&&+s^{2}v_{h}\left\vert \mathrm{grad}\psi \right\vert O\left( 1+\frac{t}{%
l^{2}}\right) +s^{2}O\left( \frac{\psi ^{2}}{\nu ^{3}}\left\vert V_{\alpha
}\right\vert \right) O\left( 1+\frac{t}{l^{2}\left\vert \cos 2t\eta
^{2,0}\right\vert }\right) \\
&&+\frac{s^{2}}{\left\vert \cos 2t\eta ^{2,0}\right\vert }O\left( 1+\frac{t}{%
l^{2}}+\frac{t^{2}}{l^{4}}\right) O\left( 1+\frac{t}{l^{2}}\right)
\end{eqnarray*}%
where the extra factor of $\frac{1}{\left\vert \cos 2t\eta ^{2,0}\right\vert 
}$ in the $\frac{t}{l^{2}\left\vert \cos 2t\eta ^{2,0}\right\vert }$ part of
the fourth term comes from the fact that we would be taking the component of 
$\nabla _{\zeta }^{\nu ,re,l}V$ in $\mathrm{span}\left\{ \eta
_{1,u}^{2,0},\eta _{2,u}^{2,0}\right\} .$ It follows that 
\begin{equation*}
\left\vert s^{2}\left\langle A_{\zeta }V_{\gamma },A_{\eta
_{u}^{2,0}}V_{\gamma }\right\rangle \right\vert \leq s^{2}v_{h}D_{\zeta
}\left( \psi \right) O\left( \frac{\psi ^{2}}{\nu ^{3}}\left\vert V_{\alpha
}\right\vert \right) \left\langle \eta _{u,V}^{2,0},\left( \eta
_{u}^{2,0}\right) ^{\perp }\right\rangle +O.
\end{equation*}

Since $V$ has the usual normalization, 
\begin{equation*}
\mathrm{curv}^{\nu ,re,l}\left( \eta _{u}^{2,0},V\right) \geq \frac{\psi ^{2}%
}{\nu ^{4}}\left\vert V_{\alpha }\right\vert _{h_{2}}^{2}
\end{equation*}%
So%
\begin{eqnarray*}
s^{2}v_{h}\frac{\psi }{\nu }\sqrt{\mathrm{curv}^{\nu ,re,l}\left( \eta
_{u}^{2,0},V\right) } &\geq &\left( s^{2}v_{h}\frac{\psi }{\nu }\right) 
\frac{\psi }{\nu ^{2}}\left\vert V_{\alpha }\right\vert _{h_{2}} \\
&=&s^{2}v_{h}\frac{\psi ^{2}}{\nu ^{3}}\left\vert V_{\alpha }\right\vert \\
&\geq &\left\vert s^{2}\left\langle A_{\zeta }V_{\gamma },A_{\eta
_{u}^{2,0}}V_{\gamma }\right\rangle \right\vert
\end{eqnarray*}%
So 
\begin{equation*}
\left\vert R^{\mathrm{diff,}s}\left( \zeta ,V,V,\eta _{u}^{2,0}\right)
\right\vert \leq \chi \left( s\right) \sqrt{\mathrm{curv}^{\nu ,re,l}\left(
\eta _{u}^{2,0},V\right) }+O
\end{equation*}%
where 
\begin{equation*}
R^{\mathrm{diff,}s}=R^{s}-R^{\nu ,re,l}.
\end{equation*}%
The partial conformal change does not add anything nearly this large so we
have proven the first statement.

To find $R^{\mathrm{diff,}s}\left( \eta _{u}^{2,0},W,V,\eta
_{u}^{2,0}\right) $ we use

\begin{eqnarray*}
A_{\eta _{u}^{2,0}}W &=&\frac{1}{\left\vert \cos 2t\eta ^{2,0}\right\vert }%
\left( \nabla _{\left( \eta ,\eta \right) }^{\nu ,re,l}W\right) ^{\mathcal{H}%
}-\mathrm{II}\left( \eta _{u}^{2,0},W^{\mathcal{H}}\right) +O\left( \frac{%
\psi ^{2}}{\nu ^{3}}\left\vert W_{\alpha }\right\vert _{h_{2}}\right) \left(
\eta _{u}^{2,0}\right) ^{\perp } \\
&=&\frac{1}{\left\vert \cos 2t\eta ^{2,0}\right\vert }\left( \nabla _{\left(
\eta ,\eta \right) }^{\nu ,re,l}W\right) ^{\mathcal{H}}+\frac{\mathrm{grad}%
\psi }{\psi }\left\langle W,\eta _{u}^{2,0}\right\rangle +O\left( \frac{\psi
^{2}}{\nu ^{3}}\left\vert W_{\alpha }\right\vert _{h_{2}}\right) \left( \eta
_{u}^{2,0}\right) ^{\perp }
\end{eqnarray*}%
and%
\begin{eqnarray*}
A_{\eta _{u}^{2,0}}V &=&\frac{1}{\left\vert \cos 2t\eta ^{2,0}\right\vert }%
\left( \nabla _{\left( \eta ,\eta \right) }^{\nu ,re,l}V\right) ^{\mathcal{H}%
}-\mathrm{II}\left( \eta _{u}^{2,0},V^{\mathcal{H}}\right) +O\left( \frac{%
\psi ^{2}}{\nu ^{3}}\left\vert V_{\alpha }\right\vert _{h_{2}}\right) \left(
\eta _{u}^{2,0}\right) ^{\perp } \\
&=&\frac{1}{\left\vert \cos 2t\eta ^{2,0}\right\vert }\left( \nabla _{\left(
\eta ,\eta \right) }^{\nu ,re,l}V\right) ^{\mathcal{H}}+\frac{\mathrm{grad}%
\psi }{\psi }\left\langle V,\eta _{u}^{2,0}\right\rangle +O\left( \frac{\psi
^{2}}{\nu ^{3}}\left\vert V_{\alpha }\right\vert _{h_{2}}\right) \left( \eta
_{u}^{2,0}\right) ^{\perp }.
\end{eqnarray*}%
Letting $\gamma _{1},\gamma _{2}$, and $\gamma _{3}$ be the unit $\gamma $%
--quaternions corresponding to $W,$ $V,$ and $\eta _{u}^{2,0},$we have 
\begin{eqnarray*}
\left\vert s^{2}\left\langle A_{\eta _{u}^{2,0}}W,A_{\eta
_{u}^{2,0}}V\right\rangle \right\vert &=&s^{2}\left( w_{h}v_{h}\right)
\left\vert \mathrm{grad}\psi \right\vert ^{2}\left\langle \gamma _{1},\gamma
_{3}\right\rangle \left\langle \gamma _{2},\gamma _{3}\right\rangle \\
&&+s^{2}\left( w_{h}+v_{h}\right) \mathrm{grad}\psi \frac{O\left( 1+\frac{t}{%
l^{2}}+\frac{t^{2}}{l^{4}}\right) }{\left\vert \cos 2t\eta ^{2,0}\right\vert 
} \\
&&+s^{2}O\left( \frac{\psi ^{2}}{\nu ^{3}}\left( \left\vert W_{\alpha
}\right\vert _{h_{2}}+\left\vert V_{\alpha }\right\vert _{h_{2}}\right)
\right) \frac{O\left( 1+\frac{\psi }{l^{2}}+\frac{\psi ^{2}}{l^{4}}\right) }{%
\left\vert \cos 2t\eta ^{2,0}\right\vert } \\
&&+s^{2}O\left( \frac{\psi ^{2}}{\nu ^{3}}\left\vert W_{\alpha }\right\vert 
\frac{\psi ^{2}}{\nu ^{3}}\left\vert V_{\alpha }\right\vert \right)
\end{eqnarray*}

We again give $V$ the usual normalization. So 
\begin{eqnarray*}
\mathrm{curv}^{\nu ,re,l}\left( \eta _{u}^{2,0},V\right) &\geq &\frac{\psi
^{2}}{\nu ^{4}}\left\vert V_{\alpha }\right\vert ^{2} \\
\frac{\psi }{\nu }\sqrt{\mathrm{curv}^{\nu ,re,l}\left( \eta
_{u}^{2,0},V\right) } &\geq &\frac{\psi }{\nu }\left( \frac{\psi }{\nu ^{2}}%
\left\vert V_{\alpha }\right\vert \right) \\
&=&\frac{\psi ^{2}}{\nu ^{3}}\left\vert V_{\alpha }\right\vert
\end{eqnarray*}%
So the third term is bounded by 
\begin{equation*}
s^{2}O\left( \frac{\psi ^{2}}{\nu ^{3}}\left( \left\vert W_{\alpha
}\right\vert +\left\vert V_{\alpha }\right\vert \right) \right) \frac{%
O\left( 1+\frac{\psi }{l^{2}}+\frac{\psi ^{2}}{l^{4}}\right) }{\left\vert
\cos 2t\eta ^{2,0}\right\vert }\leq O\left( s^{2}\frac{\psi }{\nu }\right)
\left( \sqrt{\mathrm{curv}^{\nu ,re,l}\left( \eta _{u}^{2,0},V\right) }+%
\sqrt{\mathrm{curv}^{\nu ,re,l}\left( \eta _{u}^{2,0},W\right) }\right)
\end{equation*}

and the last term is bounded by%
\begin{equation*}
s^{2}O\left( \frac{\psi ^{2}}{\nu ^{3}}\left\vert W_{\alpha }\right\vert 
\frac{\psi ^{2}}{\nu ^{3}}\left\vert V_{\alpha }\right\vert \right) \leq
s^{2}\sqrt{\mathrm{curv}^{\nu ,re,l}\left( \eta _{u}^{2,0},V\right) }\sqrt{%
\mathrm{curv}^{\nu ,re,l}\left( \eta _{u}^{2,0},W\right) }
\end{equation*}%
Combining inequalities we have \addtocounter{algorithm}{1} 
\begin{eqnarray}
\left\vert s^{2}\left\langle A_{\eta _{u}^{2,0}}W,A_{\eta
_{u}^{2,0}}V\right\rangle \right\vert &\leq &s^{2}\left( w_{h}v_{h}\right)
\left\vert \mathrm{grad}\psi \right\vert ^{2}\left\langle \gamma _{1},\gamma
_{3}\right\rangle \left\langle \gamma _{2},\gamma _{3}\right\rangle  \notag
\\
&&+O\left( s^{2}\right) \left( \sqrt{\mathrm{curv}^{\nu ,re,l}\left( \eta
_{u}^{2,0},V\right) }+\sqrt{\mathrm{curv}^{\nu ,re,l}\left( \eta
_{u}^{2,0},W\right) }\right)  \notag \\
&&+s^{2}\sqrt{\mathrm{curv}^{\nu ,re,l}\left( \eta _{u}^{2,0},V\right) }%
\sqrt{\mathrm{curv}^{\nu ,re,l}\left( \eta _{u}^{2,0},W\right) }+O
\label{A_eta W, A_eta V}
\end{eqnarray}

To dominate the first term, $s^{2}\left( w_{h}v_{h}\right) \left\vert 
\mathrm{grad}\psi \right\vert ^{2}\left\langle \gamma _{1},\gamma
_{3}\right\rangle \left\langle \gamma _{2},\gamma _{3}\right\rangle ,$ we
use the Proposition \ref{diff big sec} to get 
\begin{eqnarray*}
R^{\mathrm{diff,big}}\left( \eta _{u}^{2,0},W,W,\eta _{u}^{2,0}\right) &\geq
&e^{2f}w_{h}^{2}s^{2}\left\vert \mathrm{grad\,}\psi \right\vert ^{2}\left(
1-\left\langle \gamma _{1},\gamma _{3}\right\rangle ^{2}\right) \\
&=&e^{2f}w_{h}^{2}s^{2}\left\vert \mathrm{grad\,}\psi \right\vert ^{2}\left(
\left\langle \gamma _{2},\gamma _{3}\right\rangle ^{2}\right)
\end{eqnarray*}

and 
\begin{equation*}
R^{\mathrm{diff,big}}\left( \eta _{u}^{2,0},V,V,\eta _{u}^{2,0}\right) \geq
e^{2f}v_{h}^{2}s^{2}\left\vert \mathrm{grad\,}\psi \right\vert ^{2}\left(
\left\langle \gamma _{1},\gamma _{3}\right\rangle ^{2}\right)
\end{equation*}

Combining the previous two displays gives us 
\begin{equation*}
e^{2f}s^{2}\left( w_{h}v_{h}\right) \left\vert \mathrm{grad}\psi \right\vert
^{2}\left\langle \gamma _{1},\gamma _{3}\right\rangle \left\langle \gamma
_{2},\gamma _{3}\right\rangle \leq \sqrt{\mathrm{curv}^{\mathrm{diff,big}%
}\left( \eta _{u,3}^{2,0},V\right) }\sqrt{\mathrm{curv}^{\mathrm{diff,big}%
}\left( \eta _{u,3}^{2,0},W\right) }
\end{equation*}%
Plugging this into \ref{A_eta W, A_eta V} gives us 
\begin{eqnarray*}
&&e^{2f}\left\vert s^{2}\left\langle A_{\eta _{u}^{2,0}}W,A_{\eta
_{u,3}^{2,0}}V\right\rangle \right\vert \\
&\leq &\sqrt{\mathrm{curv}^{\mathrm{diff,big}}\left( \eta
_{u,3}^{2,0},V\right) }\sqrt{\mathrm{curv}^{\mathrm{diff,big}}\left( \eta
_{u,3}^{2,0},W\right) }+ \\
&&+O\left( s^{2}\right) \left( \sqrt{\mathrm{curv}\left( \eta
_{u}^{2,0},V\right) }+\sqrt{\mathrm{curv}\left( \eta _{u}^{2,0},w\right) }%
\right) \\
&&+s^{2}\sqrt{\mathrm{curv}\left( \eta _{u}^{2,0},V\right) }\sqrt{\mathrm{%
curv}\left( \eta _{u}^{2,0},W\right) }+O
\end{eqnarray*}%
Arguing as before this gives us 
\begin{equation*}
\left\vert R^{\mathrm{diff,big}}\left( \eta _{u}^{2,0},W,V,\eta
_{u}^{2,0}\right) \right\vert \leq \sqrt{\mathrm{curv}^{\mathrm{diff,big}%
}\left( \eta _{u,3}^{2,0},V\right) }\sqrt{\mathrm{curv}^{\mathrm{diff,big}%
}\left( \eta _{u,3}^{2,0},W\right) }
\end{equation*}
\end{proof}

\section{Appendix}

This appendix contains the calculations we omitted in section 7.

\begin{eqnarray*}
\left\vert \left( \cos 2t\right) \eta ^{2,0}\right\vert _{\nu ,l}^{2}
&=&\cos ^{2}2t+\frac{\sin ^{2}2t}{\nu ^{2}}+2l^{-2}\left( -\frac{1}{2}\sin
2t\cos 2t+\sin 2t\left( -\cos ^{2}\theta \sin ^{2}t+\sin ^{2}\theta \cos
^{2}t\right) \right) ^{2} \\
&&+\frac{2}{l^{2}}\frac{\sin ^{2}2\theta }{4}\left( \cos ^{2}2t+\sin
^{2}2t\right) ^{2} \\
&=&\cos ^{2}2t+\frac{\sin ^{2}2t}{\nu ^{2}}+2l^{-2}\left( -\frac{1}{2}\sin
2t\cos 2t+\sin 2t\left( \sin ^{2}\theta -\sin ^{2}t\right) \right) ^{2} \\
&&+\frac{\sin ^{2}2\theta }{2l^{2}} \\
&=&\cos ^{2}2t+\frac{\sin ^{2}2t}{\nu ^{2}}+2l^{-2}\sin ^{2}2t\left( -\frac{1%
}{2}\cos 2t+\left( \sin ^{2}\theta -\sin ^{2}t\right) \right) ^{2}+\frac{1}{%
l^{2}}\frac{\sin ^{2}2\theta }{2} \\
&=&\cos ^{2}2t+\frac{\sin ^{2}2t}{\nu ^{2}}+2l^{-2}\sin ^{2}2t\left( -\frac{1%
}{2}+\sin ^{2}\theta \right) ^{2}+\frac{1}{l^{2}}\frac{\sin ^{2}2\theta }{2}
\\
&=&\cos ^{2}2t+\frac{\sin ^{2}2t}{\nu ^{2}}+2l^{-2}\sin ^{2}2t\left( -\frac{1%
}{2}\cos 2\theta \right) ^{2}+\frac{1}{l^{2}}\frac{\sin ^{2}2\theta }{2} \\
&=&\cos ^{2}2t+\frac{\sin ^{2}2t}{\nu ^{2}}+\frac{1}{2l^{2}}\left( \sin
^{2}2t\left( \cos ^{2}2\theta \right) +\sin ^{2}2\theta \right) \\
&=&\cos ^{2}2t+\frac{\sin ^{2}2t}{\nu ^{2}}+\frac{1}{2l^{2}}\left( 1-\cos
^{2}2t\cos ^{2}2\theta \right)
\end{eqnarray*}%
So we can now prove

\begin{proposition}
\begin{eqnarray*}
\frac{\partial }{\partial t}\psi _{\nu ,l} &=&\frac{\left( 1+\frac{1}{2l^{2}}%
\sin ^{2}2\theta \right) \cos 2t}{\left\vert \left( \cos 2t\right) \eta
^{2,0}\right\vert _{\nu ,l}^{3}} \\
&=&\frac{\left\vert x^{2,0}\right\vert _{\nu ,l}^{2}\cos 2t}{\left\vert
\left( \cos 2t\right) \eta ^{2,0}\right\vert _{\nu ,l}^{3}} \\
\frac{\partial }{\partial \theta }\psi _{\nu ,l} &=&-\frac{1}{4l^{2}}\frac{%
\sin 2t\cos ^{2}2t\sin 4\theta }{\left\vert \left( \cos 2t\right) \eta
^{2,0}\right\vert _{\nu ,l}^{3}}
\end{eqnarray*}
\end{proposition}

\begin{proof}
We first rearrange the terms in $\left\vert \left( \cos 2t\right) \eta
^{2,0}\right\vert _{\nu ,l}$ as follows 
\begin{eqnarray*}
\left\vert \left( \cos 2t\right) \eta ^{2,0}\right\vert _{\nu ,l}^{2}
&=&\cos ^{2}2t+\frac{\sin ^{2}2t}{\nu ^{2}}+\frac{1}{2l^{2}}\left( 1-\cos
^{2}2t\cos ^{2}2\theta \right) \\
&=&1-\sin ^{2}2t+\frac{\sin ^{2}2t}{\nu ^{2}}+\frac{1}{2l^{2}}-\frac{1}{%
2l^{2}}\cos ^{2}2\theta +\frac{1}{2l^{2}}\sin ^{2}2t\cos ^{2}2\theta \\
&=&1+\frac{1}{2l^{2}}\sin ^{2}2\theta +\frac{\sin ^{2}2t}{\nu ^{2}}-\sin
^{2}2t+\frac{1}{2l^{2}}\sin ^{2}2t\cos ^{2}2\theta \\
&=&1+\frac{1}{2l^{2}}\sin ^{2}2\theta +\frac{\sin ^{2}2t}{\nu ^{2}}-\sin
^{2}2t+\frac{1}{2l^{2}}\sin ^{2}2t-\frac{1}{2l^{2}}\sin ^{2}2t\sin
^{2}2\theta \\
&=&1+\frac{\sin ^{2}2\theta }{2l^{2}}+\left( \frac{1}{\nu ^{2}}+\frac{1}{%
2l^{2}}-\left( 1+\frac{\sin ^{2}2\theta }{2l^{2}}\right) \right) \sin ^{2}2t
\end{eqnarray*}%
Setting 
\begin{equation*}
\frac{1}{\nu _{l}^{2}}=\frac{1}{\nu ^{2}}+\frac{1}{2l^{2}},
\end{equation*}%
and using the fact that 
\begin{equation*}
\left\vert x^{2,0}\right\vert _{\nu ,l}^{2}=1+\frac{\sin ^{2}2\theta }{2l^{2}%
}
\end{equation*}%
we get%
\begin{eqnarray*}
\left\vert \left( \cos 2t\right) \eta ^{2,0}\right\vert _{\nu ,l}^{2}
&=&\left\vert x^{2,0}\right\vert _{\nu ,l}^{2}+\left( \frac{1}{\nu _{l}^{2}}%
-\left\vert x^{2,0}\right\vert _{\nu ,l}^{2}\right) \sin ^{2}2t \\
&=&\left\vert x^{2,0}\right\vert _{\nu ,l}^{2}\cos ^{2}2t+\frac{1}{\nu
_{l}^{2}}\sin ^{2}2t
\end{eqnarray*}%
This gives us%
\begin{equation*}
\frac{\partial }{\partial t}\left\vert \left( \cos 2t\right) \eta
^{2,0}\right\vert _{\nu ,l}^{2}=\left( \frac{1}{\nu _{l}^{2}}-\left\vert
x^{2,0}\right\vert _{\nu ,l}^{2}\right) 4\sin 2t\cos 2t
\end{equation*}%
and using%
\begin{eqnarray*}
\frac{\partial }{\partial \theta }\left\vert x^{2,0}\right\vert _{\nu
,l}^{2} &=&\frac{\partial }{\partial \theta }\left( 1+\frac{\sin ^{2}2\theta 
}{2l^{2}}\right) \\
&=&\frac{2\sin 2\theta \cos 2\theta }{l^{2}} \\
&=&\frac{\sin 4\theta }{l^{2}}
\end{eqnarray*}%
we get%
\begin{eqnarray*}
\frac{\partial }{\partial \theta }\left\vert \left( \cos 2t\right) \eta
^{2,0}\right\vert _{\nu ,l}^{2} &=&\frac{\partial }{\partial \theta }\left(
\left\vert x^{2,0}\right\vert _{\nu ,l}^{2}+\left( \frac{1}{\nu _{l}^{2}}%
-\left\vert x^{2,0}\right\vert _{\nu ,l}^{2}\right) \sin ^{2}2t\right) \\
&=&\frac{\sin 4\theta }{l^{2}}-\frac{\sin 4\theta }{l^{2}}\sin ^{2}2t \\
&=&\frac{\sin 4\theta \cos ^{2}2t}{l^{2}}.
\end{eqnarray*}%
Thus%
\begin{eqnarray*}
\frac{\partial }{\partial t}\psi _{\nu ,l} &=&\frac{\partial }{\partial t}%
\frac{1}{2}\frac{\sin 2t}{\left\vert \left( \cos 2t\right) \eta
^{2,0}\right\vert _{\nu ,l}} \\
&=&\frac{1}{2}\frac{2\cos 2t\left( \left\vert \left( \cos 2t\right) \eta
^{2,0}\right\vert _{\nu ,l}^{2}\right) -\frac{1}{2}\sin 2t\left( \frac{%
\partial }{\partial t}\left\vert \left( \cos 2t\right) \eta
^{2,0}\right\vert _{\nu ,l}^{2}\right) }{\left\vert \left( \cos 2t\right)
\eta ^{2,0}\right\vert _{\nu ,l}^{3}} \\
&=&\frac{1}{2}\frac{2\cos 2t\left( \left\vert x^{2,0}\right\vert _{\nu
,l}^{2}+\left( \frac{1}{\nu _{l}^{2}}-\left\vert x^{2,0}\right\vert _{\nu
,l}^{2}\right) \sin ^{2}2t\right) }{\left\vert \left( \cos 2t\right) \eta
^{2,0}\right\vert _{\nu ,l}^{3}} \\
&&-\frac{1}{2}\frac{\frac{1}{2}\sin 2t\left( \left( \frac{1}{\nu _{l}^{2}}%
-\left\vert x^{2,0}\right\vert _{\nu ,l}^{2}\right) 4\sin 2t\cos 2t\right) }{%
\left\vert \left( \cos 2t\right) \eta ^{2,0}\right\vert _{\nu ,l}^{3}} \\
&=&\frac{\left\vert x^{2,0}\right\vert _{\nu ,l}^{2}\cos 2t}{\left\vert
\left( \cos 2t\right) \eta ^{2,0}\right\vert _{\nu ,l}^{3}}
\end{eqnarray*}%
Similarly%
\begin{eqnarray*}
\frac{\partial }{\partial \theta }\psi _{\nu ,l} &=&\frac{\partial }{%
\partial \theta }\frac{1}{2}\frac{\sin 2t}{\left\vert \left( \cos 2t\right)
\eta ^{2,0}\right\vert _{\nu ,l}} \\
&=&-\frac{1}{2}\frac{\sin 2t\left( \frac{\partial }{\partial \theta }\left[
\left\vert \left( \cos 2t\right) \eta ^{2,0}\right\vert _{\nu ,l}^{2}\right]
^{1/2}\right) }{\left\vert \left( \cos 2t\right) \eta ^{2,0}\right\vert
_{\nu ,l}^{2}} \\
&=&-\frac{1}{4}\frac{\sin 2t\left[ \left\vert \left( \cos 2t\right) \eta
^{2,0}\right\vert _{\nu ,l}^{2}\right] ^{-\frac{1}{2}}\left( \frac{\partial 
}{\partial \theta }\left[ \left\vert \left( \cos 2t\right) \eta
^{2,0}\right\vert _{\nu ,l}^{2}\right] \right) }{\left\vert \left( \cos
2t\right) \eta ^{2,0}\right\vert _{\nu ,l}^{2}} \\
&=&-\frac{1}{4}\frac{\sin 2t\left( \frac{\partial }{\partial \theta }%
\left\vert \left( \cos 2t\right) \eta ^{2,0}\right\vert _{\nu ,l}^{2}\right) 
}{\left\vert \left( \cos 2t\right) \eta ^{2,0}\right\vert _{\nu ,l}^{3}}%
\text{ } \\
&=&-\frac{1}{4}\frac{\sin 2t\left( \frac{\sin 4\theta \cos ^{2}2t}{l^{2}}%
\right) }{\left\vert \left( \cos 2t\right) \eta ^{2,0}\right\vert _{\nu
,l}^{3}} \\
&=&-\frac{1}{4l^{2}}\frac{\sin 2t\cos ^{2}2t\sin 4\theta }{\left\vert \left(
\cos 2t\right) \eta ^{2,0}\right\vert _{\nu ,l}^{3}}.
\end{eqnarray*}
\end{proof}

\begin{proposition}
\label{seceond derivatives}%
\begin{eqnarray*}
\frac{\partial ^{2}}{\partial t^{2}}\psi _{\nu ,l} &=&-\left\vert
x^{2,0}\right\vert _{\nu ,l}^{2}\frac{\sin 2t}{\left\vert \left( \cos
2t\right) \eta ^{2,0}\right\vert _{\nu ,l}^{5}}\left( -4\left\vert
x^{2,0}\right\vert _{\nu ,l}^{2}\cos ^{2}2t+\frac{2}{\nu _{l}^{2}}+4\left( 
\frac{1}{\nu _{l}^{2}}\right) \cos ^{2}2t\right) \\
\frac{\partial }{\partial \theta }\frac{\partial }{\partial t}\psi _{\nu ,l}
&=&\frac{\cos 2t\sin 4\theta }{l^{2}\left\vert \left( \cos 2t\right) \eta
^{2,0}\right\vert _{\nu ,l}^{5}}\left( -\frac{1}{2}\left\vert
x^{2,0}\right\vert _{\nu ,l}^{2}\cos ^{2}2t+\frac{1}{\nu _{l}^{2}}\sin
^{2}2t\right)
\end{eqnarray*}%
\begin{equation*}
\frac{\partial ^{2}}{\partial \theta ^{2}}\psi _{\nu ,l}=-\frac{\sin 2t\cos
^{2}2t}{l^{2}}\frac{\cos 4\theta \left( \left\vert x^{2,0}\right\vert _{\nu
,l}^{2}\cos ^{2}2t+\frac{1}{\nu _{l}^{2}}\sin ^{2}2t\right) }{\left\vert
\left( \cos 2t\right) \eta ^{2,0}\right\vert _{\nu ,l}^{5}}+\frac{3}{2}\frac{%
\sin 2t\cos ^{4}2t}{4l^{4}}\frac{\sin ^{2}4\theta }{\left\vert \left( \cos
2t\right) \eta ^{2,0}\right\vert _{\nu ,l}^{5}}\text{ }
\end{equation*}
\end{proposition}

\begin{proof}
\begin{eqnarray*}
\frac{\partial ^{2}}{\partial t^{2}}\psi _{\nu ,l} &=&\left\vert
x^{2,0}\right\vert _{\nu ,l}^{2}\frac{\partial }{\partial t}\frac{\cos 2t}{%
\left\vert \left( \cos 2t\right) \eta ^{2,0}\right\vert _{\nu ,l}^{3}} \\
&=&\left\vert x^{2,0}\right\vert _{\nu ,l}^{2}\frac{-2\sin 2t\left(
\left\vert \left( \cos 2t\right) \eta ^{2,0}\right\vert _{\nu ,l}^{3}\right)
-\cos 2t\left[ \frac{\partial }{\partial t}\left( \left\vert \left( \cos
2t\right) \eta ^{2,0}\right\vert _{\nu ,l}^{2}\right) ^{3/2}\right] }{%
\left\vert \left( \cos 2t\right) \eta ^{2,0}\right\vert _{\nu ,l}^{6}} \\
&=&\left\vert x^{2,0}\right\vert _{\nu ,l}^{2}\frac{-2\sin 2t\left(
\left\vert \left( \cos 2t\right) \eta ^{2,0}\right\vert _{\nu ,l}^{3}\right)
-\frac{3}{2}\left( \left\vert \left( \cos 2t\right) \eta ^{2,0}\right\vert
^{2}\right) ^{1/2}\cos 2t\left( \frac{\partial }{\partial t}\left\vert
\left( \cos 2t\right) \eta ^{2,0}\right\vert _{\nu ,l}^{2}\right) }{%
\left\vert \left( \cos 2t\right) \eta ^{2,0}\right\vert _{\nu ,l}^{6}} \\
&=&\left\vert x^{2,0}\right\vert _{\nu ,l}^{2}\frac{-2\sin 2t\left(
\left\vert \left( \cos 2t\right) \eta ^{2,0}\right\vert _{\nu ,l}^{2}\right)
-\frac{3}{2}\cos 2t\left( \frac{\partial }{\partial t}\left\vert \left( \cos
2t\right) \eta ^{2,0}\right\vert _{\nu ,l}^{2}\right) }{\left\vert \left(
\cos 2t\right) \eta ^{2,0}\right\vert _{\nu ,l}^{5}}
\end{eqnarray*}%
\begin{eqnarray*}
&=&-\left\vert x^{2,0}\right\vert _{\nu ,l}^{2}\frac{2\sin 2t\left(
\left\vert x^{2,0}\right\vert _{\nu ,l}^{2}+\left( \frac{1}{\nu _{l}^{2}}%
-\left\vert x^{2,0}\right\vert _{\nu ,l}^{2}\right) \sin ^{2}2t\right) }{%
\left\vert \left( \cos 2t\right) \eta ^{2,0}\right\vert _{\nu ,l}^{5}} \\
&&-\left\vert x^{2,0}\right\vert _{\nu ,l}^{2}\frac{\frac{3}{2}\cos 2t\left(
\left( \frac{1}{\nu _{l}^{2}}-\left\vert x^{2,0}\right\vert _{\nu
,l}^{2}\right) 4\sin 2t\cos 2t\right) }{\left\vert \left( \cos 2t\right)
\eta ^{2,0}\right\vert _{\nu ,l}^{5}}
\end{eqnarray*}%
\begin{eqnarray*}
&=&-\left\vert x^{2,0}\right\vert _{\nu ,l}^{2}\frac{\sin 2t\left(
2\left\vert x^{2,0}\right\vert _{\nu ,l}^{2}+2\left( \frac{1}{\nu _{l}^{2}}%
-\left\vert x^{2,0}\right\vert _{\nu ,l}^{2}\right) \sin ^{2}2t\right) }{%
\left\vert \left( \cos 2t\right) \eta ^{2,0}\right\vert _{\nu ,l}^{5}} \\
&&-\left\vert x^{2,0}\right\vert _{\nu ,l}^{2}\sin 2t\frac{6\left( \frac{1}{%
\nu _{l}^{2}}-\left\vert x^{2,0}\right\vert _{\nu ,l}^{2}\right) \cos ^{2}2t%
}{\left\vert \left( \cos 2t\right) \eta ^{2,0}\right\vert _{\nu ,l}^{5}}
\end{eqnarray*}%
\begin{eqnarray*}
&=&-\left\vert x^{2,0}\right\vert _{\nu ,l}^{2}\frac{\sin 2t}{\left\vert
\left( \cos 2t\right) \eta ^{2,0}\right\vert _{\nu ,l}^{5}}\left(
2\left\vert x^{2,0}\right\vert _{\nu ,l}^{2}+2\left( \frac{1}{\nu _{l}^{2}}%
-\left\vert x^{2,0}\right\vert _{\nu ,l}^{2}\right) \sin ^{2}2t+6\left( 
\frac{1}{\nu _{l}^{2}}-\left\vert x^{2,0}\right\vert _{\nu ,l}^{2}\right)
\cos ^{2}2t\right) \\
&=&-\left\vert x^{2,0}\right\vert _{\nu ,l}^{2}\frac{\sin 2t}{\left\vert
\left( \cos 2t\right) \eta ^{2,0}\right\vert _{\nu ,l}^{5}}\left(
2\left\vert x^{2,0}\right\vert _{\nu ,l}^{2}\cos ^{2}2t+\frac{2\sin ^{2}2t}{%
\nu _{l}^{2}}+6\left( \frac{1}{\nu _{l}^{2}}-\left\vert x^{2,0}\right\vert
_{\nu ,l}^{2}\right) \cos ^{2}2t\right) \\
&=&-\left\vert x^{2,0}\right\vert _{\nu ,l}^{2}\frac{\sin 2t}{\left\vert
\left( \cos 2t\right) \eta ^{2,0}\right\vert _{\nu ,l}^{5}}\left(
-4\left\vert x^{2,0}\right\vert _{\nu ,l}^{2}\cos ^{2}2t+\frac{2\sin ^{2}2t}{%
\nu _{l}^{2}}+6\left( \frac{1}{\nu _{l}^{2}}\right) \cos ^{2}2t\right) , \\
&=&-\left\vert x^{2,0}\right\vert _{\nu ,l}^{2}\frac{\sin 2t}{\left\vert
\left( \cos 2t\right) \eta ^{2,0}\right\vert _{\nu ,l}^{5}}\left(
-4\left\vert x^{2,0}\right\vert _{\nu ,l}^{2}\cos ^{2}2t+\frac{2}{\nu
_{l}^{2}}+4\left( \frac{1}{\nu _{l}^{2}}\right) \cos ^{2}2t\right)
\end{eqnarray*}%
and%
\begin{eqnarray*}
\frac{\partial }{\partial \theta }\frac{\partial }{\partial t}\psi _{\nu ,l}
&=&\cos 2t\frac{\partial }{\partial \theta }\frac{\left\vert
x^{2,0}\right\vert _{\nu ,l}^{2}}{\left\vert \left( \cos 2t\right) \eta
^{2,0}\right\vert _{\nu ,l}^{3}} \\
&=&\cos 2t\frac{\left( \frac{\sin 4\theta }{l^{2}}\right) \left( \left\vert
\left( \cos 2t\right) \eta ^{2,0}\right\vert _{\nu ,l}^{2}\right) }{%
\left\vert \left( \cos 2t\right) \eta ^{2,0}\right\vert _{\nu ,l}^{5}}-\cos
2t\frac{\left\vert x^{2,0}\right\vert _{\nu ,l}^{2}\left( \frac{\partial }{%
\partial \theta }\left[ \left\vert \left( \cos 2t\right) \eta
^{2,0}\right\vert _{\nu ,l}^{2}\right] ^{3/2}\right) }{\left\vert \left(
\cos 2t\right) \eta ^{2,0}\right\vert _{\nu ,l}^{6}} \\
&=&\cos 2t\frac{\left( \frac{\sin 4\theta }{l^{2}}\right) \left( \left\vert
\left( \cos 2t\right) \eta ^{2,0}\right\vert _{\nu ,l}^{2}\right) }{%
\left\vert \left( \cos 2t\right) \eta ^{2,0}\right\vert _{\nu ,l}^{5}} \\
&&-\frac{3}{2}\cos 2t\frac{\left\vert x^{2,0}\right\vert _{\nu ,l}^{2}\left[
\left\vert \left( \cos 2t\right) \eta ^{2,0}\right\vert _{\nu ,l}^{2}\right]
^{1/2}\left( \frac{\partial }{\partial \theta }\left\vert \left( \cos
2t\right) \eta ^{2,0}\right\vert _{\nu ,l}^{2}\right) }{\left\vert \left(
\cos 2t\right) \eta ^{2,0}\right\vert _{\nu ,l}^{6}} \\
&=&\cos 2t\frac{\left( \frac{\sin 4\theta }{l^{2}}\right) \left( \left\vert
x^{2,0}\right\vert _{\nu ,l}^{2}+\left( \frac{1}{\nu _{l}^{2}}-\left\vert
x^{2,0}\right\vert _{\nu ,l}^{2}\right) \sin ^{2}2t\right) }{\left\vert
\left( \cos 2t\right) \eta ^{2,0}\right\vert _{\nu ,l}^{5}} \\
&&-\frac{3}{2}\cos 2t\frac{\left\vert x^{2,0}\right\vert _{\nu ,l}^{2}\frac{%
\sin 4\theta \cos ^{2}2t}{l^{2}}}{\left\vert \left( \cos 2t\right) \eta
^{2,0}\right\vert _{\nu ,l}^{5}} \\
&=&\cos 2t\frac{\left( \frac{\sin 4\theta }{l^{2}}\right) \left( \left\vert
x^{2,0}\right\vert _{\nu ,l}^{2}\cos ^{2}2t+\frac{1}{\nu _{l}^{2}}\sin
^{2}2t\right) }{\left\vert \left( \cos 2t\right) \eta ^{2,0}\right\vert
_{\nu ,l}^{5}} \\
&&-\frac{3}{2}\cos 2t\frac{\left\vert x^{2,0}\right\vert _{\nu ,l}^{2}\frac{%
\sin 4\theta \cos ^{2}2t}{l^{2}}}{\left\vert \left( \cos 2t\right) \eta
^{2,0}\right\vert _{\nu ,l}^{5}} \\
&=&\frac{\cos 2t\sin 4\theta }{l^{2}\left\vert \left( \cos 2t\right) \eta
^{2,0}\right\vert _{\nu ,l}^{5}}\left( \left\vert x^{2,0}\right\vert _{\nu
,l}^{2}\cos ^{2}2t+\frac{1}{\nu _{l}^{2}}\sin ^{2}2t-\frac{3}{2}\left\vert
x^{2,0}\right\vert _{\nu ,l}^{2}\cos ^{2}2t\right) \\
&=&\frac{\cos 2t\sin 4\theta }{l^{2}\left\vert \left( \cos 2t\right) \eta
^{2,0}\right\vert _{\nu ,l}^{5}}\left( -\frac{1}{2}\left\vert
x^{2,0}\right\vert _{\nu ,l}^{2}\cos ^{2}2t+\frac{1}{\nu _{l}^{2}}\sin
^{2}2t\right)
\end{eqnarray*}%
and%
\begin{eqnarray*}
\frac{\partial ^{2}}{\partial \theta ^{2}}\psi _{\nu ,l} &=&-\frac{\sin
2t\cos ^{2}2t}{4l^{2}}\frac{\partial }{\partial \theta }\frac{\sin 4\theta }{%
\left\vert \left( \cos 2t\right) \eta ^{2,0}\right\vert _{\nu ,l}^{3}}\text{ 
} \\
&=&-\frac{\sin 2t\cos ^{2}2t}{4l^{2}}\frac{4\cos 4\theta \left( \left\vert
\left( \cos 2t\right) \eta ^{2,0}\right\vert _{\nu ,l}^{2}\right) }{%
\left\vert \left( \cos 2t\right) \eta ^{2,0}\right\vert _{\nu ,l}^{5}}+\frac{%
\sin 2t\cos ^{2}2t}{4l^{2}}\frac{\sin 4\theta \left( \frac{\partial }{%
\partial \theta }\left[ \left\vert \left( \cos 2t\right) \eta
^{2,0}\right\vert _{\nu ,l}^{2}\right] ^{3/2}\right) }{\left\vert \left(
\cos 2t\right) \eta ^{2,0}\right\vert _{\nu ,l}^{6}} \\
&=&-\frac{\sin 2t\cos ^{2}2t}{4l^{2}}\frac{4\cos 4\theta \left( \left\vert
\left( \cos 2t\right) \eta ^{2,0}\right\vert _{\nu ,l}^{2}\right) }{%
\left\vert \left( \cos 2t\right) \eta ^{2,0}\right\vert _{\nu ,l}^{5}} \\
&&+\frac{3}{2}\frac{\sin 2t\cos ^{2}2t}{4l^{2}}\frac{\sin 4\theta \left[
\left\vert \left( \cos 2t\right) \eta ^{2,0}\right\vert _{\nu ,l}^{2}\right]
^{1/2}\left( \frac{\partial }{\partial \theta }\left\vert \left( \cos
2t\right) \eta ^{2,0}\right\vert _{\nu ,l}^{2}\right) }{\left\vert \left(
\cos 2t\right) \eta ^{2,0}\right\vert _{\nu ,l}^{6}}\text{ }
\end{eqnarray*}%
\begin{eqnarray*}
&=&-\frac{\sin 2t\cos ^{2}2t}{l^{2}}\frac{\cos 4\theta \left( \left\vert
x^{2,0}\right\vert _{\nu ,l}^{2}+\left( \frac{1}{\nu _{l}^{2}}-\left\vert
x^{2,0}\right\vert _{\nu ,l}^{2}\right) \sin ^{2}2t\right) }{\left\vert
\left( \cos 2t\right) \eta ^{2,0}\right\vert _{\nu ,l}^{5}} \\
&&+\frac{3}{2}\frac{\sin 2t\cos ^{2}2t}{4l^{2}}\frac{\sin 4\theta \left(
\cos ^{2}2t\frac{\sin 4\theta }{l^{2}}\right) }{\left\vert \left( \cos
2t\right) \eta ^{2,0}\right\vert _{\nu ,l}^{5}} \\
&=&-\frac{\sin 2t\cos ^{2}2t}{l^{2}}\frac{\cos 4\theta \left( \left\vert
x^{2,0}\right\vert _{\nu ,l}^{2}\cos ^{2}2t+\frac{1}{\nu _{l}^{2}}\sin
^{2}2t\right) }{\left\vert \left( \cos 2t\right) \eta ^{2,0}\right\vert
_{\nu ,l}^{5}} \\
&&+\frac{3}{2}\frac{\sin 2t\cos ^{4}2t}{4l^{4}}\frac{\sin ^{2}4\theta }{%
\left\vert \left( \cos 2t\right) \eta ^{2,0}\right\vert _{\nu ,l}^{5}}
\end{eqnarray*}
\end{proof}

\bigskip

\bigskip

\end{document}